\documentclass[reqno,12pt]{amsart}
\usepackage{amsfonts}
\usepackage{upgreek}
\usepackage{bbm}
\usepackage{} %leqno is the option to put formula numbers on the left side
\numberwithin{equation}{section}
\usepackage{indentfirst}
 \usepackage{color}
\usepackage{amssymb}
\usepackage{mathrsfs}
\usepackage{xy}
\usepackage{hyperref}
\hypersetup{colorlinks=true,linkcolor=blue}
\xyoption{all}
\def\Ext{\mbox{\rm Ext}} \def\Hom{\mbox{\rm Hom}} \def\dim{\mbox{\rm dim}} \def\Iso{\mbox{\rm Iso}\,}
\def\lr#1{\langle #1\rangle}    
\def\Ker{\mbox{\rm Ker}\,}   \def\im{\mbox{\rm Im}\,} \def\Coker{\mbox{\rm Coker}\,}
\def\End{\mbox{\rm End}\,}

\def\rad{\mbox{\rm rad}\,}
\def\Dim{\mbox{\rm \textbf{dim}}}\def\A{\mathcal{A}} 
\def\P{\mathcal{P}}\def\I{\mathcal{I}}\def\X{\mathbb X}

\def\x{{\bf x}}

\theoremstyle{plain} %text of this environment is typesetted in italics
\newtheorem{theorem}{\bf Theorem}[section]
\newtheorem{lemma}[theorem]{\bf Lemma}
\newtheorem{corollary}[theorem]{\bf Corollary}
\newtheorem{proposition}[theorem]{\bf Proposition}

\theoremstyle{definition} %text of this environment is typesetted in roman letters
\newtheorem{definition}[theorem]{\bf Definition}
\newtheorem{remark}[theorem]{\bf Remark}
\newtheorem{example}[theorem]{\bf Example}

\newcommand{\bt}{\begin{theorem}}
\newcommand{\et}{\end{theorem}}
\newcommand{\bl}{\begin{lemma}}
\newcommand{\el}{\end{lemma}}
\newcommand{\bd}{\begin{definition}}
\newcommand{\ed}{\end{definition}}
\newcommand{\bc}{\begin{corollary}}
\newcommand{\ec}{\end{corollary}}
\newcommand{\bp}{\begin{proof}}
\newcommand{\ep}{\end{proof}}
\newcommand{\bx}{\begin{example}}
\newcommand{\ex}{\end{example}}
\newcommand{\br}{\begin{remark}}
\newcommand{\er}{\end{remark}}
\newcommand{\be}{\begin{equation}}
\newcommand{\ee}{\end{equation}}
\newcommand{\ba}{\begin{align}}
\newcommand{\ea}{\end{align}}
\newcommand{\bn}{\begin{enumerate}}
\newcommand{\en}{\end{enumerate}}
\newcommand{\bcs}{\begin{cases}}
\newcommand{\ecs}{\end{cases}}

\makeatletter
\renewcommand{\section}{\@startsection{section}{1}{0mm}
  {-\baselineskip}{0.5\baselineskip}{\bf\leftline}}
\makeatother

\begin{document}
\title[Acyclic quantum cluster algebras via Hall algebras]{Acyclic quantum cluster algebras via Hall algebras\\ associated to cluster categories}
 %title of paper and the running head option

\author{Changjian Fu and Haicheng Zhang}
\address{Department of Mathematics\\ SiChuan University\\ Chengdu 610064, P.~R.~China}
\email{changjianfu@scu.edu.cn (C. Fu)}
\address{Ministry of Education Key Laboratory of NSLSCS, School of Mathematical Sciences, Nanjing Normal University, Nanjing 210023, P.R.China}
\email{zhanghc@njnu.edu.cn (H. Zhang)}

%%%%%%%%%%%%%%% footnote %%%%%%%%%%%%%%%%
%In case \subjclass[2010] command is not effective
%(or the version of amsart.cls is old), write as follows:
%\renewcommand{\thefootnote}{\fnsymbol{footnote}}
%\footnote[0]{2010\textit{ Mathematics Subject Classification}.
%Primary 00; Secondary 00.}
%
%\thanks{$*$~Corresponding author.}
%%%%%%%%%%%% Authors addresses %%%%%%%%%%%%%

\begin{abstract}
Let $Q$ be a finite acyclic valued quiver. We establish an alternative cluster multiplication formula in the quantum cluster algebra associated with $Q$, which exhibits certain symmetries similar to 2-Calabi--Yau properties of cluster categories. Motivated by this formula, we introduce a Hall algebra $\mathcal {D}\mathcal {H}_\Lambda^{{cl}}(\widetilde{\A})$ via a certain quotient of a derived Hall subalgebra associated with $Q$, serving as an algebra counterpart to cluster categories of hereditary algebras.

As our main result, we prove that the quantum cluster algebra associated with $Q$ is isomorphic to the exceptional subalgebra of $\mathcal{D}\mathcal{H}_\Lambda^{cl}(\widetilde{\mathcal{A}})$, namely, the subalgebra generated by exceptional objects. Furthermore, when $Q$ is of Dynkin type, the quantum cluster algebra is naturally isomorphic to $\mathcal{D}\mathcal{H}_\Lambda^{{cl}}(\widetilde{\mathcal{A}})$ itself. This establishes a Hall algebra realization for all acyclic quantum cluster algebras. As an application, we introduce the Hall subalgebra of $\mathcal{D}\mathcal{H}_\Lambda^{cl}(\widetilde{\mathcal{A}})$ generated by projectives, yielding a Hall algebra analogue of the lower bound quantum cluster algebras. Then we show that the standard monomials of quantum projective cluster variables form a basis of the quantum cluster algebras via Hall algebra approach.

Moreover, we show that the Auslander--Reiten translation of cluster categories induces an algebra automorphism of $\mathcal {D}\mathcal {H}_\Lambda^{{cl}}(\widetilde{\A})$.
Finally, we apply the BGP-reflections to $\mathcal {D}\mathcal {H}_\Lambda^{{cl}}(\widetilde{\A})$, and prove that the Hall algebra $\mathcal {D}\mathcal {H}_\Lambda^{{cl}}(\widetilde{\A})$ is invariant up to isomorphism under the mutations at sink vertices of $Q$. By considering the restrictions of these algebra isomorphisms, we obtain algebra automorphisms of the quantum cluster algebras.
\end{abstract}

\maketitle
\tableofcontents

%%%%%%%%%%%%%%%%%%%%%%%%%%%%%%%%%%%%%%%%%

\section{Introduction}
\subsection{Background}
Cluster algebras were invented by Fomin and Zelevinsky~\cite{FZ} with the aim of developing a general framework for the study of total positivity in algebraic groups and canonical bases in quantum groups. To provide an additive categorification of cluster algebras, cluster categories were introduced in \cite{BMRRT} as certain orbit categories of bounded derived categories of hereditary abelian categories. By a result of Keller \cite{Keller}, cluster categories are triangulated categories satisfying the $2$-Calabi--Yau property. Since their inception, cluster categories have driven extensive developments in tilting theory and found widespread applications in the theory of cluster algebras.

The explicit connections between cluster algebras and cluster categories are established via the Caldero--Chapoton map \cite{CC} and the Caldero--Keller multiplication formulas \cite{CK2005, CK2}. For any object $M$ in a cluster category, its image under the Caldero--Chapoton map, denoted by $X_M$, is called a cluster character. The cluster characters associated with indecomposable rigid objects realize the cluster variables. The multiplication formulas of cluster characters, abbreviated as cluster multiplication formulas, play a crucial role in constructing ``good'' bases for cluster algebras and in proving various positivity and denominator conjectures (cf. \cite{CK2005, CK2, DXChen, DXX, Yang}).

Let $Q$ be a Dynkin quiver and $\mathcal{C}_Q$ the associated cluster category. Caldero and Chapoton \cite{CC} established the fundamental cluster multiplication formula:
\begin{equation}\label{jiben}
	X_M X_N = X_{M\oplus N}
\end{equation}
for any objects $M, N \in \mathcal{C}_Q$. Subsequently, Caldero and Keller \cite{CK2005} proved the general cluster multiplication formula:
\begin{equation}\label{general}
	\chi (\mathbb{P} \operatorname{Ext}^1(M,N)) X_M X_N = \sum_{[E]} \Big( \chi (\mathbb{P} \operatorname{Ext}^1(M,N)_E) + \chi (\mathbb{P} \operatorname{Ext}^1(N,M)_E) \Big) X_E
\end{equation}
for any objects $M, N \in \mathcal{C}_Q$ with $\operatorname{Ext}^1(M,N) \neq 0$, where $\chi$ denotes the Euler--Poincar\'e characteristic of the \'etale cohomology with proper support. In particular, when $\operatorname{Ext}^1(M,N)$ is one-dimensional, the formula \eqref{general} simplifies to the categorification of the exchange relation:
\begin{equation}\label{one-dim}
	X_M X_N = X_E + X_{E'},
\end{equation}
where $E$ and $E'$ are the unique objects (up to isomorphism) fitting into non-split triangles $N \rightarrow E \rightarrow M \rightarrow N[1]$ and $M \rightarrow E' \rightarrow N \rightarrow M[1]$ in $\mathcal{C}_Q$. As observed in \cite{CK2005}, the multiplication formula \eqref{general} exhibits a subtle connection with the multiplication in dual Hall algebras. Inspired by this connection, Hubery \cite{Hubery1} (for affine type) and Xiao--Xu \cite{XX, Xu} (for general acyclic quivers) successfully generalized the formula \eqref{general} by employing Hall algebra techniques.

As a quantum analogue of cluster algebras, quantum cluster algebras were introduced by Berenstein and Zelevinsky \cite{BZ05}. Since quantum cluster algebras share essential combinatorial structures with classical cluster algebras, it is natural to seek a quantum analogue of the Caldero--Chapoton map and establish the corresponding multiplication formulas.
The quantum analogue of the Caldero--Chapoton map was introduced by Rupel \cite{Rupel1} for acyclic valued quivers, and independently by Qin \cite{Qin} for acyclic equally valued quivers. The quantum version of the exchange formula \eqref{one-dim} was first established by Rupel \cite{Rupel1} for indecomposable rigid objects associated with finite type and rank 2 valued quivers, and by Qin \cite{Qin} for acyclic quivers. This formula was subsequently extended to broader classes of objects by Ding and Xu \cite{DX}, and ultimately generalized to arbitrary acyclic valued quivers by Rupel \cite{Rupel2}. We remark that Hall algebras play a crucial role in their work, and the connection between quantum cluster algebras and Hall algebras has been further developed in various contexts (cf. \cite{BR, CDX,Con, DSC, DX, Fei, FPZ}).

The Hall algebras of finite-dimensional algebras over finite fields were introduced by Ringel \cite{R90, R90a}. He used the Hall algebra of a representation-finite hereditary algebra to give a realization of the positive part of the corresponding quantum group. To\"en \cite{Toen2006} defined Hall algebras for differential graded categories satisfying certain finiteness conditions, called derived Hall algebras, which were generalized by Xiao and Xu \cite{XiaoXu} to triangulated categories satisfying suitable homological finiteness conditions. We remark that the finiteness conditions required to define derived Hall algebras fail for periodic triangulated categories. Later, Xu and Chen \cite{XuChen} defined derived Hall algebras for odd periodic triangulated categories by modifying the construction in \cite{XiaoXu}. However, defining derived Hall algebras for even periodic triangulated categories remains an open problem.

As a quantum version of the cluster multiplication formula \eqref{jiben}, the Hall multiplication formula between the quantum cluster characters $X_M$ and $X_N$ has been established via the Hall algebra approach (cf. \cite{DX, Fei, BR, CDX, DSC, FPZ}). In particular, the Hall algebras of module categories of hereditary algebras and Green's formulas on Hall numbers have played a crucial role. In recent years, derived Hall algebras of hereditary algebras have also found significant applications in the study of quantum cluster multiplication formulas (cf. \cite{CDZ, FPZ}). Explicitly, Fu, Peng, and Zhang \cite{FPZ} investigated a certain subalgebra of the derived Hall algebra of a hereditary algebra and constructed an algebra homomorphism (see Theorem \ref{mainresult}) from this subalgebra to the quantum torus by establishing a bialgebra structure and an integration homomorphism on it. Using this homomorphism, they recovered numerous quantum cluster multiplication formulas via Hall algebra techniques (see Corollary \ref{Hallcfgs}). Subsequently, Chen, Ding, and Zhang \cite{CDZ} extended the approach of \cite{FPZ} to certain quotients of derived Hall subalgebras, obtaining high-dimensional ``mutation'' multiplication formulas between the quantum cluster characters $X_M$ and $X_{P[1]}$ (see Theorem \ref{dyggs}). Furthermore, they succeeded in establishing quantum versions of Caldero--Keller's general multiplication formula \eqref{general}.

As highlighted above, Hall algebras of hereditary algebras play a vital role in the theory of quantum cluster algebras. It has long been expected to provide a direct realization of (quantum) cluster algebras via a suitable Hall algebra associated with cluster categories (cf. \cite{CC, CK2005}). However, because cluster categories are 2-Calabi--Yau triangulated categories and many of them are even periodic, their derived Hall algebras are not well-defined in the standard framework. On the other hand, the Hall multiplication formulas for quantum cluster characters reveal that the operations in quantum cluster algebras closely mirror the multiplications in (derived) Hall algebras of module categories of hereditary algebras, rather than being purely governed by the triangulated structures of cluster categories. This motivates us to define a novel Hall algebra for cluster categories that preserves the structural features of original Hall multiplications while properly incorporating the symmetries of cluster categories.

\subsection{Goal} The main goal of this paper is to provide a Hall algebra realization for all acyclic quantum cluster algebras. To this end, we introduce a certain Hall algebra $\mathcal {D}\mathcal {H}_\Lambda^{{cl}}(\widetilde{\A})$ (see Definition \ref{Halldef}) associated to cluster categories, which is defined by quotients of a derived Hall subalgebra of the hereditary algebra associated to any finite acyclic valued quiver $\widetilde{Q}$. Then we show that the Hall algebra $\mathcal {D}\mathcal {H}_\Lambda^{{cl}}(\widetilde{\A})$ has some similar properties as cluster categories of hereditary algebras. Explicitly, we show that the Auslander--Reiten translation of cluster categories induces an algebra automorphism of $\mathcal {D}\mathcal {H}_\Lambda^{{cl}}(\widetilde{\A})$ (see Theorem \ref{zmainthm}). Moreover, applying the BGP-reflections to $\mathcal {D}\mathcal {H}_\Lambda^{{cl}}(\widetilde{\A})$, we prove that the Hall algebra $\mathcal {D}\mathcal {H}_\Lambda^{{cl}}(\widetilde{\A})$ is invariant up to isomorphism under the mutations at sink vertices of $\widetilde{Q}$ (see Theorem \ref{mutationbubian}).
Thus, the Hall algebra $\mathcal {D}\mathcal {H}_\Lambda^{{cl}}(\widetilde{\A})$ can be viewed as an algebra counterpart to cluster categories of hereditary algebras.

In fact, we provide an alternative quantum cluster multiplication formula (see Theorem \ref{zhihe}) arising from the high-dimensional quantum cluster multiplication formula \eqref{ccfgs}, which exhibits certain symmetries similar to 2-Calabi--Yau properties of cluster categories, and is also
the motivation to define the quotient algebra $\mathcal {D}\mathcal {H}_\Lambda^{{cl}}(\widetilde{\A})$.
On the other hand, this alternative quantum cluster multiplication formula can also be viewed another quantum analogue of the cluster multiplication formula \eqref{jiben}, which keeps the right side of \eqref{jiben} unchanged and quantizes the left side, while the quantization via the Hall multiplication formula \eqref{qre3} keeps the left side of \eqref{jiben} unchanged and quantizes the right side.
For Dynkin quivers, we prove that $\mathcal {D}\mathcal {H}_\Lambda^{{cl}}(\widetilde{\A})$ is isomorphic to the corresponding quantum cluster algebra (see Theorem \ref{thm:hall-quantum-dynkin}). For acyclic quivers, we prove that the quantum cluster algebra is isomorphic to a subalgebra of $\mathcal {D}\mathcal {H}_\Lambda^{{cl}}(\widetilde{\A})$, which is generated by the exceptional objects (see Theorem \ref{chtgthm} and Corollary \ref{congshixian}). Hence, our Hall algebra $\mathcal {D}\mathcal {H}_\Lambda^{{cl}}(\widetilde{\A})$ can be truly viewed as the so-called ``exceptional Hall algebra", which is the name of the Hall-like algebra defined by Caldero--Chapoton \cite{CC} and Caldero--Keller \cite{CK2005} for realizing cluster algebras of finite types.
As an application, we introduce the Hall subalgebra of $\mathcal{D}\mathcal{H}_\Lambda^{cl}(\widetilde{\mathcal{A}})$ generated by projectives, which can be viewed as a Hall algebra analogue of the lower bound quantum cluster algebras. Then using the Hall algebra framework, we show that the standard monomials of quantum projective cluster variables form a basis of the quantum cluster algebras (see Theorem \ref{pcabasis} and Corollary \ref{ttll51}).

\subsection{Organization} The paper is organized as follows: we recall some preliminaries on derived Hall algebras of hereditary algebras and quantum
cluster algebras in Section 2. In Section 3, we give an alternative quantum cluster multiplication formula and apply it to a special compatible pair. We define the Hall algebra $\mathcal {D}\mathcal {H}_\Lambda^{{cl}}(\widetilde{\A})$
by using quotients of the derived Hall subalgebra associated to any finite acyclic valued quiver $\widetilde{Q}$, provide a spanning set for $\mathcal {D}\mathcal {H}_\Lambda^{{cl}}(\widetilde{\A})$, and use the Hall algebra $\mathcal {D}\mathcal {H}_\Lambda^{{cl}}(\widetilde{\A})$ to give a Hall algebra realization for all acyclic quantum cluster algebras in Section 4. We use the Hall subalgebra of $\mathcal{D}\mathcal{H}_\Lambda^{cl}(\widetilde{\mathcal{A}})$ generated by projectives to show that the standard monomials of quantum projective cluster variables form a basis of the quantum cluster algebras in Section 5.
We show that the Auslander--Reiten translation of cluster categories induces an algebra automorphism $\sigma$ of $\mathcal {D}\mathcal {H}_\Lambda^{{cl}}(\widetilde{\A})$ in Section 6, and leave the proof of that $\sigma$ is an isomorphism in Section Appendix A, where we give the inverse map of $\sigma$. Section 7 is devoted to proving that the Hall algebra $\mathcal {D}\mathcal {H}_\Lambda^{{cl}}(\widetilde{\A})$ is invariant up to isomorphism under the mutations at sink vertices of $\widetilde{Q}$.

\subsection{Convention and Notation} Let us fix some notations used throughout the paper. For a finite set $S$, we denote by $|S|$ its cardinality. Let $\mathbb{F}_q$ be a finite field with $q$ elements, and set $v=\sqrt{q}$. Let $Q$ be a finite acyclic valued quiver and $\A$ the category of finite-dimensional representations of $Q$ over $\mathbb{F}_q$, denote by $\P=\mathcal {P}_{\A}$ and $\I=\mathcal {I}_{\A}$ the subcategories of $\A$ consisting of all projective objects and all injective objects, respectively. Denote by $\mathcal{D}^b(\A)$ the bounded derived category of $\A$. For any representation $X\in\A$, we use the corresponding lowercase boldface letter ${\bf x}$ to denote its dimension vector $\Dim X$, and denote by $nX$ the direct sum of $n$ copies of $X$ for any positive integer $n$. For an abelian or a triangulated category $\mathcal {E}$, the Grothendieck group of $\mathcal {E}$ and the set of isomorphism classes $[X]$ of objects in $\mathcal {E}$ are denoted by $K(\mathcal {E})$ and $\Iso(\mathcal {E})$, respectively; for each object $M$ in $\mathcal {E}$, we denote by $\hat{M}$ the image of $M$ in $K(\mathcal {E})$, and denote by ${\rm Aut}_{\mathcal {E}} (M)$ the automorphism group of $M$.
We always assume that all the vectors are column vectors. Denote by $\mathbb{N}$ the set of non-negative integers, and $\mathbb{N}^+$ the set of positive integers.

\section*{Use of AI Tools Declaration}
The authors declare that they have not used Artificial Intelligence (AI) tools in the creation of this article.

\section*{Acknowledgments}
This work was partially supported by the National Natural Science Foundation of China (Nos. 12271257, 12571040) and Natural Science Foundation of Jiangsu Province of China (No. BK20240137).

\section{Preliminaries}
In this section, we recall some preliminaries on derived Hall algebras and quantum cluster algebras.
\subsection{Derived Hall algebras}
For any objects $M,N, X\in \mathcal{D}^b(\A)$, set $$\{M,N\}:=\prod\limits_{i>0}|\Hom_{\mathcal{D}^b(\A)}(M[i],N)|^{(-1)^i}$$ and denote by $\Hom_{\mathcal{D}^b(\A)}(M,N)_X$ the subset of $\Hom_{\mathcal{D}^b(\A)}(M,N)$ consisting of the morphisms $f: M\to N$ whose cone is isomorphic to $X$. According to \cite{Toen2006,XiaoXu}, for any objects $X,Y,L\in \mathcal{D}^b(\A)$, we have
$$\frac{|\Hom_{\mathcal{D}^b(\A)}(L,X)_{Y[1]}|}{|{\rm Aut}_{\mathcal{D}^b(\A)}(X)|}\cdot\frac{\{L,X\}}{\{X,X\}}=
\frac{|\Hom_{\mathcal{D}^b(\A)}(Y,L)_{X}|}{|{\rm Aut}_{\mathcal{D}^b(\A)}(Y)|}\cdot\frac{\{Y,L\}}{\{Y,Y\}}=:F_{X,Y}^L.$$
The derived Riedtmann--Peng formula (cf. \cite{XiaoXu2,WWZ}) states
\begin{equation*}
F_{X,Y}^L=\frac{|\Ext^1_{\mathcal{D}^b(\A)}(X,Y)_{L}|}{|\Hom_{\mathcal{D}^b(\A)}(X,Y)|}\cdot\frac{1}{\{X,Y\}}\cdot\frac{|{\rm Aut}_{\mathcal{D}^b(\A)}(L)|}{|{\rm Aut}_{\mathcal{D}^b(\A)}(X)||{\rm Aut}_{\mathcal{D}^b(\A)}(Y)|}\cdot\frac{\{L,L\}}{\{X,X\}\{Y,Y\}},
\end{equation*}
where $\Ext^1_{\mathcal{D}^b(\A)}(X,Y)_{L}:=\Hom_{\mathcal{D}^b(\A)}(X,Y[1])_{L[1]}$. In particular, if $X,Y,L\in\A$, we have
$$F_{XY}^{L}=\frac{|\Ext^1_{\A}(X,Y)_{L}|}{|\Hom_{\A}(X,Y)|}\cdot \frac{|{\rm Aut}_{\A}(L)|}{|{\rm Aut}_{\A}(X)||{\rm Aut}_{\A}(Y)|},$$
which is just the Hall number used in defining Hall algebras of abelian categories.

The (Drinfeld dual) {\em derived Hall algebra} $\mathcal {D}\mathcal {H}(\A)$ of $\mathcal{D}^b(\A)$ is the $\mathbb{Q}(v)$-linear space with the basis $\{u_{X}~|~X\in \Iso(\mathcal{D}^b(\A))\}$ and the multiplication defined by
$$u_{X}\diamond  u_{Y}=\sum\limits_{L\in {\rm Iso}(\mathcal{D}^b(\A))}H_{X,Y}^Lu_{L},$$ where
$$H_{X,Y}^L=\frac{|\Ext^1_{\mathcal{D}^b(\A)}(X,Y)_{L}|}{|\Hom_{\mathcal{D}^b(\A)}(X,Y)|}\cdot\frac{1}{\{X,Y\}}.$$

For any objects $M,N \in \mathcal{A}$, set $$\lr{M,N}:=\dim_k\Hom_{\A}(M,N)-\dim_k\Ext^1_{\A}(M,N),$$
and it descends to give a bilinear form
$$\lr{\cdot ,\cdot }: K(\mathcal{A})\times K(\mathcal{A})\longrightarrow \mathbb{Z},$$ known as the \emph{Euler form}. The \emph{symmetric Euler form} $$(\cdot ,\cdot ): K(\mathcal{A})\times K(\mathcal{A})\longrightarrow \mathbb{Z}$$
is defined by $(\alpha,\beta):=\lr{\alpha,\beta}+\lr{\beta,\alpha}$ for any $\alpha,\beta\in K(\mathcal{A})$.
For any objects ${X},{Y}\in \mathcal{D}^b(\A)$, define
\begin{equation*}
\lr{{X},{Y}}:=\sum\limits_{i\in\mathbb{Z}}(-1)^i\dim_k\Hom_{\mathcal{D}^b(\A)}({X},{Y}[i]),
\end{equation*}
and it descends to give a bilinear form on the Grothendieck group of $\mathcal{D}^b(\A)$. Moreover, this bilinear form coincides with the Euler form of $K(\A)$ over the objects in $\A$. In particular, for any $M,N\in\A$ and $i,j\in\mathbb{Z}$, we have that $\lr{M[i],N[j]}=(-1)^{i-j}\lr{M,N}$.

The \emph{twisted derived Hall algebra} $\mathcal {D}\mathcal {H}_{q}(\A)$ is the same space as $\mathcal {D}\mathcal {H}(\A)$, but with the twisted multiplication defined by
\begin{equation}u_{X}\ast u_{Y}=q^{\lr{{X},{Y}}} u_{X}\diamond u_{Y}\end{equation}
for any ${X},{Y}\in \Iso(\mathcal{D}^b(\A))$. It is easy to see that \begin{equation*}u_{M[i]}\ast u_{N[i+1]}=u_{M[i]\oplus N[i+1]}\end{equation*}
for any $M,N\in\A$ and $i\in\mathbb{Z}$.

Let $\overline{\tau}$ be the Auslander--Reiten translation of $\mathcal{D}^b(\A)$, which is a triangulated automorphism of $\mathcal{D}^b(\A)$. By the definition of derived Hall algebras, it is easy to see that $\overline{\tau}$ gives an algebra automorphism of $\mathcal {D}\mathcal {H}(\A)$. By \cite[Lemma 6.4]{CDZ}, the Euler form is $\overline{\tau}$-invariant, i.e. $\lr{X,Y}=\lr{\overline{\tau}X,\overline{\tau}Y}$ for any $X,Y\in \mathcal{D}^b(\A)$. Hence, $\overline{\tau}$ also gives an algebra automorphism of $\mathcal {D}\mathcal {H}_q(\A)$.
In what follows, for any objects $M,N,X,Y\in\A$, set
$$_X\Hom_{\A}(M,N)_Y:=\{f: M\rightarrow N~|~\Ker f\cong X~\text{and}~\Coker f\cong Y\}.$$
\begin{proposition}{\rm(\cite{Toen2006})}\label{twistderived}
The twisted derived Hall algebra $\mathcal {D}\mathcal {H}_q(\A)$ is an associative unital algebra generated by the elements in $\{u_{M[i]}~|~M\in{\rm Iso}(\A),~i\in \mathbb{Z}\}$ and the following relations for any $M, N\in{\rm Iso}(\A)$
\begin{flalign}
&u_{M[i]}\ast u_{N[i]}=q^{\lr{M,N}}\sum\limits_{[L]}{\frac{{|\Ext_\mathcal{A}^1{{(M,N)}_L}|}}{{|\Hom_\mathcal{A}(M,N)|}}}u_{L[i]},~i\in\mathbb{Z};\\
&u_{M[i+1]}\ast u_{N[i]}=q^{-\lr{M,N}}\sum\limits_{[X],[Y]}|{}_X\Hom_{\A}(M,N)_Y| u_{Y[i]}\ast u_{X[i+1]},~i\in\mathbb{Z};\label{ydchgx}\\
&u_{M[i]}\ast u_{N[j]}=q^{(-1)^{i-j}\lr{M,N}} u_{N[j]}\ast u_{M[i]},~i-j>1.
\end{flalign}
\end{proposition}

Let $C_{\A}^{e}$ be the subcategory of $\mathcal{D}^b(\A)$ consisting of all objects $I[-1]\oplus M\oplus P[1]$ with $I\in\I,M\in\A$ and $P\in\P$. Since $\A$ is hereditary, it is easy to see $C_{\A}^{e}$ is closed under extensions in $\mathcal{D}^b(\A)$. Hence,
the subspace $\mathcal {D}\mathcal {H}_q^{ec}(\A)$ of $\mathcal {D}\mathcal {H}_q(\A)$ spanned by all elements $u_{X}$ with $X\in \Iso({C_{\A}^{e}})$ is a subalgebra of $\mathcal {D}\mathcal {H}_q(\A)$.

In the following, for any object $X$ in ${\A} $ or $\mathcal{D}^b({\A})$, we denote by $a_X$ the cardinality of its automorphism group. For the need of subsequent calculations, we give the following.
\begin{proposition}\label{hallshujs}
For any $M,N,M_1,N_1,M_2,N_2\in\A$, the following equations hold.
\begin{itemize}
\item[(1)]
\begin{flalign*}
F_{M_1\oplus N_1[-1],M_2\oplus N_2[-1]}^{M\oplus N[-1]}=\sum\limits_{[X],[Y],[L]}q^{-\lr{X,Y}}\frac{a_Xa_Ya_L}{a_{M_1}a_{N_2}}F_{Y,M_2}^MF_{N_1,X}^NF_{L,Y}^{M_1}F_{X,L}^{N_2}.
\end{flalign*}
\item[(2)]
\begin{flalign*}
&H_{M_1\oplus N_1[-1],M_2\oplus N_2[-1]}^{M\oplus N[-1]}=\\&\sum\limits_{[X],[Y],[L]}q^{\lr{\hat{N}_1,\hat{M}-\hat{M}_1-\hat{M}_2}+\lr{\hat{N}-\hat{N}_2,\hat{M}_2}}\frac{a_Xa_Ya_La_{M_2}a_{N_1}}{a_{M}a_{N}}F_{Y,M_2}^MF_{N_1,X}^NF_{L,Y}^{M_1}F_{X,L}^{N_2}.
\end{flalign*}
\end{itemize}
\end{proposition}
\begin{proof}
$(1)$~By \cite[Proposition 2.6]{Zhang}, we have
$$F_{M_1\oplus N_1[-1],M_2\oplus N_2[-1]}^{M\oplus N[-1]}=\sum\limits_{[X],[Y]}F_{M_1,N_2[-1]}^{X[-1]\oplus Y}F_{Y,M_2}^MF_{N_1,X}^N.$$
Then using \cite[Lemma 3.4]{Zhang19}, we finish the proof of $(1)$.

$(2)$~Since for any $X,Y\in\A$, $$a_{X\oplus Y[-1]}=a_{X}a_{Y}|\Ext_{\A}^1(Y,X)|~\text{and}~\{X\oplus Y[-1],X\oplus Y[-1]\}=\frac{1}{|{\rm Hom}_{\A}(Y,X)|},$$ we obtain
$$H_{M_1\oplus N_1[-1],M_2\oplus N_2[-1]}^{M\oplus N[-1]}=q^{\lr{N,M}-\lr{N_1,M_1}-\lr{N_2,M_2}}\frac{a_{M_1}a_{N_1}a_{M_2}a_{N_2}}{a_Ma_N}F_{M_1\oplus N_1[-1],M_2\oplus N_2[-1]}^{M\oplus N[-1]}.$$
Hence, using $(1)$, we finish the proof of $(2)$.
\end{proof}

\begin{proposition}\label{prop:Hall-equality-P-I}
\begin{itemize}
\item[(1)] For any $M,N,Z,Z'\in\A$ and $I\in\I$, we have
\begin{flalign*}
|{}_Z\Hom_{\A}(M,N\oplus I)_{Z'}|=\sum\limits_{\substack{[X],[Y],[I']\\Y\oplus I'=Z'}}q^{\lr{\hat{M}-\hat{X},\hat{I}}}|{}_X\Hom_{\A}(M,N)_{Y}|\cdot|{}_Z\Hom_{\A}(X,I)_{I'}|.
\end{flalign*}
\item[(2)] For any $M,N,Z,Z'\in\A$ and $P\in\P$, we have
\begin{flalign*}
|{}_{Z'}\Hom_{\A}(M\oplus P,N)_{Z}|=\sum\limits_{\substack{[X],[Y],[P']\\X\oplus P'=Z'}}q^{\lr{\hat{P},\hat{M}-\hat{X}}}|{}_X\Hom_{\A}(M,N)_{Y}|\cdot|{}_{P'}\Hom_{\A}(P,Y)_{Z}|.
\end{flalign*}
\end{itemize}
\end{proposition}
\begin{proof}
$(1)$~Let us calculate $(u_{M[1]}\ast u_N)\ast u_I$ in $\mathcal {D}\mathcal {H}_q(\A):$
by the relation \eqref{ydchgx}, we have
\begin{flalign*}
&(u_{M[1]}\ast u_N)\ast u_I=\sum\limits_{[X],[Y]}q^{-\lr{\hat{M},\hat{N}}}|{}_X\Hom_{\A}(M,N)_Y| u_{Y}\ast u_{X[1]}\ast u_I\\&=\sum\limits_{[X],[Y],[Z],[I']}q^{-\lr{\hat{M},\hat{N}}-\lr{\hat{X},\hat{I}}}|{}_X\Hom_{\A}(M,N)_Y|\cdot|{}_Z\Hom_{\A}(X,I)_{I'}| u_{Y}\ast u_{I'}\ast u_{Z[1]}\\&=\sum\limits_{[X],[Y],[Z],[I']}q^{-\lr{\hat{M},\hat{N}}-\lr{\hat{X},\hat{I}}}|{}_X\Hom_{\A}(M,N)_Y|\cdot|{}_Z\Hom_{\A}(X,I)_{I'}| u_{Y\oplus I'\oplus Z[1]}.
 \end{flalign*}

On the other hand,
\begin{flalign*}
u_{M[1]}\ast (u_N\ast u_I)&=u_{M[1]}\ast u_{N\oplus I}
=\sum\limits_{[Z],[Z']}q^{-\lr{\hat{M},\hat{N}+\hat{I}}}|{}_Z\Hom_{\A}(M,N\oplus I)_{Z'}| u_{Z'\oplus Z[1]}.
\end{flalign*}
By the associativity of derived Hall algebras, for each fixed $Z,Z'\in\A$, comparing coefficients, we get the desired formula.

$(2)$~Using the associativity equation $(u_P\ast u_M)\ast u_{N[-1]}=u_P\ast(u_M\ast u_{N[-1]})$, we can similarly get the proof.
\end{proof}

\subsection{Quantum cluster algebras}\label{ss:quantum-cluster-algebras}
In this subsection, we recall some preliminaries on quantum cluster algebras from \cite{BZ05}. Fix two positive integers $n\leq m$.

Let $\Lambda$ be an $m\times m$ skew-symmetric integral matrix, and denote by $\{e_1,\cdots,e_m\}$ the standard basis of $\mathbb{Z}^{m}$. Let $\mathfrak{q}$ be an indeterminate.
Define the {\em quantum torus} to be the $\mathbb{Q}(\mathfrak{q}^{\frac{1}{2}})$-algebra $\mathcal{T}_{\mathfrak{q},\Lambda}$ with
a distinguished basis $\{X^{\alpha}: {\alpha}\in \mathbb{Z}^{m}\}$ and the
multiplication given by
\begin{equation}\label{ljhtorus}X^{\alpha} X^{\beta}=\mathfrak{q}^{\frac{1}{2}\Lambda(\alpha,\beta)}X^{\alpha+\beta},\end{equation}
where we still denote by $\Lambda$ the skew-symmetric bilinear form on $\mathbb{Z}^{m}$ associated to the skew-symmetric matrix $\Lambda$. It is well-known that $\mathcal{T}_{\mathfrak{q}}$ is an Ore domain, and thus is contained in its
skew-field of fractions $\mathcal{F}_{\mathfrak{q}}$. In what follows, we denote by $\mathcal{T}_{\Lambda}$ the quantum torus specialised at $\mathfrak{q}=q$.

Let $\widetilde{B}=(b_{ij})$ be an $m\times n$ integral matrix.  The pair
$(\Lambda, \widetilde{B})$ is called a {\em compatible pair} if $\Lambda\widetilde{B}=-{D_n\choose 0}$ for some diagonal matrix
$D_n=\operatorname{diag}\{d_1,\cdots, d_n\}$, where each $d_i\in\mathbb{N^+}$. An {\em initial  quantum seed} for $\mathcal{F}_{\mathfrak{q}}$ is a triple
$(\Lambda, \widetilde{B}, X)$  consisting of a compatible pair $(\Lambda,\widetilde{B})$ and the set  $X=\{x_1,\cdots,x_m\}$, where each $x_i$ denotes $X^{{ e}_i}$. For any $1\leq k\leq n$, one defines the {\em mutation} of $\Sigma:=(\Lambda, \widetilde{B}, X)$ in direction $k$ to obtain the new quantum seed $\mu_k(\Sigma):=(\mu_k(\Lambda),\mu_k(\widetilde{B}),\mu_k(X))$ as follows:

(1)\ $\mu_k(\Lambda)=G^{tr}\Lambda G$, where the
$m\times m$ matrix $G=(g_{ij})$ is given by
\[g_{ij}=\begin{cases}
\delta_{ij} & \text{if $j\ne k$;}\\
-1 & \text{if $i=j=k$;}\\
[b_{ik}]_{+} & \text{if $i\ne j = k$.}
\end{cases}
\]

(2)\ $\mu_k(\widetilde{B})=(b'_{ij})$ is given by
\[b'_{ij}=\begin{cases}
-b_{ij} & \text{if $i=k$ or $j=k$;}\\
b_{ij}+\frac{|b_{ik}|b_{kj}+b_{ik}|b_{kj}|}{2} & \text{otherwise.}
\end{cases}
\]

(3)\  $\mu_k(X)=\{x'_1,\cdots,x'_m\}$ is given by
\begin{align}
x_k'&=X^{\sum_{1\leq i\leq m}[b_{ik}]_{+} {e}_i -{e}_k}+X^{\sum_{1\leq i\leq m}[-b_{ik}]_{+} {e}_i -{e}_k},\nonumber\\
x_i'&=x_i ,\quad 1\leq i\neq k\leq m,\nonumber
\end{align}
where $[a]_{+}:=max\{0,a\}$ for any integer $a$.

For any two quantum seeds, they are called
{\em mutation-equivalent}, if they can be obtained from each other
by a sequence of mutations. Given an initial quantum seed $(\Lambda, \widetilde{B}, X)$, the variables from all quantum seeds which are mutation-equivalent to $(\Lambda, \widetilde{B}, X)$ are called the {\em quantum cluster
variables}. The {\em quantum cluster algebra} $\mathcal{A}_{\mathfrak{q}}(\Lambda,\widetilde{B})$ is defined to be the
$\mathbb{Q}(\mathfrak{q}^{\frac{1}{2}})$-subalgebra of $\mathcal{F}_{\mathfrak{q}}$ generated by all
quantum cluster variables.

In what follows, we always consider the quantum cluster algebras associated to finite acyclic valued quivers.
Let $Q$ be an acyclic valued quiver (cf. \cite{Rupel1,Rupel2}) with the vertex set $\{1,2,\cdots,n\}$. For each vertex $i$ of $Q$, let $d_i\in\mathbb{N^+}$ be the corresponding valuation.
Define a new acyclic quiver $\widetilde{Q}$ by attaching additional vertices $n+1,\ldots,m$ to $Q$ with the valuations $d_{n+1},\cdots, d_m$, respectively.
Let $\mathcal{A}$ and $\widetilde{\mathcal{A}}$ be always the categories of finite-dimensional representations of $Q$ and $\widetilde{Q}$ over $\mathbb{F}_q$, respectively. We may identify $\mathcal{A}$ with the full subcategory of $\widetilde{\mathcal{A}}$ consisting of representations with supports on $Q$. For each $1\leq i\leq m$, let $S_i$ be the simple representation of $\widetilde{Q}$ corresponding to the vertex $i$, and denote by $\mathcal{D}_i=\End(S_i)$ its endomorphism algebra.

Let $R(\widetilde{Q})$ and $R'(\widetilde{Q})$ be the $m\times m$ matrices with the $i$-th row and $j$-th column elements given respectively by
$$r_{ij}=\dim_{\mathcal {D}_i}\Ext_{\widetilde{\mathcal{A}}}^1(S_j,S_i)$$
and
$$r'_{ij}=\dim_{{\mathcal {D}_i}^{op}}\Ext_{\widetilde{\mathcal{A}}}^1(S_i,S_j).$$
Define $B(\widetilde{Q})=R'(\widetilde{Q})-R(\widetilde{Q})$,
$E(\widetilde{Q})=I_{m}-R'(\widetilde{Q})$ and $E'(\widetilde{Q})=I_{m}-R(\widetilde{Q})$, where $I_{m}$ is the $m\times m$ identity matrix.
We write $^\ast\x=E(\widetilde{Q})\x$ and $\x^\ast=E'(\widetilde{Q})\x$ for any $X\in\widetilde{\mathcal{A}}$.

By expanding appropriately the quiver $Q$ to $\widetilde{Q}$, we always assume that there exists a skew-symmetric $m\times m$ integral matrix $\Lambda$ such that \begin{equation*}\Lambda B(\widetilde{Q})=-\operatorname{diag}\{d_1,\cdots, d_m\}.\end{equation*}
In this case, $\Lambda\widetilde{B}=-{D_n\choose 0}$, where $\widetilde{B}$ is the left $m\times n$ submatrices of $B(\widetilde{Q})$ and $D_n=\operatorname{diag}\{d_1,\cdots, d_n\}$. Thus, we have the quantum cluster algebras ${\A}_{\mathfrak{q}}(\Lambda, B(\widetilde{Q}))$ and ${\A}_{\mathfrak{q}}(\Lambda, \widetilde{B})$.

The following equations are frequently used throughout the paper.
\begin{lemma}{\rm(\cite{FPZ})}\label{sjishu} For any $\alpha,\beta\in\mathbb{Z}^m$, the following equations hold:
\begin{itemize}
\item[$(1)$]$\Lambda(B(\widetilde{Q})\alpha,E(\widetilde{Q})\beta)=\lr{\alpha,\beta}$;~~$(2)$~~$\Lambda(B(\widetilde{Q})\alpha,E'(\widetilde{Q})\beta)=\lr{\beta,\alpha}$;
\item[$(3)$]$\Lambda(B(\widetilde{Q})\alpha,B(\widetilde{Q})\beta)=\lr{\beta,\alpha}-\lr{\alpha,\beta}$;~~$(4)$~~$\Lambda(E'(\widetilde{Q})\alpha,E'(\widetilde{Q})\beta)=\Lambda(E(\widetilde{Q})\alpha,E(\widetilde{Q})\beta)$.
\end{itemize}\end{lemma}

Following \cite[Proposition 7.1]{BZ05}, for each $1\leq j\leq n$, we denote by $\mathbf{b}^j$ the $j$-th column vector of $\widetilde{B}$, set
\begin{equation}\label{thjih}
e_j':=-e_j+\sum_{i=1}^m[b_{ij}]_+e_i~\text{and}~e_j^{''}:=-e_j+\sum_{i=1}^m[-b_{ij}]_+e_i.\end{equation}
Then it is easy to see that $e_j'=-{^\ast e}_j$ and $e_j^{''}=-e_j^\ast$.
\begin{lemma}\label{cor:exponent-relation}
For each $1\leq j,k\leq n$, the following equations hold:

$(1)$ $\Lambda(\mathbf{b}^j,\mathbf{b}^k)=\lr{e_k,e_j}-\lr{e_j,e_k}$;
$(2)$ $\Lambda(e_j'-\mathbf{b}^j,e_k^{''}+\mathbf{b}^k)=\Lambda(e_j^\ast, e_k^\ast)+\lr{e_j,e_k}$.
\end{lemma}
\begin{proof}
Noting that $\mathbf{b}^j=e_j^\ast-{^\ast e}_j$ and using \cite[Lemma 2.3(1)]{Rupel2}, we get the proof of $(1)$.
Then using Lemma \ref{sjishu} and \cite[Lemma 2.3(1)]{Rupel2}, we get the proof of $(2)$.
\end{proof}

The following result is essentially due to \cite{BZ05}.
\begin{proposition}\label{prop:quantum-cluster-presentation}
The quantum cluster algebra $\mathcal{A}_{\mathfrak{q}}(\Lambda,\widetilde{B})$ is isomorphic to the $\mathbb{Q}(\mathfrak{q}^{\frac{1}{2}})$-algebra $A_{\mathfrak{q}}(\Lambda,\widetilde{B})$ generated by the formal variables $Z_1,\ldots, Z_m,Y_1,\ldots, Y_n$ together with the following relations\,$:$
\begin{itemize}
    \item the quasi-commutative relations\,$:$
    \begin{equation}\label{gen-rel: quasi-comm-1}
        Z_iZ_j=\mathfrak{q}^{\Lambda(e_i,e_j)}Z_jZ_i~ \text{for~any~}1\leq i,j\leq m;
    \end{equation}
    \begin{equation}\label{gen-rel: quasi-comm-2}
        Z_iY_j=\mathfrak{q}^{\mu_j(\Lambda)(e_i,e_j)}Y_jZ_i ~ \text{for~any~}i\neq j~\text{with}~1\leq i\leq m, 1\leq j\leq n,
    \end{equation}
    \item the exchange relations for $1\leq k\leq n$\,$:$
    \begin{equation}
    \begin{split}\label{gen-rel: exchange-relation-left}
        Z_kY_k=\,&\mathfrak{q}^{\frac{1}{2}\Lambda(e_k,\sum_{1\leq i\leq m}[b_{ik}]_+e_i)}Z^{\sum_{1\leq i\leq m}[b_{ik}]_+e_i}+\\
&\mathfrak{q}^{\frac{1}{2}\Lambda(e_k,\sum_{1\leq i\leq m}[-b_{ik}]_+e_i)}Z^{\sum_{1\leq i\leq m}[-b_{ik}]_+e_i};
    \end{split}
    \end{equation}
    \begin{equation}
        \begin{split}\label{gen-rel: exchange-relation-right}
         Y_kZ_k=\,&\mathfrak{q}^{-\frac{1}{2}\Lambda(e_k,\sum_{1\leq i\leq m}[b_{ik}]_+e_i)}Z^{\sum_{1\leq i\leq m}[b_{ik}]_+e_i}+\\
&\mathfrak{q}^{-\frac{1}{2}\Lambda(e_k,\sum_{1\leq i\leq m}[-b_{ik}]_+e_i)}Z^{\sum_{1\leq i\leq m}[-b_{ik}]_+e_i},
    \end{split}
    \end{equation}
    \item the quasi-commutator relations for $1\leq j\neq k\leq n$ and $b_{jk}\leq 0$\,$:$
    \begin{equation}\label{gen-rel: quasi-commutation}
    \begin{split}
        \mathfrak{q}^{-\frac{1}{2}\Lambda(e_j^\ast,e_k^\ast)}Y_jY_k-\mathfrak{q}^{\frac{1}{2}\Lambda( e_j^\ast,e_k^\ast)}Y_kY_j
        =(\mathfrak{q}^{-\frac{1}{2}\lr{e_k,e_j}}-\mathfrak{q}^{\frac{1}{2}\lr{e_k,e_j}})Z^{-{^\ast} e_j-e_k^\ast},
    \end{split}
    \end{equation}
\end{itemize}
where $Z^\alpha:=\mathfrak{q}^{\frac{1}{2}\sum_{i<j}a_ia_j\lambda_{ji}}Z_{1}^{a_1}\cdots Z_{m}^{a_m}$ for any $\alpha=\sum\limits_{i=1}^ma_ie_i\in \mathbb{N}^m$. Note that, if $b_{jk}=0$, then the right hand side of \eqref{gen-rel: quasi-commutation} is zero, and if $b_{jk}<0$, then $-^\ast e_j-e_k^\ast\in \mathbb{N}^m$ by \eqref{thjih}.
\end{proposition}
\begin{proof}
For each $1\leq k\leq n$, denote by $x_k'$ the new quantum cluster variable obtained from $(\Lambda, \widetilde{B},X)$ by the mutation in direction $k$. According to the definition of quantum cluster algebras together with \cite[Proposition 7.1]{BZ05} and Lemma \ref{cor:exponent-relation}, we conclude that the relations \eqref{gen-rel: quasi-comm-1}--\eqref{gen-rel: quasi-commutation} are preserved under
the correspondence $$\pi:A_{\mathfrak{q}}(\Lambda,\widetilde{B})\longrightarrow \mathcal{A}_{\mathfrak{q}}(\Lambda,\widetilde{B}), Z_i\mapsto x_i~\text{and}~Y_j\mapsto x_{j}'$$
for any $1\leq i\leq m$ and $1\leq j\leq n$. Thus, we get $\pi:A_{\mathfrak{q}}(\Lambda,\widetilde{B})\to \mathcal{A}_{\mathfrak{q}}(\Lambda,\widetilde{B})$ is a homomorphism of algebras.
Using the relations \eqref{gen-rel: quasi-comm-1}--\eqref{gen-rel: quasi-commutation}, we obtain that the set of standard monomials $$\mathscr{S}:=\{Z_1^{a_1}\cdots Z_m^{a_m}Y_1^{b_1}\cdots Y_n^{b_n}\mid \text{all}~a_i,b_j\in \mathbb{N}~\text{and}~a_ib_i=0~\text{for~any}~1\leq i\leq n\}$$ forms a spanning set of $A_{\mathfrak{q}}(\Lambda,\widetilde{B})$.
By \cite[Theorems 7.3, 7.5]{BZ05}, the images of the elements in $\mathscr{S}$ under $\pi$ give a $\mathbb{Q}(\mathfrak{q}^{\frac{1}{2}})$-basis of $\mathcal{A}_{\mathfrak{q}}(\Lambda,\widetilde{B})$\footnote{We need a minor variation of \cite[Theorem 7.5]{BZ05}. Although it was originally proved over the Laurent polynomial ring in $x_{n+1},\ldots, x_m$, the same statement holds over the polynomial ring in $x_{n+1},\ldots, x_m$.}. It follows that $\mathscr{S}$ is a $\mathbb{Q}(\mathfrak{q}^{\frac{1}{2}})$-basis of $A_{\mathfrak{q}}(\Lambda,\widetilde{B})$ and thus $\pi$ is an algebra isomorphism.
\end{proof}

\subsection{$\Lambda$-twisted derived Hall subalgebras}
For any $X=\bigoplus\limits_{i\in I}X_i[i]\in \mathcal{D}^b(\widetilde{\A})$ for some finite subset $I$ of $\mathbb{Z}$, we write ${\bf x}:={\Dim}X=\sum\limits_{i\in I}(-1)^i{\bf x}_i$, called the {\em dimension vector} of $X$; and write $\overline{\tau}{\bf x}$ as the dimension vector of $\overline{\tau}X$.
In order to relate Hall algebras with quantum cluster algebras, we need to twist the multiplications of derived Hall algebras using the bilinear form $\Lambda$.
So, let us twist the multiplication on $\mathcal {D}\mathcal {H}_q(\widetilde{\A})$, and define $\mathcal {D}\mathcal {H}_\Lambda(\widetilde{\A})$ to be the same space as $\mathcal {D}\mathcal {H}_q(\widetilde{\A})$ but with the twisted
multiplication defined on basis elements by
\begin{equation*}
   u_{X}\star u_{Y}:=
   v^{\Lambda({\bf x}^\ast,{\bf y}^\ast)}
   u_{X}\ast u_{Y},
\end{equation*}
for any $X,Y\in \mathcal{D}^b(\widetilde{\A})$. By \cite[Lemma 6.4]{CDZ} and Lemma \ref{sjishu}, we have
\begin{equation}\Lambda((\overline{\tau}{\bf x})^\ast,(\overline{\tau}{\bf y})^\ast)=\Lambda({\bf x}^\ast,{\bf y}^\ast)\end{equation}
for any $X,Y\in \mathcal{D}^b(\widetilde{\A})$. Hence, the Auslander--Reiten translation $\overline{\tau}$ of $\mathcal{D}^b(\widetilde{\A})$ provides an algebra automorphism of $\mathcal {D}\mathcal {H}_\Lambda(\widetilde{\A})$.

Let $\mathcal {D}\mathcal {H}_\Lambda^{ec}(\widetilde{\A})$ be the subalgebra of $\mathcal {D}\mathcal {H}_\Lambda(\widetilde{\A})$, which is
the same space as $\mathcal {D}\mathcal {H}_q^{ec}(\widetilde{\A})$ but with the twisted
multiplication defined on basis elements by
\begin{equation}\label{yizhitwisting1}
\begin{split}
   u_{I[-1]\oplus M\oplus P[1]}\star u_{J[-1]\oplus N\oplus Q[1]}:=
   v^{\Lambda(({\bf m}-{\bf i}-{\bf p})^\ast,(\bf{n}
   -{\bf j}-\bf{q})^\ast)}
   u_{I[-1]\oplus M\oplus P[1]}\ast u_{J[-1]\oplus N\oplus Q[1]}\end{split}
\end{equation}
for any $M, N\in\widetilde{\A}$, $I, J\in\I_{\widetilde{\A}}$ and $P, Q\in\P_{\widetilde{\A}}$.
Using Proposition \ref{twistderived}, we have the following.
\begin{proposition}\label{elsubalgebra}
The algebra $\mathcal {D}\mathcal {H}_\Lambda^{ec}(\widetilde{\A})$ is generated by the elements $$\{u_{P[1]},u_{M},u_{I[-1]}~|~P\in\P_{\widetilde{\A}},M\in\widetilde{\A},I\in\I_{\widetilde{\A}}\}$$ and the following relations
\begin{flalign}\label{lgx2}
u_{P[1]}\star u_{Q[1]}=q^{\frac{1}{2}\Lambda({\bf p}^\ast,{\bf q}^\ast)}
u_{(P\oplus Q)[1]}
=q^{\Lambda({\bf p}^\ast,{\bf q}^\ast)}u_{Q[1]}\star u_{P[1]};\end{flalign}
\begin{flalign}\label{lgx7}u_{P[1]}\star u_{M}=q^{-\frac{1}{2}\Lambda({\bf p}^\ast,{\bf m}^\ast)-\lr{{\bf p},{\bf m}}}
\sum\limits_{[F],[P']}v^{\Lambda({\bf f}^\ast,{{\bf p}'}^\ast)}|{}_{P'}\Hom_{\widetilde{\A}}(P,M)_F|u_{F}\star u_{P'[1]};
\end{flalign}
\begin{flalign}\label{lgx3}
u_{M}\star u_{N}=q^{\frac{1}{2}\Lambda({\bf m}^\ast,{\bf n}^\ast)+\lr{{\bf m},{\bf n}}}\sum_{[L]}\frac{|\mathrm{Ext}_{\widetilde{\A}}^{1}(M,N)_{L}|}{|\mathrm{Hom}_{\widetilde{\A}}(M,N)|}u_L;\end{flalign}
\begin{flalign}\label{lgx1}u_{I[-1]}\star u_{J[-1]}=
q^{\frac{1}{2}\Lambda({\bf i}^\ast,{\bf j}^\ast)}u_{(I\oplus J)[-1]}
=q^{\Lambda({\bf i}^\ast,{\bf j}^\ast)}u_{J[-1]}\star u_{I[-1]};\end{flalign}
\begin{flalign}\label{lgx6}u_{M}\star u_{I[-1]}=q^{-\frac{1}{2}\Lambda({\bf m}^\ast,{\bf i}^\ast)-\lr{{\bf m},{\bf i}}}
\sum\limits_{[G],[I']}v^{\Lambda({{\bf i}'}^\ast,{\bf g}^\ast)}|{}_{G}\Hom_{\widetilde{\A}}(M,I)_{I'}|u_{I'[-1]}\star u_{G};\end{flalign}
\begin{flalign}\label{lgx8}u_{I[-1]}\star u_{P[1]}=q^{\Lambda({\bf i}^\ast,{\bf p}^\ast)-\lr{{\bf p},{\bf i}}}u_{P[1]}\star u_{I[-1]};\end{flalign}
for any $P,Q\in\P_{\widetilde{\A}}$, $M,N\in\widetilde{\A}$ and $I,J\in\I_{\widetilde{\A}}$.
\end{proposition}
In fact, we also have the following equations in $\mathcal {D}\mathcal {H}_\Lambda^{ec}(\widetilde{\A})$
\begin{flalign}\label{lgx5}
u_{M}\star u_{P[1]}=q^{-\frac{1}{2}\Lambda({\bf m}^\ast,{\bf p}^\ast)}u_{M\oplus
P[1]},\end{flalign}
\begin{flalign}\label{lgx4}u_{I[-1]}\star u_{M}=q^{-\frac{1}{2}\Lambda({\bf i}^\ast,{\bf m}^\ast)}u_{M\oplus
I[-1]}\end{flalign}
for any $P\in\P_{\widetilde{\A}}$, $M\in\widetilde{\A}$ and $I\in\I_{\widetilde{\A}}$.

\begin{proposition}\label{mint27}
In the generating relations of $\mathcal {D}\mathcal {H}_\Lambda^{ec}(\widetilde{\A})$, the relation \eqref{lgx6} can be replaced by the following two relations
\begin{flalign}\label{fgx1}
u_{P}\star u_{I[-1]}=q^{-\frac{1}{2}\Lambda({\bf p}^\ast,{\bf i}^\ast)-\lr{{\bf p},{\bf i}}}
\sum\limits_{[Q],[I']}v^{\Lambda({{\bf i}'}^\ast,{\bf q}^\ast)}|{}_{Q}\Hom_{\widetilde{\A}}(P,I)_{I'}|u_{I'[-1]}\star u_Q,
\end{flalign}
where $P\in\P_{\widetilde{\A}}$~and $I\in\I_{\widetilde{\A}}$;
\begin{flalign}\label{fgx2}
u_{M}\star u_{I[-1]}=q^{-\frac{1}{2}\Lambda({\bf m}^\ast,{\bf i}^\ast)-\lr{{\bf m},{\bf i}}}
\sum\limits_{[G],[I']}v^{\Lambda({{\bf i}'}^\ast,{\bf g}^\ast)}|{}_{G}\Hom_{\widetilde{\A}}(M,I)_{I'}|u_{I'[-1]}\star u_G,
\end{flalign}
where $M\in\widetilde{\A}$~has~nonzero~projective~direct~summands and $I\in\I_{\widetilde{\A}}$.
\end{proposition}
\begin{proof}
We only need to show that the relation \eqref{lgx6} can be implied from the relations \eqref{fgx1}, \eqref{fgx2} together with the other relations in Proposition \ref{elsubalgebra}. For any $M\in\widetilde{\A}$, write $M=M'\oplus P'$ such that $P'$ is the maximal projective direct summand of $M$. Then $u_M=v^{-\Lambda({{\bf p}'}^\ast,{{\bf m}'}^\ast)}u_{P'}\star u_{M'}$. Thus, we have
\begin{flalign*}
u_{M}\star u_{I[-1]}=v^{-\Lambda({{\bf p}'}^\ast,{{\bf m}'}^\ast)}u_{P'}\star u_{M'}\star u_{I[-1]}.
\end{flalign*}
By \eqref{fgx2}, we have
\begin{flalign*}u_{M'}\star u_{I[-1]}=q^{-\frac{1}{2}\Lambda({{\bf m}'}^\ast,{\bf i}^\ast)-\lr{{\bf m}',{\bf i}}}
\sum\limits_{[G_1],[J]}v^{\Lambda({\bf j}^\ast,{\bf g}_1^\ast)}|{}_{G_1}\Hom_{\widetilde{\A}}(M',I)_{J}|u_{J[-1]}\star u_{G_1}.\end{flalign*}
Thus, we obtain
\begin{flalign*}
&u_{M}\star u_{I[-1]}=\\&v^{-\Lambda({{\bf p}'}^\ast,{{\bf m}'}^\ast)-\Lambda({{\bf m}'}^\ast,{\bf i}^\ast)-2\lr{{\bf m}',{\bf i}}}\sum\limits_{[G_1],[J]}v^{\Lambda({\bf j}^\ast,{\bf g}_1^\ast)}|{}_{G_1}\Hom_{\widetilde{\A}}(M',I)_{J}|u_{P'}\star u_{J[-1]}\star u_{G_1}.
\end{flalign*}
Using \eqref{fgx1}, we get
\begin{flalign*}
u_{M}\star u_{I[-1]}&=v^{-\Lambda({{\bf p}'}^\ast,{{\bf m}'}^\ast)-\Lambda({{\bf m}'}^\ast,{\bf i}^\ast)-2\lr{{\bf m}',{\bf i}}}\sum\limits_{[G_1],[J],[G_2],[I']}v^{\Lambda({\bf j}^\ast,{\bf g}_1^\ast)-\Lambda({{\bf p}'}^\ast,{{\bf j}}^\ast)-2\lr{{\bf p}',{\bf j}}}\\&v^{\Lambda({{\bf i}'}^\ast,{\bf g}_2^\ast)}|{}_{G_1}\Hom_{\widetilde{\A}}(M',I)_{J}|\cdot|{}_{G_2}\Hom_{\widetilde{\A}}(P',J)_{I'}| u_{I'[-1]}\star u_{G_2}\star u_{G_1}.
\end{flalign*}
Noting that $G_2$ is projective, we obtain
\begin{equation*}
u_{M}\star u_{I[-1]}=\sum\limits_{[G_1],[J],[G_2],[I']}v^{x_0}|{}_{G_1}\Hom_{\widetilde{\A}}(M',I)_{J}|\cdot|{}_{G_2}\Hom_{\widetilde{\A}}(P',J)_{I'}| u_{I'[-1]}\star u_{G_1\oplus G_2},
\end{equation*}
where \begin{flalign*}x_0=&-\Lambda({{\bf p}'}^\ast,{{\bf m}'}^\ast)-\Lambda({{\bf m}'}^\ast,{\bf i}^\ast)-2\lr{{\bf m}',{\bf i}}+\Lambda({\bf j}^\ast,{\bf g}_1^\ast)\\&-\Lambda({{\bf p}'}^\ast,{{\bf j}}^\ast)-2\lr{{\bf p}',{\bf j}}+\Lambda({{\bf i}'}^\ast,{\bf g}_2^\ast)+\Lambda({{\bf g}_2}^\ast,{\bf g}_1^\ast).\end{flalign*}
Since ${\bf j}+{\bf g}_2={\bf p}'+{\bf i}'$ and ${\bf g}_1-{\bf j}={\bf m}'-{\bf i}$, we get
\begin{flalign*}x_0&=-\Lambda({{\bf m}}^\ast,{\bf i}^\ast)-2\lr{{\bf m}',{\bf i}}-2\lr{{\bf p}',{\bf j}}+\Lambda({{\bf i}'}^\ast,{\bf g}_1^\ast+{\bf g}_2^\ast)\\&=-\Lambda({{\bf m}}^\ast,{\bf i}^\ast)-2\lr{{\bf m},{\bf i}}+\Lambda({{\bf i}'}^\ast,{\bf g}_1^\ast+{\bf g}_2^\ast)+2\lr{{\bf p}',{\bf m}'-{\bf g}_1}.\end{flalign*}
Hence, by Proposition \ref{prop:Hall-equality-P-I}, we obtain the relation \eqref{lgx6}.
\end{proof}

\subsection{Cluster multiplication formulas}
In this subsection, we recall some quantum cluster multiplication formulas from \cite{CDZ} and properties of quantum cluster monomials.

For any $M\in\widetilde{\A}$, $I\in\I_{\widetilde{\A}}$ and $P\in\P_{\widetilde{\A}}$, define the specialised {\em quantum cluster character} $X_{I[-1]\oplus M\oplus P[1]}$ in the quantum torus $\mathcal{T}_\Lambda$ by
\begin{equation}\label{qcltz}X_{I[-1]\oplus M\oplus P[1]}=\sum\limits_{\mathbf{e}}v^{\lr{{\bf p}-\mathbf{e},\mathbf{m}-\mathbf{e}-\mathbf{i}}}|\mathrm{Gr}_{\mathbf{e}}M|
X^{({\bf p}-\mathbf{e})^\ast-^\ast(\mathbf{m}-\mathbf{e}-\mathbf{i})},\end{equation}
where $\mathrm{Gr}_{\bf{e}}M$ denotes the set of all submodules $V$
of $M$ with $\Dim V= \bf{e}$.
\begin{theorem}\textup{(\cite[Theorem 4.8]{CDZ})}\label{mainresult}
The map $\psi:\mathcal {D}\mathcal {H}_\Lambda^{ec}(\widetilde{\A})\longrightarrow\mathcal{T}_\Lambda$ defined by
$$\psi(u_{I[-1]\oplus M\oplus P[1]})=X_{I[-1]\oplus M\oplus P[1]},$$ for any $M\in\widetilde{\A}$, $I\in\I_{\widetilde{\A}}$ and $P\in\P_{\widetilde{\A}}$, is a homomorphism of algebras.
\end{theorem}

Using the algebra homomorphism $\psi$, one can obtain the following multiplication formulas of quantum cluster characters.
\begin{corollary}\textup{(\cite[Corollary 4.10]{CDZ})}\label{Hallcfgs}
The following cluster multiplication formulas hold:
\begin{flalign}\label{qre2}
X_{P[1]} X_{Q[1]}=q^{\frac{1}{2}\Lambda({\bf p}^\ast,{\bf q}^\ast)}
X_{(P\oplus Q)[1]}
=q^{\Lambda({\bf p}^\ast,{\bf q}^\ast)}X_{Q[1]} X_{P[1]};\end{flalign}
\begin{flalign}\label{qre5}
X_{M} X_{P[1]}=q^{-\frac{1}{2}\Lambda({\bf m}^\ast,{\bf p}^\ast)}X_{M\oplus
P[1]};\end{flalign}
\begin{flalign}\label{qre7}X_{P[1]} X_{M}=q^{-\frac{1}{2}\Lambda({\bf p}^\ast,{\bf m}^\ast)-\lr{{\bf p},{\bf m}}}
\sum\limits_{[F],[P']}|{}_{P'}\Hom_{\widetilde{\A}}(P,M)_F|X_{F\oplus P'[1]};\end{flalign}
\begin{flalign}\label{qre3}X_{M} X_{N}=q^{\frac{1}{2}\Lambda({\bf m}^\ast,{\bf n}^\ast)+\lr{{\bf m},{\bf n}}}\sum_{[L]}\frac{|\mathrm{Ext}_{\widetilde{\A}}^{1}(M,N)_{L}|}{|\mathrm{Hom}_{\widetilde{\A}}(M,N)|}X_L;\end{flalign}
\begin{flalign}\label{qre1}X_{I[-1]}X_{J[-1]}=
q^{\frac{1}{2}\Lambda({\bf i}^\ast,{\bf j}^\ast)}X_{(I\oplus J)[-1]}
=q^{\Lambda({\bf i}^\ast,{\bf j}^\ast)}X_{J[-1]}X_{I[-1]};\end{flalign}
\begin{flalign}\label{qre4}X_{I[-1]} X_{M}=q^{-\frac{1}{2}\Lambda({\bf i}^\ast,{\bf m}^\ast)}X_{M\oplus
I[-1]};\end{flalign}
\begin{flalign}\label{qre6}X_{M} X_{I[-1]}=q^{-\frac{1}{2}\Lambda({\bf m}^\ast,{\bf i}^\ast)-\lr{{\bf m},{\bf i}}}
\sum\limits_{[G],[I']}|{}_{G}\Hom_{\widetilde{\A}}(M,I)_{I'}|X_{G\oplus I'[-1]};\end{flalign}
\begin{flalign}\label{qre8}X_{I[-1]} X_{P[1]}=q^{\Lambda({\bf i}^\ast,{\bf p}^\ast)-\lr{{\bf p},{\bf i}}}X_{P[1]} X_{I[-1]};\end{flalign}
for any $P,Q\in\P_{\widetilde{\A}}$, $M,N\in\widetilde{\A}$ and $I,J\in\I_{\widetilde{\A}}$.
\end{corollary}

Let $\mathfrak{I}_1$ be the two-sided ideal of the derived Hall subalgebra $\mathcal {D}\mathcal {H}_\Lambda^{ec}(\widetilde{\A})$ generated by the elements in $$\mathcal {S}_1:=\{u_{\nu^{-1}(I)[1]}-u_{I[-1]}~|~I\in\I_{\widetilde{\A}}\},$$
where $\nu$ is the Nakayama functor in $\widetilde{\A}$.
Define \begin{equation}\mathcal {D}\mathcal {H}_\Lambda^{{cl}_1}(\widetilde{\A}):=\mathcal {D}\mathcal {H}_\Lambda^{ec}(\widetilde{\A})/\mathfrak{I}_1.\end{equation} In the following, for any element $u$ in an algebra, we use the same notation $u$ to denote its image in a quotient algebra. Thus we have that $u_{P[1]}=u_{I[-1]}$ in $\mathcal {D}\mathcal {H}_\Lambda^{{cl}_1}(\widetilde{\A})$ for any $I\in\I_{\widetilde{\A}}$ and $P=\nu^{-1}(I)$.
Since $\mathfrak{I}_1\subseteq\Ker \psi$, the algebra homomorphism $\psi:\mathcal {D}\mathcal {H}_\Lambda^{ec}(\widetilde{\A})\longrightarrow\mathcal{T}_\Lambda$ induces a homomorphism of algebras \begin{equation}\label{xuyaomap}\varphi_1:\mathcal {D}\mathcal {H}_\Lambda^{{cl}_1}(\widetilde{\A})\longrightarrow\mathcal{T}_\Lambda.\end{equation}

Let $\mathcal{A}\mathcal{H}_\Lambda^{\circ}(\widetilde{Q})$ be the subalgebra of $\mathcal{T}_\Lambda$ generated by all the quantum cluster characters $X_M, X_{I[-1]}$ with $M\in\widetilde{\A}, I\in\I_{\widetilde{\A}}$. By \cite{Rupel2}, the specialised quantum cluster algebra $\A_q(\Lambda, B(\widetilde{Q}))$ is generated by the quantum cluster characters $X_{M}, X_{I[-1]}$, where $M\in \widetilde{\A}$ and $I\in \mathcal{I}_{\widetilde{\A}}$ are indecomposable, and $M$ is rigid (i.e., $\Ext^1_{\widetilde{\A}}(M,M)=0$). Clearly, $\A_q(\Lambda, B(\widetilde{Q}))\subseteq \mathcal{A}\mathcal{H}_\Lambda^{\circ}(\widetilde{Q})$. Thus, we restate the homomorphism $\varphi_1$ as the following.

\begin{proposition}\textup{(\cite[Proposition 5.4]{CDZ})}\label{morphism1}
The map $\varphi_1:\mathcal {D}\mathcal {H}_\Lambda^{{cl}_1}(\widetilde{\A})\longrightarrow\mathcal{A}\mathcal{H}_\Lambda^{\circ}(\widetilde{Q})$ defined by
$\varphi_1(u_{I[-1]\oplus M\oplus P[1]})=X_{I[-1]\oplus M\oplus P[1]},$ for any $M\in\widetilde{\A}$, $I\in\I_{\widetilde{\A}}$ and $P\in\P_{\widetilde{\A}}$, is a surjective homomorphism of algebras.
\end{proposition}

Using the Hall algebra $\mathcal {D}\mathcal {H}_\Lambda^{{cl}_1}(\widetilde{\A})$ and the algebra homomorphism $\varphi_1$, one can obtain the following quantum cluster multiplication formulas.
\begin{theorem}\textup{(\cite[Theorem 5.7]{CDZ})}\label{dyggs}
For any $M\in\widetilde{\A}, I\in\I_{\widetilde{\A}}$ and $P=\nu^{-1}(I)$, the following cluster multiplication formulas hold
\begin{equation*}\label{xjyan1}\begin{split}
&(q^{\lr{{\bf p},{\bf m}}}-1)X_{P[1]} X_{M}=\\&q^{\frac{1}{2}\Lambda({\bf m}^\ast,{\bf p}^\ast)}(\sum\limits_{\begin{smallmatrix}[F],[P']\\P'\ncong P\end{smallmatrix}}|{}_{P'}\Hom_{\widetilde{\A}}(P,M)_F|X_{F\oplus P'[1]}+q^{-\frac{1}{2}\lr{{\bf m},{\bf i}}}\sum\limits_{\begin{smallmatrix}[G],[I']\\I'\ncong I\end{smallmatrix}}|{}_{G}\Hom_{\widetilde{\A}}(M,I)_{I'}|X_{G\oplus I'[-1]})
\end{split}\end{equation*}
and \begin{flalign*}
&(q^{\lr{{\bf m},{\bf i}}}-1)X_{M} X_{I[-1]}=\\&q^{\frac{1}{2}\Lambda({\bf i}^\ast,{\bf m}^\ast)}(q^{-\frac{1}{2}\lr{{\bf p},{\bf m}}}\sum\limits_{\begin{smallmatrix}[F],[P']\\P'\ncong P\end{smallmatrix}}|{}_{P'}\Hom_{\widetilde{\A}}(P,M)_F|X_{F\oplus P'[1]}+\sum\limits_{\begin{smallmatrix}[G],[I']\\I'\ncong I\end{smallmatrix}}|{}_{G}\Hom_{\widetilde{\A}}(M,I)_{I'}|X_{G\oplus I'[-1]}).
\end{flalign*}
\end{theorem}

Let $M,N\in\widetilde{\A}$ and write $M=M'\oplus P'$ such that $P'$ is the maximal projective direct summand of $M$. For each morphism $\theta:N\longrightarrow \tau M$, where $\tau$ is the Auslander--Reiten translation of $\widetilde{\A}$, we have an exact sequence
\begin{equation}\label{zhxl}
\xymatrix{0\ar[r]&D\ar[r]&N\ar[r]^-\theta&\tau M\ar[r]&\tau A'\oplus I\ar[r]&0,}
\end{equation}where $D=\Ker \theta, \tau A'\oplus I=\Coker \theta$, $I\in\I_{\widetilde{\A}}$, and $A'$ has no nonzero projective direct summands. In the following, we set $A=A'\oplus P'$.
For the simplicity of notation, we also set $[M,N]^1:={\rm dim}_{\mathbb{F}_q}\Ext^1_{\widetilde{\A}}(M,N)$ for any $M,N\in\widetilde{\A}$.
The following cluster multiplication formulas have also been obtained in \cite{CDZ}.
\begin{theorem}\textup{(\cite[Theorem 7.4]{CDZ})}
For any $M,N\in\widetilde{\A}$, the following cluster multiplication formulas hold
\begin{equation}\label{ccfgs}
\begin{split}
&(q^{[M,N]^1}-1)X_M X_N=q^{\frac{1}{2}\Lambda({\bf m}^*,{\bf n}^*)}\sum\limits_{[E]\neq [M\oplus N]}
|\Ext_{\widetilde{\A}}^1(M,N)_E|X_E+\\&\sum\limits_{\begin{smallmatrix}[D],[A],[I]\\ [D]\neq [N]\end{smallmatrix}}q^{\frac{1}{2}\Lambda(({\bf m}-{{\bf a}})^*,{({\bf n}+{\bf a})}^*)+\frac{1}{2}\lr{{{\bf m}}-{{\bf a}},{\bf n}}}|_D\Hom_{\widetilde{\A}}(N,\tau M)_{\tau A\oplus I}|X_{A} X_{D\oplus I[-1]},
\end{split}\end{equation}
where each $A$ has the same maximal projective direct summand as $M$, and each $I\in\I_{\widetilde{\A}}$.
\end{theorem}

According to \cite{Rupel2}, the set of quantum cluster monomials of $\A_q(\Lambda, B(\widetilde{Q}))$ is precisely given by the following set \begin{equation}\label{qclumo}
\{X_{M\oplus I[-1]}\mid M\in \widetilde{\A}~\text{is ~rigid~and}~I\in \mathcal{I}_{\widetilde{\A}}~\text{such that}~ \Hom_{\widetilde{\A}}(M,I)=0\}.\end{equation}
It is well-known that the set of quantum cluster monomials is a $\mathbb{Q}(v)$-linearly independent set (cf. \cite{BZ12, KQin}). For convenience, we give a direct proof in the following.

For any $f\in \mathcal{T}_{\Lambda}$, $f$ is said to be {\em pointed} at $g\in \mathbb{Z}^{m}$ with respect to $(\Lambda,B(\widetilde{Q}),X)$ (cf. \cite{Qin2}), if it can be expressed as
\begin{equation}\label{expansion}f=X^g(\sum_{e\in \mathbb{N}^{m}}a_eX^{B(\widetilde{Q})e}),\end{equation}
where $a_e\in \mathbb{Q}(v)$ and $a_0\neq 0$.
\begin{lemma}\label{lem:pointed-function}
    Let $f_1,\ldots, f_t\in \mathcal{T}_{\Lambda}$ be pointed at $g_1,\ldots, g_t\in \mathbb{Z}^{m}$, respectively. If $g_1,\ldots, g_t$ are pairwise different, then $f_1,\ldots, f_t$ are $\mathbb{Q}(v)$-linearly independent.
\end{lemma}

\begin{proof}
Suppose that $f_1,\ldots, f_t$ are $\mathbb{Q}(v)$-linearly dependent, i.e., there exists a nonzero vector $(b_1,\ldots, b_t)$ of $\mathbb{Q}(v)$  such that
$\sum_{i=1}^t b_i f_i=0$. Let $I=\{i~|~1\leq i\leq t, b_i\neq0\}$, then we have $\sum_{i=1}^t b_i f_i=\sum_{i\in I}b_i f_i=0.$

For any $\alpha,\beta\in \mathbb{Z}^{m}$, we define a relation $\alpha\geq \beta$ if there exists $e\in \mathbb{N}^{m}$ such that $\beta=\alpha+B(\widetilde{Q})e$. Since $B(\widetilde{Q})$ is of full rank, it follows that $\geq$ is a well-defined partial order on $\mathbb{Z}^{m}$.
Let $g_k$ be a maximal element in $\{g_i~|~i\in I\}$ with respect to this partial order.
Then for any $k'\neq k$ in $I$, in the expression \eqref{expansion} of $f_{k'}$, there is no $e\in \mathbb{N}^{m}$ such that $g_{k'}+B(\widetilde{Q})e=g_k$.
Thus, the coefficient of $X^{g_k}$ in the sum $\sum_{i\in I}b_i f_i$ is $b_k a_{k,0}\neq0$, where $a_{k,0}$ is the coefficient of the highest term of $f_k$. Hence, we obtain $\sum_{i\in I}b_i f_i\neq0$. This is a contradiction. Therefore, we finish the proof.
\end{proof}

\begin{lemma}\label{lem:pointed-vector-iso}
    Let $M,N\in \widetilde{\A}$ and $I,J\in \mathcal{I}_{\widetilde{\A}}$. Suppose that $\Hom_{\widetilde{\A}}(M,I)=0$ and $\Hom_{\widetilde{\A}}(N,J)=0$, then ${^\ast\mathbf{i}}-\mathbf{m}^\ast={^\ast\mathbf{j}}-\mathbf{n}^\ast$ if and only if $\mathbf{m}=\mathbf{n}$ and $I\cong J$.
\end{lemma}
\begin{proof}
The sufficiency is obvious, and we only need to prove the necessity.
Set $P=\nu^{-1}(I)$ and $Q=\nu^{-1}(J)$. Then $^\ast \mathbf{i}=\mathbf{p}^\ast$ and $^\ast \mathbf{j}=\mathbf{q}^\ast$.
Recall that $\mathbf{x}^\ast=E'(\widetilde{Q})\mathbf{x}$ and $E'(\widetilde{Q})$ is invertible over $\mathbb{Z}$. Thus, ${^\ast\mathbf{i}}-\mathbf{m}^\ast={^\ast\mathbf{j}}-\mathbf{n}^\ast$ if and only if $\mathbf{p}-\mathbf{m}=\mathbf{q}-\mathbf{n}$. Let $P_i$ be the projective cover of $S_i$ for each $i\in \widetilde{Q}_0$. By removing the common direct summands of $P$ and $Q$, we can assume that $P$ and $Q$ have no common indecomposable direct summands.
Suppose that $P\not \cong Q$, thus, at least one of $P,Q$ is nonzero.  Assume that $P\neq 0$ and let $P=a_1P_{i_1}\oplus \cdots\oplus a_tP_{i_t}$ for some positive integers $a_1,\ldots, a_t$. Since $\Hom_{\widetilde{\A}}(P,M)\cong\Hom_{\widetilde{\A}}(M,I)=0$, we obtain that the $i_1$-th component of $\mathbf{p}-\mathbf{m}=\mathbf{q}-\mathbf{n}$ is positive. Thus, $Q\neq 0$. Let $Q=b_1P_{j_1}\oplus \cdots\oplus b_sP_{j_s}$ for some positive integers $b_1,\ldots, b_s$, where $\{i_1,\ldots, i_t\}\cap \{j_1,\ldots, j_s\}=\emptyset$. Since $\widetilde{Q}$ is acyclic, without loss of generality, we may assume that $\Hom_{\widetilde{\A}}(P_{i_1},P_{k})=0$ for any $k\in \{i_2,\ldots, i_t,j_1,\ldots, j_s\}$. Thus, the $i_1$-th component of $\mathbf{p}-\mathbf{m}$ is $a_1$, while the $i_1$-th component of $\mathbf{q}-\mathbf{n}$ is non-positive, this is a contradiction. Hence, $P\cong Q$, and then $I\cong J$ and $\mathbf{m}=\mathbf{n}$.
\end{proof}

\begin{proposition}\label{xjlm}
The set \eqref{qclumo} of quantum cluster monomials of $\A_q(\Lambda, B(\widetilde{Q}))$
is $\mathbb{Q}(v)$-linearly independent.
\end{proposition}
\begin{proof}
Note that
$$X_{M\oplus I[-1]}=X^{{^\ast\mathbf{i}}-\mathbf{m}^\ast}\sum\limits_{\bf e}v^{-\lr{{\bf m}+\mathbf{e},{\bf m}-\mathbf{e}}+\lr{{\bf m},{\bf i}}}|\mathrm{Gr}_{\mathbf{e}}M|
X^{B(\widetilde{Q})(\mathbf{m}-\mathbf{e})}.$$
It follows that $X_{M\oplus I[-1]}$ is pointed at ${^\ast\mathbf{i}}-\mathbf{m}^\ast\in \mathbb{Z}^{m}$. Since rigid modules are determined by their dimension vectors up to isomorphisms, by Lemma \ref{lem:pointed-vector-iso}, we conclude that different $X_{M\oplus I[-1]}$ are pointed at different vectors in $\mathbb{Z}^m$. By Lemma \ref{lem:pointed-function}, we obtain that the set of quantum cluster monomials is $\mathbb{Q}(v)$-linearly independent.
\end{proof}

\section{Cluster multiplication formulas exhibiting 2-Calabi--Yau properties}
In this section, we provide a direct sum decomposition formula of quantum cluster characters, which exhibits certain symmetries similar to 2-Calabi--Yau properties of cluster categories.
\subsection{An alternative cluster multiplication formula}
Using the cluster multiplication formulas \eqref{qre3} and \eqref{ccfgs}, we obtain the following.

\begin{theorem}\label{zhihe}
For any $M,N\in\widetilde{\A}$, we have the following equation in $\mathcal{T}_\Lambda:$
\begin{equation}\label{zhihegs}
\begin{split}
X_{M\oplus N}=\sum\limits_{\begin{smallmatrix}[D],[A],[I]\end{smallmatrix}}v^{\Lambda(({\bf m}+{\bf n})^*,{\bf a}^*)+\lr{{\bf m}-{\bf a},{\bf n}}}|_D\Hom_{\widetilde{\A}}(N,\tau M)_{\tau A\oplus I}|X_{A} X_{D\oplus I[-1]},
\end{split}\end{equation}
where each $A$ has the same maximal projective direct summand as $M$, and each $I\in\I_{\widetilde{\A}}$.
\end{theorem}
\begin{proof}
Let $M=M'\oplus P'$ such that $P'$ is the maximal projective direct summand of $M$. By \eqref{ccfgs}, we have
\begin{equation}\label{diyibu}
\begin{split}
&(q^{[M,N]^1}-1)X_M X_N=v^{\Lambda({\bf m}^*,{\bf n}^*)}\sum\limits_{[E]\neq [M\oplus N]}
|\Ext_{\widetilde{\A}}^1(M,N)_E|X_E+\\&\sum\limits_{\begin{smallmatrix}[D],[A],[I]\\ [D]\neq [N]\end{smallmatrix}}v^{\Lambda(({\bf m}-{{\bf a}})^*,{({\bf n}+{\bf a})}^*)+\lr{{{\bf m}}-{{\bf a}},{\bf n}}}|_D\Hom_{\widetilde{\A}}(N,\tau M)_{\tau A\oplus I}|X_{A} X_{D\oplus I[-1]},
\end{split}\end{equation}
where each $A=A'\oplus P'$ has the same maximal projective direct summand~as $M$, and each $I$ is injective.

For the zero morphism in $\Hom_{\widetilde{\A}}(N,\tau M)$, we have $D=N, I=0, A'=M'$, and thus $A=M$.
Then $v^{\Lambda(({\bf m}-{{\bf a}})^*,{({\bf n}+{\bf a})}^*)+\lr{{{\bf m}}-{{\bf a}},{\bf n}}}X_{A} X_{D\oplus I[-1]}=X_MX_N$.
Hence, the equation \eqref{diyibu} can be rewritten as
\begin{equation*}
\begin{split}
&q^{[M,N]^1}X_M X_N=v^{\Lambda({\bf m}^*,{\bf n}^*)}\sum\limits_{[E]\neq [M\oplus N]}
|\Ext_{\widetilde{\A}}^1(M,N)_E|X_E+\\&\sum\limits_{\begin{smallmatrix}[D],[A],[I]\end{smallmatrix}}v^{\Lambda(({\bf m}-{{\bf a}})^*,{({\bf n}+{\bf a})}^*)+\lr{{{\bf m}}-{{\bf a}},{\bf n}}}|_D\Hom_{\widetilde{\A}}(N,\tau M)_{\tau A\oplus I}|X_{A} X_{D\oplus I[-1]}.
\end{split}\end{equation*}
Using the equation \eqref{qre3}, we obtain
\begin{equation*}
\begin{split}
X_{M} X_{N}&=q^{\frac{1}{2}\Lambda({\bf m}^\ast,{\bf n}^\ast)+\lr{{\bf m},{\bf n}}}\sum_{[E]}\frac{|\mathrm{Ext}_{\widetilde{\A}}^{1}(M,N)_{E}|}{|\mathrm{Hom}_{\widetilde{\A}}(M,N)|}X_E\\
&=v^{\Lambda({\bf m}^\ast,{\bf n}^\ast)}\sum_{[E]}\frac{|\mathrm{Ext}_{\widetilde{\A}}^{1}(M,N)_{E}|}{|\mathrm{Ext}_{\widetilde{\A}}^1(M,N)|}X_E.
\end{split}
\end{equation*}
That is, \begin{equation}\label{diergx}
q^{[M,N]^1}X_M X_N=v^{\Lambda({\bf m}^*,{\bf n}^*)}\sum\limits_{[E]}
|\Ext_{\widetilde{\A}}^1(M,N)_E|X_E.\end{equation}
Thus, we obtain
\begin{equation}\label{disangx}\begin{split}
&v^{\Lambda({\bf m}^*,{\bf n}^*)}X_{M\oplus N}=\\&\sum\limits_{\begin{smallmatrix}[D],[A],[I]\end{smallmatrix}}v^{\Lambda(({\bf m}-{{\bf a}})^*,{({\bf n}+{\bf a})}^*)+\lr{{{\bf m}}-{{\bf a}},{\bf n}}}|_D\Hom_{\widetilde{\A}}(N,\tau M)_{\tau A\oplus I}|X_{A} X_{D\oplus I[-1]}.\end{split}\end{equation}
Therefore, we complete the proof.
\end{proof}

\begin{corollary}\label{2CY}
For any $M,N\in\widetilde{\A}$, we have the following equation in $\mathcal{T}_\Lambda:$
\begin{equation}\label{mfcgs}
\begin{split}
&\sum\limits_{\begin{smallmatrix}[D],[A],[I]\end{smallmatrix}}v^{\Lambda(({\bf m}+{\bf n})^*,{\bf a}^*)+\lr{{\bf m}-{\bf a},{\bf n}}}|_D\Hom_{\widetilde{\A}}(N,\tau M)_{\tau A\oplus I}|X_{A} X_{D\oplus I[-1]}\\&=
\sum\limits_{\begin{smallmatrix}[C],[B],[J]\end{smallmatrix}}v^{\Lambda(({\bf m}+{\bf n})^*,{\bf b}^*)+\lr{{\bf n}-{\bf b},{\bf m}}}|_C\Hom_{\widetilde{\A}}(M,\tau N)_{\tau B\oplus J}|X_{B} X_{C\oplus J[-1]},
\end{split}\end{equation}
where each $A$ has the same maximal projective direct summand as $M$, each $B$ has the same maximal projective direct summand as $N$, and each $I,J\in\I_{\widetilde{\A}}$.
\end{corollary}
\begin{proof}
Using $X_{M\oplus N}=X_{N\oplus M}$ and Theorem \ref{zhihe}, we finish the proof.
\end{proof}
\begin{remark}\label{zhuji}
$(1)$~The equation \eqref{mfcgs} is similar to the 2-Calabi--Yau property of the cluster category $\mathcal {C}:=\mathcal {C}_{\widetilde{Q}}$ of $\widetilde{Q}$, which states that there exists a bifunctorial isomorphism
$$\Hom_\mathcal {C}(M,\widetilde{\tau} N)\cong {\rm D}\Hom_\mathcal {C}(N,\widetilde{\tau} M)$$ for any $M,N\in\mathcal {C}$, where ${\rm D}$ is the standard linear duality and $\widetilde{\tau}$ is the Auslander--Reiten translation of $\mathcal {C}$.

$(2)$~By \cite[Lemma 4.4]{Rupel2},
if the set $_D\Hom_{\widetilde{\A}}(N,\tau M)_{\tau A\oplus I}$ in Corollary \ref{2CY} is nonempty, then $A$ must be a quotient of $M$. Moreover, we have
\begin{equation}\label{sxhomjs}
|_D\Hom_{\widetilde{\A}}(N,\tau M)_{\tau A\oplus I}|=\sum\limits_{[X],[Y]}a_YF_{AX}^MF_{YD}^NF_{IY}^{\tau X},\end{equation}
where each $X$ has no nonzero projective direct summands.
\end{remark}

\begin{proposition}\label{dhallc}
For any $A,D\in\widetilde{\A}$ and $I\in\I_{\widetilde{\A}}$, we have the following equation in $\mathcal{T}_\Lambda:$
\begin{equation*}\begin{split}
X_AX_{D\oplus I[-1]}=v^{\Lambda({\bf a}^\ast,({\bf d}-{\bf i})^\ast)}\sum\limits_{[L],[R],[S]}q^{\lr{{\bf s},{\bf d}}-\lr{{\bf a},{\bf i}}}|{_S\Hom_{\widetilde{\A}}(A,I)_{R}}|H_{SD}^LX_{L\oplus R[-1]}.
\end{split}
\end{equation*}
\end{proposition}
\begin{proof}
In $\mathcal {D}\mathcal {H}_\Lambda^{ec}(\widetilde{\A})$, we have
\begin{flalign*}
u_A\star u_{D\oplus I[-1]}=v^{\Lambda({\bf a}^\ast,({\bf d}-{\bf i})^\ast)}\sum\limits_{[L],[R]}q^{\lr{{\bf a},{\bf d}-{\bf i}}}H_{A,D\oplus I[-1]}^{L\oplus R[-1]}u_{L\oplus R[-1]}.
\end{flalign*}
By Proposition \ref{hallshujs}, we get \begin{flalign*}H_{A,D\oplus I[-1]}^{L\oplus R[-1]}&=\sum\limits_{[S],[T]}q^{\lr{{\bf r}-{\bf i},{\bf d}}}\frac{a_Sa_Ta_D}{a_L}F_{TS}^AF_{RT}^IF_{SD}^L\\&=\sum\limits_{[S],[T]}q^{\lr{{\bf r}-{\bf i},{\bf d}}}{a_T}F_{TS}^AF_{RT}^IH_{SD}^L\\
&=\sum\limits_{[S]}q^{\lr{{\bf r}-{\bf i},{\bf d}}}|{_S\Hom_{\widetilde{\A}}(A,I)_{R}}|H_{SD}^L.\end{flalign*}
Thus, we obtain
\begin{flalign*}
u_A\star u_{D\oplus I[-1]}&=v^{\Lambda({\bf a}^\ast,({\bf d}-{\bf i})^\ast)}\sum\limits_{[L],[R],[S]}q^{\lr{{\bf a}+{\bf r}-{\bf i},{\bf d}}-\lr{{\bf a},{\bf i}}}|{_S\Hom_{\widetilde{\A}}(A,I)_{R}}|H_{SD}^Lu_{L\oplus R[-1]}\\
&=v^{\Lambda({\bf a}^\ast,({\bf d}-{\bf i})^\ast)}\sum\limits_{[L],[R],[S]}q^{\lr{{\bf s},{\bf d}}-\lr{{\bf a},{\bf i}}}|{_S\Hom_{\widetilde{\A}}(A,I)_{R}}|H_{SD}^Lu_{L\oplus R[-1]}.
\end{flalign*}
Therefore, using Theorem \ref{mainresult}, we complete the proof.
\end{proof}

\begin{proposition}\label{2cycfgs}
For any $M,N\in\widetilde{\A}$, we have the following equations in $\mathcal{T}_\Lambda:$
\begin{equation*}
\begin{split}
X_{M\oplus N}=\sum\limits_{[D],[A],[I],[L],[R],[S]}&v^{\Lambda(({\bf m}+{\bf n}-{\bf d}+{\bf i})^\ast,{\bf a}^\ast)+\lr{{\bf m}-{\bf a},{\bf n}}}q^{\lr{{\bf s},{\bf d}}-\lr{{\bf a},{\bf i}}}\\&|_D\Hom_{\widetilde{\A}}(N,\tau M)_{\tau A\oplus I}|\cdot|{_S\Hom_{\widetilde{\A}}(A,I)_{R}}|H_{SD}^LX_{L\oplus R[-1]}\\
=\sum\limits_{[C],[B],[J],[L],[R],[S]}&v^{\Lambda(({\bf m}+{\bf n}-{\bf c}+{\bf j})^\ast,{\bf b}^\ast)+\lr{{\bf n}-{\bf b},{\bf m}}}
q^{\lr{{\bf s},{\bf c}}-\lr{{\bf b},{\bf j}}}\\&|_C\Hom_{\widetilde{\A}}(M,\tau N)_{\tau B\oplus J}|\cdot|{_S\Hom_{\widetilde{\A}}(B,J)_{R}}|H_{SC}^LX_{L\oplus R[-1]},
\end{split}
\end{equation*}
where each $A$ has the same maximal projective direct summand as $M$, each $B$ has the same maximal projective direct summand as $N$, and each $I,J,R\in\I_{\widetilde{\A}}$.
\end{proposition}
\begin{proof}
It is proved by Corollary \ref{2CY} and Proposition \ref{dhallc}.
\end{proof}

\subsection{Cluster multiplication formulas for a special compatible pair}
In this subsection, we consider the cluster multiplication formula in Proposition \ref{2cycfgs} for a special compatible pair. In this case, the exponent of $v$ in the formula is completely given by the Euler forms.

Let $Q$ be an acyclic valued quiver with the vertex set $\{1,2,\cdots,n\}$ and the valuation $d_i\in\mathbb{N^+}$ for each vertex $i$, take $\widetilde{Q}$ to be the valued quiver obtained from $Q$ by attaching additional vertices $n+1,\ldots, 2n$ to $Q$, adding the arrow $n+i\to i$, and setting the valuation of the $(n+i)$-th vertex to be $d_{n+i}=d_i$ for each $1\leq i\leq n$. In this case, the $2n\times 2n$-matrices $R(\widetilde{Q})$ and $R'(\widetilde{Q})$ are of the forms $\left(\begin{smallmatrix}
    R&I_n\\ 0&0
\end{smallmatrix}\right)$ and $\left(\begin{smallmatrix}
    R'&0\\ I_n&0
\end{smallmatrix}\right)$, respectively.
Set $B=R'-R$, then $B(\widetilde{Q})=\left(\begin{smallmatrix}
    B&-I_n\\ I_n&0
\end{smallmatrix}\right)$.
Set $D=\operatorname{diag}\{d_1,\cdots, d_n\}$ and take $\Lambda_1=\left(\begin{smallmatrix}
    0&-D\\ D&-DB
\end{smallmatrix}\right)$. Then, we have $$\Lambda_1 (B(\widetilde{Q}))=-\left(\begin{smallmatrix}
    D&\\ &D
\end{smallmatrix}\right).$$
We remark that the matrix $E$ representing the Euler form of the representation category $\mathcal{A}$ of $Q$ under the standard basis is equal to $(I_n-R^{\rm tr})D=D(I_n-R')$, and the matrix representing the Euler form of the representation category $\widetilde{\mathcal{A}}$ of $\widetilde{Q}$ under the standard basis is equal to $(E'(\widetilde{Q}))^{\rm tr}\left(\begin{smallmatrix}
    D&\\ &D
\end{smallmatrix}\right)=\left(\begin{smallmatrix}
    D&\\ &D
\end{smallmatrix}\right)E(\widetilde{Q})$,
where \begin{equation*}E(\widetilde{Q})=\left(\begin{smallmatrix}
    D^{-1}E&0\\ -I_n&I_n
\end{smallmatrix}\right)~\text{and}~E'(\widetilde{Q})=\left(\begin{smallmatrix}
    D^{-1}E^{\rm tr}&-I_n\\ 0&I_n
\end{smallmatrix}\right).\end{equation*}

\begin{lemma}\label{zhishu1}
For any $\alpha_1,\alpha_2,\beta_1,\beta_2\in\mathbb{Z}^n$, we have
$$\Lambda_1(\left(\begin{smallmatrix}
    \alpha_1\\ \alpha_2
\end{smallmatrix}\right)^\ast,\left(\begin{smallmatrix}
    \beta_1\\ \beta_2
\end{smallmatrix}\right)^\ast)=\lr{\beta_1-\beta_2,\alpha_2}-\lr{\alpha_1-\alpha_2,\beta_2}.$$
\end{lemma}
\begin{proof}
By definitions, we have
\begin{flalign*}
\Lambda_1(\left(\begin{smallmatrix}
    \alpha_1\\ \alpha_2
\end{smallmatrix}\right)^\ast,\left(\begin{smallmatrix}
    \beta_1\\ \beta_2
\end{smallmatrix}\right)^\ast)&=(\alpha_1^{\rm tr},\alpha_2^{\rm tr})\left(\begin{smallmatrix}
    ED^{-1}~& 0\\
    -I_n~& I_n\\
 \end{smallmatrix}\right)\left(\begin{smallmatrix}
    0&-D\\ D&-DB
\end{smallmatrix}\right)\left(\begin{smallmatrix}
    D^{-1}E^{\rm tr}~& -I_n\\
    0~& I_n\\
 \end{smallmatrix}\right)\left(\begin{smallmatrix}
    \beta_1\\ \beta_2
\end{smallmatrix}\right)\\
&=(\alpha_1^{\rm tr},\alpha_2^{\rm tr})\left(\begin{smallmatrix}
    0~& -E\\
    D~& D-DB\\
 \end{smallmatrix}\right)\left(\begin{smallmatrix}
    D^{-1}E^{\rm tr}~& -I_n\\
    0~& I_n\\
 \end{smallmatrix}\right)\left(\begin{smallmatrix}
    \beta_1\\ \beta_2
\end{smallmatrix}\right)\\
&=(\alpha_1^{\rm tr},\alpha_2^{\rm tr})\left(\begin{smallmatrix}
    0~& -E\\
    E^{\rm tr}~& -DB\\
 \end{smallmatrix}\right)\left(\begin{smallmatrix}
    \beta_1\\ \beta_2
\end{smallmatrix}\right)\\
&=\alpha_2^{\rm tr}E^{\rm tr}\beta_1-\alpha_1^{\rm tr}E\beta_2-\alpha_2^{\rm tr}DB\beta_2\\
&=\lr{\beta_1,\alpha_2}-\lr{\alpha_1,\beta_2}-\alpha_2^{\rm tr}DB\beta_2.
\end{flalign*}
Since $DB=D(I-R)-D(I-R')=E^{\rm tr}-E$, we obtain
$$\alpha_2^{\rm tr}DB\beta_2=\lr{\beta_2,\alpha_2}-\lr{\alpha_2,\beta_2}.$$
Hence, we finish the proof.
\end{proof}

\begin{corollary}\label{1.21}
For any $M,N\in{\A}$, we have the following equation in $\mathcal{T}_{\Lambda_1}:$
\begin{equation*}
\begin{split}
X_{M\oplus N}&=\sum\limits_{[D],[A],[I],[L],[R],[S]}v^{\lr{{\bf m}-{\bf a},{\bf n}-{\bf a}}}q^{\lr{{\bf s},{\bf d}}-\lr{{\bf a},{\bf i}}}|_D\Hom_{\widetilde{\A}}(N,\tau M)_{\tau A\oplus I}|\\&\quad\quad\quad\quad\quad\quad\quad\quad|{_S\Hom_{\widetilde{\A}}(A,I)_{R}}|H_{SD}^LX_{L\oplus R[-1]}\\
&=\sum\limits_{[C],[B],[J],[L],[R],[S]}v^{\lr{{\bf n}-{\bf b},{\bf m}-{\bf b}}}
q^{\lr{{\bf s},{\bf c}}-\lr{{\bf b},{\bf j}}}|_C\Hom_{\widetilde{\A}}(M,\tau N)_{\tau B\oplus J}|\\&\quad\quad\quad\quad\quad\quad\quad\quad|{_S\Hom_{\widetilde{\A}}(B,J)_{R}}|H_{SC}^LX_{L\oplus R[-1]},
\end{split}
\end{equation*}
where each $A$ has the same maximal projective direct summand as $M$, each $B$ has the same maximal projective direct summand as $N$, and each $I,J,R\in\I_{\widetilde{\A}}$.
\end{corollary}
\begin{proof}
By \eqref{sxhomjs}, we have $$|_D\Hom_{\widetilde{\A}}(N,\tau M)_{\tau A\oplus I}|
=\sum\limits_{[X],[Y]}a_YF_{AX}^MF_{YD}^NF_{IY}^{\tau X},$$
where each $X$ has no nonzero projective direct summands.
Since $M,N\in\A$, we obtain that $A,D,X,Y\in\A$. However, $\tau X$ may be in $\widetilde{\A}$ and thus so is $I$. By Lemma \ref{zhishu1}, we get
$\Lambda_1(({\bf m}+{\bf n}-{\bf d})^\ast,{\bf a}^\ast)=0.$ Note that
\begin{flalign*}
\Lambda_1({\bf i}^\ast,{\bf a}^\ast)=\Lambda_1(^\ast{\bf i},{^\ast{\bf a}})=\Lambda_1(^\ast(\tau{\bf x}-{\bf y}),{^\ast{\bf a}})=\Lambda_1(^\ast(\tau{\bf x}),{^\ast{\bf a}})
\end{flalign*}
and $^\ast(\tau{\bf x})=-{\bf x}^\ast$. Thus, we obtain
\begin{flalign*}\Lambda_1({\bf i}^\ast,{\bf a}^\ast)=-\Lambda_1({\bf x}^\ast,{^\ast{\bf a}})&=-\Lambda_1({\bf x}^\ast,{\bf a}^\ast)+\Lambda_1(E'(\widetilde{Q}){\bf x},B(\widetilde{Q}){\bf a})\\
&=-\lr{{\bf x},{\bf a}}=-\lr{{\bf m}-{\bf a},{\bf a}}.\end{flalign*}
Therefore, by Proposition \ref{2cycfgs}, we complete the proof.
\end{proof}

Now, let us provide the following examples to illustrate the equation in Corollary \ref{1.21}.
\begin{example}
    Let $Q:1\to 2$ be the quiver of type $A_2$ and $\widetilde{Q}$ be the following quiver
    \[
    \xymatrix{3\ar[d] &4\ar[d]\\ 1\ar[r]&2.}
    \]
    Then we have 
    $$E(\widetilde{Q})={\tiny\begin{bmatrix}
        1&-1&0&0\\ 0&1&0&0\\ -1&0&1&0\\ 0&-1&0&1
    \end{bmatrix}}~~~~\text{and}~~~~E'(\widetilde{Q})={\tiny\begin{bmatrix}
        1&0&-1&0\\ -1&1&0&-1\\ 0&0&1&0\\ 0&0&0&1
    \end{bmatrix}}.$$

    Let $M=S_1$ and $N=S_2$. It is easy to see that $\tau M=P_4$ and ${\rm dim}_{\mathbb{F}_q}\Hom_{\widetilde{\A}}(N,\tau M)=1$. Thus, $|_D\Hom_{\widetilde{\A}}(N,\tau M)_{\tau A\oplus I}|\neq 0$ if and only if $D=N$, $A=M$, $I=0$ or $D=0,A=0, I=I_4$. Moreover, $|_N\Hom_{\widetilde{\A}}(N,\tau M)_{\tau M}|=1$ and $|_0\Hom_{\widetilde{\A}}(N,\tau M)_{I_4}|=q-1$. Hence, we obtain
    \begin{align*}
        {\rm LHS:}&=\sum\limits_{[D],[A],[I],[L],[R],[S]}v^{\lr{{\bf m}-{\bf a},{\bf n}-{\bf a}}}q^{\lr{{\bf s},{\bf d}}-\lr{{\bf a},{\bf i}}}|_D\Hom_{\widetilde{\A}}(N,\tau M)_{\tau A\oplus I}|\\&\quad\quad\quad\quad\quad\quad\quad\quad|{_S\Hom_{\widetilde{\A}}(A,I)_{R}}|H_{SD}^LX_{L\oplus R[-1]}\\
        &=q^{-1}X_{S_1\oplus S_2}+q^{-1}(q-1)X_{P_1}+v^{-1}(q-1)X_{I_4[-1]}.
    \end{align*}
    Similarly, we have
    \begin{align*}
        {\rm RHS:}&=\sum\limits_{[C],[B],[J],[L],[R],[S]}v^{\lr{{\bf n}-{\bf b},{\bf m}-{\bf b}}}
q^{\lr{{\bf s},{\bf c}}-\lr{{\bf b},{\bf j}}}|_C\Hom_{\widetilde{\A}}(M,\tau N)_{\tau B\oplus J}|\\&\quad\quad\quad\quad\quad\quad\quad\quad|{_S\Hom_{\widetilde{\A}}(B,J)_{R}}|H_{SC}^LX_{L\oplus R[-1]}\\
&=X_{S_1\oplus S_2}.
    \end{align*}
    Recall that for any $L\in\A$ and $R\in\I_{\widetilde{\A}}$,
\begin{flalign*}X_{L\oplus R[-1]}=\sum\limits_{\bf e}v^{-\lr{\mathbf{e},{\bf l}-{\bf r}-\mathbf{e}}}|\mathrm{Gr}_{\mathbf{e}}L|
X^{-\mathbf{e}^\ast-^\ast(\mathbf{l}-\mathbf{r}-\mathbf{e})}.
\end{flalign*}
Thus, we get
\begin{align*}
    &X_{S_1\oplus S_2}=X^{-e_2+e_3+e_4}+vX^{e_4}+X^{-e_1-e_2+e_3}+X^{-e_1},\\
    &X_{P_1}=X^{-e_2+e_3+e_4}+X^{-e_1-e_2+e_3}+X^{-e_1},\\
    &X_{I_4[-1]}=X^{e_4}.
\end{align*}
Therefore, we conclude that ${\rm LHS}={\rm RHS}$.
\end{example}

\begin{example}
    Let $Q:1\to2\to 3$ be the quiver of type $A_3$ and $\widetilde{Q}$ be the following quiver
    \[
    \xymatrix{4\ar[d]&5\ar[d]&6\ar[d]\\ 1\ar[r]&2\ar[r]&3.}
    \]
    Then we have
    $$E(\widetilde{Q})={\tiny\begin{bmatrix}
        1&-1&0&0&0&0\\
        0&1&-1&0&0&0\\
        0&0&1&0&0&0\\
        -1&0&0&1&0&0\\
        0&-1&0&0&1&0\\
        0&0&-1&0&0&1
    \end{bmatrix}}~~~~\text{and}~~~~E'(\widetilde{Q})={\tiny\begin{bmatrix}
        1&0&0&-1&0&0\\
        -1&1&0&0&-1&0\\
        0&-1&1&0&0&-1\\
        0&0&0&1&0&0\\
        0&0&0&0&1&0\\
        0&0&0&0&0&1
    \end{bmatrix}}.$$
    Let $M=S_1\oplus S_3$ and $N=S_2$. A direct computation shows that $\tau M$ is the unique indecomposable representation with the dimension vector $e_2+e_5$ and $\tau N$ is the unique indecomposable representation with the dimension vector $e_3+e_6$. Thus, it is easy to see that $|_D\Hom(N,\tau M)_{\tau A\oplus I}|\neq 0$ if and only if $D=N, A=M, I=0$ or $D=0,A=P_3=S_3,I=I_5=S_5$. Moreover, $|_N\Hom(N,\tau M)_{\tau M}|=1$ and $|_0\Hom(N,\tau M)_{I_5}|=q-1$. Hence, we obtain
    \begin{align*}
        {\rm LHS}:&=\sum\limits_{[D],[A],[I],[L],[R],[S]}v^{\lr{{\bf m}-{\bf a},{\bf n}-{\bf a}}}q^{\lr{{\bf s},{\bf d}}-\lr{{\bf a},{\bf i}}}|_D\Hom_{\widetilde{\A}}(N,\tau M)_{\tau A\oplus I}|\\&\quad\quad\quad\quad\quad\quad\quad\quad|{_S\Hom_{\widetilde{\A}}(A,I)_{R}}|H_{SD}^LX_{L\oplus R[-1]}\\
&=q^{-1}X_{S_1\oplus S_2\oplus S_3 }+q^{-1}(q-1)X_{S_3\oplus U}+v^{-1}(q-1)X_{S_3\oplus I_5[-1]},
    \end{align*}
    where $U$ is the unique indecomposable representation with the dimension vector $e_1+e_2$. Similarly, we have
    \begin{align*}
        {\rm RHS:}&=\sum\limits_{[C],[B],[J],[L],[R],[S]}v^{\lr{{\bf n}-{\bf b},{\bf m}-{\bf b}}}
q^{\lr{{\bf s},{\bf c}}-\lr{{\bf b},{\bf j}}}|_C\Hom_{\widetilde{\A}}(M,\tau N)_{\tau B\oplus J}|\\&\quad\quad\quad\quad\quad\quad\quad\quad|{_S\Hom_{\widetilde{\A}}(B,J)_{R}}|H_{SC}^LX_{L\oplus R[-1]}\\
&=q^{-1}X_{S_1\oplus S_2\oplus S_3}+q^{-1}(q-1)X_{S_1\oplus V}+v^{-1}(q-1)X_{S_1\oplus I_6[-1]},
    \end{align*}
   where $V$ is the unique indecomposable representation with the dimension vector $e_2+e_3$.  According to
   \begin{align*}
       X_{L\oplus R[-1]}=\sum\limits_{\bf e}v^{-\lr{\mathbf{e},{\bf l}-{\bf r}-\mathbf{e}}}|\mathrm{Gr}_{\mathbf{e}}L|
X^{-\mathbf{e}^\ast-^\ast(\mathbf{l}-\mathbf{r}-\mathbf{e})},
\end{align*}
  we get
  \begin{align*}
  &X_{S_3\oplus U}=X^{e_4+e_5+e_6-e_3}+vX^{e_4+e_6-e_1}+X^{e_4+e_5-e_2-e_3}+X^{e_4-e_1-e_2}
     +vX^{e_2+e_6-e_1}+X^{-e_1},\\
     & X_{S_3\oplus I_5[-1]}=X^{e_2-e_3+e_5+e_6}+X^{e_5-e_3},~~X_{S_1\oplus I_6[-1]}=X^{e_4+e_6-e_1}+X^{e_2+e_6-e_1},\\
     & X_{S_1\oplus V}=X^{e_4+e_5+e_6-e_3}+vX^{e_2-e_3+e_5+e_6}+X^{e_4+e_5-e_2-e_3}+X^{e_4-e_1-e_2}+vX^{e_5-e_3}+X^{-e_1}.
\end{align*}
Therefore, we conclude that ${\rm LHS}={\rm RHS}$.
\end{example}

\section{Realizations of quantum cluster algebras via Hall algebras}
In this section, we introduce a Hall algebra associated to cluster categories of hereditary algebras
by using certain quotients of derived Hall subalgebras of hereditary algebras, and then provide Hall algebra realizations of acyclic quantum cluster algebras by applying such Hall algebras.
\subsection{Hall algebras associated to cluster categories of hereditary algebras}
Let $\mathfrak{I}$ be the two-sided ideal of the Hall algebra $\mathcal {D}\mathcal {H}_\Lambda^{{cl}_1}(\widetilde{\A})$ generated by the elements $$\{u_{M\oplus N}-\sum\limits_{\begin{smallmatrix}[D],[A],[I]\end{smallmatrix}}v^{\Lambda(({\bf m}+{\bf n})^*,{\bf a}^*)+\lr{{\bf m}-{\bf a},{\bf n}}}|_D\Hom_{\widetilde{\A}}(N,\tau M)_{\tau A\oplus I}|u_{A}\star u_{D\oplus I[-1]}~|~M,N\in{\widetilde{\A}}\},$$
where each $A$ has the same maximal projective direct summand as $M$, and each $I\in\I_{\widetilde{\A}}$.

\begin{definition}\label{Halldef}
The {\em Hall algebra associated to cluster categories} is defined to be the algebra \begin{equation*}\mathcal {D}\mathcal {H}_\Lambda^{{cl}}(\widetilde{\A}):=\mathcal {D}\mathcal {H}_\Lambda^{{cl}_1}(\widetilde{\A})/\mathfrak{I}.\end{equation*}
\end{definition}
Then we have the following.
\begin{proposition}\label{prop:surjective-hom}
The map $\varphi:\mathcal {D}\mathcal {H}_\Lambda^{{cl}}(\widetilde{\A})\longrightarrow\mathcal{A}\mathcal{H}_\Lambda^{\circ}(\widetilde{Q})$ defined by
$$\varphi(u_{I[-1]\oplus M\oplus P[1]})=X_{I[-1]\oplus M\oplus P[1]},$$ for any $M\in\widetilde{\A}$, $I\in\I_{\widetilde{\A}}$ and $P\in\P_{\widetilde{\A}}$, is a surjective homomorphism of algebras.
\end{proposition}
\begin{proof}
It is proved by Theorem \ref{zhihe} and Proposition \ref{morphism1}.
\end{proof}

The following equation is the presentation of the high-dimensional quantum cluster multiplication formula \eqref{zhxl} in the Hall algebra $\mathcal {D}\mathcal {H}_\Lambda^{{cl}}(\widetilde{\A})$.
\begin{proposition}\label{hccfgs}
For any $M,N\in\widetilde{\A}$, we have the following equation in $\mathcal {D}\mathcal {H}_\Lambda^{{cl}}(\widetilde{\A}):$
\begin{equation}
\begin{split}
&(q^{[M,N]^1}-1)u_M\star u_N=q^{\frac{1}{2}\Lambda({\bf m}^*,{\bf n}^*)}\sum\limits_{[E]\neq [M\oplus N]}
|\Ext_{\widetilde{\A}}^1(M,N)_E|u_E+\\&\sum\limits_{\begin{smallmatrix}[D],[A],[I]\\ [D]\neq [N]\end{smallmatrix}}q^{\frac{1}{2}\Lambda(({\bf m}-{{\bf a}})^*,{({\bf n}+{\bf a})}^*)+\frac{1}{2}\lr{{{\bf m}}-{{\bf a}},{\bf n}}}|_D\Hom_{\widetilde{\A}}(N,\tau M)_{\tau A\oplus I}|u_{A}\star u_{D\oplus I[-1]},
\end{split}\end{equation}
where each $A$ has the same maximal projective direct summand as $M$, and each $I\in\I_{\widetilde{\A}}$.
\end{proposition}
\begin{proof}
By the ideal $\mathfrak{I}$, we have
\begin{equation*}
\begin{split}
v^{\Lambda({\bf m}^*,{\bf n}^*)}u_{M\oplus N}=\sum\limits_{\begin{smallmatrix}[D],[A],[I]\end{smallmatrix}}v^{\Lambda(({\bf m}-{{\bf a}})^*,{({\bf n}+{\bf a})}^*)+\lr{{{\bf m}}-{{\bf a}},{\bf n}}}|_D\Hom_{\widetilde{\A}}(N,\tau M)_{\tau A\oplus I}|u_{A}\star u_{D\oplus I[-1]},\end{split}\end{equation*}
where each $A$ has the same maximal projective direct summand as $M$, and each $I\in\I_{\widetilde{\A}}$.
By \eqref{lgx3}, we have
\begin{equation}\label{hhmgx}
\begin{split}
&q^{[M,N]^1}u_M\star u_N=q^{\frac{1}{2}\Lambda({\bf m}^*,{\bf n}^*)}\sum\limits_{[E]}
|\Ext_{\widetilde{\A}}^1(M,N)_E|u_E\\
&=q^{\frac{1}{2}\Lambda({\bf m}^*,{\bf n}^*)}\sum\limits_{[E]\neq [M\oplus N]}
|\Ext_{\widetilde{\A}}^1(M,N)_E|u_E+\\&\sum\limits_{\begin{smallmatrix}[D],[A],[I]\end{smallmatrix}}q^{\frac{1}{2}\Lambda(({\bf m}-{{\bf a}})^*,{({\bf n}+{\bf a})}^*)+\frac{1}{2}\lr{{{\bf m}}-{{\bf a}},{\bf n}}}|_D\Hom_{\widetilde{\A}}(N,\tau M)_{\tau A\oplus I}|u_{A}\star u_{D\oplus I[-1]}.\end{split}\end{equation}
Note that for the zero morphism in $\Hom_{\widetilde{\A}}(N,\tau M)$, we have $D=N, A=M$ and $I=0$.
Then $v^{\Lambda(({\bf m}-{{\bf a}})^*,{({\bf n}+{\bf a})}^*)+\lr{{{\bf m}}-{{\bf a}},{\bf n}}}u_{A}\star u_{D\oplus I[-1]}=u_M\star u_N$.
Hence, moving the term corresponding to the zero morphism in $\Hom_{\widetilde{\A}}(N,\tau M)$ to the left hand side of \eqref{hhmgx}, we finish the proof.
\end{proof}

Using \cite[Lemma 7.3]{CDZ},  we reformulate Proposition \ref{hccfgs} as the following.
\begin{corollary}\label{maincor}
For any $M,N\in\widetilde{\A}$, we have the following equation in $\mathcal {D}\mathcal {H}_\Lambda^{{cl}}(\widetilde{\A}):$
\begin{equation*}
\begin{split}
&(q^{[M,N]^1}-1)u_M\star u_N=q^{\frac{1}{2}\Lambda({\bf m}^*,{\bf n}^*)}\sum\limits_{[E]\neq [M\oplus N]}
|\Ext_{\widetilde{\A}}^1(M,N)_E|u_E+\\&q^{\frac{1}{2}\Lambda({\bf m}^*,{\bf n}^*)+\frac{1}{2}\lr{{\bf m},{\bf n}}}\sum\limits_{\begin{smallmatrix}[D],[A],[I]\\ [D]\neq [N]\end{smallmatrix}}q^{-\frac{1}{2}\lr{{\bf a},{\bf d}-{\bf i}}-\frac{1}{2}\Lambda({\bf a}^*,({\bf d}-{\bf i})^*)}|_D\Hom_{\widetilde{\A}}(N,\tau M)_{\tau A\oplus I}|u_{A}\star u_{D\oplus I[-1]},
\end{split}\end{equation*}
where each $A$ has the same maximal projective direct summand as $M$, and each $I\in\I_{\widetilde{\A}}$.
\end{corollary}

\begin{corollary}\label{mutatione}
Let $M,N\in\widetilde{\A}$ with a unique (up to scalar) non-trivial extension $E\in\Ext_{\widetilde{\A}}^1(M,N)$; in particular $\dim_{{\rm End}\,(M)}\Ext_{\widetilde{\A}}^1(M,N)=1$.
Let $\theta\in\Hom_{\widetilde{\A}}(N,\tau M)$ be the equivalent morphism with $A,D,I$ as before. Then we have in $\mathcal {D}\mathcal {H}_\Lambda^{{cl}}(\widetilde{\A}):$
\begin{equation}
u_M\star u_N=q^{\frac{1}{2}\Lambda({\bf m}^*,{\bf n}^*)}u_E+q^{\frac{1}{2}\Lambda({\bf m}^*,{\bf n}^*)+\frac{1}{2}\lr{{\bf m},{\bf n}}-\frac{1}{2}\lr{{\bf a},{\bf d}-{\bf i}}-\frac{1}{2}\Lambda({\bf a}^*,({\bf d}-{\bf i})^*)}u_{A}\star u_{D\oplus I[-1]}.
\end{equation}
In particular, if $\Hom_{\widetilde{\A}}(A,I)=0=\Ext^1_{\widetilde{\A}}(A,D)$, then
\begin{equation}
u_M\star u_N=q^{\frac{1}{2}\Lambda({\bf m}^*,{\bf n}^*)}u_E+q^{\frac{1}{2}\Lambda({\bf m}^*,{\bf n}^*)+\frac{1}{2}\lr{{\bf m},{\bf n}}-\frac{1}{2}\lr{{\bf a},{\bf d}}} u_{A\oplus D\oplus I[-1]}.
\end{equation}
\end{corollary}

\subsection{A spanning set of the Hall algebra $\mathcal {D}\mathcal {H}_\Lambda^{{cl}}(\widetilde{\A})$}
In this subsection, let us give the following spanning set of the Hall algebra $\mathcal {D}\mathcal {H}_\Lambda^{{cl}}(\widetilde{\A})$.
\begin{proposition}\label{prop:span-set}
The algebra $\mathcal {D}\mathcal {H}_\Lambda^{{cl}}(\widetilde{\A})$ is spanned by the elements in
\begin{flalign*}\mathfrak{\widetilde{B}}:=&\{u_{M\oplus I[-1]}~|~I\in\I_{\widetilde{\A}},~M=M_1\oplus \cdots \oplus M_t\in\widetilde{\A},~M_1,\ldots,M_t~\text{are~indecomposable}\\
&\text{such that}~\Ext^1_{\widetilde{\A}}(M_i,M_j)=0~\text{for~any}~ i\neq j, t\in\mathbb{N}\,  \text{and}~\Hom_{\widetilde{\A}}(M,I)=0\}.\end{flalign*}
\end{proposition}
\begin{proof}
Set $\mathfrak{\widetilde{B}}_1:=\{u_{M\oplus I[-1]}~|~I\in\I_{\widetilde{\A}},~M\in\widetilde{\A}\}$.
Firstly, let us prove the first claim.

\noindent{\bf Claim 1:} The algebra $\mathcal {D}\mathcal {H}_\Lambda^{{cl}}(\widetilde{\A})$ is spanned by the elements in $\mathfrak{\widetilde{B}}_1$.

For any $M\in\widetilde{\A}$, $I\in\I_{\widetilde{\A}}$ and $P\in\P_{\widetilde{\A}}$, according to the proof of \cite[Corollary 4.12]{CDZ}, in $\mathcal {D}\mathcal {H}_\Lambda^{{cl}}(\widetilde{\A})$ we have
$$u_{I[-1]\oplus M\oplus P[1]}=v^{\Lambda({\bf i}^\ast,({\bf m}-{\bf p})^\ast)}u_{I[-1]}\star u_{M\oplus P[1]}.$$
By \cite[Lemma 5.1]{CDZ}, in $\mathcal {D}\mathcal {H}_\Lambda^{{cl}}(\widetilde{\A})$ we have
$$u_{M\oplus P[1]}=v^{-\lr{{\bf m},{\bf j}}}\sum\limits_{[G],[I']}|{}_{G}\Hom_{\widetilde{\A}}(M,J)_{I'}|u_{G\oplus I'[-1]},$$
where $J=\nu(P)\in\I_{\widetilde{\A}}$. Then using \eqref{lgx4} and \eqref{lgx1}, we obtain
$$u_{I[-1]\oplus M\oplus P[1]}=v^{-\lr{{\bf m},{\bf j}}}\sum\limits_{[G],[I']}v^{\Lambda({\bf i}^\ast,({\bf m}-{\bf p}+{\bf i}'-{\bf g})^\ast)}|{}_{G}\Hom_{\widetilde{\A}}(M,J)_{I'}|u_{G\oplus (I\oplus I')[-1]}.$$
Thus, we finish the proof of Claim $1$.

For any $G\in\widetilde{\A}$, set $d_G={\rm dim}_{\mathbb{F}_q}\,G$ and $l_G={\rm dim}_{\mathbb{F}_q}\,\Ext_{\widetilde{\A}}^1(G,G)$.
We order the pairs of nonnegative integers $(d,l)$ lexicographically, i.e.
\begin{equation}\label{order}(d,l)<(d^{\prime},l^{\prime})\iff d<d^{\prime} \text{ or } d=d' \text{ and } l < l^{\prime}.\end{equation}
\begin{flalign*}\text{Set}~\mathfrak{\widetilde{B}}_2:=&\{u_{M\oplus I[-1]}~|~I\in\I_{\widetilde{\A}},~M=M_1\oplus \cdots \oplus M_t\in\widetilde{\A},~M_1,\ldots,M_t~\text{are~indecomposable}\\
&\text{such that}~\Ext^1_{\widetilde{\A}}(M_i,M_j)=0~\text{for~any}~ i\neq j\}.\end{flalign*}
Secondly, let us prove the second claim.

\noindent{\bf Claim 2:} For any $G\in \widetilde{\A}$ and $J\in \mathcal{I}_{\widetilde{\A}}$, $u_{G\oplus J[-1]}\in \mathcal {D}\mathcal {H}_\Lambda^{{cl}}(\widetilde{\A})$ is a $\mathbb{Q}(v)$-linear combination of the elements $u_{M\oplus I[-1]}$ in $\mathfrak{\widetilde{B}}_2$ with $d_M\leq d_G$.

By \eqref{lgx1} and \eqref{lgx4}, it suffices to prove that $u_G$ is a $\mathbb{Q}(v)$-linear combination of the elements $u_{M\oplus I[-1]}$ in $\mathfrak{\widetilde{B}}_2$ with $d_M\leq d_G$ for any $G\in\widetilde{\A}$.

We use the order in \eqref{order} to proceed induction on $(d_G,l_G)$ to prove Claim $2$. The claim holds trivially if $G \cong 0$, so assume $G \not\cong 0$.
Let $G=M_1\oplus \cdots \oplus M_t\in\widetilde{\A}$ such that $M_1,\ldots,M_t$~{are~indecomposable}.

Suppose that $t=1$, i.e. $G$ is indecomposable, then $u_G\in\mathfrak{\widetilde{B}}_2$.

Suppose that $t>1$ and $\Ext^1_{\widetilde{\A}}(M_i,M_j)\neq0$ for some $1\leq i\neq j\leq t$. Let us set $M=\bigoplus\limits_{1\leq k\neq j\leq t}M_k~\text{and}~N=M_j.$
Then in $\mathcal {D}\mathcal {H}_\Lambda^{{cl}}(\widetilde{\A})$ we have
\begin{flalign*}
u_G=u_{M\oplus N}&=\sum\limits_{\begin{smallmatrix}[D],[A],[I]\end{smallmatrix}}v^{\Lambda(({\bf m}+{\bf n})^*,{\bf a}^*)+\lr{{\bf m}-{\bf a},{\bf n}}}|_D\Hom_{\widetilde{\A}}(N,\tau M)_{\tau A\oplus I}|u_{A}\star u_{D\oplus I[-1]}\\
&=q^{-[M,N]^1}u_{M\oplus N}+q^{-[M,N]^1}\sum\limits_{[L]:~l_L<l_{G}}|\Ext^1_{\widetilde{\A}}(M,N)_L|u_L+\\
&\sum\limits_{\begin{smallmatrix}[D],[A],[I],[L],[R]\\ d_D<d_N,d_A<d_M\end{smallmatrix}}v^{x_0+\Lambda({\bf r}^\ast,{\bf l}^\ast)}|_D\Hom_{\widetilde{\A}}(N,\tau M)_{\tau A\oplus I}|H_{A,D\oplus I[-1]}^{L\oplus R[-1]}u_{R[-1]}\star u_{L}
\end{flalign*}
where $x_0=\Lambda(({\bf m}+{\bf n})^*,{\bf a}^*)+\Lambda({\bf a}^\ast,({\bf d}-{\bf i})^\ast)+\lr{{\bf m}-{\bf a},{\bf n}}+2\lr{{\bf a},{\bf d}-{\bf i}}.$
Thus, we get \begin{flalign*}
u_{G}&=
\frac{q^{-[M,N]^1}}{1-q^{-[M,N]^1}}\sum\limits_{[L]:~l_L<l_{G}}|\Ext^1_{\widetilde{\A}}(M,N)_L|u_L+\\&\frac{1}{1-q^{-[M,N]^1}}
\sum\limits_{\begin{smallmatrix}[D],[A],[I],[L],[R]\\ d_D<d_N,d_A<d_M\end{smallmatrix}}v^{x_0+\Lambda({\bf r}^\ast,{\bf l}^\ast)}|_D\Hom_{\widetilde{\A}}(N,\tau M)_{\tau A\oplus I}|H_{A,D\oplus I[-1]}^{L\oplus R[-1]}u_{R[-1]}\star u_{L},
\end{flalign*}
where $d_L\leq d_G$ and $(d_L,l_L)<(d_{G},l_{G})$ for each $L$ in both sums as above.
Hence, by induction, we finish the proof of Claim $2$.

In order to finish the proof of Proposition \ref{prop:span-set}, by Claim $1$ and Claim $2$, it suffices to show the following third claim.

\noindent{\bf Claim 3:} Each $u_{M\oplus I[-1]}\in \mathfrak{\widetilde{B}}_2$ is a $\mathbb{Q}(v)$-linear combination of the elements in $\mathfrak{\widetilde{B}}$.

For any $u_{M\oplus I[-1]}\in \mathfrak{\widetilde{B}}_2$, if $\Hom_{\widetilde{\A}}(M,I)=0$, then $u_{M\oplus I[-1]}\in \mathfrak{\widetilde{B}}$.
Hence, we assume that $\Hom_{\widetilde{\A}}(M,I)\neq 0$. We proceed the proof by induction on $d_M$.

By \cite[Lemma 5.1]{CDZ}, in $\mathcal {D}\mathcal {H}_\Lambda^{{cl}}(\widetilde{\A})$ we have
\begin{flalign*}
    u_{M\oplus I[-1]}&=v^{-\langle \mathbf{p},\mathbf{m}\rangle}\sum_{[F],[P']}|_{P'}\Hom_{\widetilde{\A}}(P,M)_F|u_{F\oplus P'[1]}\\
    &=\sum_{[F],[G],[P'],[I'']}v^{-\langle \mathbf{p},\mathbf{m}\rangle-\langle \mathbf{f},\mathbf{i}'\rangle}|_{P'}\Hom_{\widetilde{\A}}(P,M)_F|\cdot|_G\Hom_{\widetilde{\A}}(F,I')_{I''}|u_{G\oplus I''[-1]}\\
    &=q^{-\langle \mathbf{m},\mathbf{i}\rangle}u_{M\oplus I[-1]}+\\
    &\quad\sum_{\substack{[F],[G],[P'],[I'']\\ d_G<d_M}}v^{-\langle \mathbf{p},\mathbf{m}\rangle-\langle \mathbf{f},\mathbf{i}'\rangle}|_{P'}\Hom_{\widetilde{\A}}(P,M)_F|\cdot|_G\Hom_{\widetilde{\A}}(F,I')_{I''}|u_{G\oplus I''[-1]},
\end{flalign*}
where $P=\nu^{-1}I$ and $I'=\nu P'$. It follows that
\begin{equation}\label{eq:M-I-decreasing}
\begin{split}
   & u_{M\oplus I[-1]}=\\&\frac{1}{1-q^{-\langle \mathbf{m},\mathbf{i}\rangle}}\sum_{\substack{[F],[G],[P'],[I'']\\ d_G<d_M}}v^{-\langle \mathbf{p},\mathbf{m}\rangle-\langle \mathbf{f},\mathbf{i}'\rangle}
    |_{P'}\Hom_{\widetilde{\A}}(P,M)_F|\cdot|_G\Hom_{\widetilde{\A}}(F,I')_{I''}|u_{G\oplus I''[-1]}.
    \end{split}
\end{equation}
By applying Claim 2 to the terms $u_{G\oplus I''[-1]}$ in \eqref{eq:M-I-decreasing}, we obtain that each such term  is a $\mathbb{Q}(v)$-linear combination of the elements $u_{N\oplus J[-1]}$ in $\mathfrak{\widetilde{B}}_2$ with $d_N\leq d_G< d_M$. Then by induction, we obtain that each $u_{N\oplus J[-1]}$ is a $\mathbb{Q}(v)$-linear combination of the elements in $\mathfrak{\widetilde{B}}$. Therefore, we complete the proof.
\end{proof}

\begin{remark}
 The condition that $\Ext^1_{\widetilde{\A}}(M_i, M_j) = 0$ for any $i \neq j$ appearing in the spanning set $\mathfrak{\widetilde{B}}$ is intimately related to the theory of canonical decomposition of dimension vectors for hereditary algebras. By the fundamental results of Kac \cite{Kac} and Schofield \cite{Sc}, any dimension vector $\alpha$ of a quiver admits a unique canonical decomposition $\alpha = \sum \alpha_i$ such that a generic representation $M$ of the dimension vector $\alpha$ decomposes as $M \cong \bigoplus M_i$ with ${\bf dim}\, M_i = \alpha_i$ and $\Ext^1(M_i, M_j) = 0$ for any $i \neq j$.
\end{remark}
Let $\mathcal {D}\mathcal {H}_\Lambda^{{\tilde{c}l}}({\A})$ be the subspace of $\mathcal {D}\mathcal {H}_\Lambda^{{cl}}(\widetilde{\A})$ spanned by the elements $u_{I[-1]\oplus M\oplus P[1]}$ with $M\in\A$, $I\in\mathcal{I}_{\widetilde{\A}}$ and $P\in\mathcal{P}_{\widetilde{\A}}$. It is easy to see that the subcategory consisting of all the objects $I[-1]\oplus M\oplus P[1]$, where $M\in\A$, $I\in\mathcal{I}_{\widetilde{\A}}$ and $P\in\mathcal{P}_{\widetilde{\A}}$, is  closed under extensions in $\mathcal{D}^b(\widetilde{\A})$. Thus, $\mathcal {D}\mathcal {H}_\Lambda^{{\tilde{c}l}}({\A})$ is a subalgebra of $\mathcal {D}\mathcal {H}_\Lambda^{{cl}}(\widetilde{\A})$.
Using a similar proof as proving Proposition \ref{prop:span-set}, we have the following.
\begin{proposition}\label{sprop:span-set}
The algebra $\mathcal {D}\mathcal {H}_\Lambda^{{\tilde{c}l}}({\A})$ is spanned by the elements in
\begin{flalign*}\mathfrak{B}:=&\{u_{M\oplus I[-1]}~|~I\in\I_{\widetilde{\A}},~M=M_1\oplus \cdots \oplus M_t\in\A,~M_1,\ldots,M_t~\text{are~indecomposable}\\
&\text{such that}~\Ext^1_{\A}(M_i,M_j)=0~\text{for~any}~ i\neq j, t\in\mathbb{N}\, \text{and}~\Hom_{\widetilde{\A}}(M,I)=0\}.\end{flalign*}
\end{proposition}

Define $\mathcal{A}\mathcal{H}_\Lambda(Q)$ to be the subalgebra of $\mathcal{T}_\Lambda$ generated by all the quantum cluster characters $X_M, X_{I[-1]}$ with $M\in\A, I\in\mathcal{I}_{\widetilde{\A}}$. By \cite{Rupel2}, the specialised quantum cluster algebra $\A_q(\Lambda,\widetilde{B})$ is the subalgebra of $\mathcal{T}_\Lambda$ generated by the quantum cluster characters $X_{M}, X_{I[-1]}$, where $M\in \A$ is indecomposable and rigid, and $I\in \mathcal{I}_{\widetilde{\A}}$ is indecomposable.

By considering the restriction of $\varphi$ on the subalgebra $\mathcal {D}\mathcal {H}_\Lambda^{{\tilde{c}l}}({\A})$, we obtain a surjective homomorphism of algebras $\varphi_0:\mathcal {D}\mathcal {H}_\Lambda^{{\tilde{c}l}}({\A})\to \mathcal{A}\mathcal{H}_\Lambda(Q)$ defined by
$\varphi(u_{M\oplus I[-1]})=X_{M\oplus I[-1]}$ for any $M\in\A$ and $I\in\I_{\widetilde{\A}}$.
\begin{corollary}
The algebra $\mathcal{A}\mathcal{H}_\Lambda(Q)$ is spanned by the elements
\begin{flalign*}&\{X_{M\oplus I[-1]}~|~I\in\I_{\widetilde{\A}},~M=M_1\oplus \cdots \oplus M_t\in\A,~M_1,\ldots,M_t~\text{are~indecomposable}\\
&\text{such that}~\Ext^1_{\A}(M_i,M_j)=0~\text{for~any}~ i\neq j, t\in\mathbb{N}\, \text{and}~\Hom_{\widetilde{\A}}(M,I)=0\}.\end{flalign*}
\end{corollary}
\begin{proof}
It is proved by Proposition \ref{sprop:span-set} together with the algebra homomorphism $\varphi_0$.
\end{proof}

\subsection{Realizing quantum cluster algebras via Hall algebras: Dynkin quivers}
First of all, for valued quivers of Dynkin types, we show that $\varphi_0$ is an algebra isomorphism in the following.

\begin{theorem}\label{thm:hall-quantum-dynkin}
Assume that $Q$ is of Dynkin type. Then we have the following:
    \begin{itemize}
        \item[(1)] The set $\{u_{M\oplus I[-1]}\mid I\in \mathcal{I}_{\widetilde{\A}}, M\in{\A}~\text{is rigid such that}~ \Hom_{\widetilde{\A}}(M,I)=0 \}$ is a $\mathbb{Q}(v)$-basis of $\mathcal {D}\mathcal {H}_\Lambda^{{\tilde{c}l}}({\A})$.
        \item[(2)]  The homomorphism $\varphi_0:\mathcal {D}\mathcal {H}_\Lambda^{{\tilde{c}l}}({\A})\to \mathcal{A}\mathcal{H}_\Lambda(Q)$ is an isomorphism of algebras.
    \end{itemize}
    Consequently, $\mathcal{A}\mathcal{H}_\Lambda(Q)=\A_q(\Lambda,\widetilde{B})$ and the set of quantum cluster monomials $$\{X_{M\oplus I[-1]}\mid I\in \mathcal{I}_{\widetilde{\A}}, M\in{\A}~\text{is rigid such that}~ \Hom_{\widetilde{\A}}(M,I)=0 \}$$ forms a basis of $\A_q(\Lambda,\widetilde{B})$.
\end{theorem}
\begin{proof}
Since $Q$ is of Dynkin type, each indecomposable representation $M\in \A$ is rigid. Thus, the set $\mathfrak{B}$ can be rewritten as
    \[
    \mathfrak{B}=\{u_{M\oplus I[-1]}\mid I\in \mathcal{I}_{\widetilde{\A}}, M\in \A~\text{is rigid such that}~ \Hom_{\widetilde{\A}}(M,I)=0 \}.
    \]
    By Lemma \ref{xjlm}, the set $\{\varphi_0(u_{M\oplus I[-1]})\mid u_{M\oplus I[-1]}\in \mathfrak{B}\}$ is $\mathbb{Q}(v)$-linearly independent. Thus, $\mathfrak{B}$ is $\mathbb{Q}(v)$-linearly independent.
    By Proposition \ref{sprop:span-set}, $\mathfrak{B}$ is a $\mathbb{Q}(v)$-basis of $\mathcal {D}\mathcal {H}_\Lambda^{{\tilde{c}l}}({\A})$. Since $\varphi_0$ sends the basis $\mathfrak{B}$ to a set of linearly independent elements, we get that $\varphi_0$ is injective, and then is an algebra isomorphism. Hence, we finish the proof.
\end{proof}
\begin{remark}
For Dynkin quivers, the equation $\mathcal{A}\mathcal{H}_\Lambda(Q)=\A_q(\Lambda,\widetilde{B})$ has been proved in \cite{D}; it has been known in \cite{DX} that the quantum cluster monomials provide a basis of the quantum cluster algebra.
\end{remark}
\begin{remark}
  For Dynkin quivers, since the Hall polynomials always exist, one may define a generic version of the Hall algebra $\mathcal {D}\mathcal {H}_\Lambda^{{\tilde{c}l}}({\A})$. Consequently, Theorem \ref{thm:hall-quantum-dynkin} can be extended to the quantum cluster algebra for the indeterminate $\mathfrak{q}$ without requiring the specialization at $\mathfrak{q}=q$.
\end{remark}
\begin{remark}
   For quantum cluster algebras of finite types, the quantum cluster monomials give a basis, and the proof of Proposition \ref{prop:span-set} provides a constructive algorithm to express the products of quantum cluster monomials as linear combinations of this basis.
\end{remark}
\subsection{Realizing quantum cluster algebras via Hall algebras: acyclic quivers}
In this subsection, we show that the (specialised) quantum cluster algebra $\A_q(\Lambda,\widetilde{B})$ is a subalgebra of the Hall algebra $\mathcal {D}\mathcal {H}_\Lambda^{{\tilde{c}l}}({\A})$ up to isomorphism.

\begin{lemma}\label{injlem}
    Let $1\leq j\neq k\leq m$ such that $\Hom_{\widetilde{{\A}}}(S_j,\tau S_k)\neq 0$, then $\tau S_k/\im f\in\mathcal{I}_{\widetilde{\A}}$ for any nonzero morphism $f:S_j\to \tau S_k$. Moreover, the representation $\tau S_k/\im f$ is independent of the choice of $f$ up to isomorphism, thus can be denoted by $\tau S_k/S_j$.
\end{lemma}
\begin{proof}
Let $f:S_j\to \tau S_k$ be a nonzero morphism. It suffices to show that $$\Ext^1_{\widetilde{\A}}(S_i,\tau S_k/\im f)=0~\text{~for~any~}1\leq i\leq m.$$

For $i\neq k$, applying $\Hom_{\widetilde{\A}}(S_i,-)$ to the exact sequence \begin{equation}\label{ylzhll}0\longrightarrow S_j\stackrel{f}{\longrightarrow} \tau S_k\longrightarrow\tau S_k/\im f\longrightarrow 0,\end{equation} we get the exact sequence \begin{equation*}\xymatrix{\Ext^1_{\widetilde{\A}}(S_i,S_j)\ar[r]&\Ext^1_{\widetilde{\A}}(S_i,\tau S_k)\ar[r]&\Ext^1_{\widetilde{\A}}(S_i,\tau S_k/\im f)\ar[r]&0.}\end{equation*}
Since $\Ext^1_{\widetilde{\A}}(S_i,\tau S_k)\cong{\rm D}\Hom_{\widetilde{\A}}(S_k,S_i)=0$, we get $\Ext^1_{\widetilde{\A}}(S_i,\tau S_k/\im f)=0$.

For $i=k$, applying $\Hom_{\widetilde{\A}}(-,\tau S_k)$ to \eqref{ylzhll}, we get the exact sequence \begin{equation*}\xymatrix{0\ar[r]&\Hom_{\widetilde{\A}}(\tau S_k/\im f,\tau S_k)\ar[r]&\Hom_{\widetilde{\A}}(\tau S_k,\tau S_k)\ar[r]^-{f^*}&\Hom_{\widetilde{\A}}(S_j,\tau S_k).}\end{equation*}
Note that $\Hom_{\widetilde{\A}}(\tau S_k,\tau S_k)\cong{\rm End}_{\widetilde{\A}}(S_k)$ is a division algebra, that is, each nonzero $h\in\Hom_{\widetilde{\A}}(\tau S_k,\tau S_k)$ is an isomorphism. Then it is easy to see that $\ker f^*$ is zero, i.e. $\Hom_{\widetilde{\A}}(\tau S_k/\im f,\tau S_k)=0$. Thus, $\Ext^1_{\widetilde{\A}}(S_k,\tau S_k/\im f)\cong{\rm D}\Hom_{\widetilde{\A}}(\tau S_k/\im f,\tau S_k)=0.$

Hence, we conclude that $\tau S_k/\im f\in\mathcal{I}_{\widetilde{\A}}$. Since each injective representation is determined by its dimension vector up to isomorphism,  $\tau S_k/ \im f$ is independent of the choice of $f$ up to isomorphism. Therefore, we complete the proof.
\end{proof}

\begin{lemma}\label{lem:gen-rel-quasi-commutation}
For any $1\leq j\neq k\leq n$, the following statements hold in $\mathcal {D}\mathcal {H}_\Lambda^{{\tilde{c}l}}({\A})$\,$:$
    \begin{itemize}
        \item[(1)] If $b_{jk}=0$, then
        \[
        v^{-\Lambda(e_j^\ast,e_k^\ast)}u_{S_j}\star u_{S_k}= v^{-\Lambda(e_k^\ast,e_j^\ast)}u_{S_k}\star u_{S_j}.
        \]
        \item[(2)] If $b_{jk}<0$, then
        \[
        v^{-\Lambda(e_j^\ast,e_k^\ast)}u_{S_j}\star u_{S_k}-v^{\Lambda(e_j^\ast,e_k^\ast)}u_{S_k}\star u_{S_j}=(v^{-\lr{e_k,e_j}}-v^{\lr{e_k,e_j}})u_{I [-1]},
        \text {where}~I=\tau S_k/S_j.
        \]
    \end{itemize}
\end{lemma}
\begin{proof}
$(1)$~Since $b_{jk}=0$, we have $\Ext^1_{\A}(S_j,S_k)=\Ext^1_{\A}(S_k,S_j)=0.$ Then the first statement follows from \eqref{lgx3}.

$(2)$ Since $b_{jk}<0$, we have $\Ext^1_{\A}(S_k,S_j)\neq 0$ and $\Ext^1_{\A}(S_j,S_k)=0$.
Using the ideal $\mathcal{J}$ relation, we obtain
    \begin{align*}
        u_{S_j}\star u_{S_k}&=v^{\Lambda(e_j^\ast,e_k^\ast)}u_{S_j\oplus S_k}\\
        &=v^{\Lambda(e_j^\ast,e_k^\ast)}\sum_{[D],[A],[I]}v^{\Lambda(e_j^\ast+e_k^\ast,{\bf a}^\ast)+\lr{e_k-{\bf a},e_j}}|_D\Hom_{\widetilde{\A}}(S_j,\tau S_k)_{\tau A\oplus I}|u_A\star u_{D\oplus I[-1]}.
    \end{align*}  Note that $D=0$ or $D=S_j$, since $S_j$ is simple. If $D=0$, then $A=0$ and $I=\tau S_k/S_j$ by Lemma \ref{injlem}; if $D=S_j$, then $A=S_k$ and $I=0$. Thus, the second
statement follows by noting that \[|_0\Hom_{\widetilde{\A}}(S_j,\tau S_k)_{\tau S_k/S_j}|=q^{{\rm dim} {\rm Hom}_{\widetilde{\A}}(S_j,\tau S_k)}-1=q^{-\lr{e_k,e_j}}-1.\]
\end{proof}

It is easy to see that the quantum cluster variables $x_i$, $1\leq i\leq m$ and $x'_j$, $1\leq j\leq n$ are equal to
$X_{I_i[-1]}$, $1\leq i\leq m$ and $X_{S_j}$, $1\leq j\leq n$, respectively (cf. \cite[Proposition 3.2]{FuZh}). Then applying Proposition \ref{prop:quantum-cluster-presentation}, we give the following.

\begin{theorem}\label{chtgthm}
The map $\varphi_0':\A_q(\Lambda,\widetilde{B})\longrightarrow\mathcal {D}\mathcal {H}_\Lambda^{{\tilde{c}l}}({\A})$ defined by $$X_{I_i[-1]}\mapsto u_{I_i[-1]}~\text{and}~X_{S_j}\mapsto u_{S_j}$$ for any $1\leq i\leq m, 1\leq j\leq n$, is an injective homomorphism of algebras.
\end{theorem}
\begin{proof}
By Proposition \ref{prop:quantum-cluster-presentation}, the generators $X_{I_i[-1]}, X_{S_j}$ of $\A_q(\Lambda,\widetilde{B})$ satisfy the defining relations given by \eqref{gen-rel: quasi-comm-1}--\eqref{gen-rel: quasi-commutation}. In order to prove $\varphi_0'$ is an algebra homomorphism, we need to show that the relations \eqref{gen-rel: quasi-comm-1}--\eqref{gen-rel: quasi-commutation} of the generators $X_{I_i[-1]}, X_{S_j}$  are preserved under $\varphi_0'$.

\noindent{\bf Relation \eqref{gen-rel: quasi-comm-1}:}
It follows from \eqref{lgx1} directly.

\noindent{\bf Relation \eqref{gen-rel: quasi-comm-2}:}
For each $1\leq k\leq n$, let $G_k=(g_{ij})$ be the $m\times m$ matrix defined by
\[g_{ij}=\begin{cases}
\delta_{ij} & \text{if $j\ne k$;}\\
-1 & \text{if $i=j=k$;}\\
[b_{ik}]_{+} & \text{if $i\ne j = k$.}
\end{cases}
\]
Then $\mu_k(\Lambda)=G_k^{\rm tr}\Lambda G_k$
and $G_ke_k=-e_k+\sum_{i=1}^m[b_{ik}]_+e_i=-^\ast e_k$.
For any $1\leq j\leq m$ with $j\neq k$, we have
\begin{align*}
    \mu_k(\Lambda)(e_j,e_k)=e_j^{\rm tr}G_k^{\rm tr}\Lambda G_ke_k
    =e_j^{\rm tr}\Lambda G_ke_k
    =-\Lambda(e_j,{^\ast e_k}).
\end{align*}
In the Hall algebra $\mathcal {D}\mathcal {H}_\Lambda^{{\tilde{c}l}}({\A})$, since $\Hom_{\widetilde{\A}}(S_k, I_j)=0$, by \eqref{lgx6}, we obtain
\[
u_{I_j[-1]}\star u_{S_k}=q^{\Lambda(e_k^\ast,\mathbf{i}_j^\ast)}u_{S_k}\star u_{I_{j}[-1]}=q^{\Lambda({^\ast e}_k,e_j)}u_{S_k}\star u_{I_{j}[-1]}=q^{\mu_k(\Lambda)(e_j,e_k)}u_{S_k}\star u_{I_{j}[-1]}.\]

\noindent{\bf Relations \eqref{gen-rel: exchange-relation-left} and \eqref{gen-rel: exchange-relation-right}:}

By \eqref{lgx1}, it is easy to see
\[
q^{\frac{1}{2}\sum_{i<j}a_ia_j\lambda_{ji}}u_{I_1[-1]}^{a_1}\star\cdots \star u_{I_m[-1]}^{a_m}=u_{(a_1I_1\oplus \cdots a_mI_m)[-1]}=u_{(a_1P_1\oplus \cdots a_mP_m)[1]}
\]
for any $(a_1,\ldots, a_m)\in \mathbb{N}^m$.
For each $1\leq k\leq n$, according to \cite[Proposition 5.3]{CDZ}, we have
\begin{flalign}&u_{P_k[1]}\star u_{S_k}=q^{\frac{1}{2}\Lambda(e_k^\ast,\mathbf{p}_k^\ast)}u_{\operatorname{rad}P_k[1]}+q^{\frac{1}{2}\Lambda(e_k^\ast,\mathbf{p}_k^\ast)-\frac{1}{2}\lr{e_k,\mathbf{i}_k}}u_{(I_k/S_k)[-1]},\label{gx48}\\
&u_{S_k}\star u_{I_k[-1]}=q^{\frac{1}{2}\Lambda(\mathbf{i}_k^\ast,e_k^\ast)-\frac{1}{2}\lr{\mathbf{p}_k,e_k}}u_{\operatorname{rad}P_k[1]}+q^{\frac{1}{2}\Lambda(\mathbf{i}_k^\ast,e_k^\ast)}u_{(I_k/S_k)[-1]}\label{gx49}.
\end{flalign}
Note that
\begin{equation}\label{radsoc}
\begin{split}&E'(\widetilde{Q})\Dim\operatorname{rad} P_k={\bf p}_k^\ast-e_k^\ast=e_k-e_k^\ast=\sum_{i=1}^m[-b_{ik}]_+e_i,\\
&E(\widetilde{Q})\Dim I_k/S_k={^\ast{\bf i}}_k-{^\ast e}_k=e_k-{^\ast e}_k=\sum_{i=1}^m[b_{ik}]_+e_i.\end{split}\end{equation}
Thus, we get ${\rad} P_k=\oplus_{i=1}^ma_iP_i$ and $I_k/S_k=\oplus_{i=1}^ma'_iI_i$, where $a_i=[-b_{ik}]_+$ and $a'_i=[b_{ik}]_+$ for $1\leq i\leq m$.
Hence, we obtain \begin{flalign*}&\varphi_0'(X^{\sum_{i=1}^m[-b_{ik}]_+e_i})=q^{\frac{1}{2}\sum_{i<j}a_ia_j\lambda_{ji}}u_{I_1[-1]}^{a_1}\cdots u_{I_m[-1]}^{a_m}
=u_{(a_1P_1\oplus \cdots a_mP_m)[1]}=u_{\operatorname{rad}P_k[1]},\\
&\varphi_0'(X^{\sum_{i=1}^m[b_{ik}]_+e_i})=q^{\frac{1}{2}\sum_{i<j}a'_ia'_j\lambda_{ji}}u_{I_1[-1]}^{a'_1}\cdots u_{I_m[-1]}^{a'_m}=u_{(a'_1I_1\oplus \cdots a'_mI_m)[-1]}=u_{(I_k/S_k)[-1]}.\end{flalign*}

Now, for finishing the proof of \eqref{gen-rel: exchange-relation-left}, it remains to show that
\[
\Lambda(e_k,\sum_{1=1}^m[-b_{ik}]_+e_i)=\Lambda(e_k^\ast,\mathbf{p}_k^\ast)\quad \text{and}\quad \Lambda(e_k,\sum_{1=1}^m[b_{ik}]_+e_i)=\Lambda(e_k^\ast,\mathbf{p}_k^\ast)-\lr{e_k,\mathbf{i}_k}.
\]
In fact,
\begin{align*}
    \Lambda(e_k^\ast, \mathbf{p}_k^\ast)&=\Lambda(e_k^\ast,e_k)
    =\Lambda(e_k-\sum_{i=1}^m[-b_{ik}]_+e_i,e_k)
    =\Lambda(e_k,\sum_{1=1}^m[-b_{ik}]_+e_i).
\end{align*}
Since $\mathbf{b}^k=e_k^\ast-{^\ast e}_k=\sum_{i=1}^m[b_{ik}]_+e_i-\sum_{i=1}^m[-b_{ik}]_+e_i$, we get
\begin{align*}
    \Lambda(e_k,\sum_{1=1}^m[b_{ik}]_+e_i)&=\Lambda(e_k,\mathbf{b}^k)+\Lambda(e_k,\sum_{1=1}^m[-b_{ik}]_+e_i)\\
    &=-d_k+\Lambda(e_k^\ast, \mathbf{p}_k^\ast)\\
    &=\Lambda(e_k^\ast, \mathbf{p}_k^\ast)-\lr{e_k,\mathbf{i}_k}.
\end{align*}
For finishing the proof of \eqref{gen-rel: exchange-relation-right}, it remains to show that
\[
-\Lambda(e_k,\sum_{1=1}^m[b_{ik}]_+e_i)=\Lambda(\mathbf{i}_k^\ast,e_k^\ast)\quad \text{and}\quad -\Lambda(e_k,\sum_{1=1}^m[-b_{ik}]_+e_i)=\Lambda(\mathbf{i}_k^\ast,e_k^\ast)-\lr{\mathbf{p}_k,e_k}.
\]
In fact,
\begin{align*}
\Lambda(\mathbf{i}_k^\ast,e_k^\ast)=\Lambda(B(\widetilde{Q})\mathbf{i}_k+{^\ast\mathbf{i}}_k,e_k^\ast)
    =\lr{e_k,\mathbf{i}_k}+\Lambda(\mathbf{p}_k^\ast,e_k^\ast)
    =-\Lambda(e_k,\sum_{i=1}^m[b_{ik}]_+e_i)
\end{align*}
and
\begin{align*}
    -\Lambda(e_k,\sum_{i=1}^m[-b_{ik}]_+e_i)&=\Lambda(e_k,\mathbf{b}^k)-\Lambda(e_k,\sum_{i=1}^m[b_{ik}]_+e_i)\\
    &=-d_k+\Lambda(\mathbf{i}_k^\ast,e_k^\ast)\\
    &=\Lambda(\mathbf{i}_k^\ast,e_k^\ast)-\lr{\mathbf{p}_k,e_k}.
\end{align*}

\noindent{\bf Relation \eqref{gen-rel: quasi-commutation}:}
If $b_{jk}=0$, then it follows from Lemma \ref{lem:gen-rel-quasi-commutation}(1) directly.
If $b_{jk}<0$, then $\lr{e_j,e_k}=0$. Set $I=\tau S_k/S_j\in\mathcal{I}_{\widetilde{\A}}$ and write $I=\oplus_{i=1}^ma_iI_i$ for some $a_1,\ldots,a_m\in\mathbb{N}$. Note that ${^\ast}{\bf i}={^\ast(\tau S_k)}-{^\ast e_j}=-e_k^\ast-{^\ast e_j}$ and ${^\ast}{\bf i}=\sum\limits_{i=1}^ma_ie_i$. Thus, we obtain
$$\varphi_0'(X^{-e_k^\ast-{^\ast e_j}})=q^{\frac{1}{2}\sum_{i<j}a_ia_j\lambda_{ji}}u_{I_1[-1]}^{a_1}\cdots u_{I_m[-1]}^{a_m}=u_{(a_1I_1\oplus \cdots a_mI_m)[-1]}=u_{I[-1]}.$$
Then the relation \eqref{gen-rel: quasi-commutation} follows from Lemma \ref{lem:gen-rel-quasi-commutation}(2).

Hence, we conclude that $\varphi'_0$ is an algebra homomorphism. Noting that $\varphi_0\varphi'_0={\rm Id}$, we get that $\varphi'_0$ is injective. Therefore, we complete the proof.
\end{proof}

\begin{corollary}
For any $1\leq k\leq n$, the following equations hold in $\mathcal {D}\mathcal {H}_\Lambda^{{\tilde{c}l}}({\A})$\,$:$
\begin{flalign}&v^{\Lambda(e_k^\ast,e_k)}u_{S_k}\star u_{I_k[-1]}-v^{-\Lambda(e_k^\ast,e_k)}u_{P_k[1]}\star u_{S_k}=(v^{d_k}-v^{-d_k})u_{(I_k/S_k)[-1]};\\
&v^{\Lambda(e_k,{^\ast}e_k)}u_{P_k[1]}\star u_{S_k}-v^{-\Lambda(e_k,{^\ast}e_k)}u_{S_k}\star u_{I_k[-1]}=(v^{d_k}-v^{-d_k})u_{{\rad}P_k[1]}.\end{flalign}
\end{corollary}
\begin{proof}
It is proved by the equations \eqref{gx48} and \eqref{gx49}, and $\Lambda({^\ast}e_k,e_k)=\Lambda(e_k^\ast,e_k)-d_k$.
\end{proof}

Let $\mathcal {C}\mathcal {H}_\Lambda^{{\tilde{c}l}}({\A})$ be the subalgebra of $\mathcal {D}\mathcal {H}_\Lambda^{{\tilde{c}l}}({\A})$ generated by the elements $u_{I_i[-1]}$ and $u_{S_j}$, where $1\leq i\leq m$ and $1\leq j\leq n$, called the {\em composition subalgebra} of $\mathcal {D}\mathcal {H}_\Lambda^{{\tilde{c}l}}({\A})$.
It is well-known that $u_M\in\mathcal {C}\mathcal {H}_\Lambda^{{\tilde{c}l}}({\A})$ for any rigid $M\in\A$ (cf. \cite[Theorem 2]{ZhangP}). That is, the composition subalgebra $\mathcal {C}\mathcal {H}_\Lambda^{{\tilde{c}l}}({\A})$ is generated by the elements $u_M$ and $u_{I[-1]}$, where $M\in \A$ is indecomposable and rigid, and $I\in \mathcal{I}_{\widetilde{\A}}$ is indecomposable.
\begin{corollary}\label{congshixian}
The map $\varphi'_0$ gives an algebra isomorphism from $\A_q(\Lambda,\widetilde{B})$ to $\mathcal {C}\mathcal {H}_\Lambda^{{\tilde{c}l}}({\A})$.
\end{corollary}
\begin{proof}
Clearly, the composition subalgebra $\mathcal {C}\mathcal {H}_\Lambda^{{\tilde{c}l}}({\A})$ is just the image of the algebra homomorphism $\varphi'_0:\A_q(\Lambda,\widetilde{B})\rightarrow\mathcal {D}\mathcal {H}_\Lambda^{{\tilde{c}l}}({\A})$. Thus, $\varphi'_0:\A_q(\Lambda,\widetilde{B})\rightarrow \mathcal {C}\mathcal {H}_\Lambda^{{\tilde{c}l}}({\A})$ gives an algebra isomorphism.
\end{proof}
\begin{corollary}\label{congshixianimg}
For any rigid $M\in{\A}$ and $I\in \mathcal{I}_{\widetilde{\A}}$, it holds that $\varphi'_0(X_{M\oplus I[-1]})=u_{M\oplus I[-1]}$.
\end{corollary}
\begin{proof}
For any rigid $M\in{\A}$ and $I\in \mathcal{I}_{\widetilde{\A}}$, $X_{M\oplus I[-1]}=v^{\Lambda({\bf i}^\ast,{\bf m}^\ast)}X_{I[-1]} X_{M}\in \A_q(\Lambda,\widetilde{B})$; similarly, $u_{M\oplus I[-1]}=v^{\Lambda({\bf i}^\ast,{\bf m}^\ast)}u_{I[-1]} u_{M}\in \mathcal {C}\mathcal {H}_\Lambda^{{\tilde{c}l}}({\A})$.

Suppose that $\varphi'_0(X_{M\oplus I[-1]})=x$ for some element $x\in\mathcal {C}\mathcal {H}_\Lambda^{{\tilde{c}l}}({\A})$.
Since $\varphi_0\varphi'_0={\rm Id}$, we get $\varphi_0(x)=X_{M\oplus I[-1]}$.
On the other hand, $\varphi_0(u_{M\oplus I[-1]})=X_{M\oplus I[-1]}$. By Corollary \ref{congshixian}, $\varphi_0:\mathcal {C}\mathcal {H}_\Lambda^{{\tilde{c}l}}({\A})\rightarrow\A_q(\Lambda,\widetilde{B})$ is an isomorphism. Thus we obtain $x=u_{M\oplus I[-1]}$, i.e. $\varphi'_0(X_{M\oplus I[-1]})=u_{M\oplus I[-1]}$.
\end{proof}
\begin{corollary}\label{xxwguan}
The following subset of the spanning set $\mathfrak{B}$ of the Hall algebra $\mathcal {D}\mathcal {H}_\Lambda^{{\tilde{c}l}}({\A})$\[\{u_{M\oplus I[-1]}\mid M\in \A~\text{is ~rigid~and}~I\in \mathcal{I}_{\widetilde{\A}}~\text{such that}~ \Hom_{\widetilde{\A}}(M,I)=0\}\]
is $\mathbb{Q}(v)$-linearly independent.
\end{corollary}
\begin{proof}
It is proved by Lemma \ref{xjlm}, Corollaries \ref{congshixianimg} and \ref{congshixian}.
\end{proof}
\begin{remark}
For Dynkin quivers, the spanning set $\mathfrak{B}$ of the Hall algebra $\mathcal {D}\mathcal {H}_\Lambda^{{\tilde{c}l}}({\A})$ coincides with the set
$\{u_{M\oplus I[-1]}\mid M\in \A~\text{is ~rigid~and}~I\in \mathcal{I}_{\widetilde{\A}}~\text{such that}~ \Hom_{\widetilde{\A}}(M,I)=0\}.$
Then by Corollary \ref{xxwguan}, we also get that $\mathfrak{B}$ is a basis of $\mathcal {D}\mathcal {H}_\Lambda^{{\tilde{c}l}}({\A})$. Noting that each $u_{M\oplus I[-1]}\in\mathfrak{B}$ belongs to $\mathcal {C}\mathcal {H}_\Lambda^{{\tilde{c}l}}({\A})$, we obtain that $\mathcal {C}\mathcal {H}_\Lambda^{{\tilde{c}l}}({\A})=\mathcal {D}\mathcal {H}_\Lambda^{{\tilde{c}l}}({\A})$. Thus,
the map $\varphi'_0$ gives the algebra isomorphism between $\A_q(\Lambda,\widetilde{B})$ and $\mathcal {D}\mathcal {H}_\Lambda^{{\tilde{c}l}}({\A})$. These coincide with the main results in Theorem \ref{thm:hall-quantum-dynkin}.
\end{remark}

\section{Standard monomials of quantum projective cluster variables}
In this section, we give an application of the Hall algebra $\mathcal{D}\mathcal{H}_\Lambda^{\tilde{c}l}(\mathcal{A})$. We introduce the Hall subalgebra of $\mathcal{D}\mathcal{H}_\Lambda^{\tilde{c}l}(\mathcal{A})$ generated by projectives, yielding a Hall algebra analogue of the lower bound quantum cluster algebras from \cite{HCDX} (see also \cite{BNI}). Then we give a proof of the main result of \cite{HCDX} via Hall algebra approach, which simplifies the calculations in \cite{HCDX} and provides representation theory interpretations of the corresponding results.

Throughout this section, let $\widetilde{Q}$ be still a finite acyclic valued quiver as given in the subsection 2.2. By the acyclicity of $\widetilde{Q}$, we label the vertices of $\widetilde{Q}$ such that if there exists an arrow $i\to j$, then $i>j$. In this case, there are no arrows from the vertices $\{1,\ldots, n\}$ to $\{n+1,\ldots, m\}$. As a consequence, for each $1\leq i\leq n$, the projective representation $P_i$ of $Q$ can be identified with the corresponding projective representation of $\widetilde{Q}$.

\subsection{Hall subalgebras generated by projectives}
Let $\mathcal {P}\mathcal {H}_\Lambda^{{\tilde{c}l}}({\A})$ be the subalgebra of $\mathcal {D}\mathcal {H}_\Lambda^{{\tilde{c}l}}({\A})$ generated by the elements $u_{P_i[1]}$ and $u_{P_j}$, where $1\leq i\leq m$ and $1\leq j\leq n$. In this subsection, we show that the Hall subalgebra $\mathcal {P}\mathcal {H}_\Lambda^{{\tilde{c}l}}({\A})$ is the same as the composition subalgebra $\mathcal {C}\mathcal {H}_\Lambda^{{\tilde{c}l}}({\A})$.

\begin{lemma}\label{lem:simple-in-projective}
Given $P, Q\in\mathcal {P}_{\A}$, we have $\sum_{[M]}F_{M,Q}^{P}u_M\in \mathcal {P}\mathcal {H}_\Lambda^{{\tilde{c}l}}({\A})$. In particular, we have $u_{S_k}\in \mathcal {P}\mathcal {H}_\Lambda^{{\tilde{c}l}}({\A})$ for any $1\leq k\leq n$.
\end{lemma}
\begin{proof}
We first remark that $\dim_{\mathbb{F}_q}Q\leq \dim_{\mathbb{F}_q}P$ whenever $F_{M,Q}^{P}$ is nonzero.
We proceed the proof by induction on the dimension of $Q$.

If $\dim_{\mathbb{F}_q}Q=0$, i.e. $Q=0$, then the sum becomes $u_{P}$, which is clearly in $\mathcal {P}\mathcal {H}_\Lambda^{{\tilde{c}l}}({\A})$.

Assume that the statement holds for any projective representation of $\A$ which has the dimension less than $\dim_{\mathbb{F}_q}Q$.
By \eqref{lgx7}, we obtain
\begin{align}\label{eq:q1pk}
    u_{Q[1]}\star u_{P}&=v^{-\Lambda(\mathbf{q}^\ast,\mathbf{p}^\ast)-2\lr{\mathbf{q},\mathbf{p}}}\sum_{[G],[Q']}v^{\Lambda(\mathbf{g}^\ast,\mathbf{q'}^\ast)}|_{Q'}\Hom_{\A}(Q,P)_G|u_{G}\star u_{Q'[1]}\notag\\
    &=v^{-2\Lambda(\mathbf{q}^\ast,\mathbf{p}^\ast)-2\lr{\mathbf{q},\mathbf{p}}}u_{P}\star u_{Q[1]}+\sum_{[M]}v^{-\Lambda(\mathbf{q}^\ast,\mathbf{p}^\ast)-2\lr{\mathbf{q},\mathbf{p}}}|_0\Hom_{\A}(Q,P)_M|u_M\notag\\
    \quad &+\sum_{[0]\neq[Q']\neq[ Q],[G]}v^{-\Lambda(\mathbf{q}^\ast,\mathbf{p}^\ast)-2\lr{\mathbf{q},\mathbf{p}}+\Lambda(\mathbf{g}^\ast,\mathbf{q'}^\ast)}|_{Q'}\Hom_{\A}(Q,P)_G|u_G\star u_{Q'[1]}.
\end{align}
Note that \begin{equation}\label{diyihom}|_0\Hom_{\A}(Q,P)_M|=a_QF_{M,Q}^{P},\end{equation} where we recall that $a_X$ denotes the cardinality of the automorphism group of $X$ for any $X\in\A$.
On the other hand, given $Q'$, for any morphism $f$ in ${_{Q'}\Hom_{\A}(Q,P)_G}$, $\im f$ is projective and has the fixed dimension vector ${\bf q}-{\bf q}'$, thus all morphisms $f$ in ${_{Q'}\Hom_{\A}(Q,P)_G}$ have the same image up to isomorphism, denoted by $Q/Q'$. Hence, we obtain
\begin{equation}\label{dierhom}|_{Q'}\Hom_{\A}(Q,P)_G|=a_{Q/Q'}F_{Q/Q',Q'}^QF_{G,Q/Q'}^{P}.\end{equation}
Substituting \eqref{diyihom}, \eqref{dierhom} to \eqref{eq:q1pk}, we obtain
\begin{align}\label{eq:q1pk-rewirtten}
    &u_{Q[1]}\star u_{P}=q^{-\Lambda(\mathbf{q}^\ast,\mathbf{p}^\ast)-\lr{\mathbf{q},\mathbf{p}}}u_{P}\star u_{Q[1]}+v^{-\Lambda(\mathbf{q}^\ast,\mathbf{p}^\ast)-2\lr{\mathbf{q},\mathbf{p}}}{a_Q}\sum_{[M]}F_{M,Q}^{P}u_M\notag\\
    &+\sum_{[0]\neq[Q']\neq[ Q]}v^{-\Lambda(\mathbf{q}^\ast,\mathbf{p}^\ast)-2\lr{\mathbf{q},\mathbf{p}}+\Lambda(\mathbf{p}^\ast-\mathbf{q}^\ast,\mathbf{q'}^\ast)}a_{Q/Q'}F_{Q/Q',Q'}^Q(\sum_{[G]}F_{G,Q/Q'}^{P}u_G)\star u_{Q'[1]}.
\end{align}
Note that in the right hand side of \eqref{eq:q1pk-rewirtten}, each $Q/Q'$ belongs to $\mathcal {P}_{\A}$ and has $\dim_{\mathbb{F}_q} Q/Q'<\dim_{\mathbb{F}_q} Q$. By inductions, we have $\sum_{[G]}F_{G,Q/Q'}^{P}u_G\in \mathcal {P}\mathcal {H}_\Lambda^{{\tilde{c}l}}({\A})$. Hence, by \eqref{eq:q1pk-rewirtten}, we get that $\sum_{[M]}F_{M,Q}^{P}u_M\in \mathcal {P}\mathcal {H}_\Lambda^{{\tilde{c}l}}({\A}).$
In particular, take $P=P_k$ and $Q=\rad P_k\in\mathcal {P}_{\A}$. Then the $M$ in the sum $\sum_{[M]}F_{M,Q}^{P}u_M$ is unique and isomorphic to $S_k$. Thus, we get $u_{S_k}\in \mathcal {P}\mathcal {H}_\Lambda^{{\tilde{c}l}}({\A})$. Therefore, we complete the proof.
\end{proof}
\begin{remark}
In the equation \eqref{eq:q1pk-rewirtten}, $Q,Q'\in\mathcal {P}_{\A}$, and we need the condition $\mathcal {P}_{\A}\subseteq\mathcal {P}_{\widetilde{\A}}$ to get $u_{Q[1]}, u_{Q'[1]}\in\mathcal {P}\mathcal {H}_\Lambda^{{\tilde{c}l}}({\A})$, which is guaranteed by the convention of orientations of $\widetilde{Q}$ at the beginning of this section. Of course, the convention of orientations is also needed in the next subsection.
\end{remark}

\begin{proposition}\label{hlbxd}
It holds that  $\mathcal {P}\mathcal {H}_\Lambda^{{\tilde{c}l}}({\A})=\mathcal {C}\mathcal {H}_\Lambda^{{\tilde{c}l}}({\A})$.
\end{proposition}
\begin{proof}
    Since $u_{P_k}\in \mathcal {C}\mathcal {H}_\Lambda^{{\tilde{c}l}}({\A})$ for each $1\leq k\leq n$, we get that $\mathcal {P}\mathcal {H}_\Lambda^{{\tilde{c}l}}({\A})\subseteq \mathcal {C}\mathcal {H}_\Lambda^{{\tilde{c}l}}({\A})$. The inverse inclusion follows from Lemma \ref{lem:simple-in-projective}.
\end{proof}
\begin{definition}
The {\em lower bound quantum cluster algebra} $\mathcal {L}_q(\Lambda,\widetilde{B})$ is the subalgebra of $\mathcal{T}_\Lambda$ generated by the quantum cluster characters $X_{P_i[1]}$ and $X_{P_j}$, where $1\leq i\leq m$ and $1\leq j\leq n$.
\end{definition}
\begin{corollary}\label{tuil5.6}
It holds that  $\mathcal {L}_q(\Lambda,\widetilde{B})=\A_q(\Lambda,\widetilde{B})$.
\end{corollary}
\begin{proof}
Note that for each $1\leq i\leq m$ and $1\leq j\leq n$, the isomorphism $\varphi'_0$ in Corollary \ref{congshixian} sends $u_{P_i[1]}$, $u_{P_j}$ and $u_{S_j}$ to $X_{P_i[1]}$, $X_{P_j}$ and $X_{S_j}$, respectively. Then by Proposition \ref{hlbxd}, we finish the proof.
\end{proof}
\begin{corollary}
Given $P, Q\in\mathcal {P}_{\A}$, we have $\sum_{[M]}F_{M,Q}^{P}X_M\in \A_q(\Lambda,\widetilde{B})$.
\end{corollary}
\begin{proof}
Clearly, the homomorphism $\varphi_0:\mathcal {D}\mathcal {H}_\Lambda^{{\tilde{c}l}}({\A})\to \mathcal{A}\mathcal{H}_\Lambda(Q)$ can be restricted to be a homomorphism $\varphi_0:\mathcal {P}\mathcal {H}_\Lambda^{{\tilde{c}l}}({\A})\to \mathcal {L}_q(\Lambda,\widetilde{B})$. Then by Lemma \ref{lem:simple-in-projective}, we get $\sum_{[M]}F_{M,Q}^{P}X_M=\varphi_0(\sum_{[M]}F_{M,Q}^{P}u_M)\in\mathcal {L}_q(\Lambda,\widetilde{B})$. Hence, by Corollary \ref{tuil5.6}, we finish the proof.
\end{proof}

\subsection{Standard monomial basis}
In this subsection, we use the standard monomials of the generators of $\mathcal {P}\mathcal {H}_\Lambda^{{\tilde{c}l}}({\A})$ to construct a basis of the composition subalgebra $\mathcal {C}\mathcal {H}_\Lambda^{{\tilde{c}l}}({\A})$, which corresponds to a basis of the quantum cluster algebra $\A_q(\Lambda,\widetilde{B})$.

Let $\mathfrak{A}$ be a $\mathbb{Q}(v)$-algebra and $x_1,\ldots, x_m, y_1,\ldots, y_n\in \mathfrak{A}$. Assume there exists an $m\times m$ skew-symmetric integral matrix $\Lambda=(\lambda_{ij})$ such that the following conditions hold:
\begin{itemize}
    \item $x_1,\ldots, x_m$ are quasi-commutative: $x_ix_j=v^{\lambda_{ij}}x_jx_i$~for any $1\leq i,j\leq m$;
    \item $y_1,\ldots, y_n$ are quasi-commutative: $y_iy_j=v^{\lambda_{ij}}y_jy_i$~for any $1\leq i,j\leq n$;
    \item For each $1\leq k\leq n$, $x_ky_k=f_k(x_{k+1},\ldots, x_m,y_1,\ldots, y_{k-1})$, where $f_k$ is a finite linear combination of some monomials of the form $x_{k+1}^{a_{k+1}}\cdots x_m^{a_m}y_1^{b_1}\cdots y_{k-1}^{b_{k-1}}.$
\end{itemize}
A monomial $x_1^{a_1}\cdots x_m^{a_m}y_1^{b_1}\cdots y_n^{b_n}$ is called a {\em standard monomial} of $x_1,\ldots, x_m,y_1,\ldots, y_n$ provided that $a_ib_i=0$ for any $1\leq i\leq n$.

For any monomial $M=x_1^{a_1}\cdots x_m^{a_m}y_1^{b_1}\cdots y_n^{b_n}$, denote the exponent vector of $M$ by $\mathbf{v}(M):=(a_1,\ldots, a_m,b_1,\ldots, b_n)$. For any two monomials $M, N$ of $x_1,\ldots, x_m,y_1,\ldots, y_n$, we write $M>_{\rm lex}N$ if $\mathbf{v}(M)>_{\rm lex} \mathbf{v}(N)$, where the later $>_{\rm lex}$ is the lexicographical order on $\mathbb{N}^{m+n}$. Note that for a given $\alpha\in \mathbb{N}^{m+n}$,  a strictly decreasing sequence $\alpha>_{\rm lex}\alpha_1>_{\rm lex}\alpha_2>_{\rm lex}\cdots$ must terminate after finitely many steps.
\begin{lemma}\label{lem:rewritting-system}
In the algebra $\mathfrak{A}$, every monomial $x_1^{a_1}\cdots x_m^{a_m}y_1^{b_1}\cdots y_n^{b_n}$ is a linear combination of standard monomials of $x_1,\ldots, x_m,y_1,\ldots, y_n$.
\end{lemma}
\begin{proof}
Take any monomial $M=x_1^{a_1}\cdots x_m^{a_m}y_1^{b_1}\cdots y_n^{b_n}$, if it is not standard, then there exists $1\leq k\leq n$ such that $a_kb_k>0$ and $a_ib_i=0$ for all $i<k$. Since $x_1,\ldots, x_m$ (resp. $y_1,\ldots, y_n$) are quasi-commutative, there exists an integer $c_k$ such that
\begin{align*}
M&=v^{c_k}x_1^{a_1}\cdots x_{k-1}^{a_{k-1}}x_{k}^{a_k-1}\cdots x_m^{a_m}(x_ky_k)y_1^{b_1}\cdots y_{k-1}^{b_{k-1}}y_{k}^{b_k-1}\cdots y_n^{b_n}\notag\\
&=v^{c_k}x_1^{a_1}\cdots x_{k-1}^{a_{k-1}}x_{k}^{a_k-1}\cdots x_m^{a_m}f_k(x_{k+1},\ldots, x_m,y_1,\ldots, y_{k-1})y_1^{b_1}\cdots y_{k-1}^{b_{k-1}}y_{k}^{b_k-1}\cdots y_n^{b_n}.
\end{align*}
By the form of $f_k$, we conclude that each monomial $N$ appearing in the right hand side of the above equation satisfies $M>_{\rm lex}N$. By inductions, we get that $M$ is a linear combination of standard monomials of $x_1,\ldots, x_m,y_1,\ldots, y_n$.
\end{proof}

\begin{lemma}\label{lem:monomial-to-special-monomial}
    For any $M_1,\ldots, M_s\in I_{\widetilde{\A}}[-1]\cup \mathcal{P}_{\A}$, $u_{M_1}\star \cdots\star u_{M_s}$ is a linear combination of $u_{P\oplus I[-1]}$ with $P\in \mathcal{P}_{\A}$ and $I\in \mathcal{I}_{\widetilde{\A}}$.
\end{lemma}
\begin{proof}
   It follows from \eqref{lgx6} by noting that $\mathcal {P}_{\A}$ is closed under subobjects, and $I_{\widetilde{\A}}$ is closed under quotients.
\end{proof}

For any $P\in \mathcal{P}_{\A}$ and $I\in \mathcal{I}_{\widetilde{\A}}$, there exist unique vectors $(b_1,\ldots,b_n)\in \mathbb{N}^n$ and $(a_1,\ldots, a_m)\in \mathbb{N}^m$ such that
$P\cong b_1P_1\oplus \cdots b_nP_n$ and $I\cong a_1I_1\oplus \cdots a_mI_m.$
Then the element $u_{P\oplus I[-1]}$ is called {\em standard} if $a_ib_i=0$ for any $1\leq i\leq n$.
Clearly, the standard elements are just the standard monomials of $u_{I_1[-1]},\ldots, u_{I_m[-1]}, u_{P_1},\ldots, u_{P_n}$ up to scalars. Set $$\mathfrak{S}:=\{u_{P\oplus I[-1]}\mid \text{$u_{P\oplus I[-1]}$ is standard},~\text{where}~P\in \mathcal{P}_{\A}~\text{and}~I\in \mathcal{I}_{\widetilde{\A}}\}.$$
In what follows, we use Lemma \ref{lem:rewritting-system} to prove that the set $\mathfrak{S}$ is a basis of the composition subalgebra.

\begin{lemma}\label{yl5.9}
    The set $\mathfrak{S}$ is a spanning set of $\mathcal {P}\mathcal {H}_\Lambda^{{\tilde{c}l}}({\A})$.
\end{lemma}
\begin{proof}
By Lemma \ref{lem:monomial-to-special-monomial}, it suffices to show $u_{P\oplus I[-1]}$ is a linear combination of $\mathfrak{S}$ for any $P\in \mathcal{P}_{\A}$ and $\mathcal{I}_{\widetilde{\A}}$.
Assume $P\cong b_1P_1\oplus \cdots \oplus b_nP_n$ and $I\cong a_1I_1\oplus \cdots \oplus a_mI_m$ for some $(b_1,\ldots,b_n)\in \mathbb{N}^n$ and $(a_1,\ldots, a_m)\in \mathbb{N}^m$. Then we have
\begin{equation}\label{upi-1}
u_{P\oplus I[-1]}=v^ru_{I_1[-1]}^{a_1}\star \cdots\star u_{I_m[-1]}^{a_m}\star u_{P_1}^{b_1}\star \cdots \star u_{P_n}^{b_n}\end{equation} for some $r\in\mathbb{Z}$.

In order to use Lemma \ref{lem:rewritting-system} to show that
\eqref{upi-1} is a linear combination of standard monomials of $u_{I_1[-1]},\ldots, u_{I_m[-1]}, u_{P_1},\ldots, u_{P_n}$, we need to prove the elements of  $\mathcal {P}\mathcal {H}_\Lambda^{{\tilde{c}l}}({\A})$ $$u_{I_1[-1]},\ldots, u_{I_m[-1]}, u_{P_1},\ldots, u_{P_n}$$ satisfy the conditions as those of $x_1,\ldots, x_m, y_1,\ldots, y_n$ in $\mathfrak{A}$. By the equations \eqref{lgx2} and \eqref{lgx1}, we easily obtain the quasi-commutative relations of $u_{I_1[-1]},\ldots, u_{I_m[-1]}$ and $u_{P_1},\ldots, u_{P_n}$. Now, let us prove the third condition as follows: for any $1\leq k\leq n$,
\begin{flalign}\label{uppp1}
u_{P_k[1]}\star u_{P_k}=q^{-d_k}u_{P_k}\star u_{P_k[1]}+q^{-d_k}(q^{d_k}-1)
\end{flalign}
and
\begin{equation}\label{upii-1}
\begin{split}
u_{P_k}\star u_{P_k[1]}&=u_{P_k}\star u_{I_k[-1]}\\&=q^{\Lambda({\bf i}_k^\ast,{\bf p}_k^\ast)-d_k}u_{I_k[-1]}\star u_{P_k}+q^{-\frac{1}{2}\Lambda({\bf p}_k^\ast,{\bf i}_k^\ast)-d_k+\frac{1}{2}\Lambda({\bf i}^\ast,{\bf g}^\ast)}(q^{d_k}-1)u_{I[-1]}\star u_G,
\end{split}\end{equation}
where $I=I_k/S_k$ and $G={\rad} P_k$. Noting that
\begin{flalign*}&\Lambda({\bf i}_k^\ast,{\bf p}_k^\ast)=\Lambda(^\ast{\bf i}_k+B(\widetilde{Q}){\bf i}_k,{\bf p}_k^\ast)=\lr{{\bf p}_k,{\bf i}_k}=d_k,\\
&\Lambda({\bf i}^\ast,{\bf g}^\ast)=\Lambda({\bf i}_k^\ast-e_k^\ast,{\bf p}_k^\ast-e_k^\ast)=0
\end{flalign*}
and substituting \eqref{upii-1} to \eqref{uppp1}, we obtain
\begin{flalign}
u_{I_k[-1]}\star u_{P_k}=v^{-d_k}u_{(I_k/S_k)[-1]}\star u_{{\rm rad} P_k}+1.
\end{flalign}
Suppose ${\rm rad} P_k\cong t_1P_1\oplus \cdots \oplus t_nP_n$ and $I_k/S_k\cong s_1I_1\oplus \cdots \oplus s_mI_m$ for some $(t_1,\ldots,t_n)\in \mathbb{N}^n$ and $(s_1,\ldots, s_m)\in \mathbb{N}^m$, then
\begin{flalign*}&(t_1,\ldots,t_n,0,\ldots,0)^{\rm tr}=e_k-e_k^\ast=R(\widetilde{Q})e_k=(r_{1k},\ldots,r_{mk})^{\rm tr}\\
&(s_1,\ldots,s_m)^{\rm tr}=e_k-{^\ast e}_k=R'(\widetilde{Q})e_k=(r'_{1k},\ldots,r'_{mk})^{\rm tr}.
\end{flalign*}
By the convention of orientations of $\widetilde{Q}$, we get $r_{ik}=0$ for any $i\geq k$, and $r'_{jk}=0$ for any $j\leq k$. That is, ${\rm rad} P_k\cong t_1P_1\oplus \cdots \oplus t_{k-1}P_{k-1}$ and $I_k/S_k\cong s_{k+1}I_{k+1}\oplus \cdots \oplus s_mI_m$. Thus, we obtain
\begin{flalign*}
u_{I_k[-1]}\star u_{P_k}=v^{r'}u_{I_{k+1}[-1]}^{s_{k+1}}\star\cdots\star u_{I_{m}[-1]}^{s_{m}}\star u_{P_1}^{t_1}\star\cdots\star u_{P_{k-1}}^{t_{k-1}}+1
\end{flalign*}
for some $r'\in\mathbb{Z}$, which is just the third condition.

Hence, by Lemma \ref{lem:rewritting-system} and \eqref{upi-1}, we get that $u_{P\oplus I[-1]}$ is a linear combination of standard monomials of $u_{I_1[-1]},\ldots, u_{I_m[-1]}, u_{P_1},\ldots, u_{P_n}$. Thus, $u_{P\oplus I[-1]}$ is a linear combination of elements in $\mathfrak{S}$. Therefore, we complete the proof.
\end{proof}

\begin{theorem}\label{pcabasis}
The set  $\mathfrak{S}$ is a $\mathbb{Q}(v)$-basis of $\mathcal {C}\mathcal {H}_\Lambda^{{\tilde{c}l}}({\A})$.
\end{theorem}
\begin{proof}

By Corollaries \ref{congshixian}, \ref{congshixianimg} and Lemma \ref{yl5.9}, it suffices to show that the set $$\mathscr{T}:=\{X_{P\oplus I[-1]}\mid u_{P\oplus I[-1]}\in \mathfrak{S}\}$$ is $\mathbb{Q}(v)$-linearly independent.
Assume $P\cong b_1P_1\oplus \cdots\oplus b_nP_n$ and $I\cong a_1I_1\oplus \cdots \oplus a_mI_m$ such that $a_ib_i=0$ for any $1\leq i\leq n$, where $(b_1,\ldots,b_n)\in \mathbb{N}^n$ and $(a_1,\ldots, a_m)\in \mathbb{N}^m$.
Recall that $X_{P\oplus I[-1]}$ is pointed at ${^\ast{\bf i}}-{\bf p}^\ast=\sum_{i=1}^ma_ie_i-\sum_{j=1}^nb_je_j$.

Suppose $P'=b_1'P_1\oplus \cdots \oplus b_n'P_n$ and $I'=a_1'I_1\oplus \cdots \oplus a_mI_m$ such that $a_i'b_i'=0$ for any $1\leq i\leq n$. If $\sum_{i=1}^ma_ie_i-\sum_{j=1}^nb_je_j=\sum_{i=1}^ma_i'e_i-\sum_{j=1}^nb_j'e_j$, then $a_i-b_i=a_i'-b_i'$~for any $1\leq i\leq n$, and $a_i=a_i'$ for any $n<i\leq m$. Since $a_ib_i=a_i'b_i'=0$ for any $1\leq i\leq n$, we easily get $a_i=a_i'$ for $1\leq i\leq m$ and $b_j=b_j'$ for $1\leq j\leq n$. Thus, $P\cong P'$ and $I\cong I'$. By Lemma \ref{lem:pointed-function}, we obtain that $\mathscr{T}$ is $\mathbb{Q}(v)$-linearly independent.
\end{proof}
\begin{corollary}\label{ttll51}
The standard monomials of $X_{I_1[-1]},\ldots, X_{I_m[-1]}, X_{P_1},\ldots, X_{P_n}$ form a $\mathbb{Q}(v)$-basis of $\A_q(\Lambda,\widetilde{B})$.
\end{corollary}
\begin{proof}
It is proved by Theorem \ref{pcabasis} and Corollary \ref{congshixian}.
\end{proof}

\section{Algebra automorphisms of $\mathcal {D}\mathcal {H}_\Lambda^{{cl}}(\widetilde{\A})$ given by Auslander-Reiten translations}
In this section, let $\widetilde{Q}$ be any finite acyclic valued quiver as given in the subsection 2.2. We give an algebra automorphism of $\mathcal {D}\mathcal {H}_\Lambda^{{cl}}(\widetilde{\A})$, which is induced by the Auslander--Reiten translations of cluster categories.

First of all, let us give an equivalent characterization of the ideal $\mathfrak{I}$.
For any $M,N\in \widetilde{\A}$, set
\[
r_{M,N}:=u_{M\oplus N}-\sum\limits_{\begin{smallmatrix}[D],[A],[I]\end{smallmatrix}}v^{\Lambda(({\bf m}+{\bf n})^*,{\bf a}^*)+\lr{{\bf m}-{\bf a},{\bf n}}}|_D\Hom_{\widetilde{\A}}(N,\tau M)_{\tau A\oplus I}|u_{A}\star u_{D\oplus I[-1]}
\]
in $ \mathcal {D}\mathcal {H}_\Lambda^{{cl}_1}(\widetilde{\A})$,
where each $A$  has the same maximal projective direct summand as $M$, and each $I\in\I_{\widetilde{\A}}$.
Let $\mathfrak{I}'$ be the two sided ideal of $\mathcal {D}\mathcal {H}_\Lambda^{{cl}_1}(\widetilde{\A})$ generated by the elements
$$\{r_{M,N}, r_{M,P}~|~M,N\in{\widetilde{\A}}, P\in \mathcal{P}_{\widetilde{\A}}\ \text{and $N$ has no nonzero projective direct summands}\}.$$

\begin{lemma}\label{lxdjkh}
It holds that $\mathfrak{I}=\mathfrak{I}'$.
\end{lemma}
\begin{proof}
It suffices to show that $\mathfrak{I}\subseteq \mathfrak{I}'$. Take any $M,N\in \widetilde{\A}$ and $P\in \mathcal{P}_{\widetilde{\A}}$ such that $N$ has no nonzero projective direct summands. It remains to show $r_{M,N\oplus P}$ belongs to $\mathfrak{I}'$. In the quotient algebra $\mathcal {D}\mathcal {H}_\Lambda^{{cl}_1}(\widetilde{\A})/\mathfrak{I}'$, we have
\begin{equation}\label{eq:MNP-decom}
\begin{split}
&u_{M\oplus N\oplus P}=v^{\Lambda((\mathbf{m}+\mathbf{n})^\ast, \mathbf{p}^\ast)} u_P\star u_{M\oplus N}\\
&=v^{\Lambda((\mathbf{m}+\mathbf{n})^\ast, \mathbf{p}^\ast)} u_P\star(\sum_{[D],[A],[J]}v^{\Lambda((\mathbf{m}+\mathbf{n})^\ast,\mathbf{a}^\ast)+\langle \mathbf{m}-\mathbf{a},\mathbf{n}\rangle}|_D\Hom_{\widetilde{\A}}(N,\tau M)_{\tau A\oplus J}|u_A\star u_{D\oplus J[-1]})\\
&=\sum_{[D],[A],[J]}v^{\Lambda((\mathbf{m}+\mathbf{n})^\ast,(\mathbf{p}+\mathbf{a})^\ast)+\Lambda(\mathbf{p}^\ast,\mathbf{a}^\ast)+\langle \mathbf{m}-\mathbf{a},\mathbf{n}\rangle}|_D\Hom_{\widetilde{\A}}(N,\tau M)_{\tau A\oplus J}|u_{A\oplus P}\star u_{D\oplus J[-1]}\\
&=\sum_{\substack{[D],[A],[J]\\ [Q],[B],[J']}}v^{\Lambda((\mathbf{m}+\mathbf{n})^\ast,(\mathbf{p}+\mathbf{a})^\ast)+\Lambda(\mathbf{p}^\ast,\mathbf{a}^\ast)+\Lambda((\mathbf{p}+\mathbf{a})^\ast,\mathbf{b}^\ast)+\Lambda(\mathbf{j'}^\ast,\mathbf{q}^\ast)+\Lambda(\mathbf{j}^\ast,\mathbf{d}^\ast)+\langle \mathbf{m}-\mathbf{a},\mathbf{n}\rangle+\langle\mathbf{a}-\mathbf{b},\mathbf{p}\rangle}\\
&\quad\quad |_D\Hom_{\widetilde{\A}}(N,\tau M)_{\tau A\oplus J}|\cdot |_Q\Hom_{\widetilde{\A}}(P,\tau A)_{\tau B\oplus J'}|u_B\star u_{J'[-1]}\star u_Q\star  u_{J[-1]}\star u_D,
\end{split}\end{equation}
where $A,B$ have the same maximal projective direct summands as $M$.

By \eqref{lgx6} and \eqref{lgx4}, we obtain
\begin{flalign}\label{eq:QJ-decom}
u_Q\star u_{J[-1]}=\sum_{[G],[J'']}v^{-\Lambda(\mathbf{q}^\ast,\mathbf{j}^\ast)-2\langle \mathbf{q},\mathbf{j}\rangle}|_G\Hom_{\widetilde{\A}}(Q,J)_{J''}|u_{G\oplus J''[-1]}.
\end{flalign}
Substituting \eqref{eq:QJ-decom} into \eqref{eq:MNP-decom} and noting that $G$ is projective, we obtain
\begin{flalign*}
u_{M\oplus(N\oplus P)}=\sum_{\substack{[D],[A],[J], [Q]\\ [B],[J'], [G],[J'']}}v^{x_0}&|_D\Hom_{\widetilde{\A}}(N,\tau M)_{\tau A\oplus J}|\cdot |_Q\Hom_{\widetilde{\A}}(P,\tau A)_{\tau B\oplus J'}|\\& |_G\Hom_{\widetilde{\A}}(Q,J)_{J''}|u_B\star u_{(G\oplus D)\oplus (J'\oplus J'')[-1]},
\end{flalign*}
where
\begin{flalign*}
x_0&=\Lambda((\mathbf{m}+\mathbf{n})^\ast,(\mathbf{p}+\mathbf{a})^\ast)+\Lambda(\mathbf{p}^\ast,\mathbf{a}^\ast)+\Lambda((\mathbf{p}+\mathbf{a})^\ast,\mathbf{b}^\ast)
+\Lambda(\mathbf{j'}^\ast,\mathbf{q}^\ast)\\&\quad+\Lambda(\mathbf{j}^\ast,\mathbf{d}^\ast)+\langle\mathbf{m}-\mathbf{a},\mathbf{n}\rangle+\langle\mathbf{a}-\mathbf{b},\mathbf{p}\rangle-\Lambda(\mathbf{q}^\ast,\mathbf{j}^\ast)-2\langle \mathbf{q},\mathbf{j}\rangle\\
&\quad+\Lambda(\mathbf{j''}^\ast,\mathbf{g}^\ast)+\Lambda(\mathbf{g}^\ast,\mathbf{d}^\ast)+\Lambda(\mathbf{j'}^\ast,\mathbf{j''}^\ast)-\Lambda((\mathbf{j'}+\mathbf{j''})^\ast,(\mathbf{g}+\mathbf{d})^\ast).
\end{flalign*}
Note that
\[\mathbf{g}+\mathbf{j}=\mathbf{q}+\mathbf{j''},
\mathbf{q}-\mathbf{j'}=\mathbf{p}+\tau\mathbf{b}-\tau\mathbf{a}~\text{and}~
\mathbf{d}-\mathbf{j}=\mathbf{n}+\tau\mathbf{a}-\tau\mathbf{m}.
\]
Using \cite[Lemma 6.4]{CDZ}, $B(\widetilde{Q})\mathbf{x}=\mathbf{x}^\ast-{^\ast\mathbf{x}}$ and Lemma \ref{sjishu},
we obtain
\begin{flalign*}
&\Lambda(\mathbf{j'}^\ast,\mathbf{q}^\ast)+\Lambda(\mathbf{j}^\ast,\mathbf{d}^\ast)+\Lambda(\mathbf{j''}^\ast,\mathbf{g}^\ast)+\Lambda(\mathbf{g}^\ast,\mathbf{d}^\ast)+\Lambda(\mathbf{j'}^\ast,\mathbf{j''}^\ast)-\Lambda((\mathbf{j'}+\mathbf{j''})^\ast,(\mathbf{g}+\mathbf{d})^\ast)-\Lambda(\mathbf{q}^\ast,\mathbf{j}^\ast)\\
&=\Lambda((\mathbf{g}+\mathbf{j})^\ast,\mathbf{d}^\ast)+\Lambda(\mathbf{j'}^\ast,\mathbf{q}^\ast)+\Lambda(\mathbf{j''}^\ast,\mathbf{g}^\ast)+\Lambda(\mathbf{j'}^\ast,\mathbf{j''}^\ast)-\Lambda((\mathbf{j'}+\mathbf{j''})^\ast,(\mathbf{g}+\mathbf{d})^\ast)-\Lambda(\mathbf{q}^\ast,\mathbf{j}^\ast)\\
&=\Lambda((\mathbf{q}+\mathbf{j''})^\ast,\mathbf{d}^\ast)+\Lambda(\mathbf{j'}^\ast,\mathbf{q}^\ast)+\Lambda(\mathbf{j''}^\ast,\mathbf{g}^\ast)+\Lambda(\mathbf{j'}^\ast,\mathbf{j''}^\ast)-\Lambda((\mathbf{j'}+\mathbf{j''})^\ast,(\mathbf{g}+\mathbf{d})^\ast)-\Lambda(\mathbf{q}^\ast,\mathbf{j}^\ast)\\
&=\Lambda(\mathbf{q}^\ast,\mathbf{d}^\ast)-\Lambda(\mathbf{j'}^\ast,\mathbf{g}^\ast)-\Lambda(\mathbf{j'}^\ast,\mathbf{d}^\ast)+\Lambda(\mathbf{j'}^\ast,\mathbf{q}^\ast)+\Lambda(\mathbf{j'}^\ast,\mathbf{j''}^\ast)-\Lambda(\mathbf{q}^\ast,\mathbf{j}^\ast)\\
&=\Lambda(\mathbf{q}^\ast,\mathbf{d}^\ast)+\Lambda(\mathbf{j'}^\ast,(\mathbf{j}-\mathbf{j''})^\ast)-\Lambda(\mathbf{j'}^\ast,\mathbf{d}^\ast)+\Lambda(\mathbf{j'}^\ast,\mathbf{j''}^\ast)-\Lambda(\mathbf{q}^\ast,\mathbf{j}^\ast)\\
&=\Lambda(\mathbf{q}^\ast,\mathbf{d}^\ast)+\Lambda(\mathbf{j'}^\ast,\mathbf{j}^\ast)-\Lambda(\mathbf{j'}^\ast,\mathbf{d}^\ast)-\Lambda(\mathbf{q}^\ast,\mathbf{j}^\ast)\\
&=\Lambda((\mathbf{p}+\tau\mathbf{b}-\tau\mathbf{a})^\ast, (\mathbf{n}+\tau\mathbf{a}-\tau\mathbf{m})^\ast),\end{flalign*}
which is equal to
\begin{flalign*}&\Lambda(\mathbf{p}^\ast,\mathbf{n}^\ast)-\Lambda(\mathbf{p}^\ast-B(\widetilde{Q})\mathbf{p},(\mathbf{a}-\mathbf{m})^\ast)-\Lambda((\mathbf{b}-\mathbf{a})^\ast,\mathbf{n}^\ast-B(\widetilde{Q})\mathbf{n})+\Lambda((\mathbf{b}-\mathbf{a})^\ast,(\mathbf{a}-\mathbf{m})^\ast)\\
&=\Lambda(\mathbf{p}^\ast,\mathbf{n}^\ast)-\Lambda(\mathbf{p}^\ast,(\mathbf{a}-\mathbf{m})^\ast)-\Lambda((\mathbf{b}-\mathbf{a})^\ast,\mathbf{n}^\ast-(\mathbf{a}-\mathbf{m})^\ast)+\langle \mathbf{a}-\mathbf{m},\mathbf{p}\rangle-\langle \mathbf{b}-\mathbf{a},\mathbf{n}\rangle.
\end{flalign*}
It follows that
\begin{flalign*}
x_0&=\Lambda((\mathbf{m}+\mathbf{n}+\mathbf{p})^\ast, \mathbf{b}^\ast)+\langle\mathbf{a}-\mathbf{m},\mathbf{p}\rangle-\langle \mathbf{b}-\mathbf{a},\mathbf{n}\rangle+\langle\mathbf{m}-\mathbf{a},\mathbf{n}\rangle+\langle\mathbf{a}-\mathbf{b},\mathbf{p}\rangle-2\langle\mathbf{q},\mathbf{j}\rangle\\
&=\Lambda((\mathbf{m}+\mathbf{n}+\mathbf{p})^\ast, \mathbf{b}^\ast)+\langle \mathbf{m}-\mathbf{b},\mathbf{n}+\mathbf{p}\rangle+2\langle \mathbf{a}-\mathbf{m},\mathbf{p}\rangle-2\langle\mathbf{q},\mathbf{j}\rangle.
\end{flalign*}

Using Proposition \ref{prop:Hall-equality-P-I}, for any $W\in \widetilde{\A}$ and $I\in \mathcal{I}_{\widetilde{\A}}$, we have
\begin{flalign*}
&|_W\Hom_{\widetilde{\A}}(P\oplus N,\tau M)_{\tau B\oplus I}|=\\
&\sum_{\substack{[D],[A],[Q],[G]\\ [J],[J'],[J'']\\G\oplus D\cong W\\ J'\oplus J''\cong I}}q^{\lr{ \mathbf{p}-\mathbf{q},\mathbf{j}}+\lr{ \mathbf{p},\mathbf{n}-{\bf d}}}|_D\Hom_{\widetilde{\A}}(N,\tau M)_{\tau A\oplus J}|\cdot |_Q\Hom_{\widetilde{\A}}(P,\tau A)_{\tau B\oplus J'}|\cdot |_G\Hom_{\widetilde{\A}}(Q,J)_{J''}|.
\end{flalign*}
Noting that
\begin{flalign*}\langle \mathbf{p}-\mathbf{q},\mathbf{j}\rangle+\langle\mathbf{p},\mathbf{n}-\mathbf{d}\rangle&=\lr{\mathbf{p},\mathbf{n}+\mathbf{j}}-\lr{\mathbf{p},\mathbf{d}}-\lr{\mathbf{q},\mathbf{j}}\\
&=\lr{\mathbf{p},\mathbf{d}+\tau(\mathbf{m}-\mathbf{a})}-\lr{\mathbf{p},\mathbf{d}}-\lr{\mathbf{q},\mathbf{j}}\\
&=\langle \mathbf{a}-\mathbf{m},\mathbf{p}\rangle-\langle\mathbf{q},\mathbf{j}\rangle,
\end{flalign*}
we obtain
\begin{flalign*}
u_{M\oplus(N\oplus P)}=\sum_{[W],[B],[I]}v^{\Lambda((\mathbf{m}+\mathbf{n}+\mathbf{p})^\ast, \mathbf{b}^\ast)+\langle \mathbf{m}-\mathbf{b},\mathbf{n}+\mathbf{p}\rangle}|_W\Hom_\mathcal{A}(P\oplus N,\tau M)_{\tau B\oplus I}|u_B\star u_{W\oplus I[-1]}.
\end{flalign*}
That is, $r_{M,N\oplus P}\in \mathfrak{I}'$. Therefore, we complete the proof.
\end{proof}

As we mentioned before, the Auslander--Reiten translation $\overline{\tau}$ of the derived category $\mathcal{D}^b(\widetilde{\A})$ provides an algebra automorphism $\overline{\tau}$ of $\mathcal {D}\mathcal {H}_\Lambda(\widetilde{\A})$, but it cannot be restricted to be an algebra endomorphism of $\mathcal {D}\mathcal {H}_\Lambda^{ec}(\widetilde{\A})$, since the image of $\overline{\tau}$ on $\mathcal {D}\mathcal {H}_\Lambda^{ec}(\widetilde{\A})$ may be outside the algebra $\mathcal {D}\mathcal {H}_\Lambda^{ec}(\widetilde{\A})$.

Consider the linear map $\rho:\mathcal {D}\mathcal {H}_\Lambda^{ec}(\widetilde{\A})\rightarrow\mathcal {D}\mathcal {H}_\Lambda^{ec}(\widetilde{\A})$ defined by
$$\rho(u_{I[-1]\oplus M\oplus P[1]})=v^{\Lambda({\bf i}^\ast,({\bf m}-{\bf p})^\ast)+\Lambda({\bf m}^\ast,{\bf p}^\ast)}u_I\star u_{\tau M'\oplus \nu(P')[-1]}\star u_{\nu(P)}$$ for any $M=M'\oplus P'\in\widetilde{\A}$, $I\in\I_{\widetilde{\A}}$ and $P\in\P_{\widetilde{\A}}$, where $P'$ is the maximal projective direct summand of $M$. It is easy to see the image of $\rho$ is determined by
the Auslander--Reiten translation of the cluster category of $\widetilde{\A}$.

Thus, we get a linear map $\varrho:\mathcal {D}\mathcal {H}_\Lambda^{ec}(\widetilde{\A})\rightarrow\mathcal {D}\mathcal {H}_\Lambda^{{cl}}(\widetilde{\A})$ by considering the composition of the map $\rho$ and the natural epimorphism $\pi: \mathcal {D}\mathcal {H}_\Lambda^{ec}(\widetilde{\A})\rightarrow\mathcal {D}\mathcal {H}_\Lambda^{{cl}}(\widetilde{\A})$.
\begin{proposition}
The map $\varrho:\mathcal {D}\mathcal {H}_\Lambda^{ec}(\widetilde{\A})\longrightarrow\mathcal {D}\mathcal {H}_\Lambda^{{cl}}(\widetilde{\A})$ defined by
$$\varrho(u_{I[-1]\oplus M\oplus P[1]})=v^{\Lambda({\bf i}^\ast,({\bf m}-{\bf p})^\ast)+\Lambda({\bf m}^\ast,{\bf p}^\ast)}u_I\star u_{\tau M'\oplus \nu(P')[-1]}\star u_{\nu(P)}$$ for any $M=M'\oplus P'\in\widetilde{\A}$, $I\in\I_{\widetilde{\A}}$ and $P\in\P_{\widetilde{\A}}$, where $P'$ is the maximal projective direct summand of $M$, is a homomorphism of algebras.
\end{proposition}
\begin{proof}
It suffices to prove the relations \eqref{lgx2}-\eqref{lgx8} are preserved under the map $\varrho$.

Noting that $\varrho(u_M)=\overline{\tau}(u_M)\in\mathcal {D}\mathcal {H}_\Lambda^{ec}(\widetilde{\A})$ and $\varrho(u_{P[1]})=\overline{\tau}(u_{P[1]})\in\mathcal {D}\mathcal {H}_\Lambda^{ec}(\widetilde{\A})$ for any $M\in\widetilde{\A}$ and $P\in\P_{\widetilde{\A}}$, and $\overline{\tau}$ is an algebra automorphism of $\mathcal {D}\mathcal {H}_\Lambda(\widetilde{\A})$,
we conclude that the relations \eqref{lgx2}-\eqref{lgx3} are preserved under the map $\varrho$. Since $u_I\star u_J=v^{\Lambda({\bf i}^\ast,{\bf j}^\ast)}u_{I\oplus J}$ for any $I,J\in\I_{\widetilde{\A}}$, we get that the relation \eqref{lgx1} is preserved under the map $\varrho$.

\noindent{\bf Relation \eqref{lgx8}:} For any $P\in\P_{\widetilde{\A}}$ and $I\in\I_{\widetilde{\A}}$, set $J=\nu(P)$, then
\begin{flalign*}
\varrho(u_{I[-1]})\star \varrho(u_{P[1]})=u_I\star u_J=v^{\Lambda({\bf i}^\ast,{\bf j}^\ast)}u_{I\oplus J}
\end{flalign*}
and
\begin{flalign*}
\varrho(u_{P[1]})\star \varrho(u_{I[-1]})=u_J\star u_I=v^{\Lambda({\bf j}^\ast,{\bf i}^\ast)}u_{I\oplus J}.
\end{flalign*}
Thus, we obtain
$$\varrho(u_{I[-1]})\star \varrho(u_{P[1]})=q^{\Lambda({\bf i}^\ast,{\bf j}^\ast)}\varrho(u_{P[1]})\star \varrho(u_{I[-1]}).$$
Since \begin{flalign*}\Lambda({\bf i}^\ast,{\bf j}^\ast)=\Lambda({^*{\bf i}},{^*{\bf j}})=\Lambda({^*{\bf i}},{\bf p}^\ast)
&=\Lambda({\bf i}^*,{\bf p}^\ast)-\Lambda(B(\widetilde{Q})(\bf i),{\bf p}^\ast)\\&=\Lambda({\bf i}^*,{\bf p}^\ast)-\lr{{\bf p},{\bf i}},\end{flalign*}
we get the relation \eqref{lgx8} is preserved under the map $\varrho$.

In order to finish proving that $\varrho$ is an algebra homomorphism, by Proposition \ref{mint27}, we need to prove the relations \eqref{fgx1} and \eqref{fgx2} are preserved under the map $\varrho$.

\noindent{\bf Relation \eqref{fgx1}:} Let us show that the relation \eqref{fgx1}
\begin{flalign*}
u_{P}\star u_{J[-1]}=q^{-\frac{1}{2}\Lambda({\bf p}^\ast,{\bf j}^\ast)-\lr{{\bf p},{\bf j}}}
\sum\limits_{[Q],[J']}|{}_{Q}\Hom_{\widetilde{\A}}(P,J)_{J'}|u_{Q\oplus J'[-1]},
\end{flalign*}
where $P\in\P_{\widetilde{\A}}$~and $J\in\I_{\widetilde{\A}}$, is preserved under the map $\varrho$. In fact, set $I=\nu(P)$, then
\begin{flalign*}
&\varrho(u_P)\star\varrho(u_{J[-1]})=u_{I[-1]}\star u_J=u_{P[1]}\star u_J\\
&=q^{-\frac{1}{2}\Lambda({\bf p}^\ast,{\bf j}^\ast)-\lr{{\bf p},{\bf j}}}\sum\limits_{[Q],[J']}|{}_{Q}\Hom_{\widetilde{\A}}(P,J)_{J'}|u_{J'\oplus Q[1]}.
\end{flalign*}

On the other hand, we have
\begin{flalign*}
&q^{-\frac{1}{2}\Lambda({\bf p}^\ast,{\bf j}^\ast)-\lr{{\bf p},{\bf j}}}
\sum\limits_{[Q],[J']}|{}_{Q}\Hom_{\widetilde{\A}}(P,J)_{J'}|\varrho(u_{Q\oplus J'[-1]})\\
&=q^{-\frac{1}{2}\Lambda({\bf p}^\ast,{\bf j}^\ast)-\lr{{\bf p},{\bf j}}}
\sum\limits_{[Q],[J']}|{}_{Q}\Hom_{\widetilde{\A}}(P,J)_{J'}|v^{\Lambda({{\bf j}'}^\ast,{\bf q}^\ast)}u_{J'}\star u_{\nu(Q)[-1]}\\
&=q^{-\frac{1}{2}\Lambda({\bf p}^\ast,{\bf j}^\ast)-\lr{{\bf p},{\bf j}}}
\sum\limits_{[Q],[J']}|{}_{Q}\Hom_{\widetilde{\A}}(P,J)_{J'}|v^{\Lambda({{\bf j}'}^\ast,{\bf q}^\ast)}u_{J'}\star u_{Q[1]}\\
&=q^{-\frac{1}{2}\Lambda({\bf p}^\ast,{\bf j}^\ast)-\lr{{\bf p},{\bf j}}}
\sum\limits_{[Q],[J']}|{}_{Q}\Hom_{\widetilde{\A}}(P,J)_{J'}|u_{J'\oplus Q[1]}.
\end{flalign*}
Hence, the relation \eqref{fgx1} is preserved under the map $\varrho$.

\noindent{\bf Relation \eqref{fgx2}:} Let us show that the relation \eqref{fgx2}
\begin{equation}\label{zyygx1}
u_{M}\star u_{J[-1]}=q^{-\frac{1}{2}\Lambda({\bf m}^\ast,{\bf j}^\ast)-\lr{{\bf m},{\bf j}}}
\sum\limits_{[G],[J']}|{}_{G}\Hom_{\widetilde{\A}}(M,J)_{J'}|u_{G\oplus J'[-1]},\end{equation}
where $M\in\widetilde{\A}$ has no nonzero projective direct summands and $J\in\I_{\widetilde{\A}}$, is preserved under the map $\varrho$. In fact, let $J=P'\oplus \widetilde{J}$ such that $P'$ is the maximal projective summand of $J$, then
\begin{flalign*}
\varrho(u_M)\star\varrho(u_{J[-1]})=u_{\tau M}\star u_J=v^{\Lambda((\tau{\bf m})^\ast,{\bf j}^\ast)}u_{J\oplus\tau M}.
\end{flalign*}
Note that in $\mathcal {D}\mathcal {H}_\Lambda^{{cl}}(\widetilde{\A})$, by the ideal $\mathfrak{I}$, we have
\begin{flalign*}
&\varrho(u_M)\star\varrho(u_{J[-1]})=\\&v^{-\Lambda({\bf m}^\ast,^*{\bf j})}\sum\limits_{\begin{smallmatrix}[D],[J'],[I]\end{smallmatrix}}v^{\Lambda(({\bf j}+{\tau{\bf m}})^*,{{\bf j}'}^*)+\lr{{\bf j}-{\bf j}',\tau{\bf m}}}|_{\tau D}\Hom_{\widetilde{\A}}(\tau M,\tau J)_{\tau J'\oplus I}|u_{J'}\star u_{\tau D\oplus I[-1]},
\end{flalign*}
where $D$ has no nonzero projective direct summands and $J'=P'\oplus I'\in\I_{\widetilde{\A}}$ has the same maximal projective summand as $J$. Moreover,
\begin{flalign*}|_{\tau D}\Hom_{\widetilde{\A}}(\tau M,\tau J)_{\tau J'\oplus I}|&=|_{\tau D}\Hom_{\widetilde{\A}}(\tau M,\tau \widetilde{J})_{\tau I'\oplus I}|\\&=|\Hom_{\mathcal{D}^b(\widetilde{\A})}(\tau M,\tau \widetilde{J})_{\tau D[1]\oplus \tau I'\oplus I}|\\
&=|\Hom_{\mathcal{D}^b(\widetilde{\A})}( M, \widetilde{J})_{ D[1]\oplus  I'\oplus P[1]}|\\
&=|_{G}\Hom_{\widetilde{\A}}(M,\widetilde{J})_{I'}|,\end{flalign*}
where $P=\nu^{-1}(I)$ and $G=D\oplus P$. Noting that $\Hom_{\widetilde{\A}}(M,P')=0$ and using Proposition \ref{prop:Hall-equality-P-I}(1), we obtain
$$|_{G}\Hom_{\widetilde{\A}}(M,\widetilde{J})_{I'}|=|_{G}\Hom_{\widetilde{\A}}(M,J)_{J'}|.$$
Thus, we have
\begin{equation}\label{48zs}
\varrho(u_M)\star\varrho(u_{J[-1]})=\sum\limits_{\begin{smallmatrix}[G],[J']\end{smallmatrix}}v^{x_2}|_{G}\Hom_{\widetilde{\A}}(M,J)_{J'}|u_{J'}\star u_{\tau D\oplus I[-1]},
\end{equation}
where $x_2=-\Lambda({\bf m}^\ast,{^*{\bf j}})+\Lambda(({\bf j}+{\tau{\bf m}})^*,{{\bf j}'}^*)+\lr{{\bf j}-{\bf j}',\tau{\bf m}}$.
Using Lemma \ref{sjishu} and \cite[Lemma 6.4]{CDZ}, we obtain
\begin{equation*}
x_2=-\Lambda({\bf m}^\ast,{^*{\bf j}})+\Lambda({^*{\bf j}},{^*{\bf j}'})-\Lambda({\bf m}^\ast,{^*{\bf j}'})-\lr{{\bf m},{\bf j}-{{\bf j}'}}.
\end{equation*}

On the other hand, for each $G$ in \eqref{zyygx1}, write $G=D\oplus P$ such that $P$ is the maximal projective direct summand of $G$, then we have
\begin{equation}\label{ffzs}
\begin{split}&q^{-\frac{1}{2}\Lambda({\bf m}^\ast,{\bf j}^\ast)-\lr{{\bf m},{\bf j}}}
\sum\limits_{[G],[J']}|{}_{G}\Hom_{\widetilde{\A}}(M,J)_{J'}|\varrho(u_{G\oplus J'[-1]})\\
&=\sum\limits_{[G],[J']}|{}_{G}\Hom_{\widetilde{\A}}(M,J)_{J'}|v^{y_2}u_{J'}\star u_{\tau D\oplus I[-1]},
\end{split}\end{equation}
where $y_2=-\Lambda({\bf m}^\ast,{\bf j}^\ast)-2\lr{{\bf m},{\bf j}}+\Lambda({{\bf j}'}^\ast,{\bf g}^\ast)$ and $I=\nu(P)$.
Using ${\bf g}={\bf m}-{\bf j}+{\bf j}'$ and Lemma \ref{sjishu}, we obtain
\begin{equation*}
\begin{split}
y_2&=-\Lambda({\bf m}^\ast,{\bf j}^\ast)-2\lr{{\bf m},{\bf j}}+\Lambda({{\bf j}'}^\ast,({\bf m}-{\bf j}+{\bf j}')^\ast)\\
&=-\Lambda({\bf m}^\ast-B(\widetilde{Q}){\bf m},{^*{\bf j}})-2\lr{{\bf m},{\bf j}}+\Lambda({^*{\bf j}'},{\bf m}^\ast-B(\widetilde{Q}){\bf m})-\Lambda({^*{\bf j}'},{^*{\bf j}}).
\end{split}\end{equation*}
By Lemma \ref{sjishu}, we get $x_2=y_2$. Thus, the equations \eqref{48zs} and \eqref{ffzs} are equal. Hence, the relation \eqref{fgx2} is preserved under the map $\varrho$.

Therefore, $\varrho$ is an algebra homomorphism.
\end{proof}

The algebra homomorphism $\varrho$ induces an algebra automorphism of $\mathcal {D}\mathcal {H}_\Lambda^{{cl}}(\widetilde{\A})$ as follows.
\begin{theorem}\label{zmainthm}
The map $\sigma:\mathcal {D}\mathcal {H}_\Lambda^{{cl}}(\widetilde{\A})\longrightarrow\mathcal {D}\mathcal {H}_\Lambda^{{cl}}(\widetilde{\A})$ defined by
$$\sigma(u_{I[-1]\oplus M\oplus P[1]})=v^{\Lambda({\bf i}^\ast,({\bf m}-{\bf p})^\ast)+\Lambda({\bf m}^\ast,{\bf p}^\ast)}u_I\star u_{\tau M'\oplus \nu(P')[-1]}\star u_{\nu(P)}$$ for any $M=M'\oplus P'\in\widetilde{\A}$, $I\in\I_{\widetilde{\A}}$ and $P\in\P_{\widetilde{\A}}$, where $P'$ is the maximal projective direct summand of $M$, is an automorphism of algebras. Moreover, the inverse map of $\sigma$ is given by the map $\sigma':\mathcal {D}\mathcal {H}_\Lambda^{{cl}}(\widetilde{\A})\longrightarrow\mathcal {D}\mathcal {H}_\Lambda^{{cl}}(\widetilde{\A})$ defined by
$$\sigma'(u_{I[-1]\oplus M\oplus P[1]})=v^{\Lambda({\bf i}^\ast,({\bf m}-{\bf p})^\ast)+\Lambda({\bf m}^\ast,{\bf p}^\ast)}u_{\nu^{-1}I}\star u_{\tau^{-1} M'\oplus \nu^{-1}(I')[1]}\star u_{P}$$ for any $M=M'\oplus I'\in\widetilde{\A}$, $I\in\I_{\widetilde{\A}}$ and $P\in\P_{\widetilde{\A}}$, where $I'$ is the maximal injective direct summand of $M$.
\end{theorem}
\begin{proof}
By the definition of $\varrho$, we have $$\varrho(u_{\nu^{-1}(I)[1]})=\varrho(u_{I[-1]})=u_I$$ for any $I\in\I_{\widetilde{\A}}$. Thus, $\varrho(\mathfrak{I}_1)=0$. Hence, the algebra homomorphism $\varrho$ induces an algebra homomorphism $\theta:\mathcal {D}\mathcal {H}_\Lambda^{{cl}_1}(\widetilde{\A})\rightarrow\mathcal {D}\mathcal {H}_\Lambda^{{cl}}(\widetilde{\A}), u\mapsto \varrho(u)$.
Now, let us prove that $\theta(\mathfrak{I})=0$. By Lemma \ref{lxdjkh}, we need to prove $\theta(r_{M,Q'})=0$ and $\theta(r_{M,N})=0$ for any $M\in{\widetilde{\A}}$, $Q'\in \mathcal{P}_{\widetilde{\A}}$ and $N\in{\widetilde{\A}}$ without nonzero projective direct summands. Let us write $M=M'\oplus P'$ such that $P'$ is the maximal projective direct summand of $M$.

\noindent{\bf Relation} $\theta(r_{M,Q'})=0$: Set $I'=\nu(P')$ and $J'=\nu(Q')$, then
\begin{flalign*}
\theta(u_{M\oplus Q'})&=\theta(u_{M'\oplus P'\oplus Q'})=u_{\tau M'\oplus (I'\oplus J')[-1]}\\
&=v^{\Lambda(({\bf i}'+{\bf j}')^\ast,(\tau {\bf m}')^\ast)-\Lambda({{\bf i}'}^\ast,{{\bf j}'}^\ast)}u_{I'[-1]}\star u_{J'[-1]}\star u_{\tau M'}\\
&=v^{\Lambda(({\bf i}'+{\bf j}')^\ast,(\tau {\bf m}')^\ast)-\Lambda({{\bf i}'}^\ast,{{\bf j}'}^\ast)}u_{I'[-1]}\star u_{Q'[1]}\star u_{\tau M'}.
\end{flalign*}
Using the relation \eqref{lgx7}, we obtain
\begin{equation}\label{mpgxl}
\theta(u_{M\oplus Q'})=\sum\limits_{[Q],[A'],[I]}v^{x_0}|_Q\Hom_{\widetilde{\A}}(Q',\tau M')_{\tau A'\oplus I}|u_{I'[-1]}\star u_{\tau A'\oplus I}\star u_{Q[1]},
\end{equation}
where each $Q\in\P_{\widetilde{\A}}$, $I\in\I_{\widetilde{\A}}$, $A'\in\widetilde{\A}$ has no nonzero projective direct summands, and
$$x_0=\Lambda(({\bf i}'+{\bf j}')^\ast,(\tau {\bf m}')^\ast)-\Lambda({{\bf i}'}^\ast,{{\bf j}'}^\ast)-\Lambda({{\bf q}'}^\ast,(\tau{\bf m}')^\ast)-2\lr{{\bf q}',\tau{\bf m}'}+\Lambda((\tau {\bf a}'+{\bf i})^\ast,{\bf q}^\ast).$$

On the other hand,
\begin{flalign*}
&\sum\limits_{\begin{smallmatrix}[Q],[A],[I]\end{smallmatrix}}v^{\Lambda(({\bf m}+{\bf q}')^*,{\bf a}^*)+\lr{{\bf m}-{\bf a},{\bf q}'}}|_Q\Hom_{\widetilde{\A}}(Q',\tau M')_{\tau A'\oplus I}|\theta(u_{A})\star \theta(u_{Q\oplus I[-1]})=\\
&\sum\limits_{\begin{smallmatrix}[Q],[A'],[I]\end{smallmatrix}}v^{\Lambda(({\bf m}+{\bf q}')^*,{\bf a}^*)+\lr{{\bf m}-{\bf a},{\bf q}'}+\Lambda({\bf i}^\ast,{\bf q}^\ast)}|_Q\Hom_{\widetilde{\A}}(Q',\tau M')_{\tau A'\oplus I}|u_{\tau A'\oplus I'[-1]}\star u_I\star u_{J[-1]},
\end{flalign*}
where we recall that $A=A'\oplus P'$ for some $A'\in\widetilde{\A}$ without nonzero projective direct summands, $I\in\I_{\widetilde{\A}}$, $Q\in\P_{\widetilde{\A}}$ and $J=\nu(Q)$. Using the relations \eqref{lgx4} and \eqref{lgx3} together with $u_{J[-1]}=u_{Q[1]}$, we get
\begin{equation}\label{mpgxr}
\begin{split}
&\sum\limits_{\begin{smallmatrix}[Q],[A],[I]\end{smallmatrix}}v^{\Lambda(({\bf m}+{\bf q}')^*,{\bf a}^*)+\lr{{\bf m}-{\bf a},{\bf q}'}}|_Q\Hom_{\widetilde{\A}}(Q',\tau M')_{\tau A'\oplus I}|\theta(u_{A})\star \theta(u_{Q\oplus I[-1]})\\
&=\sum\limits_{\begin{smallmatrix}[Q],[A'],[I]\end{smallmatrix}}v^{y_0}|_Q\Hom_{\widetilde{\A}}(Q',\tau M')_{\tau A'\oplus I}|u_{I'[-1]}\star u_{\tau A'\oplus I}\star u_{Q[1]},
\end{split}\end{equation}
where $$y_0=\Lambda(({\bf m}+{\bf q}')^*,{\bf a}^*)+\lr{{\bf m}-{\bf a},{\bf q}'}+\Lambda({\bf i}^\ast,{\bf q}^\ast)+\Lambda({{\bf i}'}^\ast,(\tau{\bf a}')^\ast)+\Lambda((\tau{\bf a}')^\ast,{\bf i}^\ast).$$

Using \cite[Lemma 6.4]{CDZ}, we obtain
\begin{flalign*}
x_0=-\Lambda({{\bf p}'}^\ast+{{\bf q}'}^\ast,{{\bf m}'}^\ast)-\Lambda({{\bf p}'}^\ast,{{\bf q}'}^\ast)+\Lambda({^*{\bf q}'},{{\bf m}'}^\ast)+2\lr{{\bf m}',{\bf q}'}-\Lambda({{\bf a}'}^\ast,{^\ast{\bf q}})+\Lambda({^\ast{\bf i}},{^\ast{\bf q}})
\end{flalign*}
and
\begin{flalign*}
y_0=\Lambda({{\bf m}'}^\ast+{{\bf p}'}^\ast+{{\bf q}'}^\ast,{{\bf a}'}^\ast+{{\bf p}'}^\ast)+\lr{{\bf m}'-{\bf a}',{\bf q}'}+\Lambda({^\ast{\bf i}},{^\ast{\bf q}})-\Lambda({{\bf p}'}^\ast,{{\bf a}'}^\ast)-\Lambda({{\bf a}'}^\ast,{^\ast{\bf i}}).
\end{flalign*}
Noting that ${\bf q}={\bf q}'-\tau{\bf m}'+\tau{\bf a}'+{\bf i}$ and ${^*{\bf q}}={\bf q}^\ast-B(\widetilde{Q}){\bf q}$, by Lemma \ref{sjishu} and \cite[Lemma 6.4]{CDZ}, we obtain $x_0=y_0$. Thus, the equations \eqref{mpgxl} and \eqref{mpgxr} are equal, i.e. $\theta(r_{M,Q'})=0$.

\noindent{\bf Relation} $\theta(r_{M,N})=0$:
Set $I'=\nu(P')$, then
\begin{flalign*}
\theta(u_{M\oplus N})=u_{\tau M'\oplus \tau N\oplus I'[-1]}=v^{\Lambda({{\bf i}'}^\ast,(\tau{\bf m}'+\tau{\bf n})^\ast)}u_{I'[-1]}\star u_{\tau M'\oplus\tau N}.
\end{flalign*}
Note that in $\mathcal {D}\mathcal {H}_\Lambda^{{cl}}(\widetilde{\A})$, by the ideal $\mathfrak{I}$, we have
\begin{equation}\label{kxjih}
\begin{split}
&u_{\tau M'\oplus\tau N}=\\&\sum\limits_{\begin{smallmatrix}[K],[X],[I'']\end{smallmatrix}}v^{\Lambda((\tau{\bf m}'+\tau{\bf n})^*,{\bf x}^*)+\lr{\tau{\bf m}'-{\bf x},\tau{\bf n}}}|_K\Hom_{\widetilde{\A}}(\tau N,\tau(\tau M'))_{\tau X\oplus I''}|u_{X}\star u_{K\oplus I''[-1]},
\end{split}
\end{equation}
where $K$ has no nonzero injective direct summands, $X$ has the same maximal projective direct summands as $\tau M'$.

For each $K$ and $X$ in \eqref{kxjih}, write $K=\tau D'$ and $X=\tau A'\oplus I$ for some $D', A'\in\widetilde{\A}$ and $I\in\I_{\widetilde{\A}}$ such that $D', A', I$ have no nonzero projective direct summands. Thus,
\begin{equation*}
\begin{split}
u_{\tau M'\oplus\tau N}=\sum\limits_{\begin{smallmatrix}[D'],[A'],[I],[I'']\end{smallmatrix}}&v^{\Lambda((\tau{\bf m}'+\tau{\bf n})^*,(\tau{\bf a}'+{\bf i})^*)+\lr{\tau{\bf m}'-\tau{\bf a}'-{\bf i},\tau{\bf n}}}\\&|_{\tau D'}\Hom_{\widetilde{\A}}(\tau N,\tau(\tau M'))_{\tau (\tau A'\oplus I)\oplus I''}|u_{\tau A'\oplus I}\star u_{\tau D'\oplus I''[-1]}.
\end{split}
\end{equation*}
Noting that $u_{\tau A'\oplus I}=v^{-\Lambda((\tau{\bf a}')^\ast,{\bf i}^\ast)}u_{\tau A'}\star u_I$, we obtain
\begin{equation}\label{vzzs}
\begin{split}
&\theta(u_{M\oplus N})=\\&\sum\limits_{\begin{smallmatrix}[D'],[A'],[I],[I'']\end{smallmatrix}}v^{x_1}|_{\tau D'}\Hom_{\widetilde{\A}}(\tau N,\tau(\tau M'))_{\tau (\tau A'\oplus I)\oplus I''}|u_{I'[-1]}\star u_{\tau A'}\star u_I\star u_{\tau D'\oplus I''[-1]},
\end{split}
\end{equation}
where $$x_1=\Lambda({{\bf i}'}^\ast,(\tau{\bf m}'+\tau{\bf n})^\ast)+\Lambda((\tau{\bf m}'+\tau{\bf n})^*,(\tau{\bf a}'+{\bf i})^*)+\lr{\tau{\bf m}'-\tau{\bf a}'-{\bf i},\tau{\bf n}}-\Lambda((\tau{\bf a}')^\ast,{\bf i}^\ast).$$
Recall that $N$ has no nonzero projective direct summand, and assume that $P'''$ is the maximal projective summand of $\tau M'$. Set $P''=\nu^{-1}(I'')$, $D=D'\oplus P''$ and $I'''=\nu(P''')$. Since $\Hom_{D^b(\widetilde{\A})}(\tau N, I'''[-1])=0$, we obtain
\begin{equation*}\begin{split}
&|_{\tau D'}\Hom_{\widetilde{\A}}(\tau N,\tau(\tau M'))_{\tau (\tau A'\oplus I)\oplus I''}|\\
&=|\Hom_{\mathcal{D}^b(\widetilde{\A})}(\tau N,\tau(\tau M'))_{(\tau D')[1]\oplus\tau (\tau A'\oplus I)\oplus I''}|\\
&=|\Hom_{\mathcal{D}^b(\widetilde{\A})}(\tau N,\tau(\tau M')\oplus {I'''}[-1])_{(\tau D')[1]\oplus\tau (\tau A'\oplus I)\oplus I''\oplus{I'''}[-1]}|\\
&=|\Hom_{\mathcal{D}^b(\widetilde{\A})}(\overline{\tau} N,\overline{\tau}(\tau M'))_{\overline{\tau} (\tau A'\oplus I)\oplus \overline{\tau}D[1]}|\\
&=|\Hom_{\mathcal{D}^b(\widetilde{\A})}(N,\tau M')_{\tau A'\oplus I\oplus D[1]}|\\
&=|_D\Hom_{\widetilde{\A}}(N,\tau M')_{\tau A'\oplus I}|.\end{split}\end{equation*}

On the other hand,
\begin{flalign*}
&\sum\limits_{\begin{smallmatrix}[D],[A],[I]\end{smallmatrix}}v^{\Lambda(({\bf m}+{\bf n})^*,{\bf a}^*)+\lr{{\bf m}-{\bf a},{\bf n}}}|_D\Hom_{\widetilde{\A}}(N,\tau M)_{\tau A\oplus I}|\theta(u_{A})\star \theta(u_{D\oplus I[-1]})=\\
&\sum\limits_{\begin{smallmatrix}[D'],[A'],[I],[I'']\end{smallmatrix}}v^{\Lambda(({\bf m}+{\bf n})^*,{\bf a}^*)+\lr{{\bf m}-{\bf a},{\bf n}}}|_D\Hom_{\widetilde{\A}}(N,\tau M')_{\tau A'\oplus I}|u_{\tau A'\oplus I'[-1]}\star v^{\Lambda({\bf i}^\ast,{\bf d}^\ast)}u_{I}\star u_{\tau D'\oplus I''[-1]},
\end{flalign*}
where we recall that $A=A'\oplus P'$ for some $A'\in\widetilde{\A}$, set $D=D'\oplus P''$ for some $D'\in\widetilde{\A}$ and $P''\in\P_{\widetilde{\A}}$ such that $D', A'$ have no nonzero projective direct summands, $I'=\nu(P')$ and $I''=\nu(P'')$. Noting that $u_{\tau A'\oplus I'[-1]}=v^{\Lambda({{\bf i}'}^\ast,(\tau{\bf a}')^\ast)}u_{I'[-1]}\star u_{\tau A'}$, we obtain
\begin{equation}\label{vrzs}
\begin{split}
&\sum\limits_{\begin{smallmatrix}[D],[A],[I]\end{smallmatrix}}v^{\Lambda(({\bf m}+{\bf n})^*,{\bf a}^*)+\lr{{\bf m}-{\bf a},{\bf n}}}|_D\Hom_{\widetilde{\A}}(N,\tau M)_{\tau A\oplus I}|\theta(u_{A})\star \theta(u_{D\oplus I[-1]})\\
&=\sum\limits_{\begin{smallmatrix}[D'],[A'],[I],[I'']\end{smallmatrix}}v^{y_1}|_D\Hom_{\widetilde{\A}}(N,\tau M')_{\tau A'\oplus I}| u_{I'[-1]}\star u_{\tau A'}\star u_{I}\star u_{\tau D'\oplus I''[-1]},
\end{split}\end{equation}
where $$y_1=\Lambda(({\bf m}+{\bf n})^*,{\bf a}^*)+\Lambda({{\bf i}'}^\ast,(\tau{\bf a}')^\ast)+\Lambda({\bf i}^\ast,{\bf d}^\ast)+\lr{{\bf m}-{\bf a},{\bf n}}.$$

Set $P=\nu^{-1}(I)$. Using \cite[Lemma 6.4]{CDZ}, we obtain
\begin{flalign*}
x_1=-\Lambda({{\bf p}'}^\ast,{{\bf m}'}^\ast+{\bf n}^\ast)-\Lambda({{\bf m}'}^\ast+{\bf n}^\ast,{\bf p}^\ast-{{\bf a}'}^\ast)+\lr{{\bf m}'-{\bf a}',{\bf n}}
+\lr{{\bf n},{\bf i}}+\Lambda({{\bf a}'}^\ast,{\bf p}^\ast).
\end{flalign*}
Noting that ${\bf d}={\bf i}+{\bf n}-\tau({\bf m}'-{\bf a}')$ and using \cite[Lemma 6.4]{CDZ}, we obtain
\begin{equation*}
\begin{split}
&y_1=\\&\Lambda(({{\bf m}'}+{{\bf p}'}+{\bf n})^\ast,({{\bf a}'}+{{\bf p}'})^\ast)-\Lambda({{\bf p}'}^\ast,{{\bf a}'}^\ast)+\Lambda({\bf i}^\ast,({\bf i}+{\bf n}-\tau({\bf m}'-{\bf a}'))^\ast)+\lr{{{\bf m}'-{\bf a}'},{\bf n}}\\
&=\Lambda(({{\bf m}'}+{{\bf p}'}+{\bf n})^\ast,({{\bf a}'}+{{\bf p}'})^\ast)-\Lambda({{\bf p}'}^\ast,{{\bf a}'}^\ast)+\Lambda({\bf p}^\ast,({\bf m}'-{\bf a}')^\ast+{^*{\bf n}})+\lr{{{\bf m}'-{\bf a}'},{\bf n}}.
\end{split}
\end{equation*}
Using ${^*{\bf n}}={\bf n}^\ast-B(\widetilde{Q}){\bf n}$ and Lemma \ref{sjishu}, we get $x_1=y_1$. Hence, the equations \eqref{vzzs} and \eqref{vrzs} are equal, i.e. $\theta(r_{M,N})=0$.

Hence, the algebra homomorphism $\theta$ induces the algebra homomorphism $\sigma$.

By similar arguments, we can obtain the algebra homomorphism $\sigma'$, whose proof has been provided in Section Appendix A. Then it is easy to see $\sigma\sigma'={\rm Id}$ and $\sigma'\sigma={\rm Id}$. Therefore, $\sigma$ and $\sigma'$ are automorphisms of algebras.
\end{proof}

\begin{corollary}
We have the following commutative diagram of algebra homomorphisms in which the vertical maps are isomorphisms:
\begin{equation*}
\xymatrix{{\A}_{q}(\Lambda, B(\widetilde{Q}))\,\,\ar@{>->}[r]^-{\varphi'_0}\ar[d]^-{\Sigma}&\mathcal {D}\mathcal {H}_\Lambda^{{cl}}(\widetilde{\A})\ar[d]^-{\sigma}\\
{\A}_{q}(\Lambda, B(\widetilde{Q}))\,\,\ar@{>->}[r]^-{\varphi'_0}&\mathcal {D}\mathcal {H}_\Lambda^{{cl}}(\widetilde{\A}).}
\end{equation*}
\end{corollary}
\begin{proof}
Using Corollary \ref{congshixian} for $n=m$, we have
the algebra isomorphism $$\varphi'_0:{\A}_{q}(\Lambda, B(\widetilde{Q}))\longrightarrow\mathcal {C}\mathcal {H}_\Lambda^{{cl}}(\widetilde{\A}),$$
where $\mathcal {C}\mathcal {H}_\Lambda^{{cl}}(\widetilde{\A})$ is the subalgebra of $\mathcal {D}\mathcal {H}_\Lambda^{{cl}}(\widetilde{\A})$ generated by the elements $u_M, u_{I[-1]}$, where $M\in\widetilde{\A}$ is indecomposable and rigid, and $I\in \mathcal{I}_{\widetilde{\A}}$ is indecomposable. It is easy to see that
$\sigma(\mathcal {C}\mathcal {H}_\Lambda^{{cl}}(\widetilde{\A}))\subseteq\mathcal {C}\mathcal {H}_\Lambda^{{cl}}(\widetilde{\A})$ and $\sigma'(\mathcal {C}\mathcal {H}_\Lambda^{{cl}}(\widetilde{\A}))\subseteq\mathcal {C}\mathcal {H}_\Lambda^{{cl}}(\widetilde{\A})$. It follows that
$\sigma:\mathcal {C}\mathcal {H}_\Lambda^{{cl}}(\widetilde{\A})\to \mathcal {C}\mathcal {H}_\Lambda^{{cl}}(\widetilde{\A})$ is an algebra isomorphism.
Hence, the algebra isomorphism $\Sigma$ is defined by the following commutative diagram
\begin{equation*}
\xymatrix{{\A}_{q}(\Lambda, B(\widetilde{Q}))\ar[r]^-{\varphi'_0}_{\cong}\ar[d]^-{\Sigma}&\mathcal {C}\mathcal {H}_\Lambda^{{cl}}(\widetilde{\A})\ar[d]^-{\sigma}\\
{\A}_{q}(\Lambda, B(\widetilde{Q}))\ar[r]^-{\varphi'_0}_{\cong}&\mathcal {C}\mathcal {H}_\Lambda^{{cl}}(\widetilde{\A}).}
\end{equation*}
\end{proof}

%%%%%%%%%%%%%%%%%%%%%%%%%%%%%%%%%%%%%%%%%%%%%%%%%%%%

\section{BGP-reflection functors and mutation invariance of Hall algebra $\mathcal {D}\mathcal {H}_\Lambda^{{cl}}(\widetilde{\A})$}
In this section, let $\widetilde{Q}$ be a finite acyclic valued quiver as given in the subsection 2.2 and let $k\in \{1,\ldots, m\}$ be always a sink vertex.

Let $\widetilde{Q}'=\mu_k(\widetilde{Q})$ be the valued quiver obtained from $\widetilde{Q}$ by reversing all the arrows attached to $k$ and the valuations of
$\widetilde{Q}'$ equal those of $\widetilde{Q}$. Let $\widetilde{\mathcal{A}'}$ be the category of finite-dimensional representations of $\widetilde{Q}'$ over $\mathbb{F}_q$. In what follows, for each vertex $i$ of $\widetilde{Q}$, denote the projective cover and injective envelope of the simple representation $S_i$ by $P_i$ and $I_i$, respectively. Clearly, $P_k=S_k$. For the sake of distinction, for each vertex $i$ of $\widetilde{Q}'$, denote the projective cover and injective envelope of the simple representation associated to $i$ by $P'_i$ and $I'_i$, respectively.
As the subsection 2.2, we have the matrices $E(\widetilde{Q}')$, $E'(\widetilde{Q}')$, $B(\widetilde{Q}')$. Moreover, $B(\widetilde{Q}')$ is equal to the mutation $\mu_k(B(\widetilde{Q}))$ of $B(\widetilde{Q})$. Let ${\Lambda'}:=\mu_k(\Lambda)$ be the mutation of $\Lambda$, then $\Lambda'B(\widetilde{Q}')=-\operatorname{diag}\{d_1,\cdots, d_m\}$.
We also write $^\ast\x=E(\widetilde{Q}')\x$ and $\x^\ast=E'(\widetilde{Q}')\x$ for any $X\in\widetilde{\mathcal{A}'}$.

Let $\mathbf{r}_k^+:\widetilde{\mathcal{A}}\rightarrow\widetilde{\mathcal{A}'}$ be the classical BGP-reflection functor (cf. \cite{Dlab}). Denote by $\widetilde{\mathcal{A}}\langle k\rangle$ the full subcategory of $\widetilde{\mathcal{A}}$ consisting of all representations which do not contain $S_k$ as a direct summand. Similarly, we define the full subcategory $\widetilde{\mathcal{A}'}\langle k\rangle$ of $\widetilde{\mathcal{A}'}$. Then the BGP-reflection functor $\mathbf{r}_k^+$ restricts to an equivalence of categories $\mathbf{r}_k^+:\widetilde{\mathcal{A}}\langle k\rangle\rightarrow\widetilde{\mathcal{A}'}\langle k\rangle$, whose quasi-inverse may be denoted by $\mathbf{r}_k^-$. Moreover, the BGP-reflection functor $\mathbf{r}_k^+$ induces a triangle equivalence
$\mathbf{R}_k^+:\mathcal{D}^b(\widetilde{\A})\rightarrow\mathcal{D}^b(\widetilde{\A'})$. Note that $\mathbf{R}_k^+(X)=\mathbf{r}_k^+(X)\in\widetilde{\mathcal{A}'}\langle k\rangle$ for any $X\in\widetilde{\mathcal{A}}\langle k\rangle$. Besides, we have the following:
\begin{equation}\label{6.1}
\mathbf{R}_k^+(P_k)=I_k'[-1],~~\mathbf{R}_k^+(P_k[1])=I_k',~~\mathbf{R}_k^+(\tau^{-1}P_k)=P_k',~~\mathbf{R}_k^+(I_k[-1])=\tau I_k'[-1];
\end{equation}
and for any $j\neq k$,
\begin{equation}
\mathbf{R}_k^+(P_j)=P_j',~~\mathbf{R}_k^+(I_j)=I_j',~~\mathbf{R}_k^+(P_j[1])=P_j'[1],~~\mathbf{R}_k^+(I_j[-1])=I_j'[-1],
\end{equation}
where we denote the Auslander--Reiten translation of $\widetilde{\A'}$ by the same notation $\tau$ for $\widetilde{\A}$.

For each $1\leq i\leq m$, define the map $s_i:\mathbb{Z}^m\rightarrow \mathbb{Z}^m$ by $s_i(\alpha)=\alpha-d_i^{-1}(\alpha,e_i)e_i$, called the {\em simple reflection}. It is well-known that
\begin{equation}\label{eulerequal}
\lr{{\bf x},{\bf y}}_{\widetilde{\A}}=\lr{s_k({\bf x}),s_k({\bf y})}_{\widetilde{\A'}}\end{equation} for any $\mathbf{x}, \mathbf{y}\in \mathbb{Z}^m$.
According to \cite[Proposition 2.1]{Dlab}, it is easy to see that
\begin{equation}\label{wsxlbc}
{\Dim}\mathbf{R}_k^+(X)=s_k(\x)
\end{equation}
for any $X\in \mathcal{D}^b(\widetilde{\A})$.

\begin{lemma}\label{lem:compatible-lambda}
For any $\mathbf{x}, \mathbf{y}\in \mathbb{Z}^m$, we have $\Lambda(\mathbf{x}^\ast,\mathbf{y}^\ast)=\Lambda'((s_k(\mathbf{x}))^\ast,(s_k(\mathbf{y}))^\ast)$.
\end{lemma}
\begin{proof}
Let  $F=(f_{ij})$  be the $m\times m$ matrix with the entries given by
\[f_{ij}=\begin{cases}
\delta_{ij} & \text{if $i\ne k$;}\\
-1 & \text{if $i=k=j$;}\\
-b_{kj} & \text{if $i=k\neq j$.}
\end{cases}
\]
Then it is easy to see $s_k(\mathbf{x})=F\mathbf{x}$ for any $\mathbf{x}\in \mathbb{Z}^m$. A direct calculation shows that $GE'(\widetilde{Q}')F=E'(\widetilde{Q})$, where the matrix $G$ is given in the subsection \ref{ss:quantum-cluster-algebras}.
Thus, we have
 \begin{align*}
     \Lambda'(s_k(\mathbf{x})^\ast, s_k(\mathbf{y})^\ast)&=\mathbf{x}^{\rm tr}F^{\rm tr} E'(\widetilde{Q}')^{\rm tr}\Lambda'E'(\widetilde{Q}')F\mathbf{y}\\
     &=\mathbf{x}^{\rm tr}F^{\rm tr} E'(\widetilde{Q}')^{\rm tr}G^{\rm tr}\Lambda GE'(\widetilde{Q}')F\mathbf{y}\\
     &=\mathbf{x}^{\rm tr}E'(\widetilde{Q})^{\rm tr}\Lambda E'(\widetilde{Q})\mathbf{y}\\
     &=\Lambda(\mathbf{x}^\ast,\mathbf{y}^\ast).
 \end{align*}
 Therefore, we finish the proof.
\end{proof}

\begin{lemma}\label{lem:iso-dervied-hall-algebra-reflection}
The equivalence $\mathbf{R}_k^+$ induces an isomorphism of derived Hall algebras
\[R_k^+: \mathcal{DH}_\Lambda(\widetilde{\A})\longrightarrow\mathcal{DH}_{\Lambda'}(\widetilde{\A'})\] defined by $u_X\mapsto u_{\mathbf{R}_k^+(X)}$ for any $X\in \mathcal{D}^b(\widetilde{\A})$.
\end{lemma}
\begin{proof}
Since $\mathbf{R}_k^+$ is a triangle equivalence, by the equations \eqref{eulerequal}, \eqref{wsxlbc} and Lemma \ref{lem:compatible-lambda}, we easily finish the proof.
\end{proof}
In what follows, for the sake of simplicity in notation, we write $\acute{X}$ as $\mathbf{r}_k^+(X)$ for any $X\in\widetilde{\mathcal{A}}\langle k\rangle$. For example, we have $\acute{P}_j=P'_j$ and $\acute{I}_j=I'_j$ for any $j\neq k$, but $\acute{I}_k\neq I'_k$.

\begin{proposition}
The map $\mathfrak{b}_k^+:\mathcal{DH}_{\Lambda}^{ec}(\widetilde{\A})\to \mathcal{DH}_{\Lambda'}^{cl}(\widetilde{\A'})$ defined by
\begin{align*}
&u_{(J\oplus aI_k)[-1]}\mapsto v^{-\Lambda(\mathbf{j}^\ast,a\mathbf{i}_k^\ast)}u_{\acute{J}[-1]}\star u_{aI_k'},\\
&u_{bP_k[1]\oplus Q[1]}\mapsto v^{-\Lambda(b\mathbf{p}_k^\ast, \mathbf{q}^\ast)}u_{bI_k'}\star u_{\acute{Q}[1]},\\
&u_{cP_k\oplus M}\mapsto v^{-\Lambda(c{\mathbf{p}_k^\ast},\mathbf{m}^\ast)}u_{cI_k'[-1]}\star u_{\acute{M}}
\end{align*}
is a homomorphism of algebras,
where $a,b,c\in \mathbb{N}$, $J\in \mathcal{I}_{\widetilde{\A}}$ has no $I_k$ as direct summands,
$M, Q\in \widetilde{\mathcal{A}}\langle k\rangle$ and $Q$ is projective.
\end{proposition}
\begin{proof}
It suffices to prove the relations \eqref{lgx2}-\eqref{lgx8} are preserved under $\mathfrak{b}_k^+$.
Note that $$\mathfrak{b}_k^+(u_{bP_k[1]\oplus Q[1]})=R_k^+(u_{bP_k[1]\oplus Q[1]})\in \mathcal{DH}_{\Lambda'}^{ec}(\widetilde{\A'}),~~\mathfrak{b}_k^+(u_{cP_k\oplus M})=R_k^+(u_{cP_k\oplus M})\in \mathcal{DH}_{\Lambda'}^{ec}(\widetilde{\A'}).$$
Since $R_k^+$ is an algebra isomorphism,
we conclude that the relations \eqref{lgx2}-\eqref{lgx3} are preserved under $\mathfrak{b}_k^+$.
	
\noindent{\bf Relation \eqref{lgx1}:} Let $I=\bar{I}\oplus aI_k$ and $J=\bar{J}\oplus bI_k$, where $a,b\in \mathbb{N}$, and $\bar{I}, \bar{J}\in\mathcal{I}_{\widetilde{\A}}$ have no $I_k$ as direct summands. Note that $\Hom_{\widetilde{\A'}}(I'_k,\acute{\bar{J}})\cong\Hom_{\widetilde{\A}}(P_k[1],\bar{J})=0$.
Thus, by the relations \eqref{lgx6} and \eqref{lgx4} of $\mathcal{DH}_{\Lambda'}^{ec}(\widetilde{\A'})$, we have
\[ u_{aI_k'}\star u_{\acute{\bar{J}}[-1]}=q^{-\Lambda'(a{{\bf i}'_k}^\ast,{\acute{\bar{\bf j}}}^\ast)-\lr{ a{\bf i}_k',\acute{\bar{\bf j}}}}u_{\acute{\bar{J}}[-1]}\star u_{aI_k'}.\]
Note that \begin{equation*}\Lambda'({{\bf i}'_k}^\ast, {\acute{\bar{\bf j}}}^\ast)=-\Lambda'((s_k({\bf p}_k))^\ast,(s_k(\bar{\bf j}))^\ast)=-\Lambda({\bf p}_k^\ast,{\bar{\bf j}}^\ast)\end{equation*} and
\begin{equation*}\lr{{\bf i}_k',{\acute{\bar{\bf j}}}}=-\lr{s_k({\bf p}_k),s_k(\bar{\bf j})}=-\lr{{\bf p}_k,\bar{\bf j}}.\end{equation*}
Thus, we obtain
	 \begin{align*}
	 	\mathfrak{b}_k^+(u_{I[-1]})\star\mathfrak{b}_k^+(u_{J[-1]})&=v^{-\Lambda(\bar{\bf i}^\ast, a{\bf i}_k^\ast)-\Lambda(\bar{\bf j}^\ast, b{\bf i}_k^\ast)}u_{\acute{\bar{I}}[-1]}\star u_{aI_k'}\star u_{\acute{\bar{J}}[-1]}\star u_{bI_k'}\\
	 	&=v^{-\Lambda(\bar{\bf i}^\ast, a{\bf i}_k^\ast)-\Lambda(\bar{\bf j}^\ast, b{\bf i}_k^\ast)+2\Lambda(a{\bf p}_k^\ast,{\bar{\bf j}}^\ast)+2\lr{a{\bf p}_k,\bar{\bf j}}}u_{\acute{\bar{I}}[-1]}\star  u_{\acute{\bar{J}}[-1]}\star u_{aI_k'}\star u_{bI_k'}\\
	 	&=v^{x_0}u_{(\acute{\bar{I}}\oplus \acute{\bar{J}})[-1]}\star u_{(a+b)I_k'},
	 \end{align*}
where $x_0=-\Lambda(\bar{\bf i}^\ast, a{\bf i}_k^\ast)-\Lambda(\bar{\bf j}^\ast, b{\bf i}_k^\ast)+2\Lambda(a{\bf p}_k^\ast,{\bar{\bf j}}^\ast)+2\lr{a{\bf p}_k,\bar{\bf j}}+\Lambda(\bar{\bf i}^\ast,\bar{\bf j}^\ast)$.

On the other hand,
\[v^{\Lambda(\bar{\bf i}^\ast+a{\bf i}_k^\ast,\bar{\bf j}^\ast+b{\bf i}_k^\ast)}\mathfrak{b}_k^+(u_{I\oplus J[-1]})=v^{y_0}u_{(\acute{\bar{I}}\oplus \acute{\bar{J}})[-1]}\star u_{(a+b)I_k'},\]
where
$y_0=-\Lambda(\bar{\bf i}^\ast+\bar{\bf j}^\ast, (a+b){\bf i}_k^\ast)+\Lambda(\bar{\bf i}^\ast+a{\bf i}_k^\ast,\bar{\bf j}^\ast+b{\bf i}_k^\ast).$

Noting that \begin{align*}y_0-x_0&=2\Lambda( a{\bf i}_k^\ast,\bar{\bf j}^\ast)-2\Lambda(a{\bf p}_k^\ast,{\bar{\bf j}}^\ast)-2\lr{a{\bf p}_k,\bar{\bf j}}\\&=2a\Lambda( {^*{\bf i}_k},\bar{\bf j}^\ast)+2a\Lambda( {B(\widetilde{Q}){\bf i}_k},\bar{\bf j}^\ast)-2\Lambda(a{\bf p}_k^\ast,{\bar{\bf j}}^\ast)-2\lr{a{\bf p}_k,\bar{\bf j}}\\
&=2a\lr{ {\bar{\bf j},{\bf i}_k}}-2a\lr{{\bf p}_k,\bar{\bf j}}\\&=0,\end{align*} we conclude that the relation \eqref{lgx1} is preserved under the map $\mathfrak{b}_k^+$.

\noindent{\bf Relation \eqref{lgx8}:} Let $I=\bar{I}\oplus aI_k$ and $P=\bar{P}\oplus bP_k$, where $a,b\in \mathbb{N}$, $\bar{I}\in\mathcal{I}_{\widetilde{\A}}$ has no $I_k$ as direct summands and $\bar{P}\in\mathcal{P}_{\widetilde{\A}}$ has no $P_k$ as direct summands.
On the one hand,
	 \begin{equation*}
\begin{split}
	 	&v^{2\Lambda({\bf i}^\ast, {\bf p}^\ast)-2\langle {\bf p}, {\bf i}\rangle}\mathfrak{b}_k^+(u_{P[1]})\star \mathfrak{b}_k^+(u_{I[-1]}) \\
	 	&=v^{2\Lambda({\bf i}^\ast, {\bf p}^\ast)-2\langle {\bf p}, {\bf i}\rangle-\Lambda(\bar{\bf i}^\ast, a{\bf i}_k^\ast)-\Lambda(b{\bf p}_k^\ast,{\bar{\bf p}}^\ast)}u_{bI_k'}\star u_{\acute{\bar{P}}[1]}\star u_{\acute{\bar{I}}[-1]}\star u_{aI_k'}.
	 \end{split}\end{equation*}
By the relation \eqref{lgx8} of $\mathcal{DH}_{\Lambda'}^{ec}(\widetilde{\A'})$, we have
\begin{align*}u_{\acute{\bar{P}}[1]}\star u_{\acute{\bar{I}}[-1]}&=v^{-2\Lambda'({\acute{\bar{\bf i}}}^\ast, {\acute{\bar{\bf p}}}^\ast)+2\langle \acute{\bar{{\bf p}}},\acute{\bar{\bf i}}\rangle}u_{\acute{\bar{I}}[-1]}\star u_{\acute{\bar{P}}[1]}\\
&=v^{-2\Lambda({\bar{\bf i}}^\ast, {\bar{\bf p}}^\ast)+2\langle \bar{{\bf p}},\bar{\bf i}\rangle}u_{\acute{\bar{I}}[-1]}\star u_{\acute{\bar{P}}[1]}.\end{align*}
	 Note that $\Hom_{\widetilde{\A'}}(\acute{\bar{P}}, I_k')=0$ since $I_k'$ is a simple module and $\acute{\bar{P}}$ has no $P'_k$ as direct summands. By the relation \eqref{lgx7} of $\mathcal{DH}_{\Lambda'}^{ec}(\widetilde{\A'})$, we have
	 \[u_{\acute{\bar{P}}[1]}\star u_{aI_k'}=v^{-2\Lambda'({\acute{\bar{\bf p}}}^\ast, a{{\bf i}_k'}^\ast)-2\langle {\acute{\bar{\bf p}}},a{\bf i}_k'\rangle}u_{aI_k'}\star u_{\acute{\bar{P}}[1]},\]
where $\Lambda'({\acute{\bar{\bf p}}}^\ast, {{\bf i}_k'}^\ast)=\Lambda'((s_k(\bar{\bf p}))^\ast,(s_k(-{\bf p}_k))^\ast)=-\Lambda({\bar{\bf p}}^\ast,{\bf p}_k^\ast)$ and $\langle {\acute{\bar{\bf p}}},{\bf i}_k'\rangle=0$.
That is, \[u_{\acute{\bar{P}}[1]}\star u_{aI_k'}=v^{2a\Lambda({\bar{\bf p}}^\ast,{\bf p}_k^\ast)}u_{aI_k'}\star u_{\acute{\bar{P}}[1]}.\]
Similarly, since $\Hom_{\widetilde{\A'}}(I'_k,\acute{\bar{I}})=0$, we can get
\[u_{bI_k'}\star u_{\acute{\bar{I}}[-1]}=v^{2b\Lambda({\bf p}_k^\ast,{\bar{\bf i}}^\ast)}u_{\acute{\bar{I}}[-1]}\star u_{bI_k'}.\]
Hence, we obtain
\begin{align*}v^{2\Lambda({\bf i}^\ast, {\bf p}^\ast)-2\langle {\bf p}, {\bf i}\rangle}\mathfrak{b}_k^+(u_{P[1]})\star \mathfrak{b}_k^+(u_{I[-1]})=v^{x_1}u_{\acute{\bar{I}}[-1]}\star u_{(a+b)I_k'}\star u_{\acute{\bar{P}}[1]},\end{align*}
where \begin{align*}x_1&=2\Lambda({\bf i}^\ast, {\bf p}^\ast)-2\langle {\bf p}, {\bf i}\rangle-\Lambda(\bar{\bf i}^\ast, a{\bf i}_k^\ast)-\Lambda(b{\bf p}_k^\ast,{\bar{\bf p}}^\ast)\\&\quad-2\Lambda({\bar{\bf i}}^\ast, {\bar{\bf p}}^\ast)+2\langle \bar{{\bf p}},\bar{\bf i}\rangle+2a\Lambda({\bar{\bf p}}^\ast,{\bf p}_k^\ast)+2b\Lambda({\bf p}_k^\ast,{\bar{\bf i}}^\ast).\end{align*}
Using ${\bf i}=\bar{{\bf i}}+a{\bf i}_k$, ${\bf p}=\bar{{\bf p}}+b{\bf p}_k$, ${\bf i}_k^\ast={^\ast{{\bf i}_k}}+B(\widetilde{Q}){\bf i}_k={\bf p}_k^\ast+B(\widetilde{Q}){\bf i}_k$ and $\lr{{\bf p}_k,{\bar{\bf i}}}=0$, we get
\begin{align*}x_1&=2a\Lambda({\bf i}_k^\ast,{\bar{\bf p}}^\ast)+2ab\Lambda({\bf i}_k^\ast,{\bf p}_k^\ast)
+2a\Lambda({\bar{\bf p}}^\ast,{\bf p}_k^\ast)-2a\lr{{\bar{\bf p}},{\bf i}_k}\\&\quad-2b\lr{{\bf p}_k,{\bar{\bf i}}}-2ab\lr{{\bf p}_k,{\bf i}_k}-\Lambda(\bar{\bf i}^\ast, a{\bf i}_k^\ast)-\Lambda(b{\bf p}_k^\ast,{\bar{\bf p}}^\ast)\\
&=-\Lambda(\bar{\bf i}^\ast, a{\bf i}_k^\ast)-\Lambda(b{\bf p}_k^\ast,{\bar{\bf p}}^\ast).\end{align*}

On the other hand,
\[\mathfrak{b}_k^+(u_{I[-1]})\star \mathfrak{b}_k^+(u_{P[1]})=v^{-\Lambda(\bar{\bf i}^\ast, a{\bf i}_k^\ast)-\Lambda(b{\bf p}_k^\ast,{\bar{\bf p}}^\ast)}u_{\acute{\bar{I}}[-1]}\star u_{(a+b)I_k'}\star u_{\acute{\bar{P}}[1]}.
	 \]
Hence, we obtain that the relation \eqref{lgx8} is preserved under the map $\mathfrak{b}_k^+$.

According to Proposition \ref{mint27}, the relation \eqref{lgx6} can be replaced by the relations \eqref{fgx1} and \eqref{fgx2}.
	
\noindent{\bf Relation \eqref{fgx1}:} By using similar arguments in Proposition \ref{mint27} together with Proposition \ref{prop:Hall-equality-P-I}, we can prove that the relation \eqref{fgx1} for any $P\in\mathcal{P}_{\widetilde{\A}}$ and $I\in\mathcal{I}_{\widetilde{\A}}$ can replaced by the relations \eqref{fgx1} for the following four cases: (i) $(\bar{P},\bar{I})$; (ii) $(\bar{P}, I_k)$; (iii) $(P_k,\bar{I})$; (iv) $(P_k,I_k)$, where $\bar{I}\in\mathcal{I}_{\widetilde{\A}}$ has no $I_k$ as direct summands and $\bar{P}\in\mathcal{P}_{\widetilde{\A}}$ has no $P_k$ as direct summands.
Hence, it suffices to prove that the relations \eqref{fgx1} for such four cases are preserved under $\mathfrak{b}_k^+$. While, note that in the relation \eqref{fgx1}, $I'$ has no $I_k$ as direct summands if $I$ has no $I_k$ as direct summands. Then since $\mathfrak{b}_k^+(u_{\bar{P}})=R_k^+(u_{\bar{P}})$, $\mathfrak{b}_k^+(u_{P_k})=R_k^+(u_{P_k})$ and $\mathfrak{b}_k^+(u_{\bar{I}[-1]})=R_k^+(u_{\bar{I}[-1]})$, by Lemma \ref{lem:iso-dervied-hall-algebra-reflection}, we conclude that the relation \eqref{fgx1} for the cases (i) and (iii) are preserved under $\mathfrak{b}_k^+$.

\noindent{\bf Case(ii):} On the one hand,
$${\rm LHS}:=\mathfrak{b}_k^+(u_{\bar{P}})\star \mathfrak{b}_k^+(u_{I_k[-1]})=u_{\acute{\bar{P}}}\star u_{I_k'}=v^{\Lambda'({\acute{\bar{\bf p}}}^\ast, {{\bf i}_k'}^\ast)}u_{{\acute{\bar{P}}}\oplus I_k'}=v^{-\Lambda({\bar{\bf p}}^\ast,{\bf p}_k^\ast)}u_{{\acute{\bar{P}}}\oplus I_k'}.$$
Using the ideal $\mathfrak{I}$ relation of $\mathcal{DH}_{\Lambda'}^{cl}(\widetilde{\A'})$, we obtain
\begin{align*}
{\rm LHS}&=\sum_{[Q'],[A'],[J']}v^{-\Lambda(\bar{\bf p}^\ast, {\bf p}_k^\ast)+\Lambda'({\acute{\bar{\bf p}}}^\ast+{{\bf i}_k'}^\ast,{{\bf a}'}^\ast)+\langle {\bf i}_k'-{\bf a}',{\acute{\bar{\bf p}}}\rangle} |_{Q'}\Hom({\acute{\bar{P}}},\tau I_k')_{\tau A'\oplus J'}|u_{A'}\star u_{Q'\oplus J'[-1]}\\
&=q^{-\Lambda(\bar{\bf p}^\ast, {\bf p}_k^\ast)}u_{I_k'}\star u_{\acute{\bar{P}}}+\sum_{[Q']\neq [\acute{\bar{P}}], [J']}v^{-\Lambda(\bar{\bf p}^\ast, {\bf p}_k^\ast)+\langle {\bf i}_k',{\acute{\bar{\bf p}}}\rangle}|_{Q'}\Hom({\acute{\bar{P}}},\tau I_k')_{J'}| u_{Q'\oplus J'[-1]},
\end{align*}
where the last equality follows from the fact that $I_k'$ is simple, $A'\cong I_k'$ or $A'=0$.
	
On the other hand,
\begin{align*}
&{\rm RHS}:=\sum_{[Q],[J]}v^{-\Lambda(\bar{\bf p }^\ast,{\bf i}_k^\ast)-2\langle \bar{\bf p},{\bf i}_k\rangle+\Lambda({\bf j}^\ast,{\bf q}^\ast)}|_Q\Hom(\bar{P},I_k)_J|\mathfrak{b}_k^+(u_{J[-1]})\star \mathfrak{b}_k^+(u_Q)=\\
	 	&q^{-\Lambda(\bar{\bf p }^\ast,{\bf i}_k^\ast)-\langle \bar{\bf p},{\bf i}_k\rangle}u_{I_k'}\star u_{\acute{\bar{P}}}+\sum_{[Q]\neq [\bar{P}],[J]}v^{-\Lambda(\bar{\bf p }^\ast,{\bf i}_k^\ast)-2\langle \bar{\bf p},{\bf i}_k\rangle+\Lambda({\bf j}^\ast,{\bf q}^\ast)}|_Q\Hom(\bar{P},I_k)_J|\mathfrak{b}_k^+(u_{J[-1]})\star \mathfrak{b}_k^+(u_Q).
	 \end{align*}
For each $Q$ in the above sum, write $Q=aP_k\oplus \bar{Q}$, where $a\in\mathbb{N}$ and $\bar{Q}$ has no $P_k$ as direct summands. By applying the triangle equivalence $\mathbf{R}_k^+$, we get
\[|_{\bar{Q}\oplus aP_k}\Hom_{\widetilde{\A}}(\bar{P},I_k)_J|=|_{\acute{{\bar{Q}}}}\Hom_{\widetilde{\A'}}({\acute{\bar{P}}},\tau I_k')_{\acute{J}\oplus aI_k'}|.\]
It follows that
\begin{align*}
{\rm RHS}&=q^{-\Lambda(\bar{\bf p }^\ast,{\bf i}_k^\ast)-\langle \bar{\bf p},{\bf i}_k\rangle}u_{I_k'}\star u_{\acute{\bar{P}}}
+\sum_{[\bar{Q}]\neq [\bar{P}],[J]}v^{x'_2}|_{\acute{{\bar{Q}}}}\Hom({\acute{\bar{P}}},\tau I_k')_{\acute{J}\oplus aI_k'}|u_{\acute{J}[-1]}\star u_{aI_k'[-1]}\star u_{\acute{\bar{Q}}}\\
	 	&=q^{-\Lambda(\bar{\bf p }^\ast,{\bf i}_k^\ast)-\langle \bar{\bf p},{\bf i}_k\rangle}u_{I_k'}\star u_{\acute{\bar{P}}}+\sum_{{[{\acute{\bar{Q}}}]\neq [{\acute{\bar{P}}}],[\acute{J}]}}v^{x_2}
	 	|_{\acute{{\bar{Q}}}}\Hom({\acute{\bar{P}}},\tau I_k')_{\acute{J}\oplus aI_k'}| u_{{\acute{\bar{Q}}}\oplus({\acute{J}}\oplus aI_k')[-1]},
	 \end{align*}
where $x'_2=-\Lambda(\bar{\bf p }^\ast,{\bf i}_k^\ast)-2\langle \bar{\bf p},{\bf i}_k\rangle+\Lambda({\bf j}^\ast,\bar{\bf q}^\ast+a{\bf p}_k^\ast)-\Lambda(a{\bf p}_k^\ast,\bar{\bf q}^\ast)$ and $x_2=x'_2+\Lambda'({\acute{{\bf j}}}^\ast,a{{\bf i}_k'}^\ast)-\Lambda'(a{{\bf i}_k'}^\ast+{\acute{{\bf j}}}^\ast, {\acute{\bar{\bf q}}}^\ast)$.

Note that $\Lambda(\bar{\bf p}^\ast, {\bf p}_k^\ast)=\Lambda(\bar{\bf p}^\ast, {^\ast}{\bf i}_k)=\Lambda(\bar{\bf p}^\ast, {\bf i}_k^\ast)+\langle \bar{\bf p},{\bf i}_k\rangle$ and $\langle {\bf i}_k',{\acute{\bar{\bf p}}}\rangle=-\langle {\bf p}_k,\bar{\bf p}\rangle=-\langle \bar{\bf p},{\bf i}_k\rangle$. Thus, we get
\begin{align*}\Lambda(\bar{\bf p}^\ast, {\bf p}_k^\ast)-\langle {\bf i}_k',{\acute{\bar{\bf p}}}\rangle+x_2&=\Lambda({\bf j}^\ast,\bar{\bf q}^\ast+a{\bf p}_k^\ast)-\Lambda(a{\bf p}_k^\ast,\bar{\bf q}^\ast)+\Lambda'({\acute{{\bf j}}}^\ast,a{{\bf i}_k'}^\ast)-\Lambda'(a{{\bf i}_k'}^\ast+{\acute{{\bf j}}}^\ast, {\acute{\bar{\bf q}}}^\ast)\\
&=\Lambda({\bf j}^\ast,\bar{\bf q}^\ast+a{\bf p}_k^\ast)-\Lambda(a{\bf p}_k^\ast,\bar{\bf q}^\ast)-\Lambda({\bf j}^\ast, a{\bf p}_k^\ast)-\Lambda(-a{\bf p}_k^\ast+{\bf j}^\ast,\bar{\bf q}^\ast)\\
	&=0.\end{align*}
Hence, replacing the notations $Q'$ and $J'$ in {\rm LHS} by $\acute{\bar{Q}}$ and $\acute{J}\oplus aI_k'$, respectively,
we obtain
${\rm LHS}={\rm RHS}$, i.e. the relation \eqref{fgx1} for Case (ii) is preserved under the map $\mathfrak{b}_k^+$.

\noindent{\bf Case (iv):} On the one hand,
$${\rm LHS}:=\mathfrak{b}_k^+(u_{P_k})\star \mathfrak{b}_k^+(u_{I_k[-1]})=u_{I'_k[-1]}\star u_{I_k'}.$$
Since $u_{I_k'[-1]}=u_{P_k'[1]}$ in $\mathcal{DH}_{\Lambda'}^{cl}(\widetilde{\A'})$ and $I'_k$~is simple, we have	
\begin{align*}
{\rm LHS}&=u_{P'_k[1]}\star u_{I_k'}	
=q^{-\frac{1}{2}\Lambda'({{\bf p}_k'}^\ast,{{\bf i}_k'}^\ast)-\lr{{\bf p}'_k,{\bf i}'_k}}
\sum\limits_{[Q'],[I']}v^{\Lambda'({{\bf i}'}^\ast,{{\bf q}'}^\ast)}|{}_{Q'}\Hom(P'_k,I'_k)_{I'}|u_{I'}\star u_{Q'[1]}\\
&=q^{-\Lambda'({{\bf p}_k'}^\ast,{{\bf i}_k'}^\ast)-\lr{{\bf p}'_k,{\bf i}'_k}}u_{I'_k}\star u_{P'_k[1]}+q^{-\frac{1}{2}\Lambda'({{\bf p}_k'}^\ast,{{\bf i}_k'}^\ast)-\lr{{\bf p}'_k,{\bf i}'_k}}
\sum\limits_{[Q']}|{}_{Q'}\Hom(P'_k,I'_k)_{0}|u_{Q'[1]}.
\end{align*}

On the other hand,
\begin{align*}
&{\rm RHS}:=q^{-\frac{1}{2}\Lambda({\bf p}_k^\ast,{\bf i}_k^\ast)-\lr{{\bf p}_k,{\bf i}_k}}
\sum\limits_{[Q],[J]}v^{\Lambda({{\bf j}}^\ast,{{\bf q}}^\ast)}|{}_{Q}\Hom(P_k,I_k)_{J}|\mathfrak{b}_k^+(u_{J[-1]})\star \mathfrak{b}_k^+(u_{Q})\\
&=q^{-\Lambda({\bf p}_k^\ast,{\bf i}_k^\ast)-\lr{{\bf p}_k,{\bf i}_k}}u_{I'_k}\star u_{I'_k[-1]}+
q^{-\frac{1}{2}\Lambda({\bf p}_k^\ast,{\bf i}_k^\ast)-\lr{{\bf p}_k,{\bf i}_k}}
\sum\limits_{[J]}|{}_{0}\Hom(P_k,I_k)_{J}|u_{\acute{J}[-1]}\\
&=q^{-\Lambda({\bf p}_k^\ast,{\bf i}_k^\ast)-\lr{{\bf p}_k,{\bf i}_k}}u_{I'_k}\star u_{I'_k[-1]}+
q^{-\frac{1}{2}\Lambda({\bf p}_k^\ast,{\bf i}_k^\ast)-\lr{{\bf p}_k,{\bf i}_k}}
\sum\limits_{[J]}|{}_{0}\Hom(P_k,I_k)_{J}|u_{\nu^{-1}(\acute{J})[1]},
\end{align*}
where $\nu$ is the Nakayama functor of $\widetilde{{\A'}}$.

Noting that \begin{align*}\Lambda({\bf p}_k^\ast,{\bf i}_k^\ast)=\Lambda(^\ast{\bf i}_k,{\bf i}_k^\ast)=
-\lr{{\bf i}_k,{\bf i}_k}=-\lr{{\bf p}_k,{\bf i}_k},
\end{align*}
we get \begin{align*}&-\Lambda({\bf p}_k^\ast,{\bf i}_k^\ast)-\lr{{\bf p}_k,{\bf i}_k}=-\Lambda'({{\bf p}_k'}^\ast,{{\bf i}_k'}^\ast)-\lr{{\bf p}'_k,{\bf i}'_k}=0,\\
&-\frac{1}{2}\Lambda({\bf p}_k^\ast,{\bf i}_k^\ast)-\lr{{\bf p}_k,{\bf i}_k}=-\frac{1}{2}\lr{{\bf p}_k,{\bf i}_k}~\text{and}~-\frac{1}{2}\Lambda'({{\bf p}_k'}^\ast,{{\bf i}_k'}^\ast)-\lr{{\bf p}'_k,{\bf i}'_k}=-\frac{1}{2}\lr{{\bf p}'_k,{\bf i}'_k}.\end{align*}
It follows that $$-\frac{1}{2}\Lambda({\bf p}_k^\ast,{\bf i}_k^\ast)-\lr{{\bf p}_k,{\bf i}_k}=-\frac{1}{2}\Lambda'({{\bf p}_k'}^\ast,{{\bf i}_k'}^\ast)-\lr{{\bf p}'_k,{\bf i}'_k},$$ since $$\Hom(P'_k,I'_k)\cong \Hom(I'_k,I'_k)\cong \Hom(P_k,P_k)\cong \Hom(P_k,I_k).$$

By applying the triangle equivalence $\mathbf{R}_k^+$ and the Auslander--Reiten translation of $\mathcal{D}^b(\widetilde{{\A'}})$, we get
\begin{align*}|{}_{0}\Hom(P_k,I_k)_{J}|=|{}_{\nu^{-1}(\acute{J})}\Hom(P'_k,I'_k)_{0}|.\end{align*}
Hence, replacing the notations $Q'$ in {\rm LHS} by $\nu^{-1}(\acute{J})$, we obtain
${\rm LHS}={\rm RHS}$, i.e. the relation \eqref{fgx1} for Case (iv) is also preserved under the map $\mathfrak{b}_k^+$.
	
\noindent{\bf Relation \eqref{fgx2}:}
By using similar arguments in Proposition \ref{mint27} together with Proposition \ref{prop:Hall-equality-P-I}, we can prove that the relation \eqref{fgx2} for any $M\in\widetilde{\A}$ without nonzero projective direct summands and $I\in\mathcal{I}_{\widetilde{\A}}$ can replaced by the relations \eqref{fgx2} for the following two cases: (i) $(M,\bar{I})$; (ii) $(M, I_k)$, where $M\in\widetilde{\A}$ has no nonzero projective direct summands and $\bar{I}\in\mathcal{I}_{\widetilde{\A}}$ has no $I_k$ as direct summands.

\noindent{\bf Case (i):}
One the one hand,
\begin{align*}
{\rm LHS}:=\mathfrak{b}_k^+(u_M)\star\mathfrak{b}_k^+(u_{\bar{I}[-1]})=u_{\acute{M}}\star u_{\acute{\bar{I}}[-1]}
=\sum_{[G'],[J']}v^{x_3} |_{G'}\Hom({\acute{M}},\acute{\bar{I}})_{J'}|u_{J'[-1]}\star u_{G'},
\end{align*}
where $x_3=-\Lambda'({\acute{{\bf m}}}^\ast,{\acute{\bar{{\bf i}}}}^\ast)-2\langle \acute{{\bf m}},\acute{\bar{{\bf i}}}\rangle+\Lambda'({\bf j'}^\ast,{\bf g'}^\ast)$.

On the other hand,
\begin{align*}
{\rm RHS}:=v^{-\Lambda({\bf m}^\ast,{\bar{{\bf i}}}^\ast)-2\lr{{\bf m},\bar{{\bf i}}}}
\sum\limits_{[G],[J]}v^{\Lambda({{\bf j}}^\ast,{\bf g}^\ast)}|{}_{G}\Hom(M,\bar{I})_{J}|\mathfrak{b}_k^+(u_{J[-1]})\star \mathfrak{b}_k^+(u_G),\end{align*}
where each $J$ has no $I_k$ as direct summands.

Writing each $G$ in the above sum as $G=aP_k\oplus \bar{G}$, where $a\in\mathbb{N}$ and $\bar{G}$ has no $P_k$ as direct summands, we have
\begin{align*}{\rm RHS}&=v^{-\Lambda({\bf m}^\ast,{\bar{{\bf i}}}^\ast)-2\lr{{\bf m},\bar{{\bf i}}}}
\sum\limits_{[G],[J]}v^{\Lambda({{\bf j}}^\ast,{\bf g}^\ast)}|{}_{G}\Hom(M,\bar{I})_{J}|u_{\acute{J}[-1]}\star v^{-\Lambda(a{\bf p}_k^\ast,\bar{\bf g}^\ast)}u_{aI_k'[-1]}\star u_{\acute{\bar{G}}}\\
&=
\sum\limits_{[G],[J]}v^{y_3}|{}_{G}\Hom(M,\bar{I})_{J}|u_{(\acute{J}\oplus aI_k')[-1]}\star u_{\acute{\bar{G}}},
\end{align*}
where $y_3=-\Lambda({\bf m}^\ast,{\bar{{\bf i}}}^\ast)-2\lr{{\bf m},\bar{{\bf i}}}+\Lambda({{\bf j}}^\ast,{\bf g}^\ast)-\Lambda(a{\bf p}_k^\ast,\bar{\bf g}^\ast)+\Lambda'(\acute{{\bf j}}^\ast,a{{\bf i}'_k}^\ast)$.

Replacing the notations $G'$ and $J'$ in {\rm LHS} by $\acute{\bar{G}}$ and $\acute{J}\oplus aI_k'$, respectively,
we obtain
\[|_{G'}\Hom({\acute{M}},\acute{\bar{I}})_{J'}|=|_{\acute{\bar{G}}}\Hom({\acute{M}},\acute{\bar{I}})_{\acute{J}\oplus aI_k'}|=|_{aP_k\oplus \bar{G}}\Hom(M,\bar{I})_J|=|_G\Hom(M,\bar{I})_J|\]
by applying the  triangle equivalence $\mathbf{R}_k^+$.

Noting that \begin{align*}x_3&=-\Lambda({\bf m}^\ast,{\bar{{\bf i}}}^\ast)-2\lr{{\bf m},\bar{{\bf i}}}+\Lambda'(\acute{{\bf j}}^\ast+a{{\bf i}'_k}^\ast,\acute{\bar{{\bf g}}}^\ast)\\
&=-\Lambda({\bf m}^\ast,{\bar{{\bf i}}}^\ast)-2\lr{{\bf m},\bar{{\bf i}}}+\Lambda({\bf j}^\ast,\bar{{\bf g}}^\ast)-\Lambda(a{{\bf p}_k}^\ast,\bar{{\bf g}}^\ast)\end{align*}
and $y_3-x_3=\Lambda({{\bf j}}^\ast,a{\bf p}_k^\ast)+\Lambda'(\acute{{\bf j}}^\ast,a{{\bf i}'_k}^\ast)=\Lambda({{\bf j}}^\ast,a{\bf p}_k^\ast)-\Lambda({{\bf j}}^\ast,a{\bf p}_k^\ast)=0$, we get ${\rm LHS}={\rm RHS}$, i.e. the relation \eqref{fgx2} for Case (i) is preserved under the map $\mathfrak{b}_k^+$.

\noindent{\bf Case (ii):}  One the one hand,
\begin{align*}
{\rm LHS}:=\mathfrak{b}_k^+(u_M)\star\mathfrak{b}_k^+(u_{I_k[-1]})=u_{\acute{M}}\star u_{I_k'}
=v^{\Lambda'(\acute{{\bf m}}^\ast,{{\bf i}_k'}^\ast)}u_{\acute{M}\oplus I_k'}.\end{align*}

Using the ideal $\mathfrak{I}$ relation of $\mathcal{DH}_{\Lambda'}^{cl}(\widetilde{\A'})$, we obtain
\begin{align*}
&{\rm LHS}=\sum_{[D'], [A'],[I']}v^{-\Lambda({\bf m}^\ast,{\bf p}_k^\ast)+\Lambda'(\acute{{\bf m}}^\ast+{{\bf i}_k'}^\ast, {\bf a'}^\ast)+\lr{{\bf i}_k'-{\bf a'},\acute{{\bf m}}}}|_{D'}\Hom(\acute{M},\tau I_k')_{\tau A'\oplus I'}|u_{A'}\star u_{D'\oplus I'[-1]}.\end{align*}
Noting that $I'_k$ is simple and each $A'$ as above is a quotient of $I'_k$, we have
\begin{align*}{\rm LHS}=v^{-\Lambda({\bf m}^\ast,{\bf p}_k^\ast)+\Lambda'(\acute{{\bf m}}^\ast, {{\bf i}_k'}^\ast)}u_{I_k'}\star u_{\acute{M}}
+\sum_{[D']\neq [\acute{M}],[I']}v^{x_4}|_{D'}\Hom(\acute{M},\tau I_k')_{ I'}|u_{D'\oplus I'[-1]},
\end{align*}
where $x_4=-\Lambda({\bf m}^\ast,{\bf p}_k^\ast)+\lr{{\bf i}_k',\acute{{\bf m}}}$.

On the other hand,	
\begin{align*}
{\rm RHS}:=v^{-\Lambda({\bf m}^\ast,{\bf i}_k^\ast)-2\lr{{\bf m},{{\bf i}_k}}}\sum_{[G],[J]}v^{\Lambda({\bf j}^\ast,{\bf g}^\ast)}|_G\Hom(M,I_k)_J|\mathfrak{b}_k^+(u_{J[-1]})\star \mathfrak{b}_k^+(u_G).\end{align*}
For each $G$ in the above sum, write it as $G=\bar{G}\oplus aP_k$, where $a\in\mathbb{N}$ and $\bar{G}$ has no $P_k$ as direct summands. Then we have
\begin{align*}
{\rm RHS}=v^{-2\Lambda({\bf m}^\ast,{\bf i}_k^\ast)-2\lr{{\bf m},{{\bf i}_k}}}u_{I_k'}\star u_{\acute{M}}
+\sum_{[\bar{G}\oplus aP_k]\neq [M],[J]}v^{y_4}|_{\bar{G}\oplus aP_k}\Hom(M,I_k)_J|u_{{\acute{\bar{G}}}\oplus (\acute{J}\oplus aI_k')[-1]},
\end{align*}
where \begin{align*}y_4=-\Lambda({\bf m}^\ast,{\bf i}_k^\ast)-2\lr{{\bf m},{{\bf i}_k}}+\Lambda({\bf j}^\ast, \bar{\bf g}^\ast+a{\bf p}_k^\ast)-\Lambda(a{\bf p}_k^\ast,\bar{\bf g}^\ast)+\Lambda'(\acute{{\bf j}}^\ast,a{{\bf i}_k'}^\ast)-\Lambda'(\acute{{\bf j}}^\ast+a{{\bf i}_k'}^\ast,\acute{\bar{\bf g}}^\ast).\end{align*}

Replacing the notations $D'$ and $I'$ in {\rm LHS} by $\acute{\bar{G}}$ and $\acute{J}\oplus aI_k'$, respectively,
we obtain
\[|_{D'}\Hom(\acute{M},\tau I_k')_{ I'}|=|_{\acute{\bar{G}}}\Hom(\acute{M},\tau I_k')_{ \acute{J}\oplus aI_k'}|=|_{\bar{G}\oplus aP_k }\Hom(M,I_k)_J|\]
by applying the  triangle equivalence $\mathbf{R}_k^+$.

Noting that
\[-\Lambda({\bf m}^\ast,{\bf p}_k^\ast)+\Lambda'(\acute{{\bf m}}^\ast, {{\bf i}_k'}^\ast)=-2\Lambda({\bf m}^\ast,{\bf p}_k^\ast)=-2\Lambda({\bf m}^\ast,{\bf i}_k^\ast)-2\lr{{\bf m},{\bf{i}_k}}\]
and
\begin{align*}
y_4=&-\Lambda({\bf m}^\ast,{\bf i}_k^\ast)-2\lr{{\bf m},{{\bf i}_k}}+\Lambda({\bf j}^\ast, \bar{\bf g}^\ast+a{\bf p}_k^\ast)\\&-\Lambda(a{\bf p}_k^\ast,\bar{\bf g}^\ast)-\Lambda({\bf j}^\ast, a{\bf p}_k^\ast)-\Lambda({\bf j}^\ast, \bar{\bf g}^\ast)+\Lambda(a{\bf p}_k^\ast,\bar{\bf g}^\ast)\\
=&-\Lambda({\bf m}^\ast,{\bf i}_k^\ast)-2\lr{{\bf m},{{\bf i}_k}}=x_4,
\end{align*}
we obtain ${\rm LHS}={\rm RHS}$, i.e. the relation \eqref{fgx2} for Case (ii) is also preserved under $\mathfrak{b}_k^+$.

Therefore, we complete the proof.
\end{proof}

\begin{proposition}
The homomorphism  $\mathfrak{b}_k^+:\mathcal{DH}_{\Lambda}^{ec}(\widetilde{\A})\to \mathcal{DH}_{\Lambda'}^{cl}(\widetilde{\A'})$ induces an algebra homomorphism
$b_k^+:\mathcal{DH}_{\Lambda}^{cl}(\widetilde{\A})\to \mathcal{DH}_{\Lambda'}^{cl}(\widetilde{\A'})$ defined by
\begin{align*}
&u_{(J\oplus aI_k)[-1]}\mapsto v^{-\Lambda(\mathbf{j}^\ast,a\mathbf{i}_k^\ast)}u_{\acute{J}[-1]}\star u_{aI_k'},\\
&u_{bP_k[1]\oplus Q[1]}\mapsto v^{-\Lambda(b\mathbf{p}_k^\ast, \mathbf{q}^\ast)}u_{bI_k'}\star u_{\acute{Q}[1]},\\
&u_{cP_k\oplus M}\mapsto v^{-\Lambda(c{\mathbf{p}_k^\ast},\mathbf{m}^\ast)}u_{cI_k'[-1]}\star u_{\acute{M}},
\end{align*}
where $a,b,c\in \mathbb{N}$, $J\in \mathcal{I}_{\widetilde{\A}}$ has no $I_k$ as direct summands,
$M, Q\in \widetilde{\mathcal{A}}\langle k\rangle$ and $Q$ is projective.
\end{proposition}
\begin{proof}
Since $\mathfrak{b}_k^+(P_k[1])=\mathfrak{b}_k^+(I_k[-1])=u_{I'_k}$,
$\mathfrak{b}_k^+(P_j[1])=u_{P'_j[1]}$ and $\mathfrak{b}_k^+(I_j[-1])=u_{I'_j[-1]}$ for any $j\neq k$, we get
$\mathfrak{b}_k^+(P_t[1])=\mathfrak{b}_k^+(I_t[-1])$ in $\mathcal{DH}_{\Lambda'}^{cl}(\widetilde{\A'})$ for any $1\leq t\leq m$. Thus, it is easy to see
$\mathfrak{b}_k^+(u_{\nu^{-1}(I)[1]})=\mathfrak{b}_k^+(u_{I[-1]})$ for any $I\in\I_{\widetilde{\A}}$,
i.e. $\mathfrak{b}_k^+(\mathfrak{I}_1)=0$. Hence, the algebra homomorphism $\mathfrak{b}_k^+$ induces an algebra homomorphism $\mathcal {D}\mathcal {H}_\Lambda^{{cl}_1}(\widetilde{\A})\rightarrow\mathcal {D}\mathcal {H}_{\Lambda'}^{{cl}}(\widetilde{\A'})$, also denoted by $\mathfrak{b}_k^+$.
Now, let us prove $\mathfrak{b}_k^+(\mathfrak{I})=0$. By Lemma \ref{lxdjkh}, we need to prove $\mathfrak{b}_k^+(r_{M,N})=0$ and $\mathfrak{b}_k^+(r_{M,P})=0$ for any $M\in{\widetilde{\A}}$, $P\in \mathcal{P}_{\widetilde{\A}}$ and $N\in{\widetilde{\A}}$ without nonzero projective direct summands.

\noindent{\bf Relation} $\mathfrak{b}_k^+(r_{M,N})=0$:
We remark that the following proof holds for any $N\in{\widetilde{\A}}$ which has no $P_k$ as direct summands. Let $M=aP_k\oplus \bar{P}_M\oplus \bar{M}$, where $a\in\mathbb{N}$, $\bar{P}_M\in \mathcal{P}_{\widetilde{\A}}$ has no $P_k$ as direct summands, and $\bar{M}$ has no nonzero projective direct summands. By definition, we have
\begin{align*}
{\rm LHS}:=\mathfrak{b}_k^+(u_{M\oplus N})=v^{-\Lambda(a\mathbf{p}_k^\ast,(\bar{\bf p}_M+\bar{\bf m}+{\bf n})^\ast)}u_{aI_k'[-1]}\star u_{\acute{\bar{P}}_M\oplus \acute{\bar{M}}\oplus \acute{N}}.
\end{align*}
	By the ideal $\mathfrak{I}$ relation of $\mathcal{DH}_{\Lambda'}^{cl}(\widetilde{\A'})$, we have
	\begin{align*}
		u_{\acute{\bar{P}}_M\oplus \acute{\bar{M}}\oplus \acute{N}}&=\sum_{[D'],[A'],[I']}v^{\Lambda'((\acute{\bar{\mathbf{p} }}_M+\acute{\bar{\mathbf{m} }}+\acute{\mathbf{n} })^\ast,{\mathbf{a}'}^\ast)+\lr{\acute{\bar{\mathbf{p} }}_M+\acute{\bar{\mathbf{m} }}-{\mathbf{a}'},\acute{\mathbf{n}}}}|_{D'}\Hom(\acute{N},\tau \acute{\bar{M}})_{\tau A'\oplus I'}|u_{A'}\star u_{D'\oplus I'[-1]},
	\end{align*}
	where $A'$ has the same maximal projective direct summand as $\acute{\bar{P}}_M\oplus\acute{\bar{M}}$. Write $A'=\bar{A}'\oplus \acute{\bar{P}}_M$. We remark that $\bar{A}'$ may have $P_k'$ as direct summands, and the multiplicity of $P_k'$ in $\bar{A}'$ is equal to the multiplicity of $\tau^{-1}P_k$ in $\bar{M}$.  Thus, we get
	\begin{equation}\label{zuoshoubian}
		{\rm LHS}=\sum_{[D'],[A'],[I']}v^{z_0}
		|_{D'}\Hom(\acute{N},\tau \acute{\bar{M}})_{\tau A'\oplus I'}|u_{aI_k'[-1]}\star u_{A'}\star u_{D'\oplus I'[-1]},
	\end{equation}
where $z_0=-\Lambda(a\mathbf{p}_k^\ast,(\bar{\bf p}_M+\bar{\bf m}+{\bf n})^\ast)+\Lambda'((\acute{\bar{\mathbf{p} }}_M+\acute{\bar{\mathbf{m} }}+\acute{\mathbf{n} })^\ast,{\mathbf{a}'}^\ast)+\lr{\acute{\bar{\mathbf{p} }}_M+\acute{\bar{\mathbf{m} }}-{\mathbf{a}'},\acute{\mathbf{n}}}$.

Let
	\begin{align*}
		&{\rm RHS}:=\sum_{[D],[A],[I]}v^{\Lambda(({\bf m}+{\bf n})^\ast, {\bf a}^\ast)+\lr{{\bf m}-{\bf a},{\bf n}}}|_D\Hom(N,\tau \bar{M})_{\tau A\oplus I}|\mathfrak{b}_k^+(u_A)\star \mathfrak{b}_k^+(u_{D\oplus I[-1]})\\
&=\sum_{[D],[A],[I]}v^{\Lambda(({\bf m}+{\bf n})^\ast, {\bf a}^\ast)+\lr{{\bf m}-{\bf a},{\bf n}}+\Lambda({\bf i}^\ast,{\bf d}^\ast)}|_D\Hom(N,\tau \bar{M})_{\tau A\oplus I}|\mathfrak{b}_k^+(u_A)\star \mathfrak{b}_k^+(u_{I[-1]})\star \mathfrak{b}_k^+(u_{D}),
	\end{align*}
where each $A$ has the same maximal projective direct summand as $M$ and $I\in\I_{\widetilde{\A}}$.

Write $A=\bar{A}\oplus \bar{P}_M\oplus aP_k$, $D=\bar{D}\oplus dP_k$ and $I=\bar{I}\oplus cI_k$, where $\bar{A}$ has no nonzero projective direct summands, $c,d\in\mathbb{N}$, $\bar{D}$ has no $P_k$ as direct summands and $\bar{I}$ has no $I_k$ as direct summands. By the assumption of $N$, $\Hom(N,P_k)=0$. It follows that the multiplicity of $\tau^{-1}P_k$ in $\bar{A}$ is equal to the multiplicity of $\tau^{-1}P_k$ in $\bar{M}$. Then we have
\begin{align*}
{\rm RHS}=\sum_{[D],[A],[I]}v^{y'_0}|_D\Hom(N,\tau \bar{M})_{\tau A\oplus I}|u_{aI_k'[-1]}\star u_{\acute{\bar{A}}\oplus \acute{\bar{P}}_M}\star u_{\acute{\bar{I}}[-1]}\star u_{cI'_k}\star u_{dI_k'[-1]}\star u_{\acute{\bar{D}}},
\end{align*}
where
\begin{align*}
y'_0=\Lambda(({\bf m}+{\bf n})^\ast, {\bf a}^\ast)+\lr{{\bf m}-{\bf a},{\bf n}}+\Lambda({\bf i}^\ast,{\bf d}^\ast)-\Lambda(a{\bf p}_k^\ast,(\bar{\bf a}+\bar{\bf p}_M)^\ast)-\Lambda(\bar{\bf i}^\ast,c{\bf i}_k^\ast)-\Lambda(d{\bf p}_k^\ast, \bar{\bf d}^\ast).
\end{align*}
Since $\Hom_{\widetilde{\A'}}(I'_k,\acute{\bar{I}})=0$, we have
$$u_{\acute{\bar{I}}[-1]}\star u_{cI'_k}=q^{-\Lambda'({\acute{\bar{{\bf i}}}}^\ast,c{{\bf i}'_k}^\ast)}u_{cI'_k}\star u_{\acute{\bar{I}}[-1]}=q^{c\Lambda({\bar{{\bf i}}}^\ast,{\bf p}_k^\ast)}u_{cI'_k}\star u_{\acute{\bar{I}}[-1]}.$$
Note that
\begin{align*}2c\Lambda({\bar{{\bf i}}}^\ast,{\bf p}_k^\ast)-\Lambda(\bar{\bf i}^\ast,c{\bf i}_k^\ast)&=-2c\Lambda(^\ast{\bf i}_k,{\bar{{\bf i}}}^\ast)+c\Lambda({\bf i}_k^\ast,\bar{\bf i}^\ast)\\&=-\Lambda(c{\bf i}_k^\ast,\bar{\bf i}^\ast)+2c\lr{\bar{\bf i},{\bf i}_k}\\&=-\Lambda(c{\bf i}_k^\ast,\bar{\bf i}^\ast).\end{align*}
Hence, we obtain
\begin{align}\label{youshoubian}
{\rm RHS}=\sum_{[D],[A],[I]}v^{y_0}|_D\Hom(N,\tau \bar{M})_{\tau A\oplus I}|u_{aI_k'[-1]}\star u_{\acute{\bar{A}}\oplus \acute{\bar{P}}_M\oplus cI_k'}\star u_{\acute{\bar{D}}\oplus \acute{\bar{I}}[-1]\oplus dI_k'[-1]},
\end{align}
where
\begin{align*}
y_0=&\Lambda(({\bf m}+{\bf n})^\ast, {\bf a}^\ast)+\lr{{\bf m}-{\bf a},{\bf n}}+\Lambda({\bf i}^\ast,{\bf d}^\ast)-\Lambda(a{\bf p}_k^\ast,(\bar{\bf a}+\bar{\bf p}_M)^\ast)-\Lambda(c{\bf i}_k^\ast,\bar{\bf i}^\ast)-\Lambda(d{\bf p}_k^\ast, \bar{\bf d}^\ast)\\
&+\Lambda'((\acute{\bar{\mathbf{a}}}+\acute{\bar{\mathbf{p}}}_M)^\ast,c{{\bf i}_k'}^\ast)+\Lambda'(\acute{\bar{\mathbf{i} }}^\ast,d{{\bf i}_k'}^\ast)-\Lambda'(\acute{\bar{\mathbf{i}}}^\ast+d{{\bf i}_k'}^\ast,{\acute{\bar{\mathbf{d} }}}^\ast).
\end{align*}

By applying the  triangle equivalence $\mathbf{R}_k^+$, we get
\[|_{\bar{D}\oplus dP_k}\Hom_{\widetilde{\A}}(N,\tau \bar{M})_{\tau \bar{A}\oplus \bar{I}\oplus cI_k}|=|_{\acute{\bar{D}}}\Hom_{\widetilde{\A'}}(\acute{N},\tau \acute{\bar{M}})_{\tau (\acute{\bar{A}}\oplus cI_k')\oplus ({\acute{\bar{I}}}\oplus dI_k')}|.
\]
Replacing the notations $\bar{A}'$, $D'$ and $I'$ in \eqref{zuoshoubian} by ${\acute{\bar{A}}}\oplus cI'_k$, $\acute{\bar{D}}$ and ${\acute{\bar{I}}}\oplus dI'_k$, respectively,
we have
\[z_0=-\Lambda(a\mathbf{p}_k^\ast,(\bar{\bf p}_M+\bar{\bf m}+{\bf n})^\ast)+\Lambda'((\acute{\bar{\mathbf{p} }}_M+\acute{\bar{\mathbf{m} }}+\acute{\mathbf{n} })^\ast,(\acute{\bar{\mathbf{a} }}+\acute{\bar{\mathbf{p}}}_M+c{\bf i}_k')^\ast)+\lr{\acute{\bar{\mathbf{m} }}-\acute{\bar{\mathbf{a} }}-c{\bf i}_k',\acute{\bf n}}.\]
Noting that ${\bf m}=\bar{\bf m}+a{\bf p}_k+\bar{\bf p}_M$ and ${\bf a}=\bar{\bf a}+a{\bf p}_k+\bar{\bf p}_M$, we obtain
\begin{align*}
	z_0=&-\Lambda(a{\bf p}_k^\ast,({\bf m}+{\bf n}-a{\bf p}_k)^\ast)+\Lambda((\bar{\bf p}_M+\bar{\bf m}+{\bf n})^\ast,{\bf a}^\ast-(a{\bf p}_k+c{\bf p}_k)^\ast)+\lr{\bar{\bf m}-\bar{\bf a}+c{\bf p}_k,{\bf n}}\\
	=&\Lambda(({\bf m}+{\bf n}-a{\bf p}_k)^\ast,a{\bf p}_k^\ast)+\Lambda(({\bf m}+{\bf n})^\ast,{\bf a}^\ast)-\Lambda(({\bf m}+{\bf n})^\ast,a{\bf p}_k^\ast)\\
	&-\Lambda(({\bf m}+{\bf n})^\ast,c{\bf p}_k^\ast)-\Lambda(a{\bf p}_k^\ast, {\bf a}^\ast)+\lr{\bar{\bf m}-\bar{\bf a}+c{\bf p}_k,{\bf n}}\\
	=&\Lambda(({\bf m}+{\bf n})^\ast,{\bf a}^\ast)-\Lambda(({\bf m}+{\bf n})^\ast,c{\bf p}_k^\ast)-\Lambda(a{\bf p}_k^\ast, {\bf a}^\ast)+\lr{{\bf m}-{\bf a}+c{\bf p}_k,{\bf n}}.
\end{align*}
Thus, we have
\begin{align*}
	y_0-z_0=&\Lambda({\bf i}^\ast,{\bf d}^\ast){-\Lambda(a{\bf p}_k^\ast,(\bar{\bf a}+\bar{\bf p}_M)^\ast)}-\Lambda(c{\bf i}_k^\ast,\bar{\bf i}^\ast)-{\Lambda(d{\bf p}_k^\ast, \bar{\bf d}^\ast)}-\Lambda({\bar{\mathbf{i} }}^\ast,d{\bf p}_k^\ast)\\
	&{-\Lambda({\bar{\mathbf{i}}}^\ast-d{\bf p}_k^\ast,{\bar{\bf d}^\ast})}-\Lambda((\bar{\mathbf{a}}+\bar{\mathbf{p}}_M)^\ast,c{\mathbf{p}_k^\ast})+\Lambda(({\bf m}+{\bf n})^\ast,c{\bf p}_k^\ast)+{\Lambda(a{\bf p}_k^\ast, {\bf a}^\ast)}-\lr{c{\bf p}_k,{\bf n}}\\
	=&\Lambda({\bf i}^\ast,{\bf d}^\ast)-\Lambda(c{\bf i}_k^\ast,\bar{\bf i}^\ast)-\Lambda({\bar{\mathbf{i} }}^\ast,d{\bf p}_k^\ast)-\Lambda({\bar{\mathbf{i}}}^\ast,{\bar{\bf d}^\ast})-\Lambda((\bar{\mathbf{a}}+\bar{\mathbf{p}}_M)^\ast,c{\mathbf{p}_k^\ast})\\
	&+\Lambda(({\bf m}+{\bf n})^\ast,c{\bf p}_k^\ast)-\lr{c{\bf p}_k,{\bf n}}.
\end{align*}
Since $\mathbf{i}=\bar{\mathbf{i}}+c{\bf i}_k$ and $\mathbf{d}=\bar{\mathbf{d}}+d\mathbf{p}_k$, we obtain
\begin{align*}
	y_0-z_0=\Lambda(c\mathbf{i}_k^\ast,\bar{\mathbf{d}}^\ast)+\Lambda(c\mathbf{i}_k^\ast,d\mathbf{p}_k^\ast)-\Lambda(c\mathbf{i}_k^\ast,\bar{\mathbf{i}}^\ast)+\Lambda(\bar{\mathbf{m}}^\ast-\bar{\mathbf{a}}^\ast,c\mathbf{p}_k^\ast)+\Lambda(\mathbf{n}^\ast,c\mathbf{p}_k^\ast)-\lr{c\mathbf{p}_k,\mathbf{n}}.
\end{align*}
Noting that $\tau \bar{\bf m}-\tau\bar{\bf a}=\mathbf{i}+\mathbf{n}-\mathbf{d}$, we get
\begin{align*}
	\Lambda(\bar{\mathbf{m}}^\ast-\bar{\mathbf{a}}^\ast,c\mathbf{p}_k^\ast)&=\Lambda( ^\ast(\tau\bar{\mathbf{a}}-\tau \bar{\mathbf{m}}),c ^\ast\mathbf{i}_k)\\
	&=\Lambda( (\tau\bar{\mathbf{a}}-\tau \bar{\mathbf{m}})^\ast,c \mathbf{i}_k^\ast)\\
	&=\Lambda((\mathbf{d}-\mathbf{i}-\mathbf{n})^\ast, c\mathbf{i}_k^\ast).
\end{align*}
So, we obtain
\begin{align*}
	y_0-z_0&=\Lambda(\mathbf{n}^\ast,c(\mathbf{p}_k-\mathbf{i}_k)^\ast)-\lr{c\mathbf{p}_k,\mathbf{n}}\\
	&=c\Lambda(B(\widetilde{Q})\mathbf{i}_k,\mathbf{n}^\ast)-\lr{c\mathbf{p}_k,\mathbf{n}}\\
	&=c\lr{\mathbf{n},\mathbf{i}_k}-c\lr{\mathbf{p}_k,\mathbf{n}}\\
	&=0.
\end{align*}
Hence, we conclude ${\rm LHS}={\rm RHS}$, i.e. $\mathfrak{b}_k^+(r_{M,N})=0$ in this case.

\noindent{\bf Relation} $\mathfrak{b}_k^+(r_{M,P})=0$:
In the proof of Lemma \ref{lxdjkh}, by replacing the object $N$ therein by an arbitrary $R\in\mathcal{P}_{\widetilde{\A}}$, we obtain the following observations:

In the Hall algebra $\mathcal {D}\mathcal {H}_\Lambda^{{cl}_1}(\widetilde{\A})$, for any $M\in\widetilde{\A}$ and $P,R\in\mathcal{P}_{\widetilde{\A}}$, if we impose the relations $r_{M,R}=0$ and $r_{X,P}=0$ for any $X\in \widetilde{\A}$, then we have the relation $r_{M,R\oplus P}=0$.

Using this observation, we can reduce the proof of this case to the following two cases:
\begin{itemize}
	\item[(i)] $\mathfrak{b}_k^+(r_{M,\bar{P}})=0$ for any $M\in\widetilde{\A}$ and any $\bar{P}\in \mathcal{P}_{\widetilde{\A}}$ without $P_k$ as direct summands.
	\item[(ii)] $\mathfrak{b}_k^+(r_{M,P_k})=0$ for any $M\in\widetilde{\A}$.
\end{itemize}
By the remark given at the beginning of the proof of $\mathfrak{b}_k^+(r_{M,N})=0$, the case (i) can be proved by using the same proof as above. So, we only need to prove the case (ii) in the following.

Let $M=aP_k\oplus \bar{P}_M\oplus \bar{M}$, where $a\in\mathbb{N}$, $\bar{P}_M\in \mathcal{P}_{\widetilde{\A}}$ has no $P_k$ as direct summands, and $\bar{M}$ has no nonzero projective direct summands. By definition, we have
\begin{align*}
{\rm LHS}:=\mathfrak{b}_k^+(u_{M\oplus P_k})=v^{-\Lambda((a+1)\mathbf{p}_k^\ast,(\bar{\bf p}_M+\bar{\bf m})^\ast)}u_{aI_k'[-1]}\star u_{I_k'[-1]}\star u_{\acute{\bar{P}}_M\oplus \acute{\bar{M}}}.
\end{align*}
Noting that $u_{I_k'[-1]}=u_{P_k'[1]}$ in $\mathcal{DH}_{\Lambda'}^{cl}(\widetilde{\A'})$, we have
\begin{align*}
{\rm LHS}&=v^{-\Lambda((a+1)\mathbf{p}_k^\ast,(\bar{\bf p}_M+\bar{\bf m})^\ast)}u_{aI_k'[-1]}\star u_{P_k'[1]}\star u_{\acute{\bar{P}}_M\oplus \acute{\bar{M}}}\\
&=\sum\limits_{[F'],[P']}v^{z_1}|_{P'}\Hom(P'_k,\acute{\bar{P}}_M\oplus \acute{\bar{M}})_{F'}|u_{aI_k'[-1]}\star u_{F'}\star u_{P'[1]}\\
&=\sum\limits_{[F'],[P']}v^{z_1}|_{P'}\Hom(P'_k,\acute{\bar{P}}_M\oplus \acute{\bar{M}})_{F'}|u_{aI_k'[-1]}\star u_{F'}\star u_{I'[-1]},
\end{align*}
where $z_1=-\Lambda((a+1)\mathbf{p}_k^\ast,(\bar{\bf p}_M+\bar{\bf m})^\ast)-\Lambda'({{\bf p}_k'}^\ast,({\acute{\bar{\bf p}}}_M+{\acute{\bar{\bf m}}})^\ast)-2\lr{{{\bf p}_k'},{\acute{\bar{\bf p}}}_M+{\acute{\bar{\bf m}}}}+\Lambda'({{\bf f}'}^\ast,{{\bf p}'}^\ast)$ and $I'=\nu(P')$.

Let
\begin{align*}
&{\rm RHS}:=\sum_{[D],[A],[I]}v^{\Lambda(({\bf m}+{\bf p}_k)^\ast, {\bf a}^\ast)+\lr{{\bf m}-{\bf a},{\bf p}_k}}|_D\Hom(P_k,\tau \bar{M})_{\tau A\oplus I}|\mathfrak{b}_k^+(u_A)\star \mathfrak{b}_k^+(u_{D\oplus I[-1]})\\
&=\sum_{[D],[A],[I]}v^{\Lambda(({\bf m}+{\bf p}_k)^\ast, {\bf a}^\ast)+\lr{{\bf m}-{\bf a},{\bf p}_k}+\Lambda({\bf i}^\ast,{\bf d}^\ast)}|_D\Hom(P_k,\tau \bar{M})_{\tau A\oplus I}|\mathfrak{b}_k^+(u_A)\star \mathfrak{b}_k^+(u_{I[-1]})\star \mathfrak{b}_k^+(u_{D}),
\end{align*}
where each $A$ has the same maximal projective direct summand as $M$ and $I\in\I_{\widetilde{\A}}$.

Write $A=\bar{A}\oplus \bar{P}_M\oplus aP_k$, $D=dP_k$ and $I=\bar{I}\oplus cI_k$, where $\bar{A}$ has no nonzero projective direct summands, $d=0,1$, and $\bar{I}$ has no $I_k$ as direct summands. By \eqref{youshoubian}, similarly, we have
\begin{align}\label{youshoubian2}
{\rm RHS}=\sum_{[D],[A],[I]}v^{y_1}|_D\Hom(P_k,\tau \bar{M})_{\tau A\oplus I}|u_{aI_k'[-1]}\star u_{\acute{\bar{A}}\oplus \acute{\bar{P}}_M\oplus cI_k'}\star u_{\acute{\bar{I}}[-1]\oplus dI_k'[-1]},
\end{align}
where
\begin{align*}
y_1=&\Lambda(({\bf m}+{\bf p}_k)^\ast, {\bf a}^\ast)+\lr{{\bf m}-{\bf a},{\bf p}_k}+\Lambda({\bf i}^\ast,{\bf d}^\ast)-\Lambda(a{\bf p}_k^\ast,(\bar{\bf a}+\bar{\bf p}_M)^\ast)-\Lambda(c{\bf i}_k^\ast,\bar{\bf i}^\ast)\\
&+\Lambda'((\acute{\bar{\mathbf{a}}}+\acute{\bar{\mathbf{p}}}_M)^\ast,c{{\bf i}_k'}^\ast)+\Lambda'(\acute{\bar{\mathbf{i} }}^\ast,d{{\bf i}_k'}^\ast).
\end{align*}

Replace the notations $F'$ and $I'$ in LHS by $\acute{\bar{A}}\oplus \acute{\bar{P}}_M\oplus cI_k'$ and $\acute{\bar{I}}\oplus dI_k'$, respectively.
Then the notation $P'$ is replaced by $\acute{\bar{P}}\oplus dP'_k$, where $\bar{P}=\nu^{-1}(\bar{I})$. Thus,
\begin{align*}|_{P'}\Hom(P'_k,\acute{\bar{P}}_M\oplus \acute{\bar{M}})_{F'}|&=|\Hom_{\mathcal{D}^b(\widetilde{\A'})}(P'_k,\acute{\bar{P}}_M\oplus \acute{\bar{M}})_{P'[1]\oplus F'}|\\
&=|\Hom_{\mathcal{D}^b(\widetilde{\A'})}(P'_k,\acute{\bar{M}}\oplus\acute{\bar{P}}_M)_{\acute{\bar{P}}[1]\oplus dP'_k[1]\oplus\acute{\bar{A}}\oplus cI_k'\oplus \acute{\bar{P}}_M}|.\end{align*}
Since $\Hom(P'_k,\acute{\bar{P}}_M)=0$, by \cite[Lemma 2.5]{PX}, we get
\begin{align*}|\Hom_{\mathcal{D}^b(\widetilde{\A'})}(P'_k,\acute{\bar{M}}\oplus\acute{\bar{P}}_M)_{\acute{\bar{P}}[1]\oplus dP'_k[1]\oplus\acute{\bar{A}}\oplus cI_k'\oplus \acute{\bar{P}}_M}|=|\Hom_{\mathcal{D}^b(\widetilde{\A'})}(P'_k,\acute{\bar{M}})_{\acute{\bar{P}}[1]\oplus dP'_k[1]\oplus\acute{\bar{A}}\oplus cI_k'}|.\end{align*}
By applying the Auslander--Reiten translation of $\mathcal{D}^b(\widetilde{{\A'}})$ and the triangle equivalence $\mathbf{R}_k^+$, we get
\begin{align*}|\Hom_{\mathcal{D}^b(\widetilde{\A'})}(P'_k,\acute{\bar{M}})_{\acute{\bar{P}}[1]\oplus dP'_k[1]\oplus\acute{\bar{A}}\oplus cI_k'}|&=|\Hom_{\mathcal{D}^b(\widetilde{\A})}(P_k,\tau\bar{M})_{\bar{I}\oplus dP_k[1]\oplus\tau{\bar{A}}\oplus cI_k}|\\
&=|_D\Hom(P_k,\tau \bar{M})_{\tau A\oplus I}|.\end{align*}
Hence, we obtain \begin{align*}|_{P'}\Hom(P'_k,\acute{\bar{P}}_M\oplus \acute{\bar{M}})_{F'}|=|_D\Hom(P_k,\tau \bar{M})_{\tau A\oplus I}|.\end{align*}
Moreover, \begin{align*}\Lambda'({{\bf f}'}^\ast,{{\bf p}'}^\ast)&=\Lambda'(({\acute{\bar{{\bf a}}}}+{\acute{\bar{{\bf p}}}}_M+c{{\bf i}_k'})^\ast,{^\ast}({\acute{\bar{{\bf i}}}}+d{{\bf i}_k'}))\\
&=\Lambda'(({\acute{\bar{{\bf a}}}}+{\acute{\bar{{\bf p}}}}_M+c{{\bf i}_k'})^\ast,({\acute{\bar{{\bf i}}}}+d{{\bf i}_k'})^\ast)+\lr{{\acute{\bar{{\bf a}}}}+{\acute{\bar{{\bf p}}}}_M+c{{\bf i}_k'},{\acute{\bar{{\bf i}}}}+d{{\bf i}_k'}}\\
&=\Lambda(({\bar{{\bf a}}}+{\bar{{\bf p}}}_M-c{{\bf p}_k})^\ast,({\bar{{\bf i}}}-d{{\bf p}_k})^\ast)+\lr{{\bar{{\bf a}}}+{\bar{{\bf p}}}_M-c{{\bf p}_k},{\bar{{\bf i}}}-d{{\bf p}_k}}.\end{align*}
Thus, we have
\begin{align*}z_1=&-\Lambda((a+1)\mathbf{p}_k^\ast,(\bar{\bf p}_M+\bar{\bf m})^\ast)-\Lambda'({{\bf i}_k'}^\ast,({\acute{\bar{\bf p}}}_M+{\acute{\bar{\bf m}}})^\ast)+\lr{{\acute{\bar{\bf p}}}_M+{\acute{\bar{\bf m}}},{\bf i}_k'}-2\lr{{\acute{\bar{\bf p}}}_M+{\acute{\bar{\bf m}}},{\bf i}_k'}\\&
+\Lambda(({\bar{{\bf a}}}+{\bar{{\bf p}}}_M-c{{\bf p}_k})^\ast,({\bar{{\bf i}}}-d{{\bf p}_k})^\ast)+\lr{{\bar{{\bf a}}}+{\bar{{\bf p}}}_M-c{{\bf p}_k},{\bar{{\bf i}}}-d{{\bf p}_k}}\\
=&-\Lambda((a+1)\mathbf{p}_k^\ast,(\bar{\bf p}_M+\bar{\bf m})^\ast)+\Lambda({\bf p}_k^\ast,({\bar{\bf p}}_M+{\bar{\bf m}})^\ast)+\lr{{\bar{\bf p}}_M+{\bar{\bf m}},{\bf p}_k}\\&
+\Lambda(({\bar{{\bf a}}}+{\bar{{\bf p}}}_M-c{{\bf p}_k})^\ast,({\bar{{\bf i}}}-d{{\bf p}_k})^\ast)+\lr{{\bar{{\bf a}}}+{\bar{{\bf p}}}_M-c{{\bf p}_k},{\bar{{\bf i}}}-d{{\bf p}_k}}\\
=&-\Lambda(a\mathbf{p}_k^\ast,(\bar{\bf p}_M+\bar{\bf m})^\ast)+\lr{{\bar{\bf p}}_M+{\bar{\bf m}},{\bf p}_k}
+\Lambda(({\bar{{\bf a}}}+{\bar{{\bf p}}}_M-c{{\bf p}_k})^\ast,({\bar{{\bf i}}}-d{{\bf p}_k})^\ast)\\&+\lr{{\bar{{\bf a}}}+{\bar{{\bf p}}}_M-c{{\bf p}_k},{\bar{{\bf i}}}-d{{\bf p}_k}}.\end{align*}
Since $\bar{{\bf i}}-d{\bf p}_k=\tau(\bar{{\bf m}}-\bar{{\bf a}})-{\bf p}_k-c{\bf i}_k$, we get
\begin{align*}&\Lambda(({\bar{{\bf a}}}+{\bar{{\bf p}}}_M-c{{\bf p}_k})^\ast,({\bar{{\bf i}}}-d{{\bf p}_k})^\ast)=\Lambda(({\bar{{\bf a}}}+{\bar{{\bf p}}}_M-c{{\bf p}_k})^\ast,(\tau(\bar{{\bf m}}-\bar{{\bf a}}))^\ast-{\bf p}_k^\ast-c{\bf i}_k^\ast)\\
&=-\Lambda({^\ast}({\bar{{\bf a}}}+{\bar{{\bf p}}}_M-c{{\bf p}_k}),({\bar{{\bf m}}}-{\bar{{\bf a}}})^\ast)-\Lambda((\bar{{\bf a}}+\bar{{\bf p}}_M)^\ast,{\bf p}_k^\ast)-\Lambda((\bar{{\bf a}}+\bar{{\bf p}}_M)^\ast,c{\bf i}_k^\ast)+\Lambda(c{\bf p}_k^\ast,c{\bf i}_k^\ast)\\
&=-\Lambda(({\bar{{\bf a}}}+{\bar{{\bf p}}}_M-c{{\bf p}_k})^\ast,({\bar{{\bf m}}}-{\bar{{\bf a}}})^\ast)+\lr{\bar{{\bf m}}-\bar{{\bf a}},\bar{{\bf a}}+\bar{{\bf p}}_M-c{\bf p}_k}\\&\quad-\Lambda((\bar{{\bf a}}+\bar{{\bf p}}_M)^\ast,{\bf p}_k^\ast)-\Lambda((\bar{{\bf a}}+\bar{{\bf p}}_M)^\ast,c{\bf i}_k^\ast)+\Lambda(c{\bf p}_k^\ast,c{\bf i}_k^\ast)\end{align*}
and
\begin{align*}
\lr{{\bar{{\bf a}}}+{\bar{{\bf p}}}_M-c{{\bf p}_k},{\bar{{\bf i}}}-d{{\bf p}_k}}=-\lr{\bar{{\bf m}}-\bar{{\bf a}},{\bar{{\bf a}}}+{\bar{{\bf p}}}_M-c{{\bf p}_k}}-\lr{{\bar{{\bf a}}}+{\bar{{\bf p}}}_M-c{{\bf p}_k},{\bf p}_k+c{\bf i}_k}.
\end{align*}

While
\begin{align*}
y_1=&\Lambda((\bar{{\bf m}}+\bar{{\bf p}}_M+a{\bf p}_k+{\bf p}_k)^\ast,(\bar{{\bf a}}+\bar{{\bf p}}_M+a{\bf p}_k)^\ast)+\lr{\bar{{\bf m}}-\bar{{\bf a}},{\bf p}_k}+\Lambda((\bar{{\bf i}}+c{\bf i}_k)^\ast,d{\bf p}_k^\ast)\\
&-\Lambda(a{\bf p}_k^\ast,(\bar{{\bf a}}+\bar{{\bf p}}_M)^\ast)-\Lambda(c{\bf i}_k^\ast,{\bar{{\bf i}}}^\ast)-\Lambda((\bar{{\bf a}}+\bar{{\bf p}}_M)^\ast,c{\bf p}_k^\ast)-\Lambda({\bar{{\bf i}}}^\ast,d{\bf p}_k^\ast).
\end{align*}
Since $d{\bf p}_k-\bar{{\bf i}}={\bf p}_k-\tau(\bar{{\bf m}}-\bar{{\bf a}})+c{\bf i}_k$, we get
\begin{align*}
\Lambda((\bar{{\bf i}}+c{\bf i}_k)^\ast,d{\bf p}_k^\ast)-\Lambda(c{\bf i}_k^\ast,{\bar{{\bf i}}}^\ast)-\Lambda({\bar{{\bf i}}}^\ast,d{\bf p}_k^\ast)
&=\Lambda(c{\bf i}_k^\ast,{\bf p}_k^\ast)-\Lambda(c{\bf i}_k^\ast,(\tau(\bar{{\bf m}}-\bar{{\bf a}}))^\ast)\\
&=\Lambda(c{\bf i}_k^\ast,{\bf p}_k^\ast)+\Lambda(c{\bf p}_k^\ast,(\bar{{\bf m}}-\bar{{\bf a}})^\ast).
\end{align*}
Thus, we have
\begin{align*}
&\Lambda((\bar{{\bf i}}+c{\bf i}_k)^\ast,d{\bf p}_k^\ast)-\Lambda(c{\bf i}_k^\ast,{\bar{{\bf i}}}^\ast)-\Lambda({\bar{{\bf i}}}^\ast,d{\bf p}_k^\ast)-\Lambda((\bar{{\bf a}}+\bar{{\bf p}}_M)^\ast,c{\bf p}_k^\ast)\\
&=\Lambda(c{\bf i}_k^\ast,{\bf p}_k^\ast)+\Lambda(c{\bf p}_k^\ast,(\bar{{\bf m}}-\bar{{\bf a}})^\ast)+\Lambda(c{\bf p}_k^\ast,(\bar{{\bf a}}+\bar{{\bf p}}_M)^\ast)\\
&=\Lambda(c{\bf i}_k^\ast,{\bf p}_k^\ast)+\Lambda(c{\bf p}_k^\ast,(\bar{{\bf m}}+\bar{{\bf p}}_M)^\ast).
\end{align*}
So, we get
\begin{align*}
y_1=&\Lambda((\bar{{\bf m}}+\bar{{\bf p}}_M)^\ast,(\bar{{\bf a}}+\bar{{\bf p}}_M)^\ast)+\Lambda((\bar{{\bf m}}+\bar{{\bf p}}_M)^\ast,a{\bf p}_k^\ast)
+\Lambda({\bf p}_k^\ast,(\bar{{\bf a}}+\bar{{\bf p}}_M)^\ast)\\&+\lr{\bar{{\bf m}}-\bar{{\bf a}},{\bf p}_k}+\Lambda(c{\bf i}_k^\ast,{\bf p}_k^\ast)+\Lambda(c{\bf p}_k^\ast,(\bar{{\bf m}}+\bar{{\bf p}}_M)^\ast)\\
=&-\Lambda((\bar{{\bf a}}+\bar{{\bf p}}_M-c{\bf p}_k)^\ast,(\bar{{\bf m}}+\bar{{\bf p}}_M)^\ast)+\Lambda((\bar{{\bf m}}+\bar{{\bf p}}_M)^\ast,a{\bf p}_k^\ast)
\\&+\Lambda({\bf p}_k^\ast,(\bar{{\bf a}}+\bar{{\bf p}}_M)^\ast)+\lr{\bar{{\bf m}}-\bar{{\bf a}},{\bf p}_k}+\Lambda(c{\bf i}_k^\ast,{\bf p}_k^\ast).
\end{align*}
Thus, we obtain
\begin{align*}
y_1-z_1=&\Lambda(c{\bf p}_k^\ast,(\bar{{\bf a}}+\bar{{\bf p}}_M)^\ast)+\Lambda((\bar{{\bf a}}+\bar{{\bf p}}_M)^\ast,{\bf p}_k^\ast)+\Lambda((\bar{{\bf a}}+\bar{{\bf p}}_M)^\ast,c{\bf i}_k^\ast)-\Lambda(c{\bf p}_k^\ast,c{\bf i}_k^\ast)\\&+\Lambda({\bf p}_k^\ast,(\bar{{\bf a}}+\bar{{\bf p}}_M)^\ast)+\Lambda(c{\bf i}_k^\ast,{\bf p}_k^\ast)+\lr{\bar{{\bf m}}-\bar{{\bf a}},{\bf p}_k}-\lr{\bar{{\bf m}}-\bar{{\bf a}},\bar{{\bf a}}+\bar{{\bf p}}_M-c{\bf p}_k}\\&-\lr{{\bar{\bf p}}_M+{\bar{\bf m}},{\bf p}_k}+\lr{\bar{{\bf m}}-\bar{{\bf a}},{\bar{{\bf a}}}+{\bar{{\bf p}}}_M-c{{\bf p}_k}}+\lr{{\bar{{\bf a}}}+{\bar{{\bf p}}}_M-c{{\bf p}_k},{\bf p}_k+c{\bf i}_k}\\
=&c\Lambda(({\bf p}_k-{\bf i}_k)^\ast,(\bar{{\bf a}}+\bar{{\bf p}}_M)^\ast)-\Lambda(c{\bf p}_k^\ast,c{\bf i}_k^\ast)+\Lambda(c{\bf i}_k^\ast,{\bf p}_k^\ast)\\&+\lr{\bar{{\bf m}}-\bar{{\bf a}},{\bf p}_k}-\lr{{\bar{\bf p}}_M+{\bar{\bf m}},{\bf p}_k}+\lr{{\bar{{\bf a}}}+{\bar{{\bf p}}}_M-c{{\bf p}_k},{\bf p}_k+c{\bf i}_k}\\
=&-c\lr{\bar{{\bf a}}+\bar{{\bf p}}_M,{\bf i}_k}+c^2\lr{{\bf i}_k,{\bf i}_k}+c\lr{{\bf i}_k,{\bf i}_k}+\lr{\bar{{\bf m}}-\bar{{\bf a}},{\bf p}_k}\\&-\lr{{\bar{\bf p}}_M+{\bar{\bf m}},{\bf p}_k}+\lr{{\bar{{\bf a}}}+{\bar{{\bf p}}}_M-c{{\bf p}_k},{\bf p}_k+c{\bf i}_k}\\
=&0.
\end{align*}
Hence, we conclude ${\rm LHS}={\rm RHS}$, i.e. $\mathfrak{b}_k^+(r_{M,P_k})=0$.

Therefore, we complete the proof.
\end{proof}
Now, we are in a position to give the main theorem of this section as follows.
\begin{theorem}\label{mutationbubian}
The map $b_k^+:\mathcal{DH}_{\Lambda}^{cl}(\widetilde{\A})\to \mathcal{DH}_{\Lambda'}^{cl}(\widetilde{\A'})$ is an algebra isomorphism.
\end{theorem}
\begin{proof}
Recall that $k$ is a sink vertex and set $i_1=k$. Let $i_1,i_2,\ldots, i_m$ be an admissible sequence of sink vertices in $\widetilde{Q}$, i.e. for each $1<s\leq m$, $i_s$ is a sink vertex in $\mu_{i_{s-1}}\cdots \mu_{i_1}(\widetilde{Q})$. Then $\mu_{i_{m}}\cdots \mu_{i_1}(\widetilde{Q})=\widetilde{Q}$ and $\mu_{i_{m}}\cdots \mu_{i_1}(\Lambda)=\Lambda$.
Set $c^+=\mathbf{r}_{i_m}^+\cdots\mathbf{r}_{i_1}^+$ and $\mathbf{C}^+=\mathbf{R}_{i_m}^+\cdots\mathbf{R}_{i_1}^+$.
Then $\mathbf{C}^+:\mathcal{D}^b(\widetilde{\A})\rightarrow\mathcal{D}^b(\widetilde{\A})$ is a triangle auto-equivalence. Moreover, we have $c^+M=\tau M$ for any $M\in\widetilde{\A}$ without nonzero projective direct summands and $\mathbf{C}^+X=\bar{\tau}X$ for any $X\in\mathcal{D}^b(\widetilde{\A})$.

Let ${b}^+:={b}_{i_m}^+\cdots {b}_{i_1}^+:\mathcal{DH}_\Lambda^{cl}(\widetilde{\A})\to \mathcal{DH}_\Lambda^{cl}(\widetilde{\A})$. Then $b^+$ is just the algebra automorphism $\sigma$ of $\mathcal{DH}_\Lambda^{cl}(\widetilde{\A})$, which is given in Theorem \ref{zmainthm}. Hence, we conclude that the algebra homomorphism $b_k^+=b_{i_1}^+:\mathcal{DH}_{\Lambda}^{cl}(\widetilde{\A})\to \mathcal{DH}_{\Lambda'}^{cl}(\widetilde{\A'})$ is injective.

By Proposition \ref{elsubalgebra}, it is easy to see the Hall algebra $\mathcal{DH}_{\Lambda'}^{cl}(\widetilde{\A'})$ is generated by the elements $u_{\bar{M}'}, u_{I'_k}, u_{I_k'[-1]}$ and $ u_{I_j'[-1]}$ for all $\bar{M}'\in\widetilde{\A'}\langle k\rangle$ and $j\neq k$, whose preimages under $b_k^+$ are $u_{\mathbf{r}_k^-(\bar{M}')}, u_{P_k[1]}, u_{P_k}$ and $ u_{I_j[-1]}$, respectively. It follows that $b_k^+$ is surjective.

Therefore, we complete the proof.
\end{proof}

\begin{corollary}
Let $1\leq k\leq n$ be a sink vertex of $\widetilde{Q}$. Then
we have the following commutative diagram of algebra homomorphisms in which the vertical maps are isomorphisms:
\begin{equation*}
\xymatrix{\A_q(\Lambda,\widetilde{B})\,\,\ar@{>->}[r]^-{\varphi'_0}\ar[d]^-{\mu_k}&\mathcal {D}\mathcal {H}_\Lambda^{{\tilde{c}l}}({\A})\ar[d]^-{b_k^+}\\
\A_q(\Lambda',\widetilde{B}')\,\,\ar@{>->}[r]^-{\varphi'_0}&\mathcal{DH}_{\Lambda'}^{\tilde{c}l}({\A'}).}
\end{equation*}
Here we recall that ${\Lambda'}=\mu_k(\Lambda)$, $\widetilde{B}'=\mu_k(\widetilde{B})$ and $\A'$ is the representation subcategory associated with $\mu_k(\widetilde{Q})$.
\end{corollary}
\begin{proof}
Since for any $M\in{\mathcal{A}}\langle k\rangle\subseteq\widetilde{\A}$, ${\Dim}\,\mathbf{r}_k^+(M)=s_k({\bf m})={\bf m}-d_k^{-1}({\bf m},e_k)e_k$ has supports only on $\{1,\ldots,n\}$. Thus, $\mathbf{r}_k^+(M)\in\mathcal{A}'$.
Hence, the isomorphism $b_k^+:\mathcal{DH}_{\Lambda}^{cl}(\widetilde{\A})\to \mathcal{DH}_{\Lambda'}^{cl}(\widetilde{\A'})$ can be restricted to an injective homomorphism of algebras $b_k^+:\mathcal {D}\mathcal {H}_\Lambda^{{\tilde{c}l}}({\A}) \to \mathcal{DH}_{\Lambda'}^{\tilde{c}l}({\A'})$.
It is easy to see that the restricted map $b_k^+$ is also surjective, i.e. it is also an isomorphism.
Clearly, the map $b_k^+$ can also be restricted to an algebra isomorphism $b_k^+:\mathcal {C}\mathcal {H}_\Lambda^{{\tilde{c}l}}({\A})\to \mathcal {C}\mathcal {H}_{\Lambda'}^{{\tilde{c}l}}({\A'})$. By Corollary \ref{congshixian}, $\varphi'_0:\A_q(\Lambda,\widetilde{B})\to\mathcal {C}\mathcal {H}_\Lambda^{{\tilde{c}l}}({\A})$ is an isomorphism. Thus, the algebra isomorphism $\mu_k$ is defined by the following commutative diagram
\begin{equation*}
\xymatrix{\A_q(\Lambda,\widetilde{B})\ar[r]^-{\varphi'_0}_{\cong}\ar[d]^-{\mu_k}&\mathcal {C}\mathcal {H}_\Lambda^{{\tilde{c}l}}({\A})\ar[d]^-{b_k^+}\\
\A_q(\Lambda',\widetilde{B}')\ar[r]^-{\varphi'_0}_{\cong}&\mathcal {C}\mathcal {H}_{\Lambda'}^{{\tilde{c}l}}({\A'}).}
\end{equation*}
\end{proof}

In order to distinguish,
we write the quantum cluster character \eqref{qcltz} associated to the quiver $\widetilde{Q}$ as $X^{\widetilde{Q}}_{I[-1]\oplus M\oplus P[1]}$, and for the quiver $\widetilde{Q}':=\mu_k(\widetilde{Q})$, the corresponding quantum cluster character in the quantum torus $\mathcal{T}_{\Lambda'}$ is denoted by
\begin{equation*}\mathbb{X}^{\widetilde{Q}'}_{I[-1]\oplus M\oplus P[1]}=\sum\limits_{\mathbf{e}}v^{\lr{{\bf p}-\mathbf{e},\mathbf{m}-\mathbf{e}-\mathbf{i}}}|\mathrm{Gr}_{\mathbf{e}}M|
\mathbb{X}^{({\bf p}-\mathbf{e})^\ast-^\ast(\mathbf{m}-\mathbf{e}-\mathbf{i})},\end{equation*}
where ${\Lambda'}:=\mu_k(\Lambda)$, $M\in\widetilde{\A'}$, $I\in\I_{\widetilde{\A'}}$ and $P\in\P_{\widetilde{\A'}}$.
\begin{corollary}
Let $1\leq k\leq n$ be a sink vertex of $\widetilde{Q}$. The algebra isomorphism $$\mu_k:\A_q(\Lambda,\widetilde{B})\longrightarrow\A_q(\Lambda',\widetilde{B}')$$
sends $X_M^{\widetilde{Q}}$ and $X_{P[1]}^{\widetilde{Q}}$ to $\mathbb{X}_{\mathbf{R}_k^+(M)}^{\widetilde{Q}'}$ and $\mathbb{X}_{\mathbf{R}_k^+(P)[1]}^{\widetilde{Q}'}$, respectively, where $M\in \A$ is indecomposable and rigid, and $P\in \mathcal{P}_{\widetilde{\A}}$ is indecomposable. In particular, $\mu_k(X_{P_i[1]}^{\widetilde{Q}})=\mathbb{X}_{I'_i[-1]}^{\widetilde{Q}'}$ for any $1\leq i\neq k\leq m$ and $\mu_k(X_{P_k[1]}^{\widetilde{Q}})=\mathbb{X}_{I'_k}^{\widetilde{Q}'}$.
\end{corollary}

\appendix
\section{Proof of the algebra homomorphism $\sigma'$}
In this section, we prove the map $\sigma'$ given in Theorem \ref{zmainthm} is an algebra homomorphism.
\begin{proposition}\label{dmint27}
In the generating relations of $\mathcal {D}\mathcal {H}_\Lambda^{ec}(\widetilde{\A})$, the relation \eqref{lgx7} can be replaced by the following two relations
\begin{flalign}\label{dfgx1}
u_{P[1]}\star u_{I}=q^{-\frac{1}{2}\Lambda({\bf p}^\ast,{\bf i}^\ast)-\lr{{\bf p},{\bf i}}}
\sum\limits_{[Q],[I']}v^{\Lambda({{\bf i}'}^\ast,{\bf q}^\ast)}|{}_{Q}\Hom_{\widetilde{\A}}(P,I)_{I'}|u_{I'}\star u_{Q[1]},
\end{flalign}
where $P\in\P_{\widetilde{\A}}$~and $I\in\I_{\widetilde{\A}}$;
\begin{flalign}\label{dfgx2}
u_{P[1]}\star u_{M}=q^{-\frac{1}{2}\Lambda({\bf p}^\ast,{\bf m}^\ast)-\lr{{\bf p},{\bf m}}}
\sum\limits_{[F],[P']}v^{\Lambda({{\bf f}}^\ast,{{\bf p}'}^\ast)}|{}_{P'}\Hom_{\widetilde{\A}}(P,M)_{F}|u_{F}\star u_{P'[1]},
\end{flalign}
where $M\in\A$~has~nonzero~injective~direct~summands and $P\in\P_{\widetilde{\A}}$.
\end{proposition}
\begin{proof}
We only need to show that the relation \eqref{lgx7} can be implied from the relations \eqref{dfgx1}, \eqref{dfgx2} together with the other relations in Proposition \ref{elsubalgebra}. For any $M\in\widetilde{\A}$, write $M=I'\oplus M'$ such that $I'$ is the maximal injective direct summand of $M$. Then $u_M=v^{-\Lambda({{\bf m}'}^\ast,{{\bf i}'}^\ast)}u_{M'}\star u_{I'}$. Thus, we have
\begin{flalign*}
u_{P[1]}\star u_{M}=v^{-\Lambda({{\bf m}'}^\ast,{{\bf i}'}^\ast)}u_{P[1]}\star u_{M'}\star u_{I'}.
\end{flalign*}
By \eqref{dfgx2}, we have
\begin{flalign*}
u_{P[1]}\star u_{M'}=q^{-\frac{1}{2}\Lambda({\bf p}^\ast,{{\bf m}'}^\ast)-\lr{{\bf p},{\bf m}'}}
\sum\limits_{[F_1],[Q]}v^{\Lambda({\bf f}_1^\ast,{{\bf q}}^\ast)}|{}_{Q}\Hom_{\widetilde{\A}}(P,M')_{F_1}|u_{F_1}\star u_{Q[1]}.
\end{flalign*}
Thus, we obtain
\begin{flalign*}
&u_{P[1]}\star u_{M}=\\&v^{-\Lambda({\bf p}^\ast,{{\bf m}'}^\ast)-2\lr{{\bf p},{\bf m}'}-\Lambda({{\bf m}'}^\ast,{{\bf i}'}^\ast)}
\sum\limits_{[F_1],[Q]}v^{\Lambda({\bf f}_1^\ast,{{\bf q}}^\ast)}|{}_{Q}\Hom_{\widetilde{\A}}(P,M')_{F_1}|u_{F_1}\star u_{Q[1]}\star u_{I'}.
\end{flalign*}
Using \eqref{dfgx1}, we get
\begin{flalign*}
u_{P[1]}\star u_{M}&=v^{-\Lambda({\bf p}^\ast,{{\bf m}'}^\ast)-2\lr{{\bf p},{\bf m}'}-\Lambda({{\bf m}'}^\ast,{{\bf i}'}^\ast)}\sum\limits_{[F_1],[Q],[P'],[F_2]}v^{\Lambda({\bf f}_1^\ast,{{\bf q}}^\ast)-\Lambda({{\bf q}}^\ast,{{\bf i}'}^\ast)-2\lr{{\bf q},{\bf i}'}}\\&v^{\Lambda({\bf f}_2^\ast,{{\bf p}'}^\ast)}|{}_{Q}\Hom_{\widetilde{\A}}(P,M')_{F_1}|\cdot |{}_{P'}\Hom_{\widetilde{\A}}(Q,I')_{F_2}|u_{F_1}\star u_{F_2}\star u_{P'[1]}.
\end{flalign*}
Noting that $F_2$ is injective, we obtain
\begin{equation*}
u_{P[1]}\star u_{M}=\sum\limits_{[F_1],[Q],[P'],[F_2]}v^{x_0}|{}_{Q}\Hom_{\widetilde{\A}}(P,M')_{F_1}|\cdot |{}_{P'}\Hom_{\widetilde{\A}}(Q,I')_{F_2}|u_{F_1\oplus F_2}\star u_{P'[1]},
\end{equation*}
where \begin{flalign*}x_0=&-\Lambda({\bf p}^\ast,{{\bf m}'}^\ast)-2\lr{{\bf p},{\bf m}'}-\Lambda({{\bf m}'}^\ast,{{\bf i}'}^\ast)+\Lambda({\bf f}_1^\ast,{{\bf q}}^\ast)\\&-\Lambda({{\bf q}}^\ast,{{\bf i}'}^\ast)-2\lr{{\bf q},{\bf i}'}+\Lambda({\bf f}_2^\ast,{{\bf p}'}^\ast)+\Lambda({\bf f}_1^\ast,{\bf f}_2^\ast).\end{flalign*}
Since ${\bf m}'-{\bf p}={\bf f}_1-{\bf q}$ and ${\bf p}'+{\bf i}'={\bf q}+{\bf f}_2$, we get
\begin{flalign*}x_0&=-\Lambda({{\bf p}}^\ast,{\bf m}^\ast)-2\lr{{\bf p},{\bf m}'}-2\lr{{\bf q},{\bf i}'}+\Lambda({\bf f}_1^\ast+{\bf f}_2^\ast,{{\bf p}'}^\ast)\\&=-\Lambda({{\bf p}}^\ast,{\bf m}^\ast)-2\lr{{\bf p},{\bf m}}+\Lambda({\bf f}_1^\ast+{\bf f}_2^\ast,{{\bf p}'}^\ast)+2\lr{{\bf p}-{\bf q},{\bf i}'}.\end{flalign*}
Hence, by Proposition \ref{prop:Hall-equality-P-I}, we obtain the relation \eqref{lgx7}.
\end{proof}

\begin{proposition}
The map $\varrho':\mathcal {D}\mathcal {H}_\Lambda^{ec}(\widetilde{\A})\longrightarrow\mathcal {D}\mathcal {H}_\Lambda^{{cl}}(\widetilde{\A})$ defined by
$$\varrho'(u_{I[-1]\oplus M\oplus P[1]})=v^{\Lambda({\bf i}^\ast,({\bf m}-{\bf p})^\ast)+\Lambda({\bf m}^\ast,{\bf p}^\ast)}u_{\nu^{-1}I}\star u_{\tau^{-1} M'\oplus \nu^{-1}(I')[1]}\star u_{P}$$ for any $M=M'\oplus I'\in\widetilde{\A}$, $I\in\I_{\widetilde{\A}}$ and $P\in\P_{\widetilde{\A}}$, where $I'$ is the maximal injective direct summand of $M$, is a homomorphism of algebras.
\end{proposition}
\begin{proof}
It suffices to prove the relations \eqref{lgx2}-\eqref{lgx8} are preserved under the map $\varrho'$.

Noting that $\varrho'(u_M)=\overline{\tau}^{-1}(u_M)$ and $\varrho'(u_{I[-1]})=\overline{\tau}^{-1}(u_{I[-1]})$ for any $M\in\widetilde{\A}$ and $I\in\I_{\widetilde{\A}}$, and $\overline{\tau}^{-1}$ is an algebra automorphism of $\mathcal {D}\mathcal {H}_\Lambda(\widetilde{\A})$,
we conclude that the relations \eqref{lgx3}-\eqref{lgx6} are preserved under the map $\varrho'$. Since $u_P\star u_Q=v^{\Lambda({\bf p}^\ast,{\bf q}^\ast)}u_{P\oplus Q}$ for any $P,Q\in\P_{\widetilde{\A}}$, we get that the relation \eqref{lgx2} is preserved under the map $\varrho'$.

\noindent{\bf Relation \eqref{lgx8}:} For any $P\in\P_{\widetilde{\A}}$ and $I\in\I_{\widetilde{\A}}$, set $Q=\nu^{-1}(I)$, then
\begin{flalign*}
\varrho'(u_{I[-1]})\star \varrho'(u_{P[1]})=u_Q\star u_P=v^{\Lambda({\bf q}^\ast,{\bf p}^\ast)}u_{P\oplus Q}
\end{flalign*}
and
\begin{flalign*}
\varrho'(u_{P[1]})\star \varrho'(u_{I[-1]})=u_P\star u_Q=v^{\Lambda({\bf p}^\ast,{\bf q}^\ast)}u_{P\oplus Q}.
\end{flalign*}
Thus, we obtain
$$\varrho'(u_{I[-1]})\star \varrho'(u_{P[1]})=q^{\Lambda({\bf q}^\ast,{\bf p}^\ast)}\varrho'(u_{P[1]})\star \varrho'(u_{I[-1]}).$$
Since \begin{flalign*}\Lambda({\bf q}^\ast,{\bf p}^\ast)=\Lambda({^*{\bf i}},{{\bf p}^*})
&=\Lambda({\bf i}^*,{\bf p}^\ast)-\Lambda(B(\widetilde{Q})(\bf i),{\bf p}^\ast)\\&=\Lambda({\bf i}^*,{\bf p}^\ast)-\lr{{\bf p},{\bf i}},\end{flalign*}
we get that the relation \eqref{lgx8} is preserved under the map $\varrho'$.

In order to finish proving that $\varrho'$ is an algebra homomorphism, by Proposition \ref{dmint27}, we need to prove the relations \eqref{dfgx1} and \eqref{dfgx2} are preserved under the map $\varrho'$.

\noindent{\bf Relation \eqref{dfgx1}:} Let us show that the relation \eqref{dfgx1}
\begin{flalign*}
u_{P[1]}\star u_{J}=q^{-\frac{1}{2}\Lambda({\bf p}^\ast,{\bf j}^\ast)-\lr{{\bf p},{\bf j}}}
\sum\limits_{[P'],[J']}|{}_{P'}\Hom_{\widetilde{\A}}(P,J)_{J'}|u_{J'\oplus P'[1]},
\end{flalign*}
where $P\in\P_{\widetilde{\A}}$~and $J\in\I_{\widetilde{\A}}$, is preserved under the map $\varrho'$. In fact, set $Q=\nu^{-1}(J)$, then
\begin{flalign*}
&\varrho'(u_{P[1]})\star\varrho'(u_{J})=u_{P}\star u_{Q[1]}=u_{P}\star u_{J[-1]}\\
&=q^{-\frac{1}{2}\Lambda({\bf p}^\ast,{\bf j}^\ast)-\lr{{\bf p},{\bf j}}}\sum\limits_{[P'],[J']}|{}_{P'}\Hom_{\widetilde{\A}}(P,J)_{J'}|u_{P'\oplus J'[-1]}.
\end{flalign*}

On the other hand, we have
\begin{flalign*}
&q^{-\frac{1}{2}\Lambda({\bf p}^\ast,{\bf j}^\ast)-\lr{{\bf p},{\bf j}}}
\sum\limits_{[P'],[J']}|{}_{P'}\Hom_{\widetilde{\A}}(P,J)_{J'}|\varrho'(u_{J'\oplus P'[1]})\\
&=q^{-\frac{1}{2}\Lambda({\bf p}^\ast,{\bf j}^\ast)-\lr{{\bf p},{\bf j}}}
\sum\limits_{[P'],[J']}|{}_{P'}\Hom_{\widetilde{\A}}(P,J)_{J'}|v^{\Lambda({{\bf j}'}^\ast,{{\bf p}'}^\ast)}u_{\nu^{-1}(J')[1]}\star u_{P'}\\
&=q^{-\frac{1}{2}\Lambda({\bf p}^\ast,{\bf j}^\ast)-\lr{{\bf p},{\bf j}}}
\sum\limits_{[P'],[J']}|{}_{P'}\Hom_{\widetilde{\A}}(P,J)_{J'}|v^{\Lambda({{\bf j}'}^\ast,{{\bf p}'}^\ast)}u_{J'[-1]}\star u_{P'}\\
&=q^{-\frac{1}{2}\Lambda({\bf p}^\ast,{\bf j}^\ast)-\lr{{\bf p},{\bf j}}}
\sum\limits_{[P'],[J']}|{}_{P'}\Hom_{\widetilde{\A}}(P,J)_{J'}|u_{P'\oplus J'[-1]}.
\end{flalign*}
Hence, the relation \eqref{dfgx1} is preserved under the map $\varrho'$.

\noindent{\bf Relation \eqref{dfgx2}:} Let us show that the relation \eqref{dfgx2}
\begin{equation}\label{dzyygx1}
u_{P[1]}\star u_{M}=q^{-\frac{1}{2}\Lambda({\bf p}^\ast,{\bf m}^\ast)-\lr{{\bf p},{\bf m}}}
\sum\limits_{[F],[P']}|{}_{P'}\Hom_{\widetilde{\A}}(P,M)_{F}|u_{F\oplus P'[1]},\end{equation}
where $M\in\widetilde{\A}$ has no nonzero injective direct summands and $P\in\P_{\widetilde{\A}}$, is preserved under the map $\varrho'$. In fact,
\begin{flalign*}
\varrho'(u_{P[1]})\star\varrho'(u_{M})=u_{P}\star u_{\tau^{-1}M}=v^{\Lambda({\bf p}^\ast,(\tau^{-1}\bf m)^\ast)}u_{P\oplus\tau^{-1} M}.
\end{flalign*}
Note that in $\mathcal {D}\mathcal {H}_\Lambda^{{cl}}(\widetilde{\A})$, by the ideal $\mathfrak{I}$, we have
\begin{equation*}
\begin{split}
&\varrho'(u_{P[1]})\star\varrho'(u_{M})=\\&v^{-\Lambda({\bf p}^\ast,^*{\bf m})}\sum\limits_{\begin{smallmatrix}[P'],[A],[I]\end{smallmatrix}}v^{\Lambda({\bf p}^\ast+{(\tau^{-1}{\bf m}})^*,{{\bf a}}^*)+\lr{(\tau^{-1}{\bf m})-{\bf a},{\bf p}}}|_{P'}\Hom_{\widetilde{\A}}(P,M)_{\tau A\oplus I}|u_{A}\star u_{P'\oplus I[-1]},
\end{split}\end{equation*}
where $P'\in\P_{\widetilde{\A}}$ and $A$ has no nonzero projective direct summands, since $\tau^{-1}M$ has no nonzero projective direct summands. Thus,
\begin{flalign}\label{dAA}
&\varrho'(u_{P[1]})\star\varrho'(u_{M})=\sum\limits_{\begin{smallmatrix}[P'],[A],[I]\end{smallmatrix}}v^{x_2}|_{P'}\Hom_{\widetilde{\A}}(P,M)_{\tau A\oplus I}|u_{A}\star u_{I[-1]}\star u_{P'},
\end{flalign}
where $x_2=-\Lambda({\bf p}^\ast,{^*{\bf m}})+\Lambda({\bf p}^\ast+{(\tau^{-1}{\bf m}})^*,{{\bf a}}^*)+\lr{(\tau^{-1}{\bf m})-{\bf a},{\bf p}}+\Lambda({\bf i}^\ast,{{\bf p}'}^\ast).$

On the other hand, for each $F$ in \eqref{dzyygx1}, write $F=\tau A\oplus I$ such that $A$ has no nonzero projective direct summands and $I\in\I_{\widetilde{\A}}$, then we have
\begin{equation}\label{dffzs}
\begin{split}&q^{-\frac{1}{2}\Lambda({\bf p}^\ast,{\bf m}^\ast)-\lr{{\bf p},{\bf m}}}
\sum\limits_{[F],[P']}|{}_{P'}\Hom_{\widetilde{\A}}(P,M)_{F}|\varrho'(u_{F\oplus P'[1]})\\
&=\sum\limits_{[P'],[A],[I]}v^{y_2}|{}_{P'}\Hom_{\widetilde{\A}}(P,M)_{\tau A\oplus I}|u_{A}\star u_{Q[1]}\star u_{P'}\\
&=\sum\limits_{[P'],[A],[I]}v^{y_2}|{}_{P'}\Hom_{\widetilde{\A}}(P,M)_{\tau A\oplus I}|u_{A}\star u_{I[-1]}\star u_{P'},
\end{split}\end{equation}
where $y_2=-\Lambda({\bf p}^\ast,{\bf m}^\ast)-2\lr{{\bf p},{\bf m}}+\Lambda((\tau{\bf a}+{\bf i})^\ast,{{\bf p}'}^\ast)+\Lambda({\bf a}^\ast,{\bf q}^\ast)$ and $Q=\nu^{-1}(I)$.

Note that $(\tau^{-1}{\bf m})^*=-{^\ast{\bf m}}$, $\lr{\tau^{-1}{\bf m},{\bf p}}=-\lr{{\bf p},{\bf m}}$ and ${\bf m}-{\bf p}=\tau{\bf a}+{\bf i}-{\bf p}'$. Using \cite[Lemma 6.4]{CDZ} and Lemma \ref{sjishu}, we obtain $x_2=y_2$. Thus, the equations \eqref{dAA} and \eqref{dffzs} are equal. Hence, the relation \eqref{dfgx2} is preserved under the map $\varrho'$.

Therefore, $\varrho'$ is an algebra homomorphism.
\end{proof}

\begin{lemma}\label{yla3}
For any $J,J'\in\mathcal{I}_{\widetilde{\A}}$ such that $J'$ has no nonzero projective direct summands, we have the following equation in the twisted derived Hall algebra $\mathcal {D}\mathcal {H}_q(\A):$
\begin{flalign*}u_{J}\ast u_{\tau J'}=q^{-\lr{{\bf j}',{\bf j}}}
\sum\limits_{[P],[J'']}|{}_{P}\Hom_{\widetilde{\A}}(\nu^{-1}(J),J')_{J''}|u_{\nu(P)\oplus \tau J''}.\end{flalign*}
\end{lemma}
\begin{proof}
Let $Q=\nu^{-1}(J)$. In $\mathcal {D}\mathcal {H}_q(\A)$, we have
\begin{flalign}\label{dcxgs}u_{Q}\ast u_{J'[-1]}=q^{-\lr{{\bf q},{\bf j}'}}
\sum\limits_{[P],[J'']}|{}_{P}\Hom_{\widetilde{\A}}(Q,J')_{J''}|u_{P\oplus J''[-1]}.\end{flalign}
Applying the algebra automorphism $\overline{\tau}$ of $\mathcal {D}\mathcal {H}_q(\A)$ to the equation \eqref{dcxgs}, we get
\begin{flalign*}u_{J[-1]}\ast u_{\tau J'[-1]}=q^{-\lr{{\bf q},{\bf j}'}}
\sum\limits_{[P],[J'']}|{}_{P}\Hom_{\widetilde{\A}}(Q,J')_{J''}|u_{\nu(P)[-1]\oplus \tau J''[-1]},\end{flalign*}
where we have used the condition that $J'$ has no nonzero projective direct summands.
By the shift automorphism of $\mathcal {D}\mathcal {H}_q(\A)$ and $\lr{{\bf q},{\bf j}'}=\lr{{\bf j}',{\bf j}}$, we finish the proof.
\end{proof}

\begin{lemma}\label{gjyla4}
Let $A,M,N,K\in\widetilde{\A}$ and $I, R\in\mathcal{I}_{\widetilde{\A}}$ such that $A, M$ have the same maximal projective direct summands and $I$ has no nonzero projective direct summands. Then we have
\begin{flalign*}
|_K\Hom_{\widetilde{\A}}(N,\tau M\oplus \tau I)_{\tau A\oplus R}|=&\sum_{\substack{[D],[B],[J], [J']\\ [I'], [P''],[J'']\\B\oplus J''\cong A,I'\oplus \nu(P'')\cong R}}q^{\lr{{\bf i},{\bf d}-{\bf n}}-\lr{{\bf j}',{\bf j}}}|_D\Hom_{\widetilde{\A}}(N,\tau M)_{\tau B\oplus J}|\\& \quad\quad|_K\Hom_{\widetilde{\A}}(D,\tau I)_{\tau J'\oplus I'}|\cdot
|_{P''}\Hom_{\widetilde{\A}}(Q,J')_{J''}|,
\end{flalign*}
where each $B$ has the same maximal projective direct summands as $M$.
\end{lemma}
\begin{proof}
Let us calculate $(u_{N[1]}\ast u_{\tau M})\ast u_{\tau I}$ in $\mathcal {D}\mathcal {H}_q(\A):$
by the relation \eqref{ydchgx}, we have
\begin{flalign*}
&(u_{N[1]}\ast u_{\tau M})\ast u_{\tau I}=\sum\limits_{[D],[B],[J]}q^{-\lr{{\bf n},\tau{\bf m}}}|{}_D\Hom_{\A}(N,\tau M)_{\tau B\oplus J}| u_{\tau B\oplus J}\ast u_{D[1]}\ast u_{\tau I}\\
&=\sum\limits_{\substack{[D],[B],[J]\\ [K],[J'],[I']}}q^{x_0}|{}_D\Hom_{\A}(N,\tau M)_{\tau B\oplus J}|\cdot|{}_K\Hom_{\A}(D,\tau I)_{\tau J'\oplus I'}| u_{\tau B\oplus J}\ast u_{\tau J'\oplus I'}\ast u_{K[1]},
\end{flalign*}
where $x_0=-\lr{{\bf n},\tau{\bf m}}-\lr{{\bf d},\tau{\bf i}}$, each $B$ has the same maximal projective direct summands as $M$ and $J,I'\in\mathcal{I}_{\widetilde{\A}}$. Remark that each $J'\in\mathcal{I}_{\widetilde{\A}}$ has no nonzero projective direct summands, since $I$ has no nonzero projective direct summands. Thus, we have
\begin{flalign*}
&(u_{N[1]}\ast u_{\tau M})\ast u_{\tau I}=\\&\sum\limits_{\substack{[D],[B],[J]\\ [K],[J'],[I']}}q^{x_0}|{}_D\Hom_{\A}(N,\tau M)_{\tau B\oplus J}|\cdot|{}_K\Hom_{\A}(D,\tau I)_{\tau J'\oplus I'}| u_{\tau B}\ast (u_{ J}\ast u_{\tau J'})\ast u_{ I'}\ast u_{K[1]}.
\end{flalign*}
By Lemma \ref{yla3}, we obtain
\begin{flalign*}
(u_{N[1]}\ast u_{\tau M})\ast u_{\tau I}=\sum\limits_{\substack{[D],[B],[J],[K]\\ [J'],[I'],[P''],[J'']}}q^{x_1}&|{}_D\Hom_{\A}(N,\tau M)_{\tau B\oplus J}|\cdot|{}_K\Hom_{\A}(D,\tau I)_{\tau J'\oplus I'}|\\& |{}_{P''}\Hom_{\A}(Q,J')_{J''}| u_{\tau B}\ast u_{\tau J''}\ast u_{I'' }\ast u_{ I'}\ast u_{K[1]},
\end{flalign*}
where $x_1=x_0-\lr{{\bf j}',{\bf j}}$, $Q=\nu^{-1}(J)$, $I''=\nu(P'')$ and each $J''\in\mathcal{I}_{\widetilde{\A}}$ has no nonzero projective direct summands.
Noting that $$\Ext^1_{\widetilde{\A}}(\tau B,\tau J'')\cong {\rm D}\Hom_{\widetilde{\A}}(J'',\tau B)\cong\Ext^1_{\widetilde{\A}}(B,J'')=0,$$ we get
\begin{flalign*}
(u_{N[1]}\ast u_{\tau M})\ast u_{\tau I}=\sum\limits_{\substack{[D],[B],[J],[K]\\ [J'],[I'],[P''],[J'']}}q^{x_1}&|{}_D\Hom_{\A}(N,\tau M)_{\tau B\oplus J}|\cdot|{}_K\Hom_{\A}(D,\tau I)_{\tau J'\oplus I'}|\\& |{}_{P''}\Hom_{\A}(Q,J')_{J''}| u_{\tau B\oplus\tau J''\oplus I'\oplus I''\oplus K[1]}.
\end{flalign*}

On the other hand, since $\Ext^1_{\widetilde{\A}}(\tau M,\tau I)=0$, we have
\begin{flalign*}
u_{N[1]}\ast (u_{\tau M}\ast u_{\tau I})&=u_{N[1]}\ast (u_{\tau M\oplus \tau I})\\
&=q^{-\lr{{\bf n},\tau({\bf m}+{\bf i})}}\sum\limits_{[K],[A],[R]}|{}_K\Hom_{\A}(N,\tau M\oplus \tau I)_{\tau A\oplus R}| u_{\tau A\oplus R\oplus K[1]},
\end{flalign*}
where each $A$ has the same maximal projective direct summands as $M$.
By the associativity of derived Hall algebras, comparing the coefficients of the term $u_{\tau A\oplus R\oplus K[1]}$, we get the desired formula.
\end{proof}

Let ${\mathfrak{I}''}$ be the two sided ideal of $\mathcal {D}\mathcal {H}_\Lambda^{{cl}_1}(\widetilde{\A})$ generated by the elements
$$\{r_{I,N}, r_{M,N}~|~M,N\in{\widetilde{\A}}, I\in \mathcal{I}_{\widetilde{\A}}\ \text{and $M$ has no nonzero injective direct summands}\}.$$

\begin{lemma}\label{dlxdjkh}
It holds that $\mathfrak{I}={\mathfrak{I}''}$.
\end{lemma}
\begin{proof}
It suffices to show that $\mathfrak{I}\subseteq \mathfrak{I}''$. Let $M,N\in \widetilde{\A}$ and $I\in \mathcal{I}_{\widetilde{\A}}$ such that $M$ has no nonzero injective direct summands. It remains to show that $r_{I\oplus M,N}$ belongs to $\mathfrak{I}''$.

First of all, it is easy to see
$$r_{P\oplus I\oplus M,N}=v^{-\Lambda({\bf p}^\ast,{\bf i}^\ast+{\bf m}^\ast+{\bf n}^\ast)}u_P\star r_{I\oplus M,N}$$ for any $P\in \mathcal{P}_{\widetilde{\A}}$.
Hence, for proving $r_{I\oplus M,N}\in\mathfrak{I}''$, without loss of generality, we can assume that $I$ has no nonzero projective direct summands.

In the quotient algebra $\mathcal {D}\mathcal {H}_\Lambda^{{cl}_1}(\widetilde{\A})/\mathfrak{I}''$, we have
\begin{equation*}
\begin{split}
&u_{I\oplus M\oplus N}
=v^{-\Lambda((\mathbf{m}+\mathbf{n})^\ast, \mathbf{i}^\ast)} u_{M\oplus N}\star u_I\\
&=v^{-\Lambda((\mathbf{m}+\mathbf{n})^\ast, \mathbf{i}^\ast)} (\sum_{[D],[B],[J]}v^{\Lambda((\mathbf{m}+\mathbf{n})^\ast,\mathbf{b}^\ast)+\langle \mathbf{m}-\mathbf{b},\mathbf{n}\rangle}|_D\Hom_{\widetilde{\A}}(N,\tau M)_{\tau B\oplus J}|u_B\star u_{D\oplus J[-1]})\star u_I\\
&=\sum_{[D],[B],[J]}v^{\Lambda((\mathbf{m}+\mathbf{n})^\ast,(\mathbf{b}-\mathbf{i})^\ast)+\langle \mathbf{m}-\mathbf{b},\mathbf{n}\rangle+\Lambda({\bf j}^\ast,{\bf d}^\ast)+\Lambda({\bf d}^\ast,{\bf i}^\ast)}|_D\Hom_{\widetilde{\A}}(N,\tau M)_{\tau B\oplus J}|u_B\star u_{J[-1]}\star u_{I\oplus D},
\end{split}\end{equation*}
where each $B$ has the same maximal projective direct summand as $M$ and $J\in\mathcal{I}_{\widetilde{\A}}$. Moreover,
\begin{equation*}
\begin{split}
u_{I\oplus D}=\sum_{[K],[J'],[I']}v^{\Lambda(({\bf i}+{\bf d})^\ast,{{\bf j}'}^\ast)+\lr{{\bf i}-{\bf j}',{\bf d}}+\Lambda({{\bf i}'}^\ast,{\bf k}^\ast)}|_K\Hom_{\widetilde{\A}}(D,\tau I)_{\tau J'\oplus I'}|u_{J'}\star u_{I'[-1]}\star u_K,
\end{split}\end{equation*}
where each $J'\in\mathcal{I}_{\widetilde{\A}}$ has no nonzero projective direct summands and $I'\in\mathcal{I}_{\widetilde{\A}}$,
and
\begin{equation*}
\begin{split}
u_{J[-1]}\star u_{J'}=u_{Q[1]}\star u_{J'}
=\sum_{[P''],[J'']}v^{-\Lambda({\bf q}^\ast,{{\bf j}'}^\ast)+\Lambda({{\bf j}''}^\ast,{{\bf p}''}^\ast)-2\lr{{\bf q},{\bf j}'}}|_{P''}\Hom_{\widetilde{\A}}(Q,J')_{J''}|u_{J''}\star u_{P''[1]},
\end{split}\end{equation*}
where $Q=\nu^{-1}(J)$ and $J''\in\mathcal{I}_{\widetilde{\A}}$ has no nonzero projective direct summands. Hence, we obtain
\begin{equation*}
\begin{split}
u_{I\oplus M\oplus N}
=\sum_{\substack{[D],[B],[J], [K]\\ [J'],[I'], [P''],[J'']}}&v^{x'_0}|_D\Hom_{\widetilde{\A}}(N,\tau M)_{\tau B\oplus J}|\cdot |_K\Hom_{\widetilde{\A}}(D,\tau I)_{\tau J'\oplus I'}|\\
&|_{P''}\Hom_{\widetilde{\A}}(Q,J')_{J''}| u_B\star u_{J''}\star u_{P''[1]}\star u_{I'[-1]}\star u_K,
\end{split}\end{equation*}
where $x'_0=\Lambda((\mathbf{m}+\mathbf{n})^\ast,(\mathbf{b}-\mathbf{i})^\ast)+\langle \mathbf{m}-\mathbf{b},\mathbf{n}\rangle+\Lambda({\bf j}^\ast,{\bf d}^\ast)+\Lambda({\bf d}^\ast,{\bf i}^\ast)+\Lambda(({\bf i}+{\bf d})^\ast,{{\bf j}'}^\ast)+\lr{{\bf i}-{\bf j}',{\bf d}}+\Lambda({{\bf i}'}^\ast,{\bf k}^\ast)-\Lambda({\bf q}^\ast,{{\bf j}'}^\ast)+\Lambda({{\bf j}''}^\ast,{{\bf p}''}^\ast)-2\lr{{\bf q},{\bf j}'}.$
Thus, we have
\begin{equation*}
\begin{split}
u_{I\oplus M\oplus N}
=\sum_{\substack{[D],[B],[J], [K]\\ [J'],[I'], [P''],[J'']}}&v^{x_0}|_D\Hom_{\widetilde{\A}}(N,\tau M)_{\tau B\oplus J}|\cdot |_K\Hom_{\widetilde{\A}}(D,\tau I)_{\tau J'\oplus I'}|\\
&|_{P''}\Hom_{\widetilde{\A}}(Q,J')_{J''}| u_{B\oplus J''}\star u_{K\oplus(I'\oplus I'')[-1]},
\end{split}\end{equation*}
where $x_0=x'_0+\Lambda({\bf b}^\ast,{{\bf j}''}^\ast)+\Lambda({{\bf i}''}^\ast,{{\bf i}'}^\ast)-\Lambda({{\bf i}'}^\ast+{{\bf i}''}^\ast,{\bf k}^\ast)$ and $I''=\nu(P'')$.

Set $x_1=x_0-2\lr{{\bf i},{\bf d}-{\bf n}}+2\lr{{\bf j}',{\bf j}}$ and note that
\[
{\bf d}+{\bf \tau m}={\bf n}+{\bf \tau b}+{\bf j},~{\bf k}+ \tau{\bf i}={\bf d}+\tau{\bf  j}'+{\bf i}'~\text{and}~{\bf p}''+{\bf j}'={\bf q}+{\bf j}''.
\]
Then we have
\begin{equation}\label{fduan}
\begin{split}
&\Lambda({\bf i}^\ast+{\bf d}^\ast,{{\bf j}'}^\ast)+{\Lambda({{\bf i}'}^\ast,{\bf k}^\ast)}-\Lambda({\bf q}^\ast,{{\bf j}'}^\ast)+\Lambda({{\bf j}''}^\ast,{{\bf p}''}^\ast)+\Lambda({\bf b}^\ast,{{\bf j}''}^\ast)+\Lambda({{\bf i}''}^\ast,{{\bf i}'}^\ast)-{\Lambda({{\bf i}'}^\ast+{{\bf i}''}^\ast,{\bf k}^\ast)}\\
&=\Lambda({\bf i}^\ast+{\bf d}^\ast,{{\bf j}'}^\ast)-\Lambda({\bf q}^\ast,{{\bf j}'}^\ast)+\Lambda({{\bf j}''}^\ast,{{\bf p}''}^\ast)+\Lambda({\bf b}^\ast,{{\bf j}''}^\ast)+{\Lambda({{\bf i}''}^\ast,{{\bf i}'}^\ast)}-{\Lambda({{\bf i}''}^\ast,{\bf k}^\ast)}\\
&=\Lambda({\bf i}^\ast+{\bf d}^\ast,{{\bf j}'}^\ast)-\Lambda({\bf q}^\ast,{{\bf j}'}^\ast)+\Lambda({{\bf j}''}^\ast,{{\bf p}''}^\ast)+\Lambda({\bf b}^\ast,{{\bf j}''}^\ast)+\Lambda({{\bf i}''}^\ast, (\tau{\bf i}-{\bf d}-\tau {{\bf j}'})^\ast)\\
&=\Lambda({\bf i}^\ast+{\bf d}^\ast,{{\bf j}'}^\ast)-{\Lambda({\bf q}^\ast,{{\bf j}'}^\ast)}+\Lambda({{\bf j}''}^\ast,{{\bf p}''}^\ast)+\Lambda({\bf b}^\ast,{{\bf j}''}^\ast)+{\Lambda({{\bf p}''}^\ast,-{\bf i}^\ast- {^\ast {\bf d}}+{{\bf j}'}^\ast)}\\
&=\Lambda({\bf i}^\ast+{\bf d}^\ast,{{\bf j}'}^\ast)+\Lambda({{\bf j}''}^\ast,{{\bf p}''}^\ast)+\Lambda({\bf b}^\ast,{{\bf j}''}^\ast)+{\Lambda(({\bf p}''-{\bf q})^\ast,{{\bf j}'}^\ast)}-\Lambda({{\bf p}''}^\ast,{\bf i}^\ast)-\Lambda({{\bf p}''}^\ast, {{^\ast {\bf d}}})\\
&=\Lambda({\bf i}^\ast+{\bf d}^\ast,{{\bf j}'}^\ast)+{\Lambda({{\bf j}''}^\ast,{{\bf p}''}^\ast)}+\Lambda({\bf b}^\ast,{{\bf j}''}^\ast)+ {\Lambda({{\bf j}''}^\ast,{{\bf j}'}^\ast)}-\Lambda({{\bf p}''}^\ast,{\bf i}^\ast)-\Lambda({{\bf p}''}^\ast, {^\ast {\bf d}})\\
&=\Lambda({\bf i}^\ast+{\bf d}^\ast,{{\bf j}'}^\ast)+{\Lambda({\bf b}^\ast,{{\bf j}''}^\ast)}+{\Lambda({{\bf j}''}^\ast,{\bf q}^\ast)}-\Lambda({{\bf p}''}^\ast,{\bf i}^\ast)-\Lambda({{\bf p}''}^\ast,{^\ast {\bf d}})\\
&=\Lambda({\bf i}^\ast+{\bf d}^\ast,{{\bf j}'}^\ast)+{\Lambda({\bf b}^\ast,{{\bf j}''}^\ast)}+{\Lambda({{\bf j}''}^\ast,{^\ast{\bf j}})}-\Lambda({{\bf p}''}^\ast,{\bf i}^\ast)-\Lambda({{\bf p}''}^\ast,{^\ast {\bf d}}).
\end{split}\end{equation}
Since ${^\ast{\bf j}}={^\ast({\bf d}-{\bf n})}+{^\ast(\tau({\bf m}-{\bf b}))}={^\ast({\bf d}-{\bf n})}-({\bf m}-{\bf b})^\ast$, we get
\begin{flalign*} \Lambda({\bf b}^\ast,{{\bf j}''}^\ast)+\Lambda({{\bf j}''}^\ast,{^\ast{\bf j}})&=\Lambda({\bf b}^\ast,{{\bf j}''}^\ast)+\Lambda({{\bf j}''}^\ast,{^\ast({\bf d}-{\bf n})})-\Lambda({{\bf j}''}^\ast,({\bf m}-{\bf b})^\ast)\\
&=\Lambda({{\bf j}''}^\ast,{^\ast({\bf d}-{\bf n})})-\Lambda({{\bf j}''}^\ast,{\bf m}^\ast).\end{flalign*}
Thus, the equations \eqref{fduan} are equal to the following
\begin{flalign*}
&\Lambda({\bf i}^\ast+{\bf d}^\ast,{{\bf j}'}^\ast)+\Lambda({{\bf j}''}^\ast,{^\ast({\bf d}-{\bf n})})-\Lambda({{\bf j}''}^\ast,{\bf m}^\ast)-\Lambda({{\bf p}''}^\ast,{\bf i}^\ast)-\Lambda({{\bf p}''}^\ast,{^\ast {\bf d}})\\
&=\Lambda({\bf i}^\ast,{{\bf j}'}^\ast)+\Lambda({^\ast{\bf d}},{{\bf j}'}^\ast)+\lr{{\bf j}',{\bf d}}+\Lambda({{\bf j}''}^\ast,{^\ast{\bf d}})-\Lambda({{\bf j}''}^\ast,{^\ast{\bf n}})\\&\quad-\Lambda({{\bf j}''}^\ast,{\bf m}^\ast)+\Lambda({\bf i}^\ast,{{\bf p}''}^\ast)-\Lambda({{\bf p}''}^\ast,{^\ast{\bf d}})\\
&=\Lambda({\bf i}^\ast,{{\bf j}'}^\ast+{{\bf p}''}^\ast)+\Lambda(({\bf j}''-{\bf p}''-{\bf j}')^\ast,{^\ast{\bf d}})+\lr{{\bf j}',{\bf d}}-\Lambda({{\bf j}''}^\ast,{^\ast{\bf n}})-\Lambda({{\bf j}''}^\ast,{\bf m}^\ast)\\
&=\Lambda({\bf i}^\ast,{{\bf q}}^\ast+{{\bf j}''}^\ast)-\Lambda({\bf j}^\ast,{\bf d}^\ast)+\lr{{\bf j}',{\bf d}}-\Lambda({{\bf j}''}^\ast,{^\ast{\bf n}})-\Lambda({{\bf j}''}^\ast,{\bf m}^\ast).
\end{flalign*}
Set $A=B\oplus J''$ and $R=I'\oplus I''$. Note that ${\bf a}={\bf b}+{\bf {j''}}$ and $\langle {\bf q},{\bf j'}\rangle =\lr{{\bf j}',{\bf j}}$. Putting all of these together, we compute
\begin{flalign*}
&x_1-\Lambda(({\bf m}+{\bf i}+{\bf n})^*,{\bf a}^*)-\lr{{\bf m}+{\bf i}-{\bf a},{\bf n}}\\
&=-\Lambda({\bf m}^\ast+{\bf n}^\ast,{\bf i}^\ast)-{\Lambda({\bf m}^\ast+{\bf n}^\ast,{{\bf j}''}^\ast)}-\Lambda({\bf i}^\ast,{\bf b}^\ast+{{\bf j}''}^\ast)+\Lambda({\bf d}^\ast,{\bf i}^\ast)+\Lambda({\bf i}^\ast,{\bf q}^\ast+{{\bf j}''}^\ast)\\
&\quad -{\Lambda({{\bf j}''}^\ast,{^\ast{\bf  n}})}-{\Lambda({{\bf j}''}^\ast,{\bf m}^\ast)}-\lr{{\bf i},{\bf d}} +\lr{ {\bf i},{\bf n}}+\lr{ {{\bf j}''},{\bf n}}\\
&=-\Lambda({\bf m}^\ast+{\bf n}^\ast,{\bf i}^\ast)-\Lambda({\bf i}^\ast,{\bf b}^\ast+{{\bf j}''}^\ast)+\Lambda({\bf d}^\ast,{\bf i}^\ast)+\Lambda({\bf i}^\ast,{\bf q}^\ast+{{\bf j}''}^\ast)-\lr{ {\bf i},{\bf d}} +\lr{ {\bf i},{\bf n}}\\
&=\Lambda({\bf i}^\ast,({\bf m}-{\bf b})^\ast+({\bf n}-{\bf d})^\ast+{\bf q}^\ast)-\lr{ {\bf i},{\bf d}} +\lr{ {\bf i},{\bf n}}\\
&=-\Lambda({\bf i}^\ast,{^\ast(\tau{\bf m}-\tau{\bf b})})+\Lambda({\bf i}^\ast,{^\ast({\bf n}-{\bf d})})-\lr{{\bf i},{\bf n}-{\bf d}}+\Lambda({\bf i}^\ast,{^\ast{\bf j}})-\lr{ {\bf i},{\bf d}} +\lr{ {\bf i},{\bf n}}\\
&=0.
\end{flalign*}
That is,
\begin{flalign*}
x_1=\Lambda(({\bf m}+{\bf i}+{\bf n})^*,{\bf a}^*)+\lr{{\bf m}+{\bf i}-{\bf a},{\bf n}}.
\end{flalign*}
Thus, by Lemma \eqref{gjyla4}, we obtain
\begin{equation*}
\begin{split}
u_{I\oplus M\oplus N}
=\sum\limits_{\begin{smallmatrix}[K],[A],[R]\end{smallmatrix}}v^{\Lambda(({\bf m}+{\bf i}+{\bf n})^*,{\bf a}^*)+\lr{{\bf m}+{\bf i}-{\bf a},{\bf n}}}|_K\Hom_{\widetilde{\A}}(N,\tau M\oplus \tau I)_{\tau A\oplus R}|u_{A}\star u_{K\oplus R[-1]},
\end{split}\end{equation*}
where each $A$ has the same maximal projective direct summand as $M$ and $R\in\I_{\widetilde{\A}}$.

\end{proof}

\begin{proposition}
The map $\sigma':\mathcal {D}\mathcal {H}_\Lambda^{{cl}}(\widetilde{\A})\longrightarrow\mathcal {D}\mathcal {H}_\Lambda^{{cl}}(\widetilde{\A})$ defined by
$$\sigma'(u_{I[-1]\oplus M\oplus P[1]})=v^{\Lambda({\bf i}^\ast,({\bf m}-{\bf p})^\ast)+\Lambda({\bf m}^\ast,{\bf p}^\ast)}u_{\nu^{-1}I}\star u_{\tau^{-1} M'\oplus \nu^{-1}(I')[1]}\star u_{P}$$ for any $M=M'\oplus I'\in\widetilde{\A}$, $I\in\I_{\widetilde{\A}}$ and $P\in\P_{\widetilde{\A}}$, where $I'$ is the maximal injective direct summand of $M$, is an endomorphism of algebras.
\end{proposition}
\begin{proof}
By the definition of $\varrho'$, we have $$\varrho'(u_{\nu^{-1}(I)[1]})=\varrho'(u_{I[-1]})=u_{\nu^{-1}(I)}$$ for any $I\in\I_{\widetilde{\A}}$. Thus, $\varrho'(\mathfrak{I}_1)=0$. Hence, the algebra homomorphism $\varrho'$ induces an algebra homomorphism $\theta':\mathcal {D}\mathcal {H}_\Lambda^{{cl}_1}(\widetilde{\A})\longrightarrow\mathcal {D}\mathcal {H}_\Lambda^{{cl}}(\widetilde{\A}), u\mapsto \varrho'(u)$.
Now, let us prove that $\theta'(\mathfrak{I})=0$. By Lemma \ref{dlxdjkh}, we need to prove $\theta'(r_{I,N})=0$ and $\theta'(r_{M,N})=0$ for any $N\in{\widetilde{\A}}$, $I\in \mathcal{I}_{\widetilde{\A}}$ and $M\in{\widetilde{\A}}$ without nonzero injective direct summands. Let us write $N=N'\oplus J'$ such that $J'$ is the maximal injective direct summand of $N$.

\noindent{\bf Relation} $\theta'(r_{I,N})=0$: Set $P=\nu^{-1}(I)$ and $Q'=\nu^{-1}(J')$, then
\begin{flalign*}
\theta'(u_{I\oplus N})&=\theta'(u_{N'\oplus J'\oplus I})=u_{\tau^{-1} N'\oplus (P\oplus Q')[1]}\\&=v^{\Lambda((\tau^{-1}{{\bf n}'})^\ast,({\bf p}+{\bf q}')^\ast)-\Lambda({\bf p}^\ast,{{\bf q}'}^\ast)}u_{\tau^{-1}N'}\star u_{P[1]}\star u_{Q'[1]}\\&=v^{\Lambda((\tau^{-1}{{\bf n}'})^\ast,({\bf p}+{\bf q}')^\ast)-\Lambda({\bf p}^\ast,{{\bf q}'}^\ast)}u_{\tau^{-1}N'}\star u_{I[-1]}\star u_{Q'[1]}.
\end{flalign*}
Using the relation \eqref{lgx6}, we obtain
\begin{equation}\label{dmpgxl}
\theta'(u_{I\oplus N})=\sum\limits_{[D],[J]}v^{x'_0}|_D\Hom_{\widetilde{\A}}(\tau^{-1} N',I)_{J}|u_{D\oplus J[-1]}\star u_{Q'[1]},
\end{equation}
where $J\in\I_{\widetilde{\A}}$ has the same maximal projective direct summand as $I$, and $$x'_0=\Lambda((\tau^{-1}{{\bf n}'})^\ast,({\bf p}+{\bf q}')^\ast)-\Lambda({\bf p}^\ast,{{\bf q}'}^\ast)-\Lambda((\tau^{-1}{{\bf n}'})^\ast,{\bf i}^\ast)-2\lr{\tau^{-1}{{\bf n}'},{\bf i}}.$$
Writing each $D=\tau^{-1}D'\oplus P'$ in \eqref{dmpgxl} such that $P'\in\P_{\widetilde{\A}}$ is the maximal projective direct summand of $D$ and $D'$ has no nonzero injective direct summands, we have
\begin{equation}\label{dmpgxl2}
\begin{split}
\theta'(u_{I\oplus N})&=\sum\limits_{[D'],[P'],[J]}v^{x'_0}|_{\tau^{-1}D'\oplus P'}\Hom_{\widetilde{\A}}(\tau^{-1} N',I)_{J}|u_{\tau^{-1}D'\oplus P'\oplus J[-1]}\star u_{Q'[1]}\\
&=\sum\limits_{[D'],[P'],[J]}v^{x_0}|_{\tau^{-1}D'\oplus P'}\Hom_{\widetilde{\A}}(\tau^{-1} N',I)_{J}|u_{Q[1]}\star u_{P'}\star u_{\tau^{-1}D'\oplus Q'[1]},
\end{split}
\end{equation}
where $x_0=x'_0+\Lambda({\bf j}^\ast,(\tau^{-1}{\bf d}'+{\bf p}')^\ast)-\Lambda({{\bf p}'}^\ast,(\tau^{-1}{\bf d}')^\ast)-\Lambda((\tau^{-1}{\bf d}')^\ast,{{\bf q}'}^\ast)$ and $Q=\nu^{-1}(J)$.

On the other hand,
\begin{equation}\label{dmpgxl3}
\begin{split}
&\sum\limits_{\begin{smallmatrix}[D],[J],[I']\end{smallmatrix}}v^{\Lambda(({\bf n}+{\bf i})^*,{\bf j}^*)+\lr{{\bf i}-{\bf j},{\bf n}}}|_D\Hom_{\widetilde{\A}}(N,\tau I)_{\tau J\oplus I'}|\theta'(u_{J})\star \theta'(u_{D\oplus I'[-1]})=\\
&\sum\limits_{\begin{smallmatrix}[D'],[J],[I']\end{smallmatrix}}v^{\Lambda(({\bf n}+{\bf i})^*,{\bf j}^*)+\lr{{\bf i}-{\bf j},{\bf n}}+\Lambda({{\bf i}'}^\ast,{\bf d}^\ast)}|_D\Hom_{\widetilde{\A}}(N,\tau I)_{\tau J\oplus I'}|u_{Q[1]}\star u_{P'}\star u_{\tau^{-1}D'\oplus Q'[1]},
\end{split}\end{equation}
where $J\in\I_{\widetilde{\A}}$ has the same maximal projective direct summand as $I$, and $D=D'\oplus J'$ has the same maximal injective direct summand $J'$ as $N$, set $Q=\nu^{-1}(J)$, $P'=\nu^{-1}(I')$ and $Q'=\nu^{-1}(J')$.

Let $I=I''\oplus Q''$ such that $Q''$ is the maximal projective direct summand of $I$.
In \eqref{dmpgxl2}, writing $J=R''\oplus Q''$ such that $J$ has the same maximal projective direct summand $Q''$ as $I$, and setting $J''=\nu(Q'')$, $I'=\nu(P')$, we have
\begin{equation*}
\begin{split}
|_{\tau^{-1}D'\oplus P'}\Hom_{\widetilde{\A}}(\tau^{-1} N',I)_{J}|&=|\Hom_{\mathcal{D}^b(\widetilde{\A})}(\tau^{-1} N',I)_{\tau^{-1}D'[1]\oplus P'[1]\oplus J}|\\
&=|\Hom_{\mathcal{D}^b(\widetilde{\A})}(\tau^{-1} N',I''\oplus Q'')_{\tau^{-1}D'[1]\oplus P'[1]\oplus R''\oplus Q''}|\\
&=|\Hom_{\mathcal{D}^b(\widetilde{\A})}( N',\tau I''\oplus J''[-1])_{D'[1]\oplus I'\oplus \tau R''\oplus J''[-1]}|.
\end{split}\end{equation*}
Since $\Hom_{\mathcal{D}^b(\widetilde{\A})}(N',J''[-1])=0$, we get
\begin{flalign*}
|\Hom_{\mathcal{D}^b(\widetilde{\A})}( N',\tau I''\oplus J''[-1])_{D'[1]\oplus I'\oplus \tau R''\oplus J''[-1]}|
&=|\Hom_{\mathcal{D}^b(\widetilde{\A})}( N',\tau I'')_{D'[1]\oplus I'\oplus \tau R''}|\\
&=|_{D'}\Hom_{\widetilde{\A}}(N',\tau I'')_{I'\oplus \tau R''}|\\
&=|_{D'}\Hom_{\widetilde{\A}}(N',\tau I)_{\tau J\oplus I' }|.
\end{flalign*}
While in \eqref{dmpgxl3}, since $\Hom_{\widetilde{\A}}(J',\tau I)=0$, we have
\begin{flalign*}
|_D\Hom_{\widetilde{\A}}(N,\tau I)_{\tau J\oplus I'}|=|_{D'}\Hom_{\widetilde{\A}}(N',\tau I)_{\tau J\oplus I' }|.
\end{flalign*}

Using $(\tau^{-1}{\bf d}')^*=-{^\ast{\bf d}'}$ and $\lr{\tau^{-1}{\bf d}',{\bf i}}=-\lr{{\bf i},{\bf d}'}$, we obtain
\begin{flalign*}
x_0=&-\Lambda({^\ast{\bf n}'},{^\ast{\bf i}}+{^\ast{\bf j}'})-\Lambda({^\ast{\bf i}},{^\ast{\bf j}'})+\Lambda({^\ast{\bf n}'},{{\bf i}^\ast})
+2\lr{{\bf i},{\bf n}'}\\&-\Lambda({\bf j}^\ast,{^\ast{\bf d}'})+\Lambda({\bf j}^\ast,{^\ast{\bf i}'})+\Lambda({^\ast{\bf i}'},{^\ast{\bf d}'})+\Lambda({^\ast{\bf d}'},{^\ast{\bf j}'}).
\end{flalign*}
By Lemma \ref{sjishu}, we get
\begin{flalign*}
x_0
=\Lambda({^\ast({\bf d}'-{\bf n}')},{^\ast{\bf j}'})-\Lambda({{\bf i}}^\ast,{{\bf j}'}^\ast)+\Lambda({^\ast{\bf j}},{^\ast({\bf i}'-{\bf d}')})+\Lambda({{\bf i}'}^\ast,{{\bf d}'}^\ast)+\lr{{\bf i},{\bf n}'}-\lr{{\bf j},{\bf d}'-{\bf i}'}.
\end{flalign*}
Noting that ${\bf d}'-{\bf n}'={\bf i}'-\tau({\bf i}-{\bf j})$ and ${\bf i}'-{\bf d}'=\tau({\bf i}-{\bf j})-{\bf n}'$, we get
$$\Lambda({^\ast({\bf d}'-{\bf n}')},{^\ast{\bf j}'})=\Lambda({{\bf i}'}^\ast,{{\bf j}'}^\ast)-\Lambda({{\bf j}}^\ast,{{\bf j}'}^\ast)+\Lambda({{\bf i}}^\ast,{{\bf j}'}^\ast)+\lr{{\bf i}-{\bf j},{\bf j}'}$$
and
$$\Lambda({^\ast{\bf j}},{^\ast({\bf i}'-{\bf d}')})=\Lambda({{\bf n}'}^\ast,{{\bf j}}^\ast)+\Lambda({{\bf i}}^\ast,{{\bf j}}^\ast)+\lr{{\bf i}-{\bf j},{\bf j}}.$$
Thus, we obtain $$x_0=\Lambda(({\bf n}+{\bf i})^*,{\bf j}^*)+\lr{{\bf i}-{\bf j},{\bf n}}+\Lambda({{\bf i}'}^\ast,{\bf d}^\ast).$$
Hence, the equations \eqref{dmpgxl2} and \eqref{dmpgxl3} are equal, i.e. $\theta'(r_{I,N})=0$.

\noindent{\bf Relation} $\theta'(r_{M,N})=0$:
Set $Q'=\nu^{-1}(J')$, then
\begin{flalign*}
\theta'(u_{M\oplus N})=\theta'(u_{M\oplus N'\oplus J'})=u_{\tau^{-1} M\oplus \tau^{-1} N'\oplus Q'[1]}=v^{\Lambda((\tau^{-1}{\bf m}+\tau^{-1}{\bf n}')^\ast,{{\bf q}'}^\ast)}u_{\tau^{-1} M\oplus \tau^{-1}N'}\star u_{Q'[1]}.
\end{flalign*}
Note that in $\mathcal {D}\mathcal {H}_\Lambda^{{cl}}(\widetilde{\A})$, by the ideal $\mathfrak{I}$, we have
\begin{equation}\label{dkxjih}
\begin{split}
u_{\tau^{-1} M\oplus\tau^{-1} N'}=\sum\limits_{\begin{smallmatrix}[K],[X],[J'']\end{smallmatrix}}v^{x'_1}|_K\Hom_{\widetilde{\A}}(\tau^{-1} N',M)_{\tau X\oplus J''}|u_{X}\star u_{K\oplus J''[-1]},
\end{split}
\end{equation}
where $x'_1=\Lambda((\tau^{-1}{\bf m}+\tau^{-1}{\bf n}')^\ast,{\bf x}^*)+\lr{\tau^{-1}{\bf m}-{\bf x},\tau^{-1}{\bf n}'}$, each $X$ has no nonzero projective direct summands and $J''\in\I_{\widetilde{\A}}$.

Let $M=M'\oplus P'$ such that $P'$ is the maximal projective direct summand of $M$. For each $K$ and $X$ in \eqref{dkxjih}, write $K=\tau^{-1} D'\oplus P$ and $X=\tau^{-1} X'$ for some $P\in P_{\widetilde{\A}}$ such that $D'$ and $X'$ have no nonzero injective direct summands. Thus,
\begin{equation*}
\begin{split}
&\theta'(u_{M\oplus N})=\\&\sum\limits_{\begin{smallmatrix}[D'],[P],[X'],[J'']\end{smallmatrix}}v^{x''_1}|_{\tau^{-1} D'\oplus P}\Hom_{\widetilde{\A}}(\tau^{-1} N',M'\oplus P')_{X'\oplus J''}|u_{\tau^{-1} X'}\star u_{\tau^{-1} D'\oplus P\oplus J''[-1]}\star u_{Q'[1]},
\end{split}
\end{equation*}
where $x''_1=x'_1+\Lambda((\tau^{-1}{\bf m}+\tau^{-1}{\bf n}')^\ast,{{\bf q}'}^\ast)$.

Note that $X'\oplus J''$ has the same maximal projective direct summand $P'$ as $M$, since $\tau^{-1}N'$ has no nonzero projective direct summands.
Thus, if we write $X'=A''\oplus S'$ and $J''=I''\oplus R'$ such that $S'$ and $R'$ are the maximal projective direct summands of $X'$ and $J''$, respectively, then $S'\oplus R'=P'$. Remark that $R'$ is both projective and injective.

Set $I'=\nu(P')$, $I=\nu(P)$, $T'=\nu(S')$ and $T''=\nu(R')$, then $T'\oplus T''=I'$. Since $\Hom_{\mathcal{D}^b(\widetilde{\A})}(N',I'[-1])=0$, we get
\begin{flalign*}
&|_{\tau^{-1} D'\oplus P}\Hom_{\widetilde{\A}}(\tau^{-1} N',M'\oplus P')_{X'\oplus J''}|\\
&=|\Hom_{\mathcal{D}^b(\widetilde{\A})}( \tau^{-1} N',M'\oplus P')_{ \tau^{-1}D'[1]\oplus P[1]\oplus A''\oplus S'\oplus I''\oplus R'}|\\
&=|\Hom_{\mathcal{D}^b(\widetilde{\A})}( N',\tau M'\oplus I'[-1])_{ D'[1]\oplus I\oplus \tau A''\oplus T'[-1]\oplus \tau I''\oplus T''[-1]}|\\
&=|\Hom_{\mathcal{D}^b(\widetilde{\A})}( N',\tau M'\oplus I'[-1])_{ D'[1]\oplus I\oplus \tau A''\oplus \tau I''\oplus I'[-1]}|\\
&=|\Hom_{\mathcal{D}^b(\widetilde{\A})}( N',\tau M')_{ D'[1]\oplus I\oplus \tau A''\oplus \tau I''}|\\
&=|_{D'}\Hom_{\widetilde{\A}}(N',\tau M)_{\tau (A''\oplus I'')\oplus I }|\\
&=|_{D'}\Hom_{\widetilde{\A}}(N',\tau M)_{\tau (X'\oplus J'')\oplus I }|.
\end{flalign*}
Set $A=X'\oplus J''$ and $A'=A''\oplus I''$, then $A=A'\oplus P'$.
Thus, we have
\begin{equation}\label{dstt2}
\begin{split}
&\theta'(u_{M\oplus N})=\\&\sum\limits_{\begin{smallmatrix}[D'],[P],[X'],[J'']\end{smallmatrix}}v^{x''_1}|_{D'}\Hom_{\widetilde{\A}}(N',\tau M)_{\tau A\oplus I }|u_{\tau^{-1} (A''\oplus S')}\star u_{\tau^{-1} D'\oplus P\oplus (I''\oplus R')[-1]}\star u_{Q'[1]}
\\&=\sum\limits_{\begin{smallmatrix}[D'],[P],[X'],[J'']\end{smallmatrix}}v^{x_1}|_{D'}\Hom_{\widetilde{\A}}(N',\tau M)_{\tau A\oplus I }|u_{\tau^{-1} A''\oplus \tau^{-1}S'\oplus P''[1]\oplus T[1]}\star u_{P}\star u_{\tau^{-1} D'\oplus Q'[1]},
\end{split}
\end{equation}
where $P''=\nu^{-1}(I'')$, $T=\nu^{-1}(R')$ and \begin{flalign*}x_1=&x''_1+\Lambda({{\bf i}''}^\ast+{{\bf r}'}^\ast,(\tau^{-1}{\bf d}')^\ast+{\bf p}^\ast)-\Lambda({\bf p}^\ast,(\tau^{-1}{\bf d}')^\ast)\\&-\Lambda((\tau^{-1}{\bf d}')^\ast,{{\bf q}'}^\ast)-\Lambda((\tau^{-1}{{\bf a}''}+\tau^{-1}{{\bf s}'})^\ast,({\bf p}''+{\bf t})^\ast).\end{flalign*}
Hence,  \begin{flalign*}x_1=&-\Lambda({^\ast{\bf m}}+{^\ast{\bf n}'},{\bf x}^\ast)+\lr{{\bf m}-{\bf x}',{\bf n}'}
-\Lambda({^\ast{\bf m}},{{\bf q}'}^\ast)-\Lambda({^\ast{\bf n}'},{{\bf q}'}^\ast)-\Lambda({{\bf i}''}^\ast,{^\ast{\bf d}'})+\Lambda({{\bf i}''}^\ast,{\bf p}^\ast)\\
&-\Lambda({{\bf r}'}^\ast,{^\ast{\bf d}'})+\Lambda({{\bf r}'}^\ast,{{\bf p}}^\ast)+\Lambda({{\bf p}}^\ast,{^\ast{\bf d}'})+\Lambda({^\ast{\bf d}'},{{\bf q}'}^\ast)+\Lambda({^\ast{\bf a}''}+{^\ast{\bf s}'},{{\bf p}''}^\ast+{\bf t}^\ast).
\end{flalign*}
Since ${\bf d}'-{\bf n}'={\bf i}-\tau({\bf m}-{\bf a})$, we have ${^\ast({\bf d}'-{\bf n}')}=({\bf p}+{\bf m}-{\bf a})^\ast$, and then
\begin{flalign*}
-\Lambda({^\ast{\bf m}},{{\bf q}'}^\ast)-\Lambda({^\ast{\bf n}'},{{\bf q}'}^\ast)+\Lambda({^\ast{\bf d}'},{{\bf q}'}^\ast)=\Lambda({\bf p}^\ast-{\bf a}^\ast,{{\bf q}'}^\ast)+\lr{{\bf q}',{\bf m}}.
\end{flalign*}
Noting that ${\bf x}^\ast=-{^\ast{\bf x}'}$, ${^\ast{\bf a}''}+{^\ast{\bf s}'}={^\ast{\bf x}'}$ and ${{\bf p}''}^\ast+{\bf t}^\ast={^\ast{\bf i}''}+{^\ast{\bf r}'}={^\ast{\bf j}''}$, we obtain
\begin{flalign*}x_1=&\Lambda({^\ast{\bf m}}+{^\ast{\bf n}'},{^\ast{\bf x}'})+\lr{{\bf m}-{\bf x}',{\bf n}'}+\Lambda({\bf p}^\ast-{\bf a}^\ast,{{\bf q}'}^\ast)+\lr{{\bf q}',{\bf m}}\\
&-\Lambda({{\bf j}''}^\ast,{^\ast{\bf d}'})+\Lambda({{\bf j}''}^\ast,{{\bf p}^\ast})+\Lambda({^\ast{\bf i}},{^\ast{\bf d}'})+\Lambda({^\ast{\bf x}'},{^\ast{\bf j}''}).
\end{flalign*}
By ${\bf x}'={\bf a}-{\bf j}''$, we get
\begin{flalign*}x_1=&\Lambda({{\bf m}^\ast}+{{\bf n}'}^\ast,{\bf a}^\ast)-\Lambda({{\bf m}^\ast}+{{\bf n}'}^\ast,{{\bf j}''}^\ast)+\lr{{\bf m}-{\bf a},{\bf n}'}+\lr{{{\bf j}''},{\bf n}'}+\Lambda({^\ast{\bf i}},{^\ast{\bf j}'})\\
&-\Lambda({\bf a}^\ast,{^\ast{\bf j}'})+\lr{{\bf q}',{\bf m}}-\Lambda({{\bf j}''}^\ast,{^\ast{\bf d}'})+\Lambda({{\bf j}''}^\ast,{^\ast{\bf i}})+\Lambda({^\ast{\bf i}},{^\ast{\bf d}'})+\Lambda({\bf a}^\ast,{{\bf j}''}^\ast).
\end{flalign*}
Using ${^\ast({\bf i}-{\bf d}')}=-({\bf m}-{\bf a})^\ast-{^\ast{\bf n}'}$, we get
\begin{flalign*}-\Lambda({{\bf m}^\ast}+{{\bf n}'}^\ast,{{\bf j}''}^\ast)-\Lambda({{\bf j}''}^\ast,{^\ast{\bf d}'})+\Lambda({{\bf j}''}^\ast,{^\ast{\bf i}})+\Lambda({\bf a}^\ast,{{\bf j}''}^\ast)=-\lr{{{\bf j}''},{\bf n}'}.\end{flalign*} Set $D=D'\oplus J'$, then we obtain
$x_1=\Lambda(({\bf m}+{\bf n})^*,{\bf a}^*)+\lr{{\bf m}-{\bf a},{\bf n}}+\Lambda({\bf i}^\ast,{\bf d}^\ast)$.

On the other hand,
\begin{equation}\label{ddstt2}
\begin{split}
&\sum\limits_{\begin{smallmatrix}[D],[A],[I]\end{smallmatrix}}v^{\Lambda(({\bf m}+{\bf n})^*,{\bf a}^*)+\lr{{\bf m}-{\bf a},{\bf n}}}|_D\Hom_{\widetilde{\A}}(N,\tau M)_{\tau A\oplus I}|\theta'(u_{A})\star \theta'(u_{D\oplus I[-1]})=\\
&\sum\limits_{\begin{smallmatrix}[D'],[A'],[I],[I'']\end{smallmatrix}}v^{y_1}|_D\Hom_{\widetilde{\A}}(N,\tau M)_{\tau A\oplus I}|u_{\tau^{-1} A''\oplus \tau^{-1}S'\oplus P''[1]\oplus T[1]}\star u_{P}\star u_{\tau^{-1} D'\oplus Q'[1]},
\end{split}\end{equation}
where $y_1=\Lambda(({\bf m}+{\bf n})^*,{\bf a}^*)+\lr{{\bf m}-{\bf a},{\bf n}}+\Lambda({\bf i}^\ast,{\bf d}^\ast)$, each $A$ has the same maximal projective direct summand $P'$ as $M$, $D=D'\oplus J'$ has the same maximal injective direct summand $J'$ as $N$, write $A=A'\oplus P'$, $A'=A''\oplus I''$ and $P'=S'\oplus R'$ such that $I''$ and $R'$ are the maximal injective direct summands of $A'$ and $P'$, respectively, $P''=\nu^{-1}(I'')$, $T=\nu^{-1}(R')$ and $P=\nu^{-1}(I)$.

Since $\Hom_{\widetilde{\A}}(J',\tau M)=0$, we have
\begin{flalign*}
|_D\Hom_{\widetilde{\A}}(N,\tau M)_{\tau A\oplus I}|=|_{D'}\Hom_{\widetilde{\A}}(N',\tau M)_{\tau A\oplus I }|.
\end{flalign*}
Hence, the equations \eqref{dstt2} and \eqref{ddstt2} are equal, i.e. $\theta'(r_{M,N})=0$.

Therefore, the algebra homomorphism $\theta'$ induces the algebra homomorphism $\sigma'$.
\end{proof}

\end{document}